\documentclass[a4paper,10pt]{article}

 \usepackage{latexstyle}
 \usepackage{bbm}
 
\usepackage{mathtools}
 
\usepackage{amsmath}
\makeatletter
\renewcommand*\env@matrix[1][*\c@MaxMatrixCols c]{%
  \hskip -\arraycolsep
  \let\@ifnextchar\new@ifnextchar
  \array{#1}}
\makeatother

\title{On the role of total variation in compressed sensing \\- structure dependence}
\author{Clarice Poon\thanks{cmhsp2@cam.ac.uk}\\ Department of Applied Mathematics and Theoretical Physics,\\ University of Cambridge}

\begin{document}

\maketitle

\begin{abstract}
This paper considers the use of total variation regularization in the recovery of approximately gradient sparse signals from their noisy discrete Fourier samples in the context of compressed sensing. It has been observed over the last decade that a reconstruction which is robust to noise and stable to inexact sparsity can be achieved when we observe a highly incomplete subset of the Fourier samples for which the samples have been drawn in a random manner. Furthermore, in order to minimize the cardinality of the set of Fourier samples, the sampling set needs to be drawn in a non-uniform manner and the use of randomness is far more complex than the notion of uniform random sampling often considered in the theoretical results of compressed sensing. The purpose of this paper is to derive recovery guarantees in the case where the sampling set is drawn in a non-uniform random manner. We will show how the sampling set is dependent on the sparsity structure of the underlying signal.
\end{abstract}

\section{Introduction}

In \cite{candes2006robust}, Cand\`{e}s, Romberg and Tao presented a numerical experiment which demonstrated that the Logan-Shepp phantom can be exactly recovered from its partial Fourier coefficients. Specifically, exact recovery is achieved by solving  the following minimization problem, where $\nm{\cdot}_{TV}$ is the isotropic total variation norm, $x$ is taken to be the Logan-Shepp phantom at resolution $512\times 512$ and $\rP_\Omega\rA$ is the discrete Fourier transform restricted to the index set $\Omega$  which is taken to be 22  radial lines on a uniform grid as shown in Figure \ref{fig:phantom}. The reader is referred to Section \ref{sec:notation} for precise definitions of these notations.

\be{ \label{eq:tv_orig}
\min_{z \in \bbC^{N\times N}} \nm{z}_{TV} \text{ subject to } \rP_\Omega\rA z =  \rP_\Omega\rA x.
}

\begin{figure}[ht] 
\centering
$\begin{array}{cc}
\includegraphics[ trim=0.8cm 0.8cm 1cm 0.2cm,clip=true,width=0.33\textwidth]{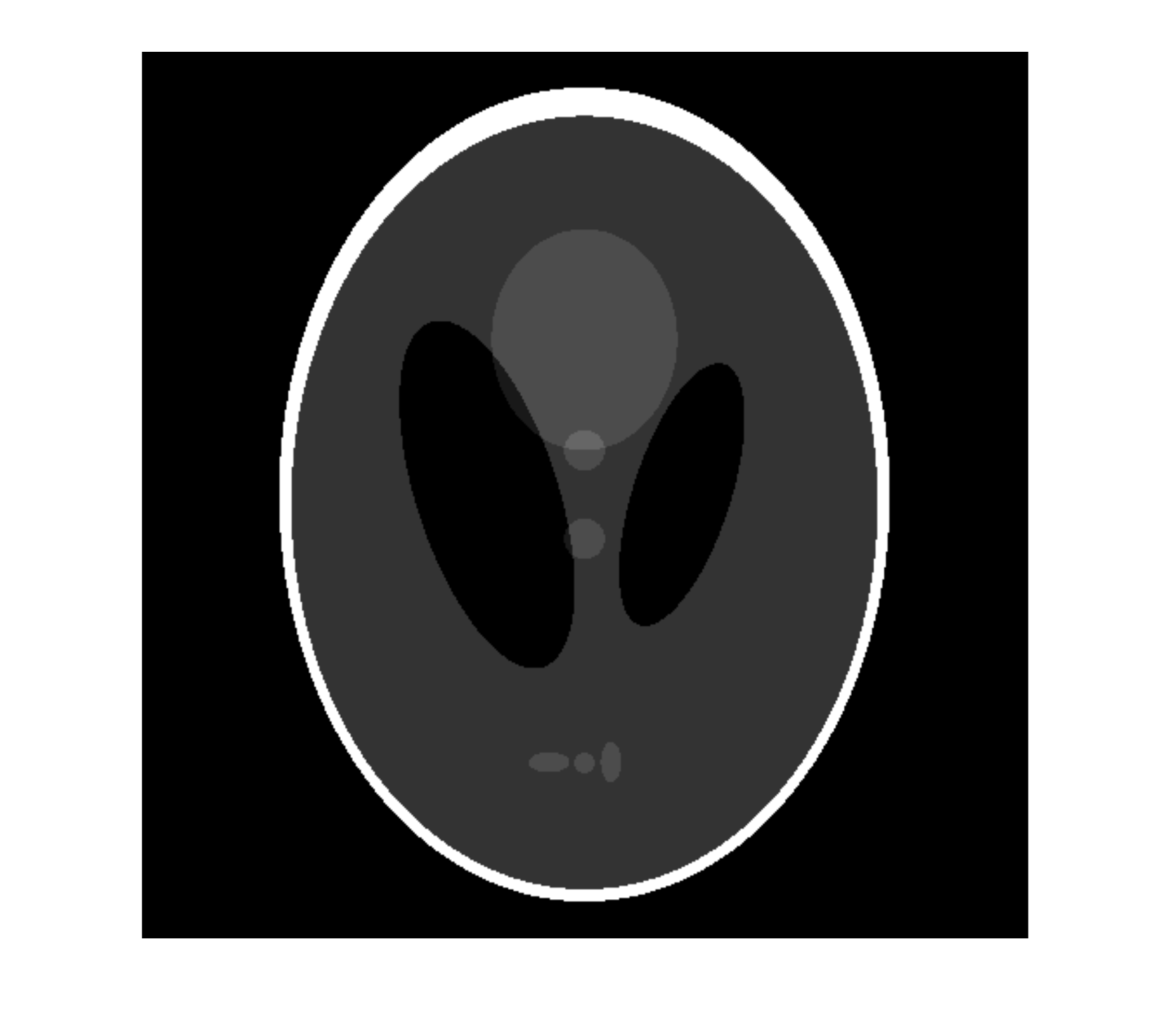} &
\includegraphics[ trim=0.8cm 0.8cm 1cm 0.2cm,clip=true,width=0.33\textwidth]{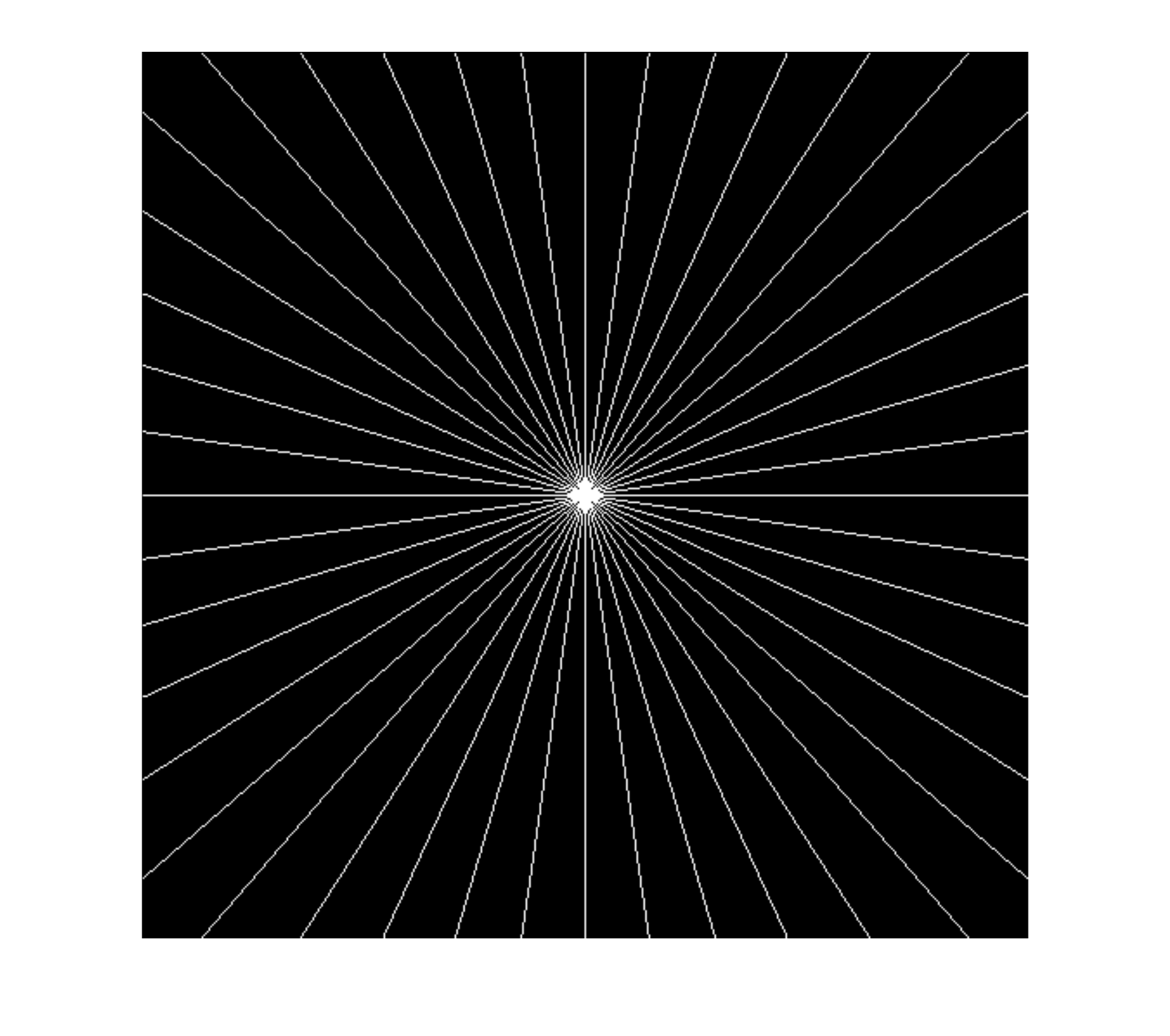} 
\end{array}$
\caption{
 The  Logan Shepp Phantom (left), which can be exactly recovered from 22 radial lines of its Fourier coefficients (right).
\label{fig:phantom}}
\end{figure}

To theoretically justify this experiment, \cite{candes2006robust} proved that given any signal $x\in\bbC^{N\times N}$ which is $s$-sparse in its discrete gradient, with probability exceeding $1-\epsilon$, $x$ can be exactly recovered from its discrete Fourier data supported on an index set $\Omega = \Omega'\cup \br{0}$  where $\Omega'$ is of cardinality $\ord{s\cdot(\log N +\log(\epsilon^{-1})}$ chosen uniformly at random by solving (\ref{eq:tv_orig}). Furthermore, it can be proved (see \cite{tv1}) that solutions of   
\be{\label{eq:noise_tv}
\min_{z\in \bbC^{N\times N}}  \nm{z}_{TV} \text{ subject to } \nm{\rP_\Omega \rA - \rP_\Omega \rA x}_2 \leq \delta
} are robust to noise $\delta >0$ and stable to inexact sparsity for this uniform random choice of the sampling set $\Omega$.

Due to the close links of the Logan Shepp phantom experiment with practical applications, such as computed tomography imaging, the work \cite{candes2006robust}  has motivated much of the research in compressed sensing over the last decade. In particular, there have been many studies on how to optimally choosing the sampling set $\Omega$ when solving (\ref{eq:noise_tv}) \cite{Lustig,Lustig2,puy2011variable,benning2014phase} and one common quality of the sampling patterns proposed is that they are not uniform random patterns. They all sample densely at low Fourier frequencies, and less densely at higher Fourier frequencies. This is evident even in the radial line sampling pattern of Figure \ref{fig:phantom}.
Furthermore, despite the theoretical results of \cite{candes2006robust} and \cite{tv1},  letting the sampling set be $\Omega = \Omega' \cup \br{0}$ with $\Omega'$ chosen uniformly at random generally result in inferior reconstructions and is not used in practice. As an example, consider the reconstruction of the image shown on the left of Figure \ref{fig:boat_test} from $5.22\%$ of its discrete Fourier coefficients supported on $\Omega_V$ and $\Omega_U$ shown in Figure \ref{fig:maps_boat}. $\Omega_U$ is chosen uniformly at random, whilst $\Omega_V$ is chosen randomly with higher concentration at low Fourier frequencies. The reconstructions and their relative errors are shown in Figure \ref{fig:boat_test}, where the relative error of a reconstruction $R$ is $\epsilon_{rel} =\norm{R-I}_2/\norm{I}_2$ with $I$ denoting the original image.
\begin{figure}[ht] 
\centering
$\begin{array}{cc}
\Omega_{V} &  \Omega_{U}\\
\includegraphics[ width=0.31\textwidth]{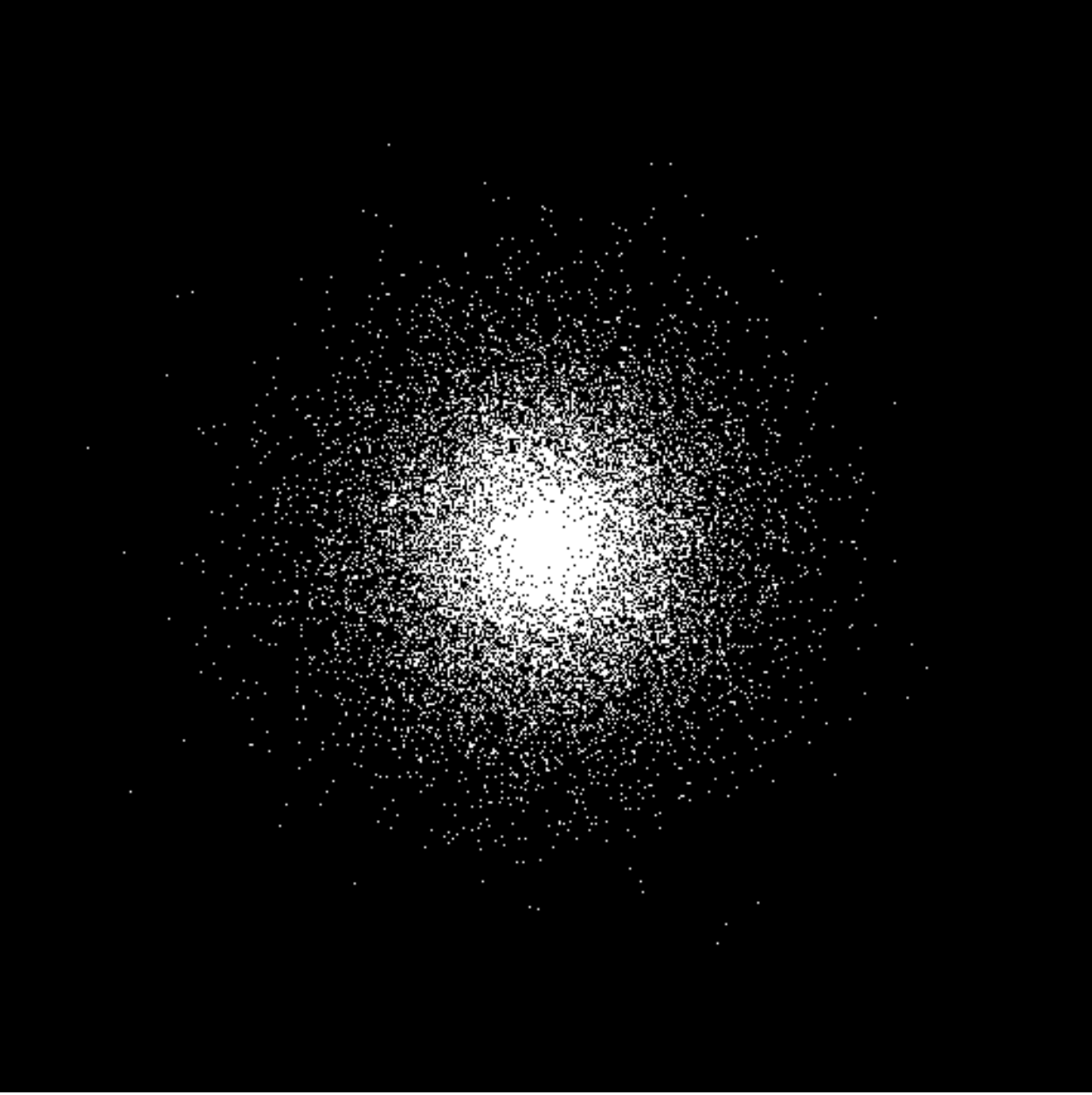} &
\includegraphics[ width=0.31\textwidth]{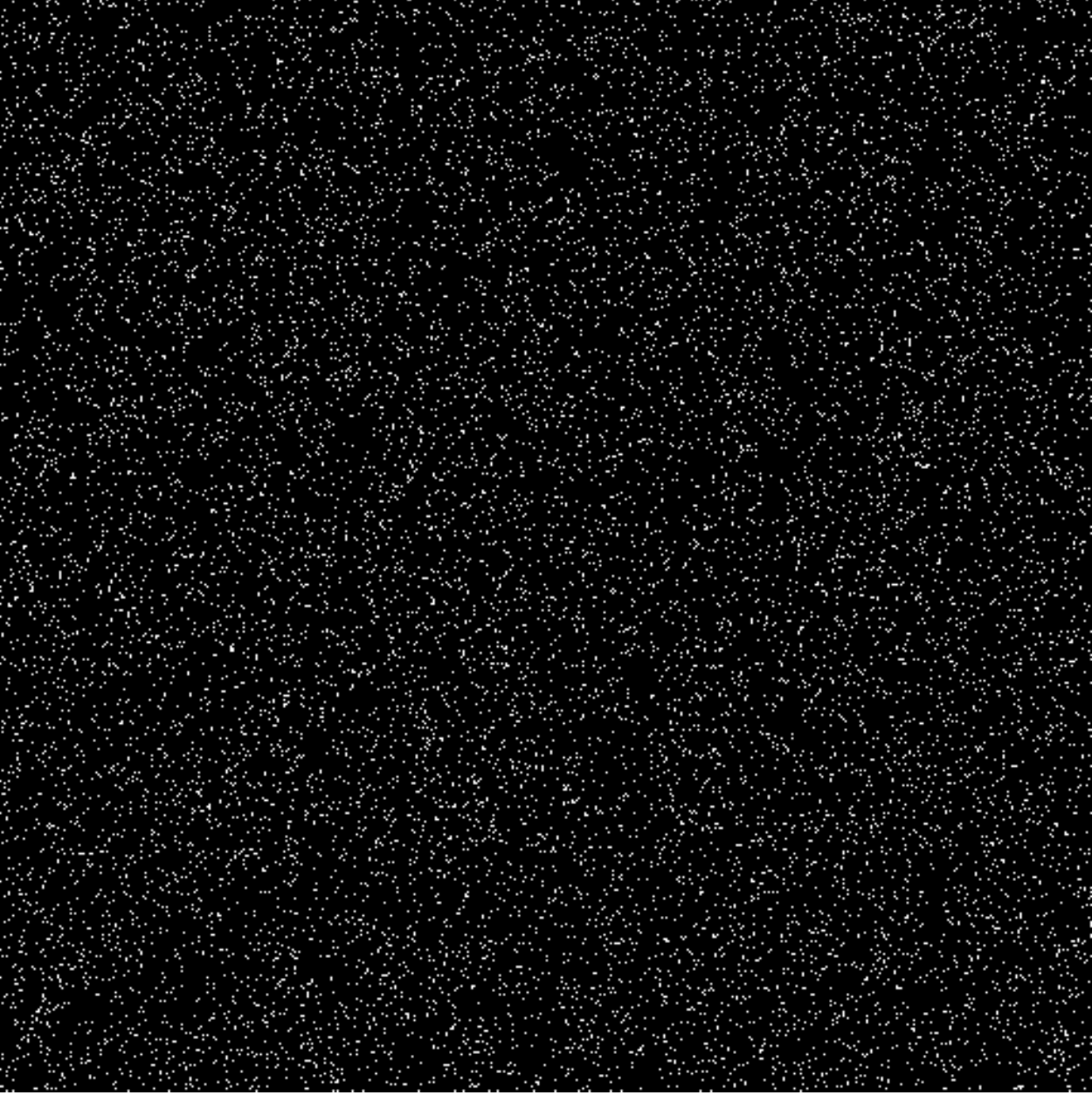}
\end{array}$
\caption{the Fourier sampling maps, each indexing 5.22\% of the available Fourier samples. 
\label{fig:maps_boat}}
\end{figure}

\begin{figure}[ht] 
\centering
$\begin{array}{ccc}
\text{Original} &\text{Reconstruction from }\Omega_{V} &\text{Reconstruction from } \Omega_{U}\\
&
\epsilon_{rel} = 8.04\% &\epsilon_{rel} = 33.25\%\\
\includegraphics[ width=0.31\textwidth]{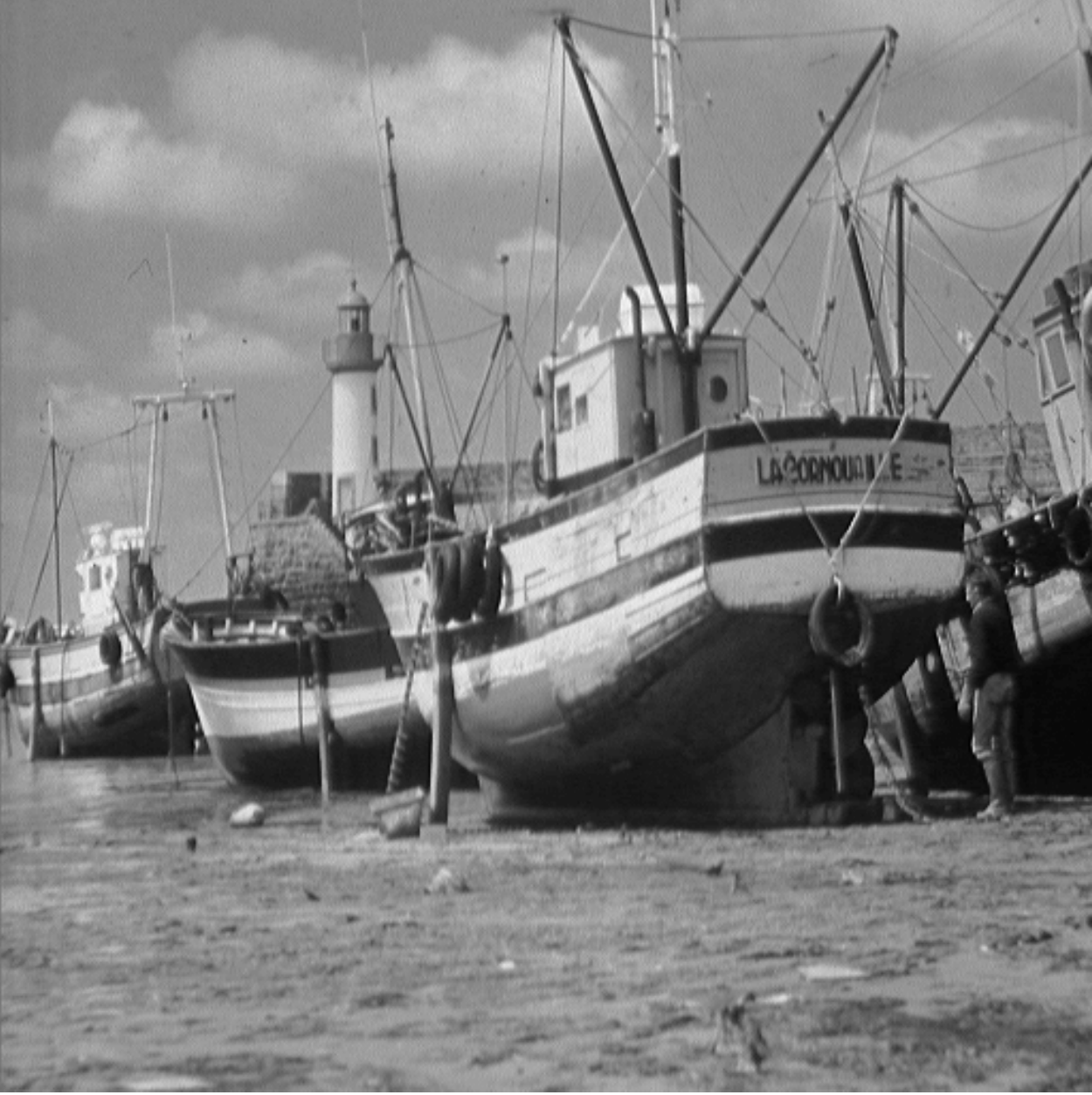}  &
\includegraphics[ width=0.31\textwidth]{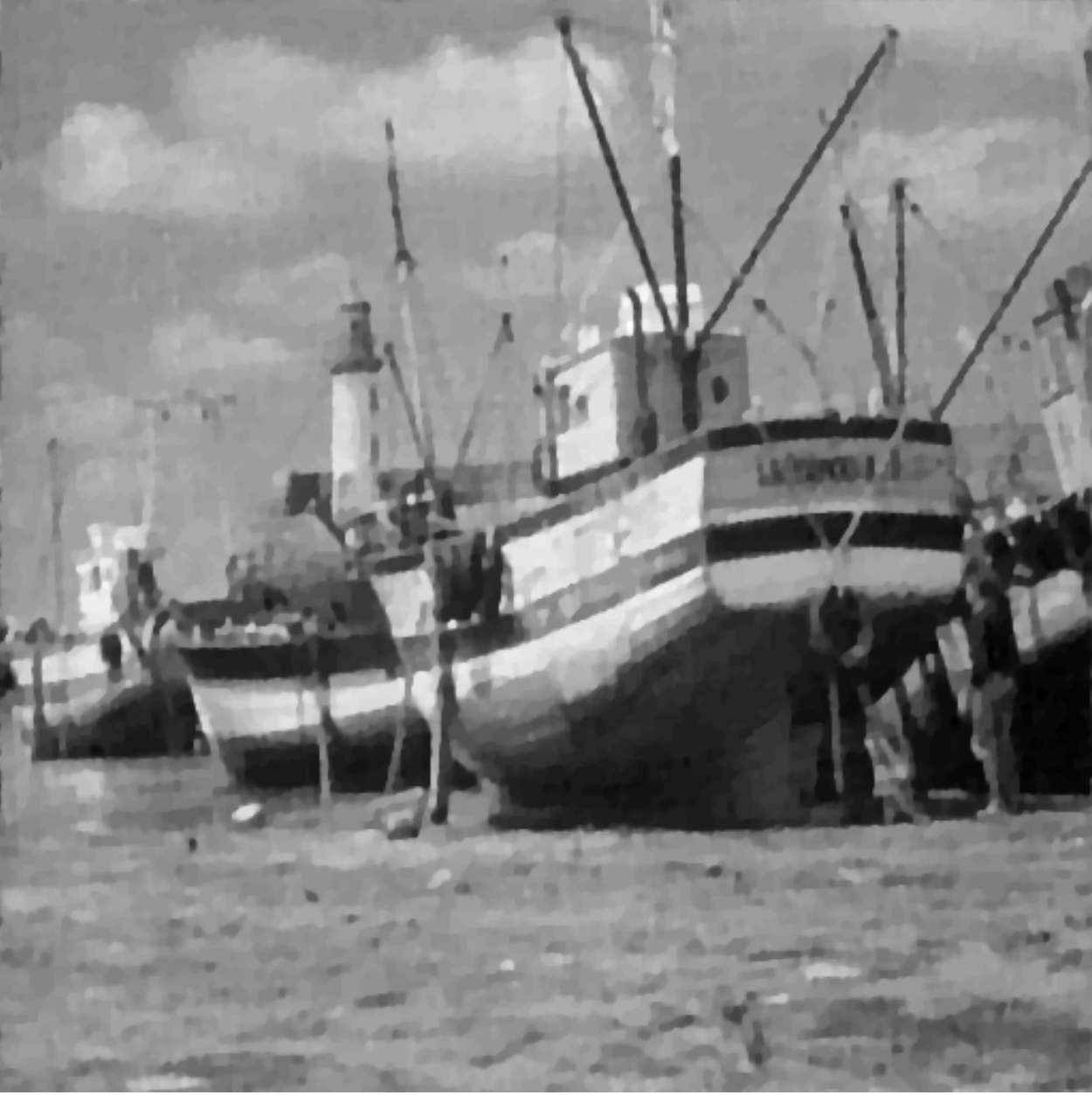} &
\includegraphics[ width=0.31\textwidth]{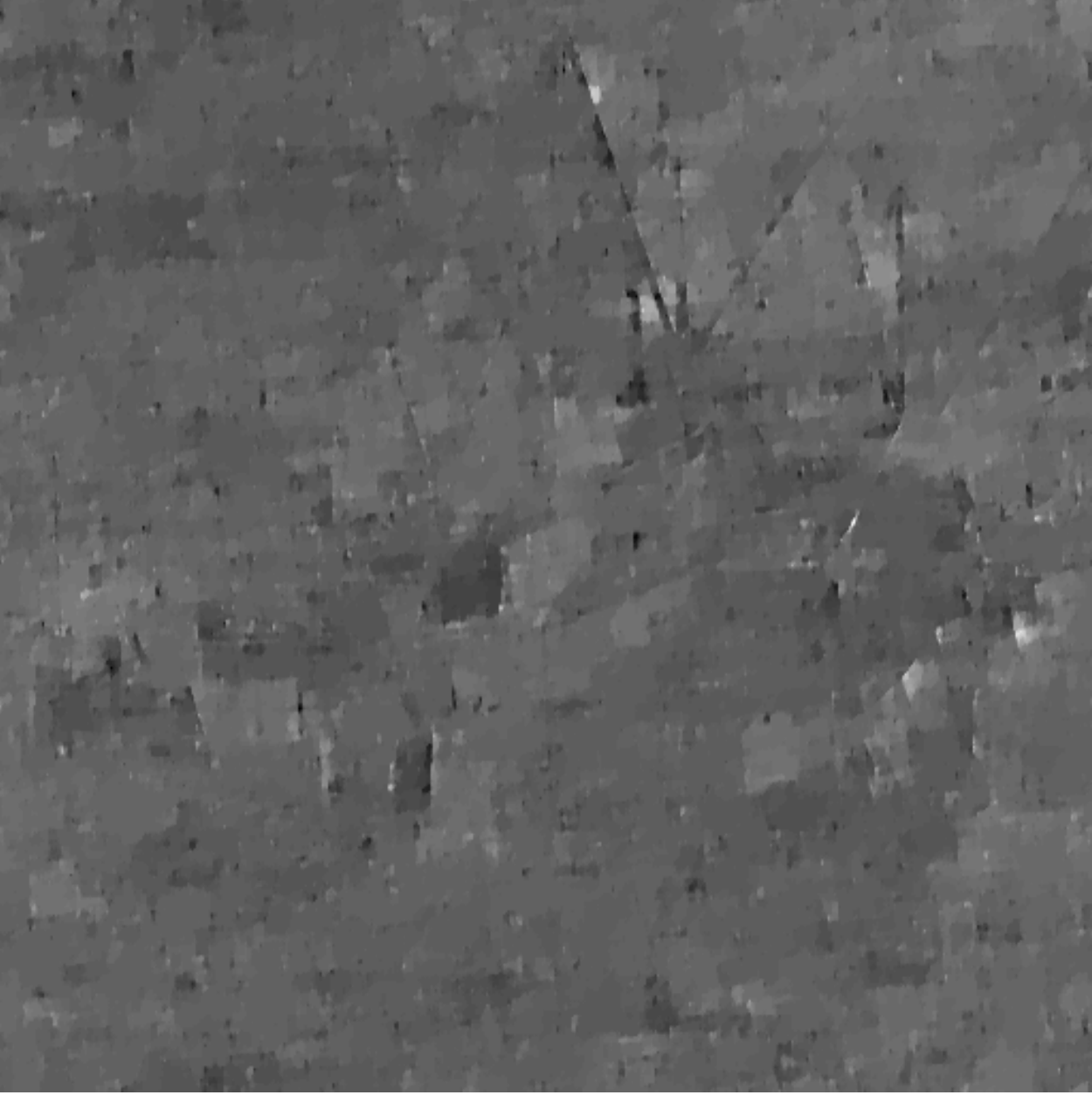}
\end{array}$
\caption{
The original image of size $512\times 512$  taken from the USC-SIPI image database, the reconstructed images and their relative errors.
\label{fig:boat_test}}
\end{figure}

However, to date, most of the theoretical results of compressed sensing  consider only a uniform random choice of the sampling set $\Omega$.
The purpose of this paper is to derive theoretical statement on how a non-uniform choice of the sampling  set $\Omega$ will impact how solutions of the following minimization problem approximates  $x\in \bbC^N $ (or $\bbC^{N\times N}$).
\bes{
\min_{z\in\bbC^N (\text{or } \bbC^{N\times N})}  \nm{z}_{TV} \text{ subject to } \nm{\rP_\Omega \rA - \rP_\Omega \rA x}_2 \leq \delta
}
where $\nm{\cdot}_{TV}$ is either the one dimensional total variation norm or the two dimensional isotropic total variation norm with Neumann boundary conditions and $\rA$ is the unitary discrete Fourier transform on $\bbC^N$ or $\bbC^{N\times N}$ and $\delta\geq 0$. We first demonstrate the choice of $\Omega$ is not dependent on sparsity alone.

\subsection{Sparsity is insufficient in itself}\label{subsec:sparsity_insuff}
Theoretical results in compressed sensing which concern recovery statements based on a uniform random choice of the sampling set often account only for the sparsity of the underlying signal.
However, to understand how to choose the samples $\Omega$ optimally, one must consider more than simply sparsity.

Consider the recovery of two signals, as shown in Figure \ref{fig:1d_recons}. Signal 1 and signal 2 are both vectors of length $N=256$, with $16$ non-zero entries in their total variation coefficients. For each signal, say $x$, we now consider the reconstruction obtained by solving \footnote{The numerical algorithm used was the split Bregman method described in \cite{goldstein2009split}}
\bes{
\min_{z\in\bbC^N} \nm{z}_{TV} \text{ subject to } \rP_{\Omega}\rA x =\rP_\Omega \rA z
}
where $\nm{\cdot}_{TV}$ is now the one-dimensional total variation norm, $\rA$ is the one-dimensional discrete Fourier transform  and $\Omega\subset \br{-N/2+1,\ldots, N/2}$ indexes the $32$ Fourier coefficients of lowest frequencies plus $10\%$  of the remaining Fourier coefficients, drawn uniformly at random. The same sampling pattern $\Omega$ is used in the reconstruction of  signal 1 and signal 2. Although the sampling pattern and the sparsity in the total variation of signals are identical, the reconstruction for signal 1 is exact and the reconstruction for signal 2 has a relative error of $78.34\%$.

This effect is also visible in the recovery of signals via total variation regularisation in two-dimensions. Consider image 1 and image 2 as displayed in Figure \ref{fig:lines_recons_orig}. Both images are of dimension $128\times 128$ and have exactly $489$ nonzero entries in their total variation coefficients. Suppose that we are given the Fourier coefficients indexed by the sampling map A shown in Figure \ref{fig:lines_recons}, that is $16\times 16$ Fourier coefficients of the lowest frequencies, plus $5\%$ of the remaining Fourier coefficients. Then, as shown in Figure \ref{fig:lines_recons} image 1 can be recovered exactly by solving (\ref{eq:tv_orig}), whilst the solution of  (\ref{eq:tv_orig}) for image 2 yields a relative error of $53.7\%$. We again emphasize that the total variation sparsity and the given Fourier data are exactly the same for both images. Yet, there is a stark difference in the reconstruction quality. So, this suggests that the sampling strategy cannot depend on sparsity alone, but on some additional signal structure.

Suppose now that we are restricted to  between 5\% and 8\% of the available samples, but the samples can be  distributed arbitrarily. Is it possible to recover Image 2? 
Figure \ref{fig:lines_recons} shows the  reconstructions  obtained when restricted to the Fourier data specified by the  sampling patterns, B, C, D  and E,  shown in Figure \ref{fig:sampling_half}.
 The reconstruction of image 1 is exact, whilst the reconstruction of image 2 has a relative error of   $57.7\%$  and $70.8\%$ when reconstructed from the Fourier samples specified by maps B and C respectively. Furthermore, taking samples uniformly at random or with greater density at higher Fourier frequencies resulted in poor reconstructions for both images. 
This not only suggests that the  optimal choice of the sampling pattern   cannot be dependent on sparsity alone, but also that the amount of subsampling possible may also be dependent on some  additional signal structure.
So, a theory which assumes only sparsity cannot fully explain the sub-Nyquist phenomenon of compressed sensing with total variation in practice.

\begin{figure}[ht] 
\centering
$\begin{array}{cc}
\text{Signal 1} & \text{Reconstruction of signal 1}\\
\includegraphics[ trim=0.8cm 0cm 1cm 0.8cm,clip=true,width=0.45\textwidth]{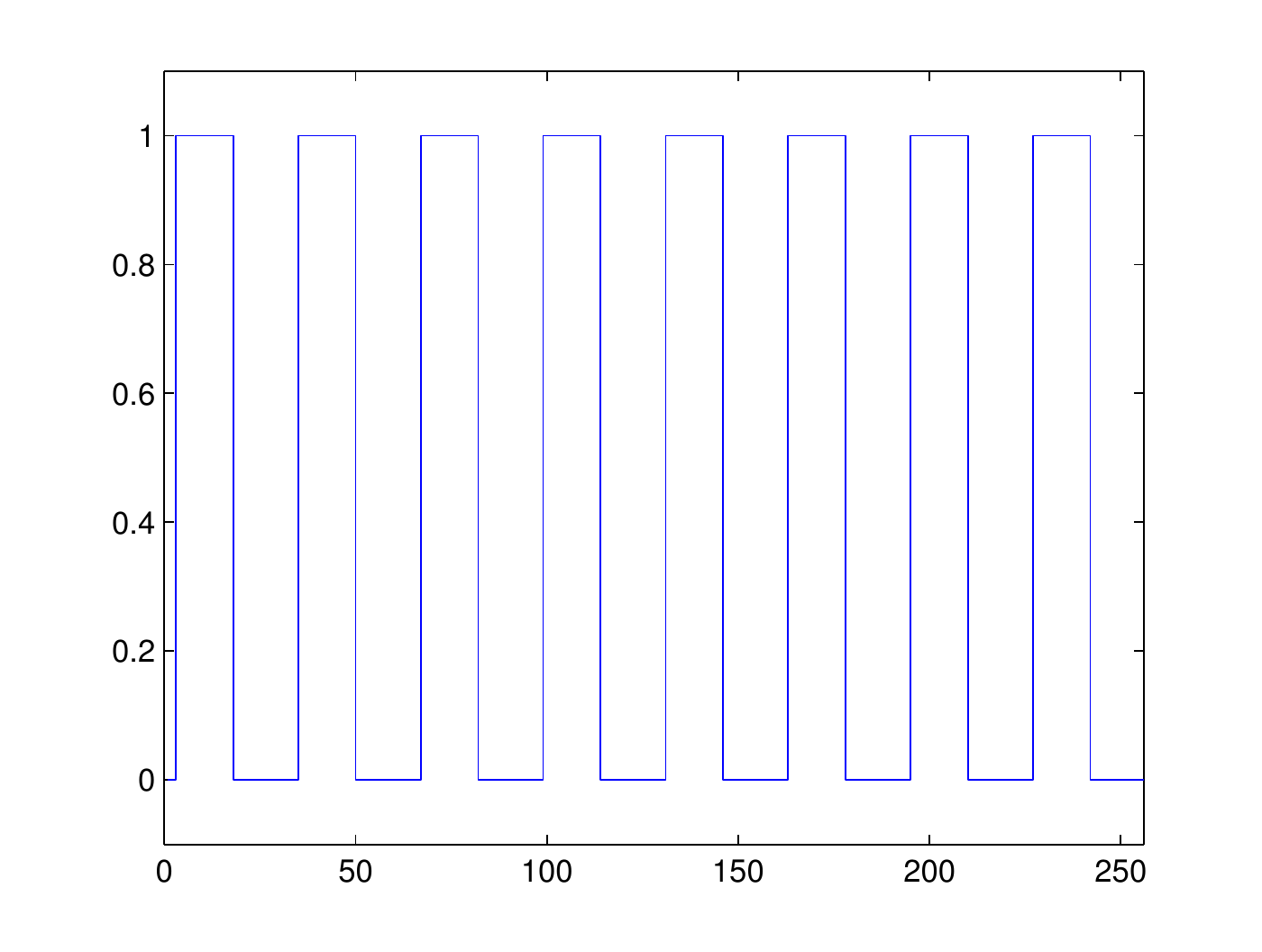} &
\includegraphics[ trim=0.8cm 0cm 1cm 0.8cm,clip=true,width=0.45\textwidth]{1d_orig_coarse.pdf} \\
\text{Signal 2} & \text{Reconstruction of signal 2}\\
\includegraphics[ trim=0.8cm 0.2cm 1cm 0.8cm,clip=true,width=0.45\textwidth]{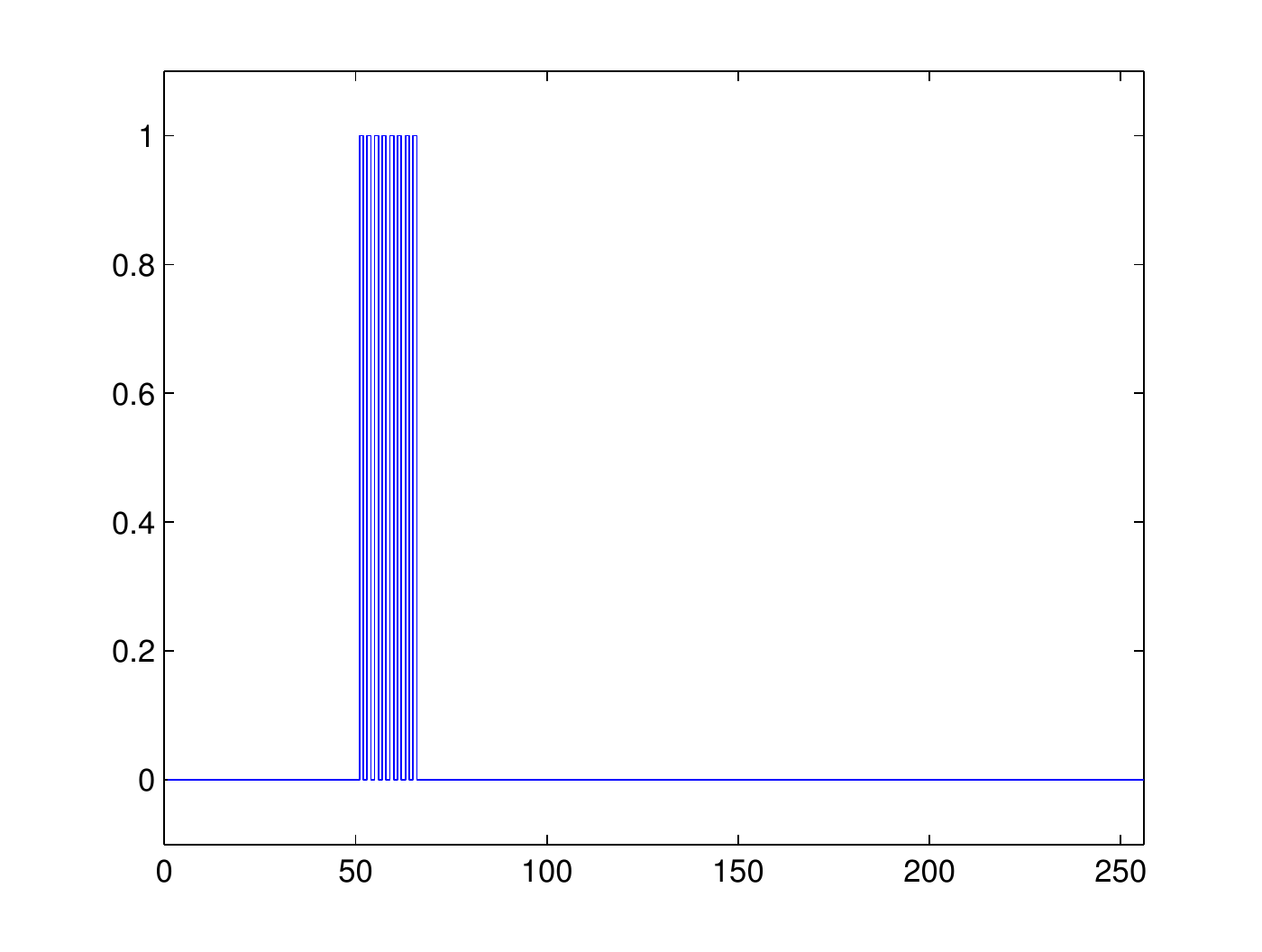} &
\includegraphics[ trim=0.8cm 0.2cm 1cm 0.8cm,clip=true,width=0.45\textwidth]{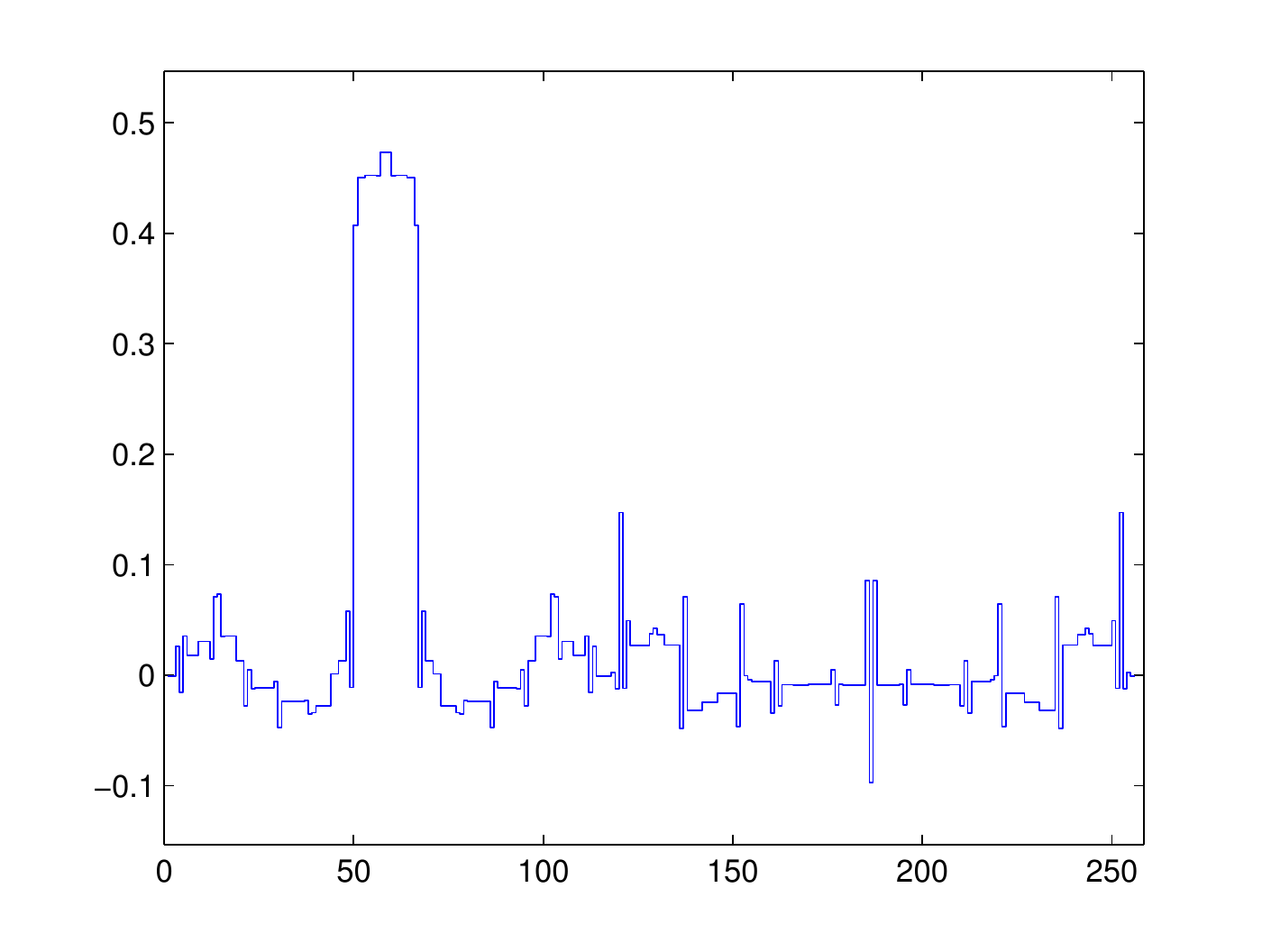} 
\end{array}$
\caption{
the Fourier sampling map associated with the reconstruction of these signals consists of the first 32 samples plus $10\%$ of the remaining samples chosen uniformly at random. Both signal 1 and signal 2 have the same sparsity in their total variation coefficients. However, signal 1 is exactly reconstructed whilst the reconstruction of signal 2 has a relative error of  78.34\%.
\label{fig:1d_recons}}
\end{figure}

\begin{figure}[htp!] 
\centering
$\begin{array}{cc}
 \includegraphics[ trim=0.8cm 0.7cm 1cm 0.8cm,clip=true,width=0.45\textwidth]{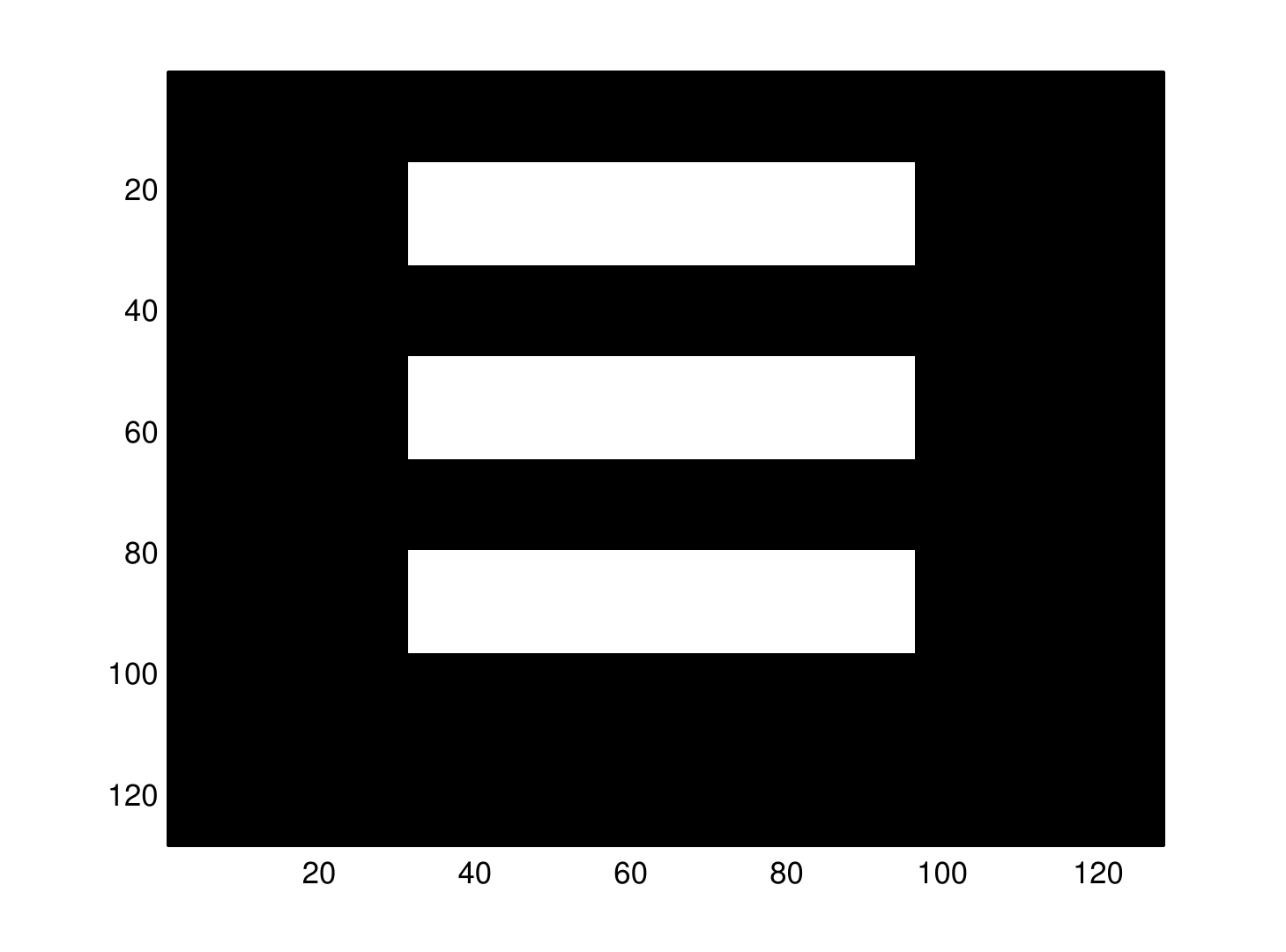} &
\includegraphics[ trim=0.8cm 0.7cm 1cm 0.8cm,clip=true,width=0.45\textwidth]{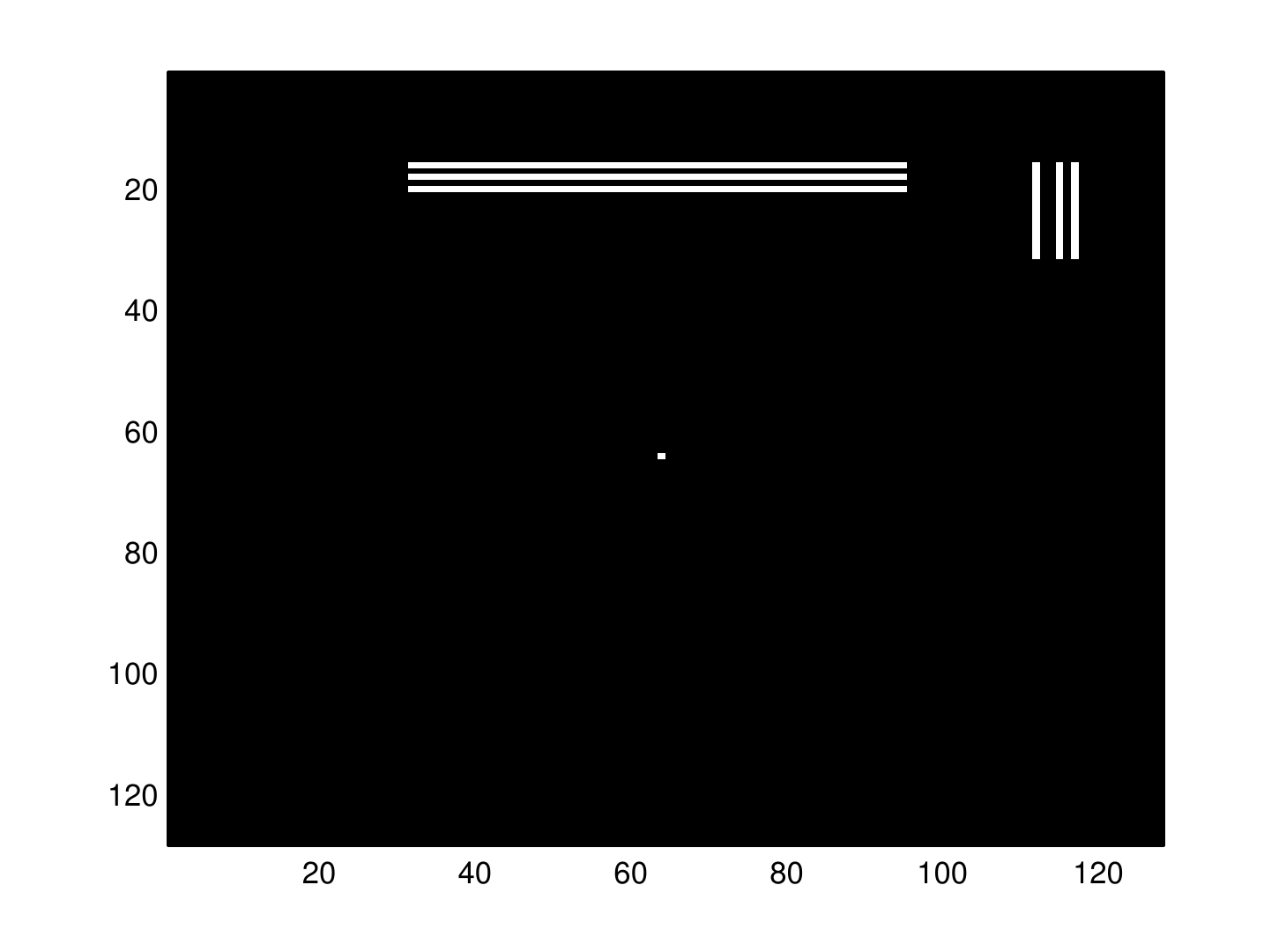}
\end{array}
$
\caption{
image 1 (left) and image 2 (right).
\label{fig:lines_recons_orig}}
\end{figure}

\begin{figure}[htp!] 
\centering
$\begin{array}{cc}
\text{ A, sampling 6.4 \%} &\text{ B, sampling 5.5\%}  \\
\includegraphics[ trim=1cm 0cm 1.3cm 0.8cm,clip=true,width=0.4\textwidth]{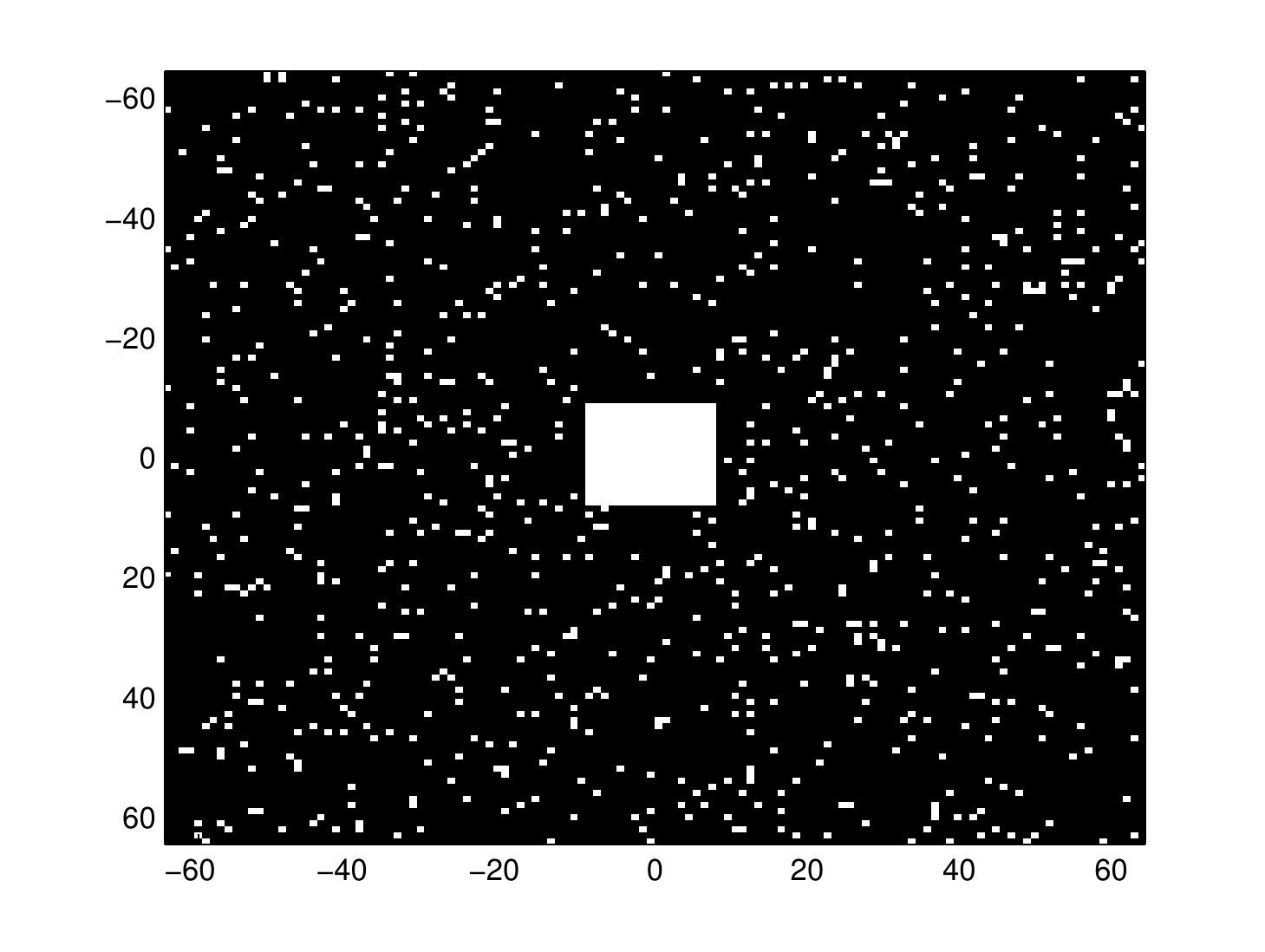} &
\includegraphics[ trim=1cm 0cm 1.3cm 0.8cm,clip=true,width=0.4\textwidth]{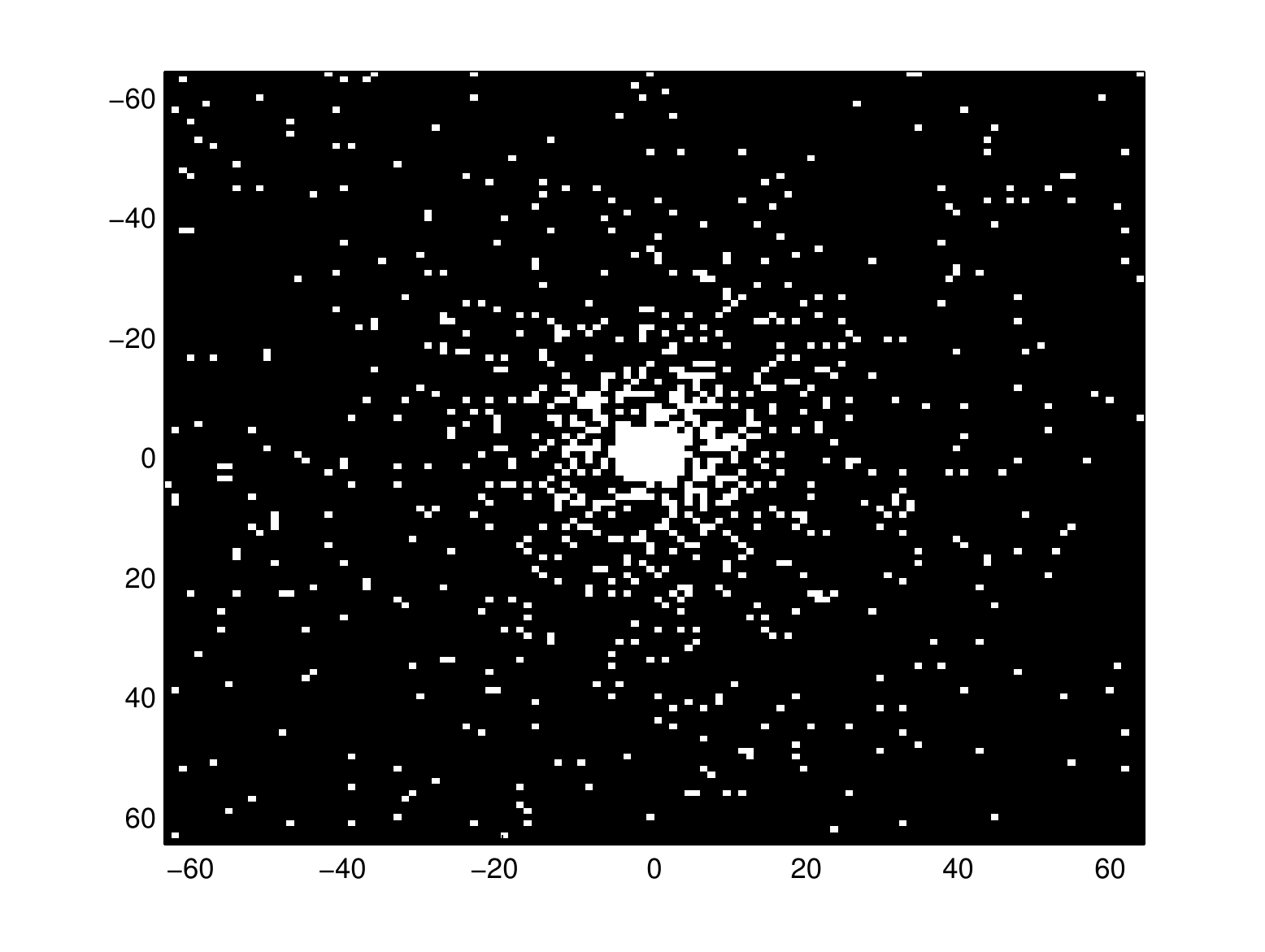}
\\
\text{ C, sampling 8.3\%} & \text{ D, sampling 5.7\%}\\
\includegraphics[ trim=1cm 0cm 1.3cm 0.8cm,clip=true,width=0.4\textwidth]{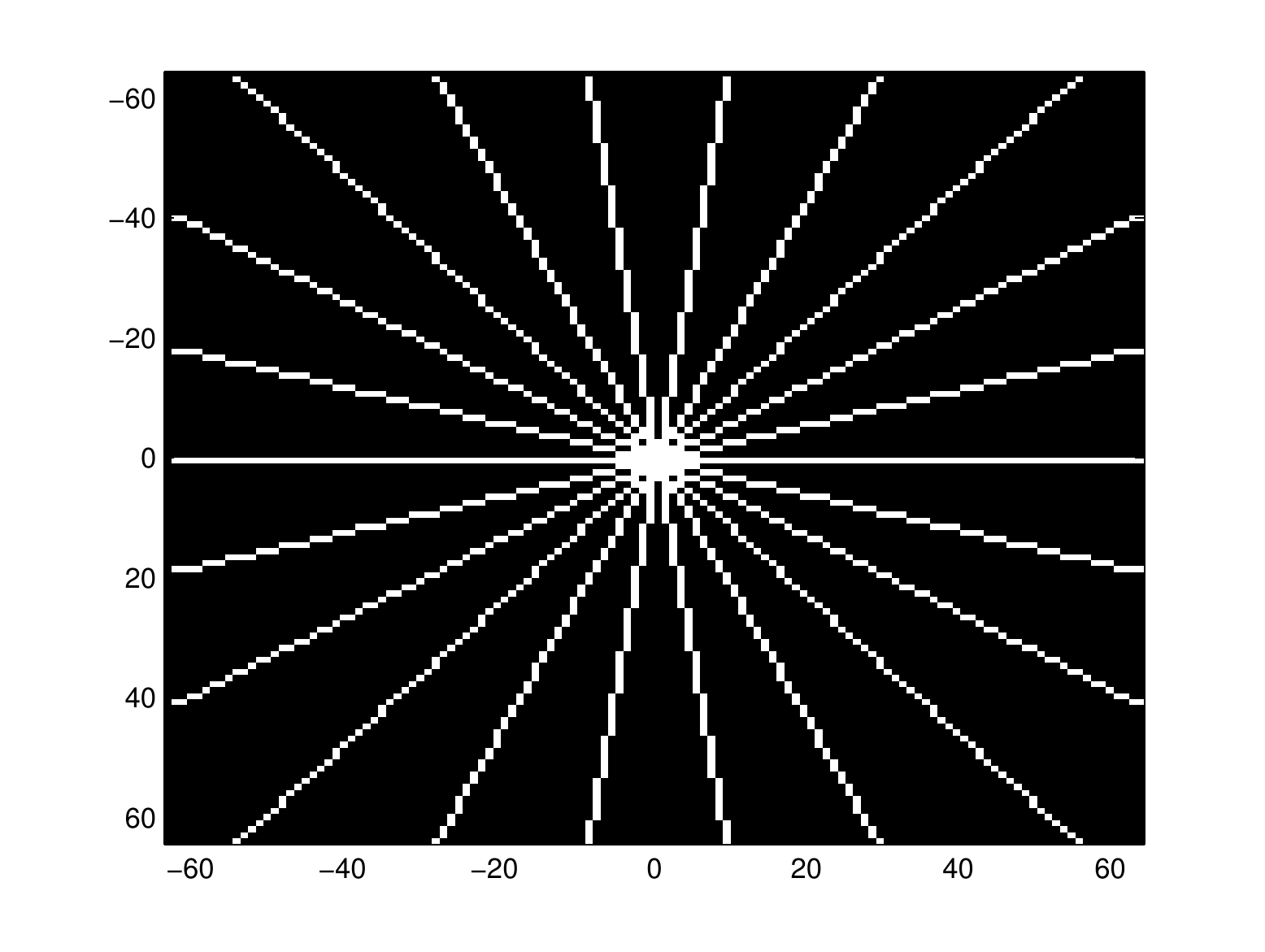} 
&
\includegraphics[ trim=1cm 0cm 1.3cm 0.8cm,clip=true,width=0.4\textwidth]{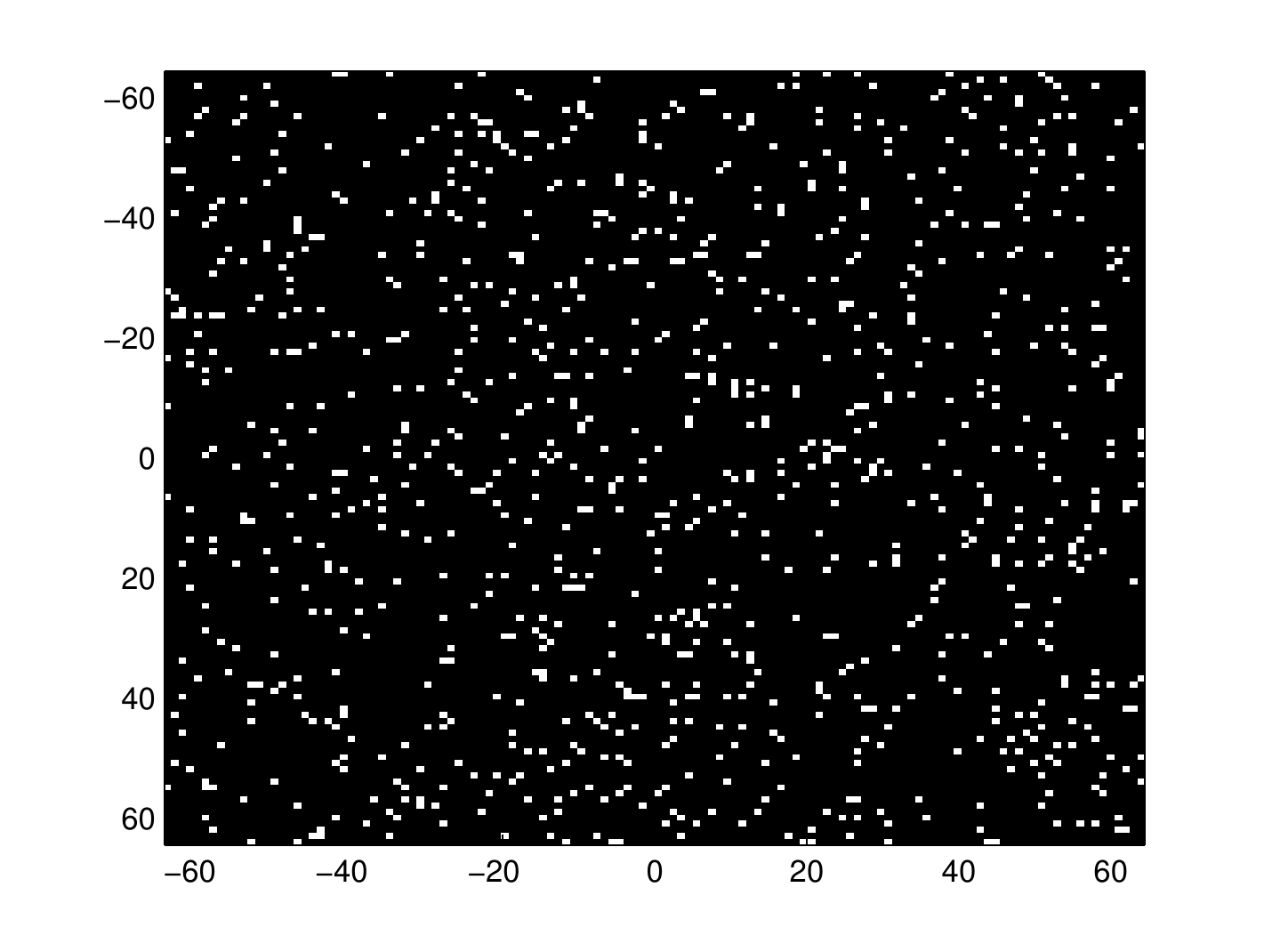} \\
\text{E, sampling 7.5\%}\\
\includegraphics[ trim=1cm 0cm 1.3cm 0.8cm,clip=true,width=0.4\textwidth]{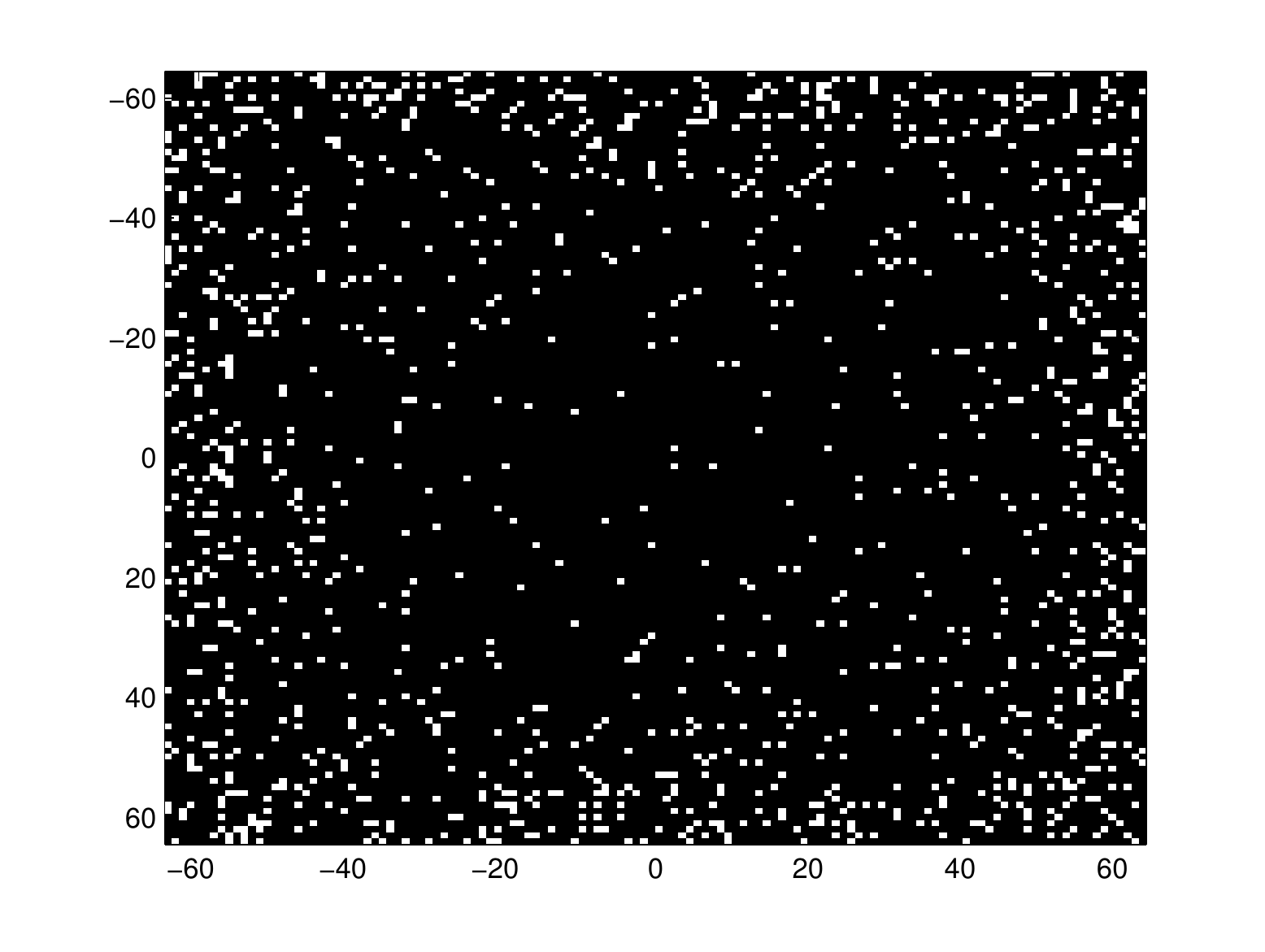} 
\end{array}$
\caption{ the Fourier sampling maps used in the reconstructions of image 1 and image 2. \label{fig:sampling_half}}
\end{figure}

\begin{figure}[htp!] 
\centering
$\begin{array}{cc}
\text{\footnotesize Reconstruction from map A, $\varepsilon_{rel} = 0\%$ } & \text{ \footnotesize  Reconstruction from map A, $\varepsilon_{rel} = 53.7\%$}\\ 
 \includegraphics[ trim=0.8cm 0cm 1cm 0.8cm,clip=true,width=0.32\textwidth]{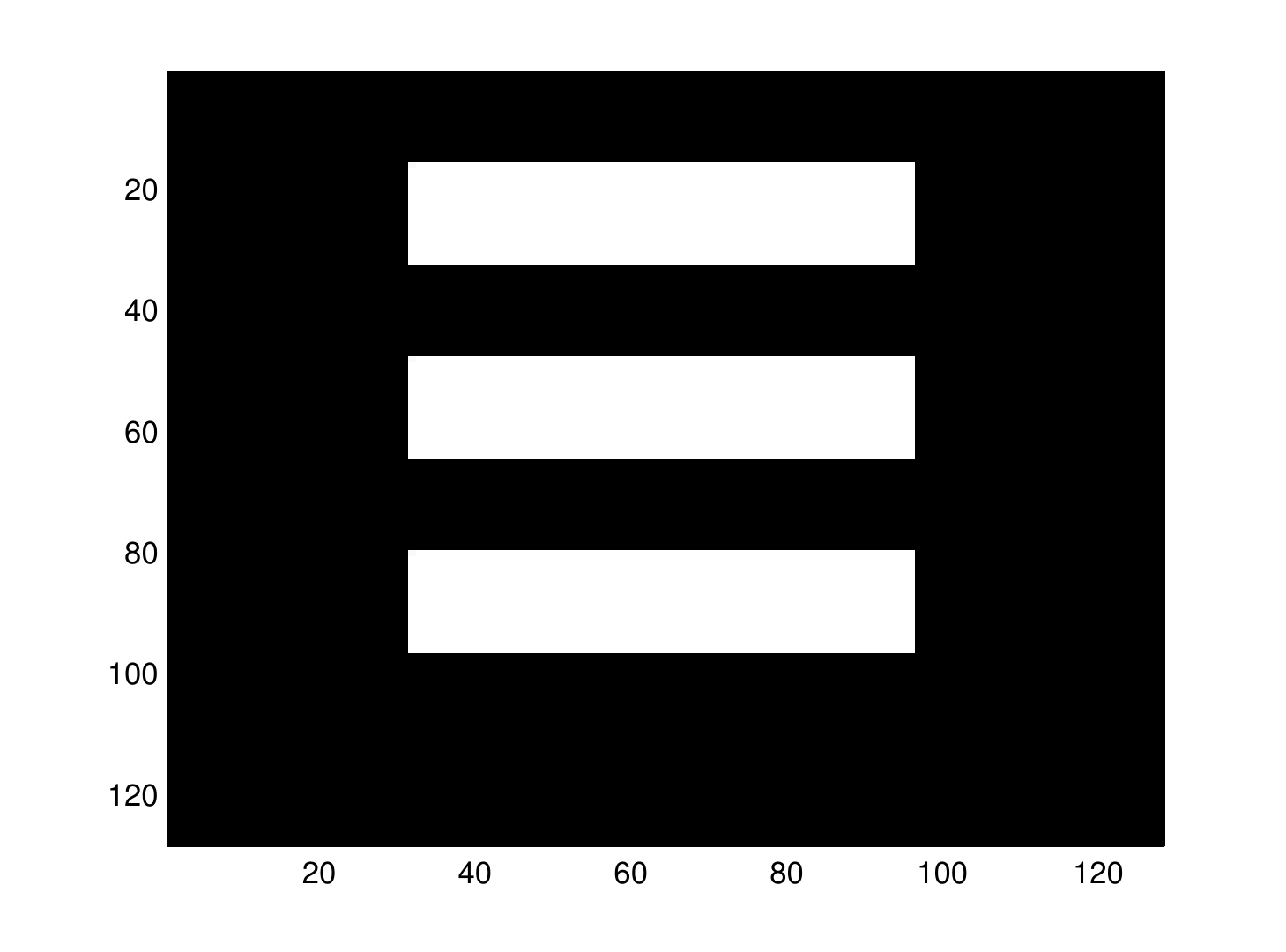} &
\includegraphics[ trim=0.8cm 0cm 1cm 0.8cm,clip=true,width=0.32\textwidth]{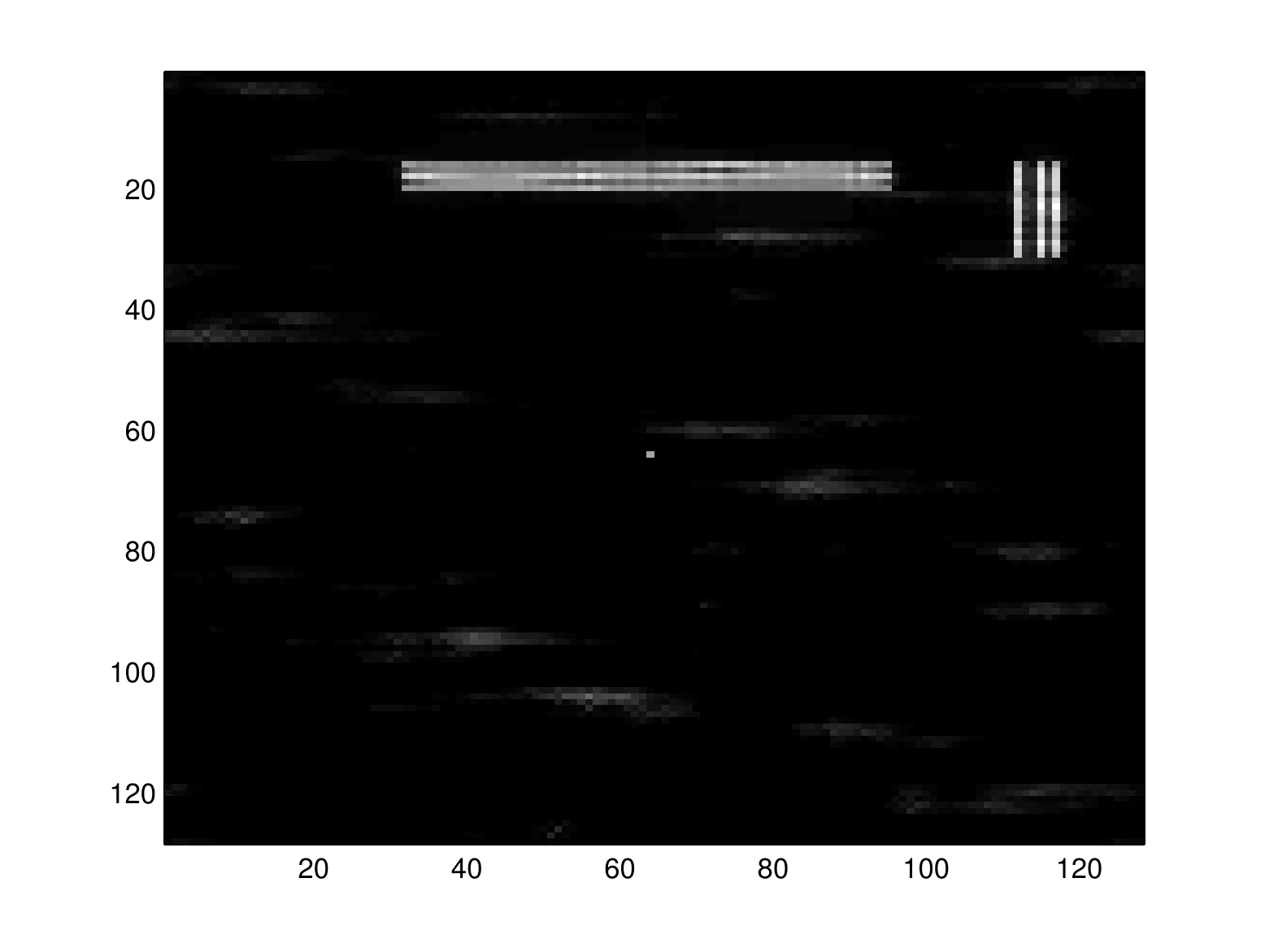} \\
\text{\footnotesize Reconstruction from map B, $\varepsilon_{rel} = 0\%$ } & \text{\footnotesize Reconstruction from map B, $\varepsilon_{rel} = 59.7\%$}\\
 \includegraphics[ trim=0.8cm 0cm 1cm 0.8cm,clip=true,width=0.32\textwidth]{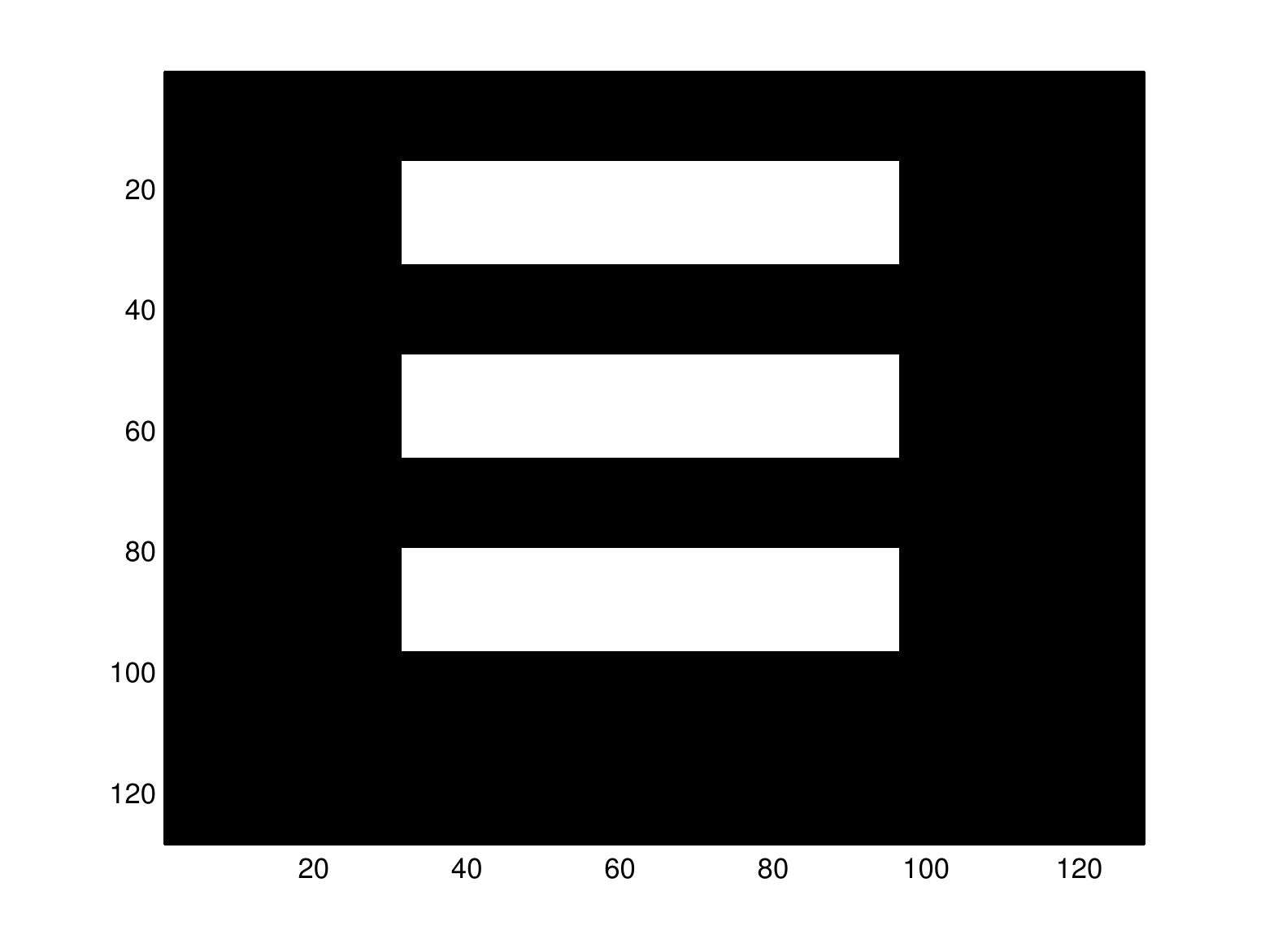} &
\includegraphics[ trim=0.8cm 0cm 1cm 0.8cm,clip=true,width=0.32\textwidth]{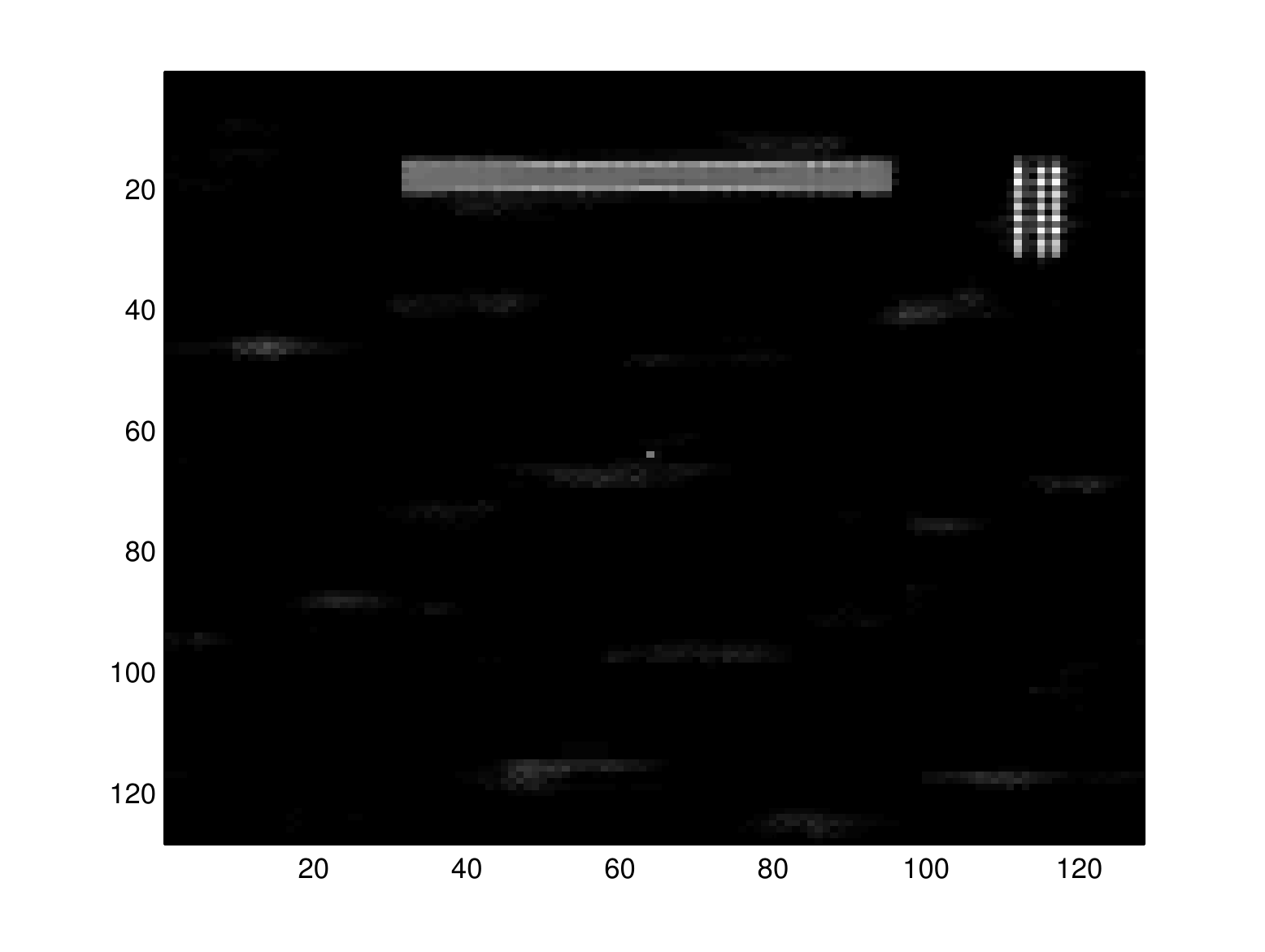} \\
\text{\footnotesize Reconstruction from map C, $\varepsilon_{rel} = 0\%$} & \text{\footnotesize Reconstruction from map C, $\varepsilon_{rel} = 70.8\%$}\\
 \includegraphics[ trim=0.8cm 0cm 1cm 0.8cm,clip=true,width=0.32\textwidth]{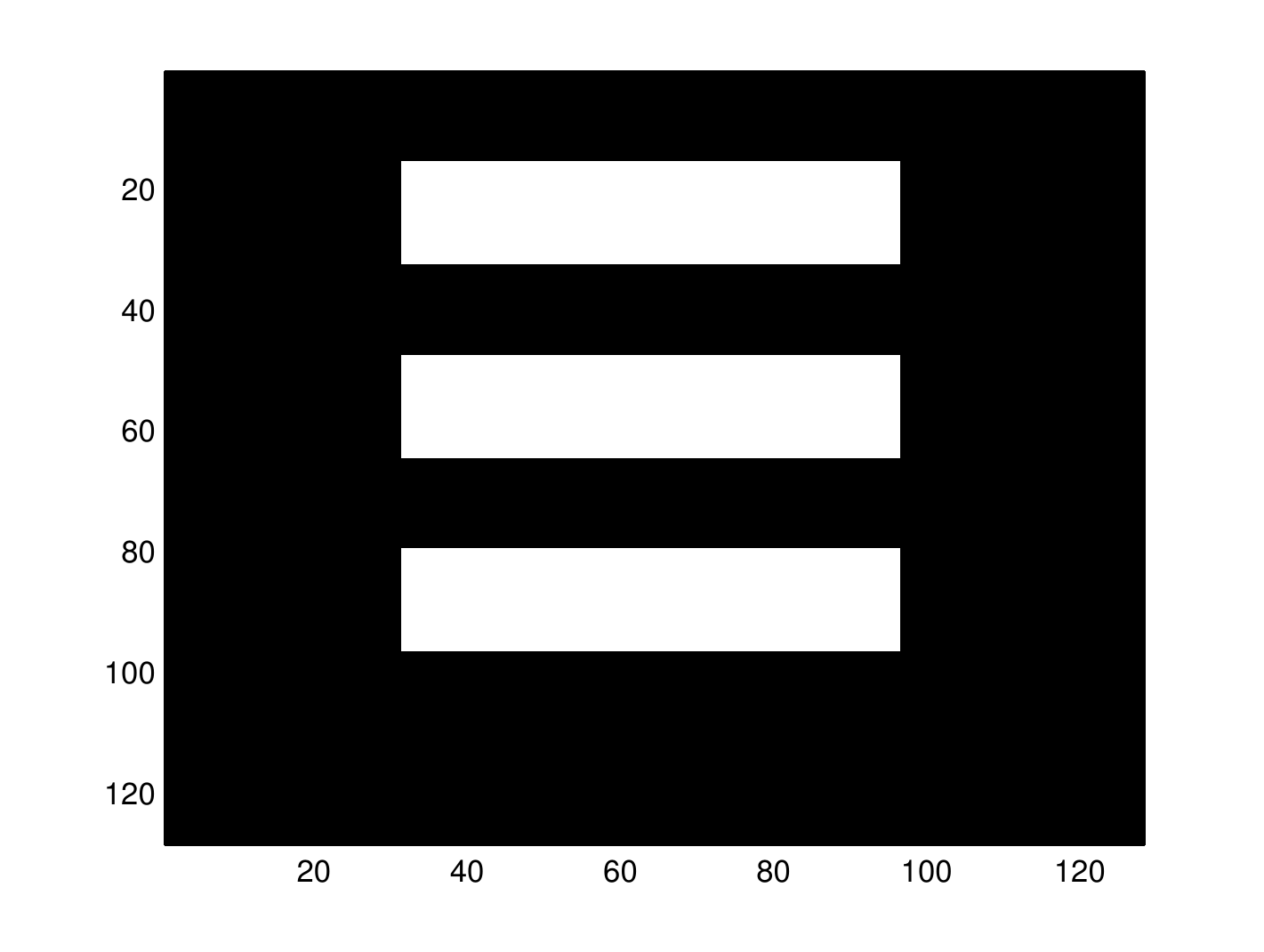} &
\includegraphics[ trim=0.8cm 0cm 1cm 0.8cm,clip=true,width=0.32\textwidth]{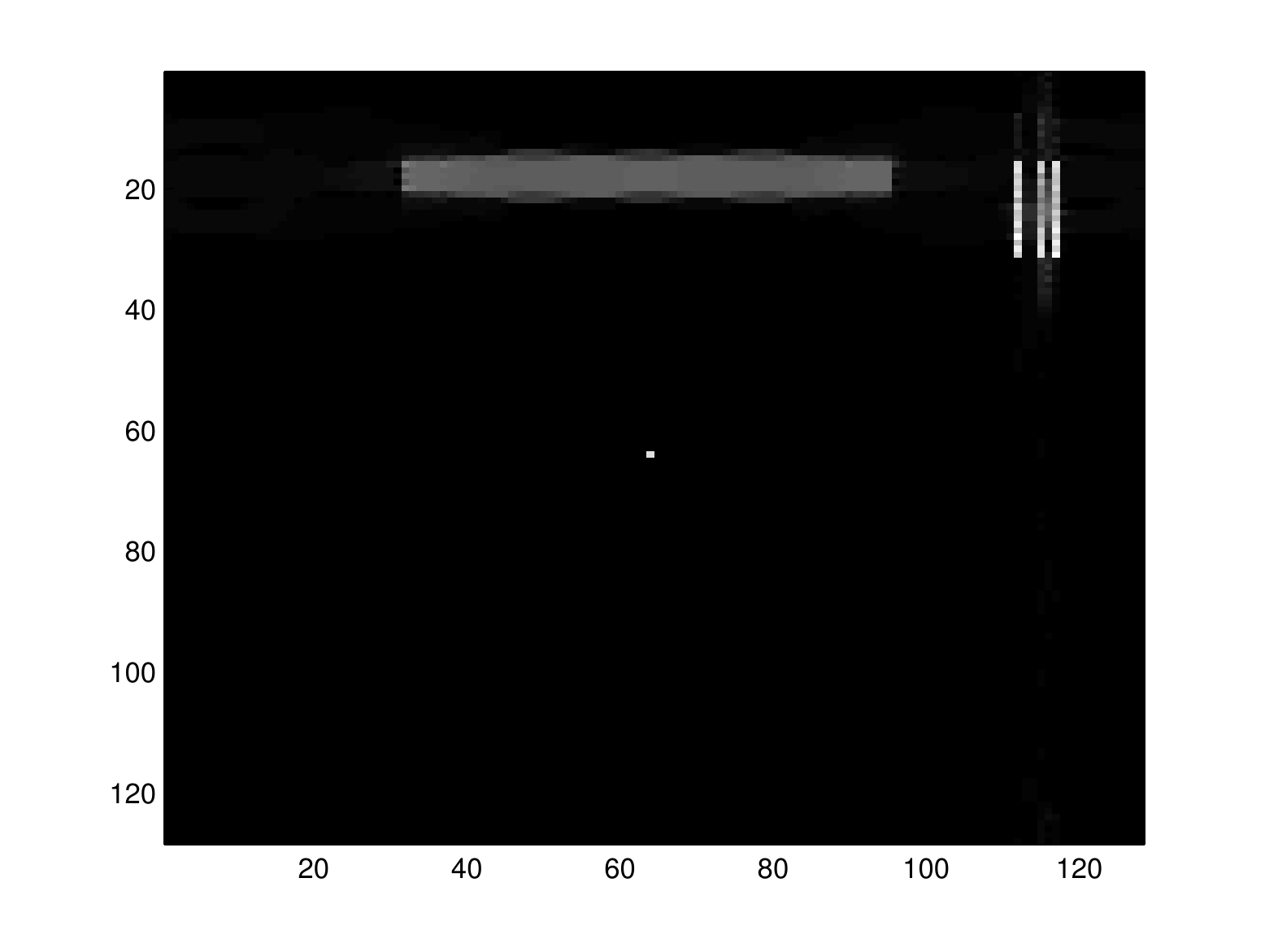} \\
\text{\footnotesize Reconstruction from map D, $\varepsilon_{rel} = 89.9\%$ } & \text{\footnotesize Reconstruction from map D, $\varepsilon_{rel} = 64.0\%$} \\
 \includegraphics[ trim=0.8cm 0cm 1cm 0.8cm,clip=true,width=0.32\textwidth]{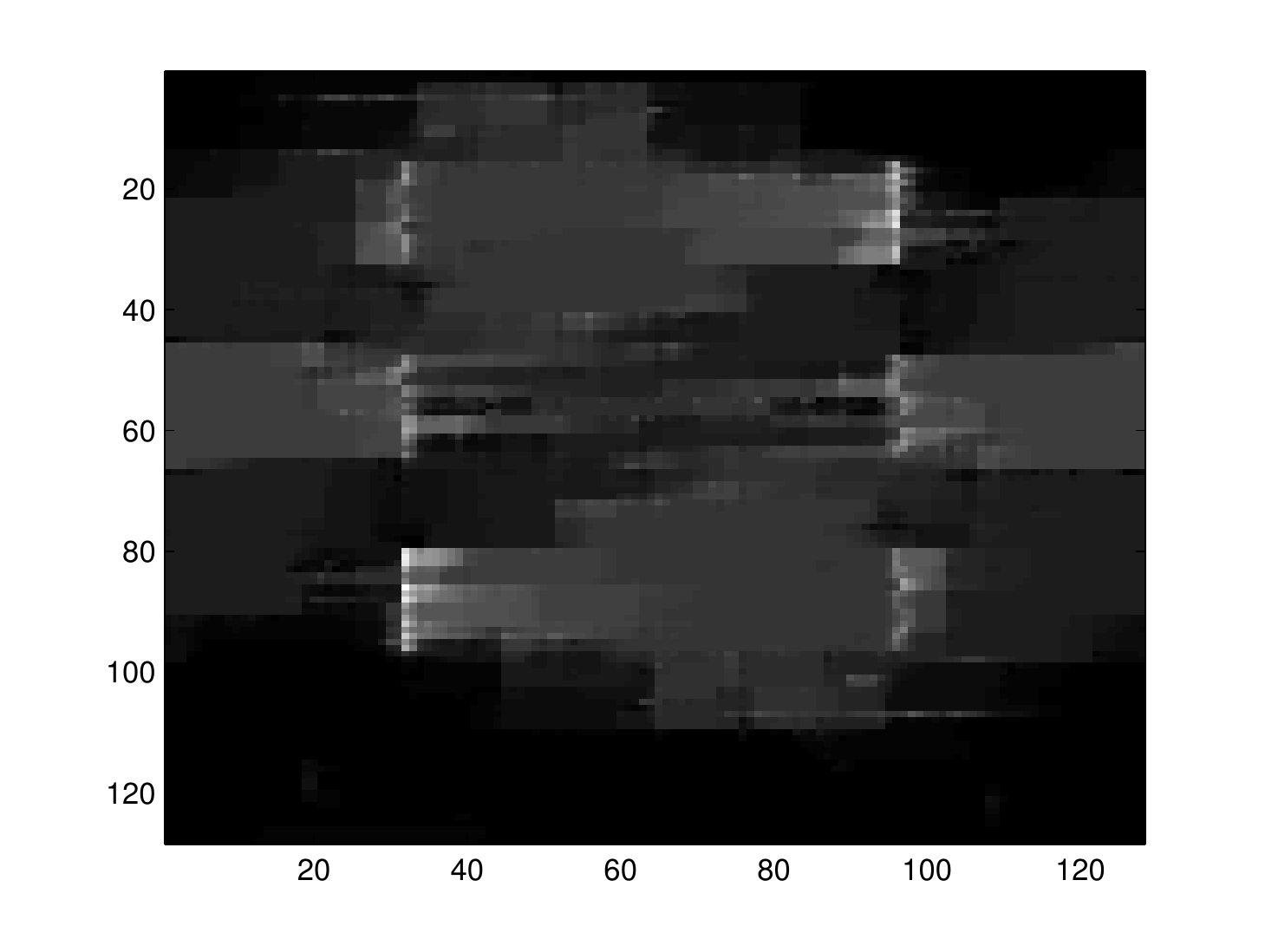} &
\includegraphics[ trim=0.8cm 0cm 1cm 0.8cm,clip=true,width=0.32\textwidth]{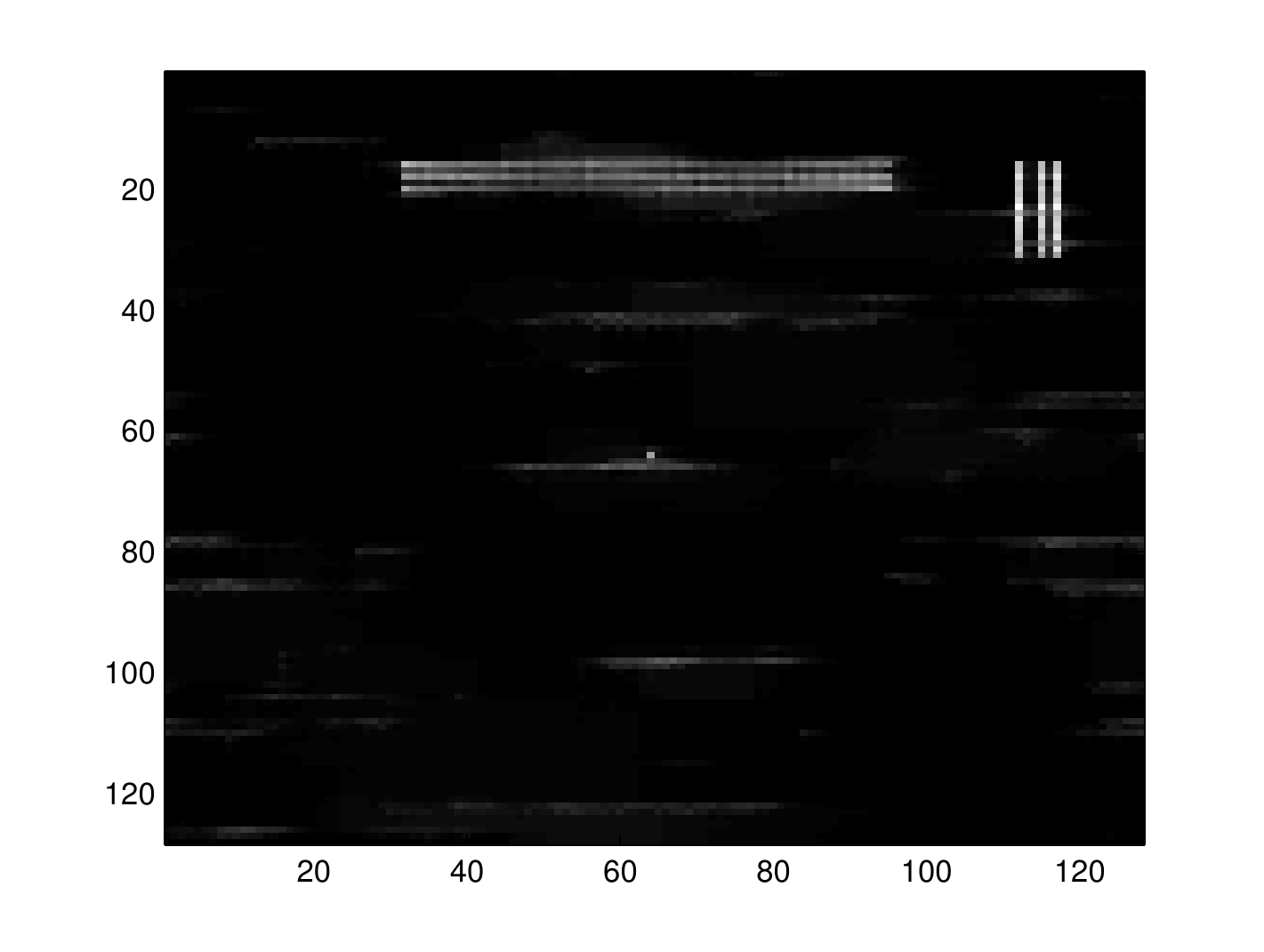}\\
\text{\footnotesize Reconstruction from map E, $\varepsilon_{rel} = 96.4\%$ } & \text{\footnotesize Reconstruction from map E, $\varepsilon_{rel} = 77.3\%$} \\
 \includegraphics[ trim=0.8cm 0cm 1cm 0.8cm,clip=true,width=0.32\textwidth]{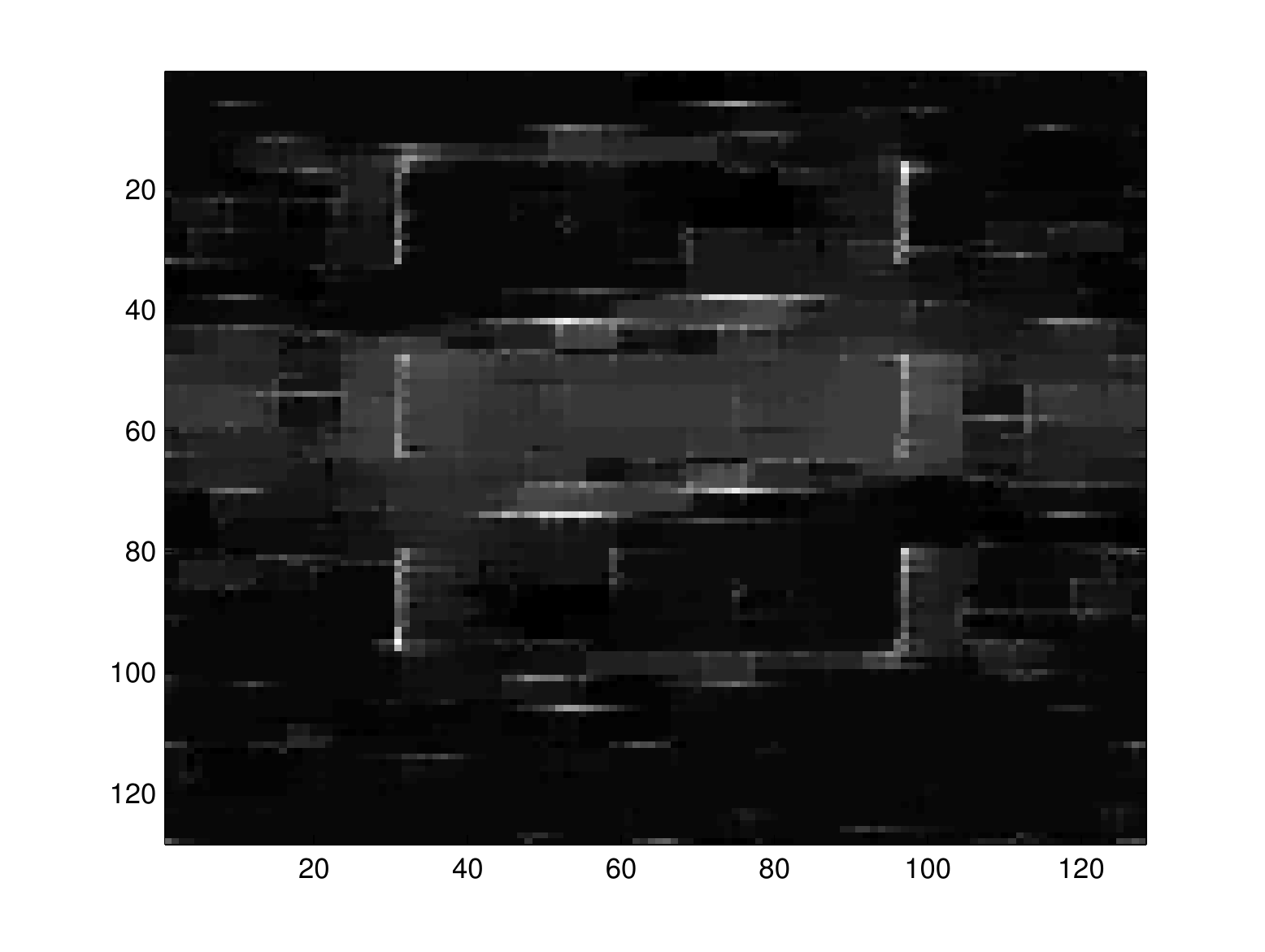} &
\includegraphics[ trim=0.8cm 0cm 1cm 0.8cm,clip=true,width=0.32\textwidth]{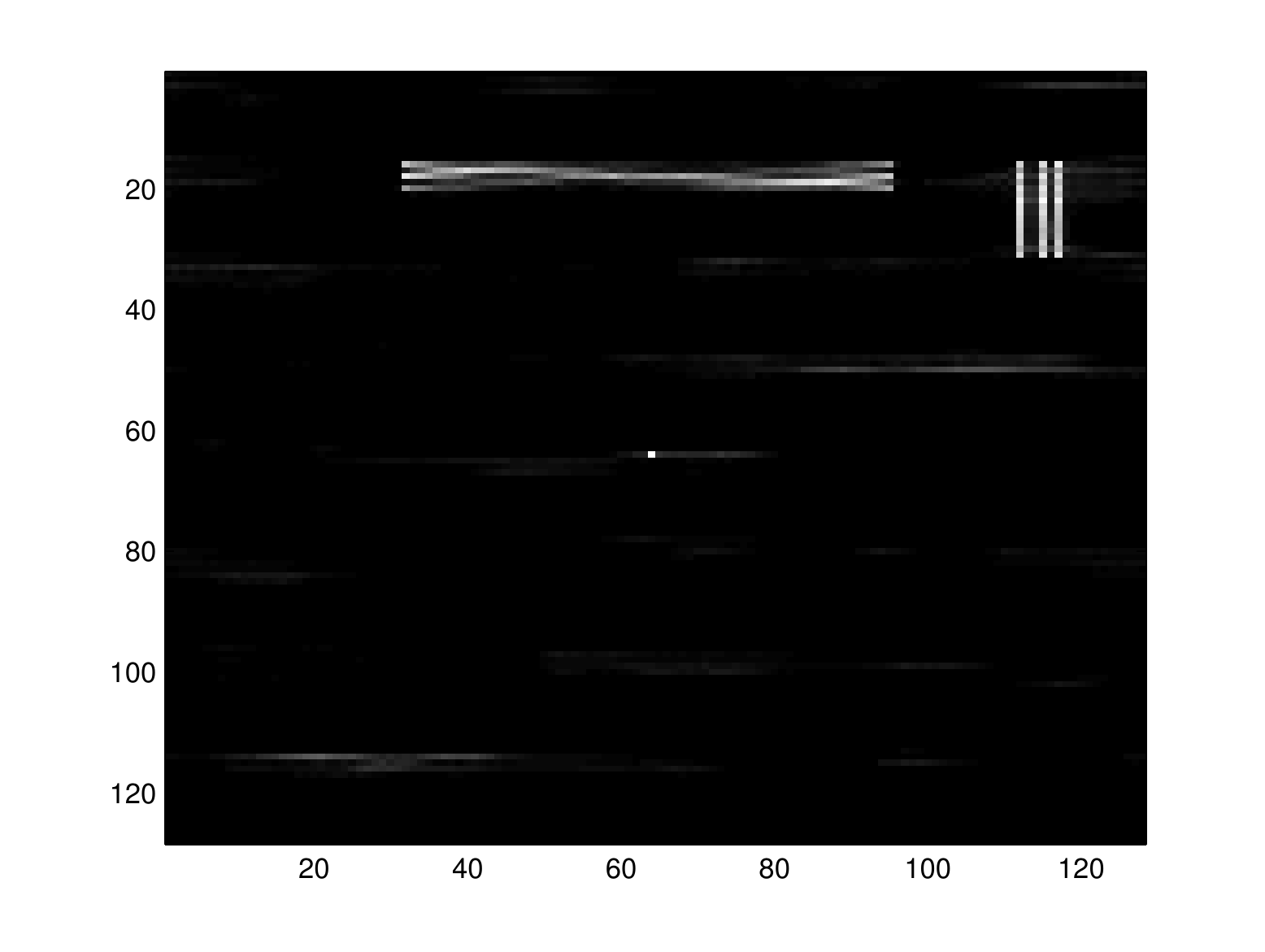} 
\end{array}$
\caption{
reconstructions of image 1 (left) and of image 2 (right) from their  partial Fourier coefficients specified by the sampling maps in Figure \ref{fig:sampling_half}. Relative error is denoted by $\varepsilon_{rel}$.
\label{fig:lines_recons}}
\end{figure}

\subsection{Related works}

\subsubsection{Stable and robust recovery of gradient sparse signals from Fourier samples}
As mentioned in the introduction it was proved in \cite{tv1} that solutions of (\ref{eq:noise_tv}) are stable to inexact sparsity and robust to noise if $\Omega$ is chosen uniformly at random and of cardinality $\ord{s\log(N)}$. However, this choice of $\Omega$ does not reflect how compressed sensing is used in practice and \cite{tv1} also highlighted the need for analysis when $\Omega$ is chosen in a non-uniform random manner and demonstrated that under additional assumptions on the sparsity pattern of the underlying signal, one is required to only sample low Fourier frequencies and thereby sample less than would be required when sampling uniformly at random. Another recent development is the work \cite{ward2013stable} by Krahmer and Ward, in which they  derived sufficient conditions on the number of samples required for  stable and robust recovery when solving (\ref{eq:noise_tv}) where $\rA$ is a weighted discrete Fourier operator and $\Omega$ is chosen in a non-uniform manner which concentrates at low Fourier frequencies. However, their theory  assumes that the underlying signal is sparse and does not consider any further signal structure.  Consequently, the  estimates  on the number of samples required were pessimistic when compared to actual examples and only suggest that subsampling is possible in cases of extreme sparsity. 
Moreover, a sparsity based theory will not only give overly pessimistic results that do not match the observations seen in experiments, it hides the highly important fact that the optimal sampling strategy is signal structure dependent. As we will demonstrate in this paper, the theory of how to choose the sampling procedure depending on the signal structure is highly complex.

\subsubsection{Other instances of signal structure dependence}
The concept that the randomness of our sampling strategy should be tailored  to some specific  signal structure  other than sparsity is relevant not only when using the  total variation norm as argued in this paper, but also for regularization with other sparsifying systems.  \cite{adcockbreaking} discusses several experiments which demonstrate this phenomonen for a range of sparsifying sytems (including shearlets, wavelets, curvelets) and different sampling systems (Fourier, Hadamard). 

One particular example which has motivated much of this work is the case of  recovery of wavelet coefficients from Fourier samples. As discussed in \cite{adcockbreaking}, this problem does not fit into the standard framework of compressed sensing and the  notions of uniform random sampling, sparsity and incoherence cannot be used to explain the recovery of signals from highly incomplete Fourier samples. With this example in mind,   \cite{adcockbreaking} introduced an extended theory of compressed sensing which allows us to divide the available samples into levels and obtain estimates of the number of samples required at each level. Moreover those estimates reveals the dependence of the Fourier sampling pattern on the underlying wavelet structure. 

\subsection{Overview}
 In this paper, we derive recovery statement which will demonstrate how the sparsity pattern of the underlying signal will impact the choice of the sampling set $\Omega$ when solving (\ref{eq:noise_tv}). In order to do this, we develop a theoretical framework which caters for non-uniform choices of the sampling set $\Omega$ and the sparsity structure of the underlying signal. This can be seen as a generalization of the framework introduced in \cite{adcockbreaking}. The main results concerning the recovery of signals from their partial Fourier data via total variation minimization are in Section \ref{sec:main_tv_results}. The main results concerning a more general compressed sensing problem are  presented in Section \ref{general_main_results}.
The results of Section \ref{sec:main_tv_results} are proved in Section \ref{sec:proof_tv1} and Section \ref{sec:proof_tv2}. The results of Section \ref{general_main_results} are proved in Section \ref{sec:proofs}.

\section{Recovery from partial Fourier data via total variation regularization}\label{sec:main_tv_results}
This paper concerns the following two problems.
\begin{enumerate}
\item[(P1)]
For some fixed underlying signal $x\in\bbC^N$ and noise level $\delta\geq 0$, how should $$\Omega\subset \br{-\lfloor N/2\rfloor +1,\ldots, \lceil N/2\rceil }$$ be chosen such that given $y = \rP_\Omega \rA x$, solutions of the following minimization problem `accurately approximates' $x$?
\be{\label{eq:min_tv}
\min_{\eta\in\bbC^{N}}\norm{\eta}_{TV}\text{ subject to } \norm{\rP_\Omega\rA \eta - y}_2 \leq \delta.
}

\item[(P2)]
For some fixed underlying image $x\in\bbC^{N\times N}$ and noise level  $\delta\geq 0$, how should $$\Omega\subset \br{-\lfloor N/2\rfloor +1,\ldots, \lceil N/2\rceil}^2$$ be chosen such that given $y = \rP_\Omega \rA x$, solutions of the following minimization problem `accurately approximates' $x$?
\be{\label{eq:min_tv_2d}
\min_{\eta\in\bbC^{N\times N}}\norm{ \eta}_{TV}\text{ subject to } \norm{\rP_\Omega\rA \eta - y}_2 \leq \delta.
}
\end{enumerate}

\subsection{Notation} \label{sec:notation}
Before stating the main results, we first introduce some notations.
\subsection*{Notation for one dimensional total variation}
Let $N\in\bbN$. Given $J\subset \br{1,\ldots, N}$, let $\mathbbm{1}_J\in\br{0,1}^N$ be such that $(\mathbbm{1}_{J})_k = 1$ for $k\in J$ and $(\mathbbm{1}_{J})_k = 0$ for $k\not \in J$.
We define the operator
 $$
 \rD: \bbC^N \to \bbC^{N-1}, \qquad x\mapsto (-x_j +  x_{j+1})_{j=1}^{N-1}
 $$
 and the total variation norm of $x\in\bbC^N$ as
 $$
 \nm{x}_{TV} := \nm{\rD x}_1.
 $$
 We refer to $\rD x$ as the total variation coefficients of $x$.
 Let $\rA \in \bbC^{N\times N}$ be the unitary discrete Fourier transform.
We index the matrix $\rA$ with $k=-N/2+1,\ldots,N/2$ and $j=0,\ldots,N-1$, and let $\rA[k,j]= N^{-1/2} e^{2\pi i k j/N}$.
For $\Omega \subset \br{-\lfloor N/2\rfloor +1,\ldots, \lceil N/2\rceil}$, let $\rP_\Omega$ be the orthogonal projection matrix such that 
$$
\rP_\Omega : \bbC^N\to \bbC^N, \qquad 
(\rP_{\Omega}  x)_j = \begin{cases}
 x_j & j\in\Omega\\
0 &\text{otherwise.}
\end{cases}
$$

\subsection*{Notation for two dimensional total variation}

We will also be dealing with two dimensional images in this paper, and we now define the analogous operators and norm. To avoid clutter of notation, we will use the same notation for analogous concepts, however it should be clear from the context as to whether we are using the one-dimensional definitions or the two-dimensional definitions.

For $J\subset \br{1,\ldots, N}^2$, let $\mathbbm{1}_{J}\in\br{0,1}^{N\times N}$ 
be such that $(\mathbbm{1}_{J})_k = 1$ for $k\in J$ and  $(\mathbbm{1}_{J})_k = 0$ for $k\not \in J$.
Let $\rD_1: \bbC^{N\times N} \to \bbC^{N \times N}$, $\rD_2: \bbC^{N\times N} \to \bbC^{N\times N}$
and $\rD: \bbC^{N\times N} \to \bbC^{ N\times N} \times \bbC^{ N\times N} $ be defined as follows.
Given $x\in\bbC^{N\times N}$,
\eas{
&(\rD_1 x)_{t_1,t_2} = \begin{cases} x_{t_1+1,t_2} - x_{t_1,t_2} &t_1 <N\\
0 &t_1 = N,
\end{cases}\\
&
(\rD_2 x)_{t_1,t_2} =\begin{cases} x_{t_1,t_2+1} - x_{t_1,t_2} &t_2<N\\
0 &t_2 = N,
\end{cases}\\
&(\rD x)_{t_1,t_2} = ((\rD_1 x)_{t_1,t_2}, (\rD_2 x)_{t_1, t_2} ).
}
We refer to $\abs{(\rD x)_{t_1,t_2}}$ as the $(t_1,t_2)$ total variation coefficients of $x$, where $\abs{\cdot}$ is the Euclidean norm and define the isotropic total variation norm as
$$
\nm{x}_{TV} = \sum_{i,j=1}^N \abs{(\rD x)_{i,j}}.
$$
Given $J\subset \br{1,\ldots, N}^2$, we define its perimeter as
$$
\mathrm{Per}(J) = \nm{\mathbbm{1}_J}_{TV,1}
$$
where $\nm{\cdot}_{TV,1}$ is the anisotropic total variation norm and is defined as
$$
\nm{x}_{TV, 1} = \sum_{i,j=1}^N \abs{(\rD_1 x)_{i,j}} + \abs{(\rD_2 x)_{i,j}}.
$$

For $\Omega \subset \bbZ^2$, we will let $\rP_{\Omega}$ be the orthogonal projection such that
$$
\rP_\Omega: \bbC^{N\times N}  \to \bbC^{N\times N} , \qquad 
(\rP_{\Omega} x)_{k_1,k_2} = \begin{cases}
 x_{k_1,k_2} & (k_1,k_2) \in\Omega\\
0 &\text{otherwise}
\end{cases},\qquad x\in\bbC^N.
$$
For $\Lambda \subset \br{1,\ldots, N}^2$, let $\tilde \rP_{\Lambda} :\bbC^{N\times N}\times \bbC^{N\times N} \to \bbC^{N\times N}\times \bbC^{N\times N}$ be such that
$$
\tilde \rP_{\Lambda} \left((x,y)\right) = (\rP_\Lambda x, \rP_\Lambda  y), \qquad x,y\in\bbC^{N\times N}
$$
Let $\rA:\bbC^{N\times N}\to \bbC^{N\times N}$ be the two dimensional unitary discrete Fourier transform, where we index $\rA x$ by $(k_1,k_2)$ such that $-\lfloor N/2\rfloor +1\leq k_1,k_2\leq \lceil N/2\rceil$. Specifically, we define
$$
(\rA x)_{k_1,k_2} = \frac{1}{N}\sum_{j_1,j_2\in\br{1,\ldots,N}} x_{j_1,j_2}e^{2\pi i \left(\frac{j_1k_1}{N}+ \frac{j_2k_2}{N}\right)}, \qquad -\left \lfloor \frac{N}{2}\right \rfloor+1\leq k_1,k_2\leq \left \lceil \frac{N}{2} \right \rceil.
$$

\subsection*{General notation}
\begin{enumerate}
\item For $p\in[1,\infty]$, let $\cB(\ell^p(\bbN))$ denote the set of bounded linear operators on $\ell^p(\bbN)$. We will not differentiate between infinite matrices and bounded linear operators on $\ell^p(\bbN)$ in this paper.
\item Given an operator $\rA$, let $\cN(\rA)$ denote the null space of $\rA$ and let $\cR(\rA)$ denote the range of $\rA$.
\item For $p,q\in[1,\infty]$, let $\norm{\cdot}_{p\to q}$ denote the operator norm from $\ell^p(\bbN)$ to $\ell^q(\bbN)$.
\item Given any index set $\Delta$, let $\rP_\Delta$ denote the orthogonal projection onto the canonical basis indexed by $\Delta$.
\item Given any subspace $\cW\subset \ell^2(\bbN)$, let $\rQ_{\cW}$ denote the orthogonal projection onto $\cW$. 
\item Given $\blam = (\blam_j)_{j\in\bbN} \in \ell^\infty(\bbN)$, let $\blam^{-1}$ be the vector whose $j^{th}$ entry is $\blam_j$ if $\blam_j \neq 0$ and is zero otherwise. Given any matrix $\rW$, $\rW \circ\blam := \rW \rL$ and $\blam\circ \rW := \rL \rW$, where $\rL$ is the diagonal matrix whose diagonal is $\blam$.
\item Given $x,y \in \bbR$, we write $x \lesssim y$ to denote $x\leq C \cdot y$ for some numerical constant $C$. 
\item Given $x\in\ell^2(\bbN)$, $\sgn(x) \in\ell^\infty(\bbN)$ is such that its $j^{th}$ entry is $x_j/\abs{x_j} $ if $x_j\neq 0$ and zero otherwise.
\item For $M\in\bbN$, let $[M] = \br{1,\ldots, M}$.
\end{enumerate}

\subsection{Signal structure}\label{sec:structure_tv}
As demonstrated in Section \ref{subsec:sparsity_insuff}, any theory which explains the success of non-uniform sampling patterns cannot be based on sparsity alone, but should take into account more specific signal structures. The examples therein suggest that the sampling pattern depends on the number of `fine details' present in the underlying signal. 

\subsubsection{One dimensional case}

\subsubsection*{Active sparsity}
 For some $x\in\bbC^{N}$,
 suppose that $\rD x$ has $s$ non-zero entries at the indices $t_1,\ldots, t_{s}$ where  $1\leq t_1<\ldots< t_{s}\leq  N-1$. So, the signal $x$ is $s$-sparse in its gradient.  Then, it follows that $x\in \cN(\rP_\Lambda\rD) = \br{\mathbbm{1}_{\br{t_{j-1}+1,\ldots, t_j}}: j=1,\ldots,s+1}$ and $x$ is the composition of $s+1$ constant vectors, i.e. there exists some $\alpha \in\bbC^{s+1}$ such that
$$
x = \sum_{j=1}^{s+1} \alpha_j \mathbbm{1}_{\br{ t_{j-1}+1,\ldots, t_j}}
$$
where $t_0 = 0$, $t_{s+1} = N$ and for any $J\subset \br{1,\ldots, N}$, $\mathbbm{1}_J\in \br{0,1}^N$ is the vector which is one on the index set $J$ and zero elsewhere.

The example in Figure \ref{fig:1d_recons} suggest that different signal structures require different sampling patterns. Moreover the key difference in feature between signal 1 and signal 2 are the widths of the constant vectors which make up the signal.
 With this in mind, we propose the following  notions of active sparsity and fineness with respect to $\Lambda$.

\begin{definition}[Active sparsity of a signal]\label{def:active_s_1d}
Given $\Lambda^c = \br{t_1,t_2,\ldots, t_s} \subset \br{1,\ldots, N-1}$ such that 
$$
0=t_0 <t_1<\cdots < t_{s} < t_{s+1} = N,
$$
the active sparsity at $p\in[0,\infty]$ is defined to be 
$$
S\left(\Lambda , p\right) := \frac{N}{p}\cdot \sum_{j: t_j -t_{j-1}>(N/p)} \frac{1}{ (t_j - t_{j-1})} + \abs{\br{j: t_j -t_{j-1} \leq \frac{N}{p}}}.
$$
\end{definition}
Observe that $S\left(\Lambda, p\right)$ is a non-increasing function in $p$ and if $\rD x$ is zero on $\Lambda$, then for all $p\in [0,1]$, $S\left(\Lambda, p\right)-1=s = \abs{\Lambda^c}$. So, $S\left(\Lambda, p\right)-1$ is the number of nonzero entries in $\rD x$.
If $\rD x$  is zero on $\Lambda$, then $x = \sum_{j=1}^{s+1} \alpha_j \mathbbm{1}_{\br{t_{j-1}+1,\ldots, t_j}}$ for some $\br{\alpha_j}_{j=1}^{s+1} \in\bbC^{s+1}$. As demonstrated numerically, the sampling strategy should depend on the widths $\br{t_j - t_{j-1}: j=1,\ldots, s+1}$. $S(\Lambda, p)$ can be thought of as quantifying the number of widths less than $N/p$.
For the examples in Figure \ref{fig:1d_recons}, plots of their active sparsity values are shown in Figure \ref{fig:active_sparsity}.

\subsubsection*{Fineness of a signal}

The notion of sparsity in compressed sensing has been useful because it is one quantity describing the amount of information needed to recover a signal. In the case of recovering a signal from its partial Fourier data, the sparsity of the gradient of a signal is not sufficient for this. Thus, we introduce the concept of the `fineness' of a signal. This is meant to provide a succinct means of differentiating between the structure of two signals.

\begin{definition}[Fineness of a signal]\label{def:fineness_1d}
Given $\Lambda^c = \br{t_1,t_2,\ldots, t_s} \subset \br{1,\ldots, N-1}$ such that 
$$
0=t_0 <t_1<\cdots < t_{s} < t_{s+1} = N,
$$
the fineness is defined as
$$
F\left(\Lambda\right) := \sum_{j=1}^{s+1} \frac{1}{t_j - t_{j-1}}.
$$
\end{definition}
For the examples in Figure \ref{fig:1d_recons}, the fineness of signal 1 is 1.51 whilst the fineness of signal 2 is 15.03.

\subsubsection{Two dimensional case}

\subsubsection*{Active sparsity}
 For two-dimensional case, suppose that $x\in\bbC^{N\times N}$ and $\Lambda \subset \br{1,\ldots, N}^2$ is the largest index set such that
$
\tilde \rP_\Lambda \rD x = 0
$
and there exists a partition of $\br{1,\ldots, N}^2$, $\br{I_j : j=1,\ldots, n}$  such that 
$$
\cN(\tilde \rP_\Lambda \rD) = \br{\mathbbm{1}_{I_j}: j=1,\ldots, n}.
$$
Then $x = \sum_{j=1}^n \alpha_j \mathbbm{1}_{I_j}$ for some $\alpha \in\bbC^n$.

We first remark that the perimeters of $\br{\mathrm{Per}(I_j)}_{j=1}^n$ are closely related to the sparsity in the total variation coefficients of $x$. In general,
 the anisotropic total variation and the associated notion of perimeter satisfy the following for any $x\in\bbC^{N\times N}$ \cite{chambolle2009total} (although this is not true for the isotropic total variation norm and its associated perimeter \cite{chambolle2010continuous})
$$
\norm{z}_{TV, 1} = \int_{-\infty}^\infty \mathrm{Per}(\br{z \geq t}) \mathrm{d}t.
$$
Also, we have that  $2\abs{\Lambda^c} \leq \sum_{j=1}^n\mathrm{Per}(I_j)\leq  4\abs{\Lambda^c}$.

 The example in presented in Figure \ref{fig:lines_recons} demonstrated  that the subsampling pattern cannot depend on sparsity alone, but suggest that the signal structure which determines the sampling pattern are the 'fineness' of the different components of the image. We will characterize this notion of 'fineness' by considering the ratio between the perimeter and the area of each region,  $\br{\mathrm{Per}(I_j)/\abs{I_j}}_{j=1}^n$. 
We now introduce some definitions to formalise these ideas.

\begin{definition}[Active sparsity of an image]\label{def:active_s_2d}
Given $\Lambda \subset \br{1,\ldots, N}^2$ such that 
$$
\cN(\tilde \rP_\Lambda \rD) = \br{ \mathbbm{1}_{I_j} : j=1,\ldots, n},
$$
 the active sparsity at $p\in[0,\infty]$ is defined as
$$
S\left(\Lambda, p\right) := \frac{N}{p}\cdot \sum_{j\not \in\Delta_p} \frac{ \mathrm{Per}(I_j)^2}{\abs{I_j}} +\sum_{j\in\Delta_p} \mathrm{Per}(I_j)
$$
where
$$
\Delta_p = \br{j: \frac{\abs{I_j}}{\mathrm{Per}(I_j)} \leq \frac{N}{p}}.
$$
\end{definition}

Observe that $S(\Lambda, 0) = \sum_{j=1}^n \mathrm{Per}(I_j)$, which is up to a constant the number of nonzero entries in the gradient of $x$. $S(\Lambda, p)$ is non-increasing in $p$, and describes the sparsity in the gradient of $x$ coming from the components $I_j$ for which $\abs{I_j}/\mathrm{Per}(I_j)$ is sufficiently small. 

For the images in Figure \ref{fig:lines_recons}, the fineness of image 1  is   90.43 and the fineness of image 2 is 1008.05. Plots of their active sparsity values are shown in Figure \ref{fig:active_sparsity}. Observe that the active sparsities have much faster decay  for signal 1 and image 1 when compared with signal 2 and image 2 respectively.

\subsubsection*{Fineness of an image}
As in the one dimensional case, it is desirable to have a succinct concept to differentiate between two images for the purpose of total variation regularization. Intuitively, this should be dependent on how sparsity in the gradient, and how complex the boundaries of components which make up the image. 

\begin{definition}[Finenss of an image]\label{def:fineness_2d}
Given $\Lambda \subset \br{1,\ldots, N}^2$ such that 
$$
\cN(\tilde \rP_\Lambda \rD) = \br{ \mathbbm{1}_{I_j} : j=1,\ldots, n}.
$$
The fineness is defined as
$$
F\left(\Lambda\right) := \sum_{j=1}^n \frac{\mathrm{Per}(I_j)^2}{\abs{I_j}}.
$$
\end{definition}
For the images in Figure \ref{fig:lines_recons}, the fineness of image 1  is   90.43 and the fineness of image 2 is 1008.05.

\begin{figure}[ht] 
\centering
$\begin{array}{cc}
\includegraphics[ trim=0cm 0.2cm 0.8cm 0.2cm,clip=true,width=0.45\textwidth]{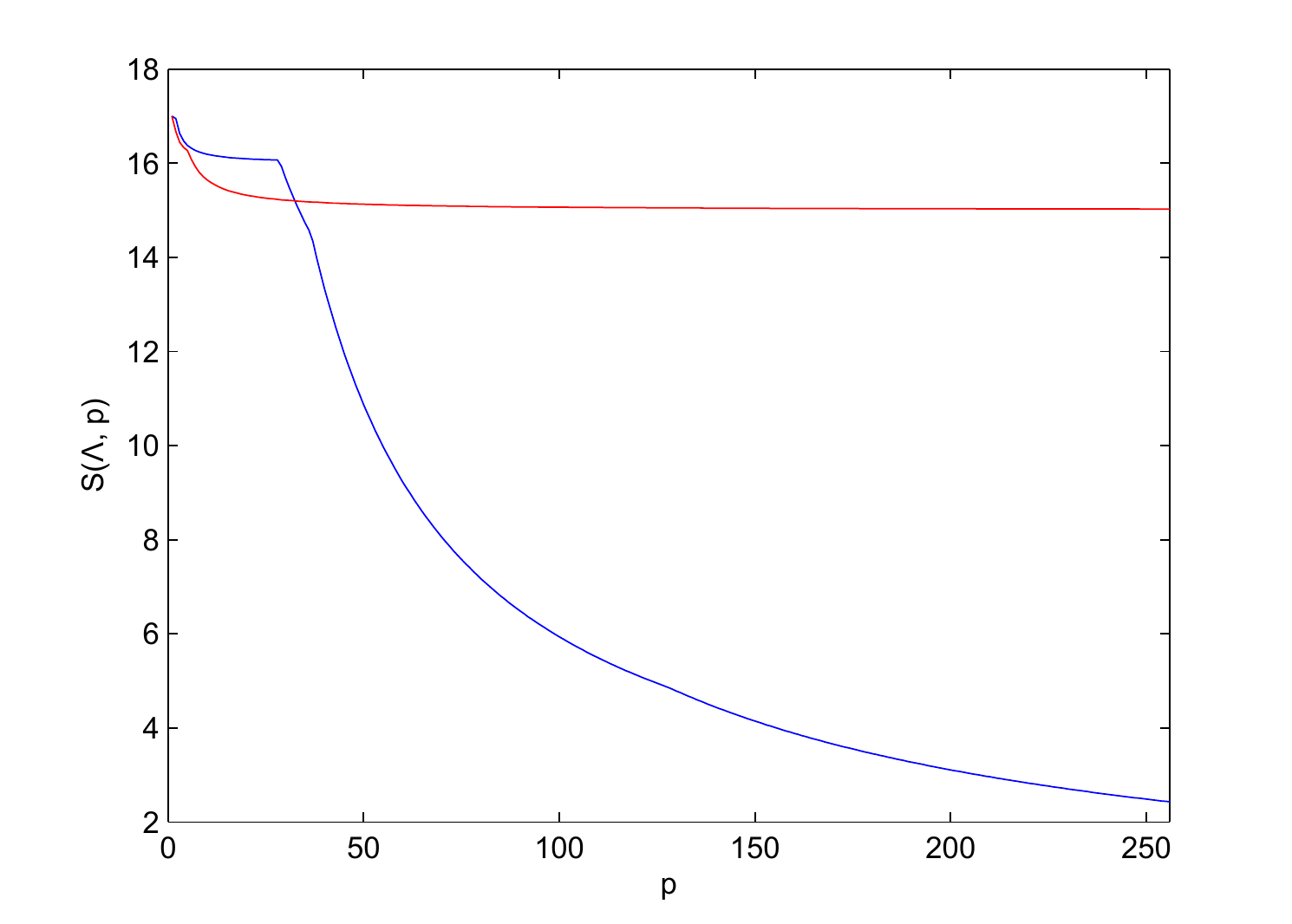} &
\includegraphics[ trim=0cm 0.2cm 1cm 0.2cm,clip=true,width=0.45\textwidth]{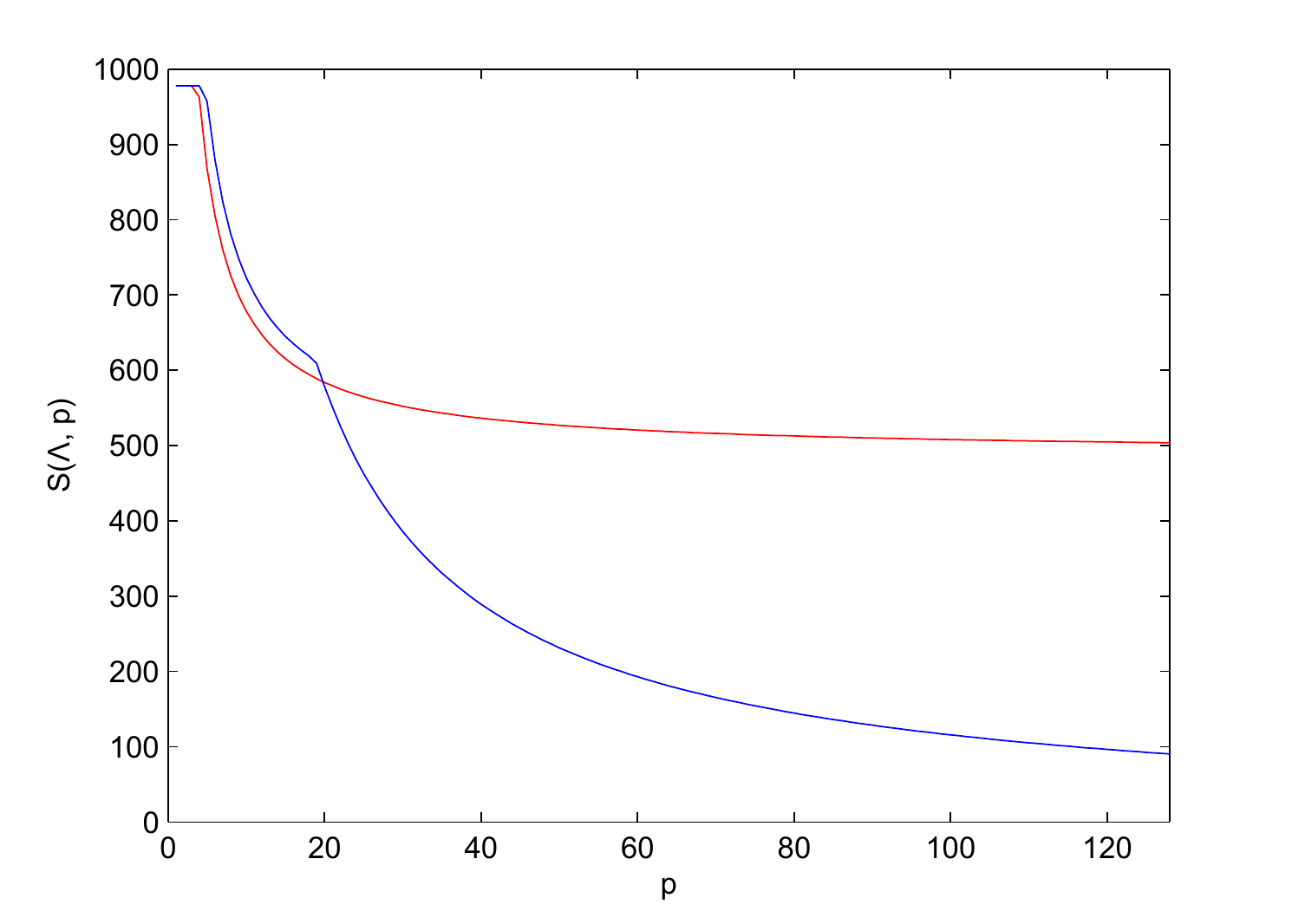} 
\end{array}$
\caption{
Left: active sparsity values for signal 1 (blue) and signal 2 (red). Right: active sparsity values for image 1 (blue) and image 2 (red). 
\label{fig:active_sparsity}}
\end{figure}

\subsection{Main results I}\label{sec:tv_results}

\subsection{The one dimensional case} 
To understand the success of variable density sampling patterns, we define the following multilevel sampling model.

\subsection*{The multilevel sampling model}
Let $N\in\bbN$. For $r\in\bbN$, let $\mathbf{M}= \br{M_k}_{k=1}^r\in\bbN^r$ be such that
$0=M_0 < M_1<\cdots < M_r = N$ and for each $k=1,\ldots, r$, let $$
\Gamma_k = \br{j\in\bbZ: -\lfloor M_{k}/2\rfloor \leq  j \leq -\lfloor M_{k-1}/2\rfloor -1, \, \lceil M_{k-1}/2\rceil \leq j\leq \lceil M_k/2 \rceil -1}
$$
Let $ (m_k)_{k=1}^r \in\bbN^r $ and $\Omega_{\mathbf{M}, \mathbf{m}} = \Omega_1 \cup \cdots \Omega_r$ be such that for each $k=1,\ldots, r$, 
$$
\Omega_k \subset  \Gamma_k, \qquad \abs{\Omega_k} = m_k \leq M_k - M_{k-1}
$$
is drawn uniformly at random. 
\\
\\
Suppose that $x\in\bbC^N$ is approximately $(s-1)$-sparse in its gradient, meaning that there is $\Lambda  \subset \br{1,\ldots, N-1}$ be such that  $\abs{\Lambda^c}=s-1$ and $\nm{\rP_\Lambda \rD x}_1 <<1$. By definition of the function $S$ from Definition \ref{def:active_s_1d}, $s = S(\Lambda, 1)$.
Then, we have the following theorem on the impact of $\mathbf{M}$ and $\mathbf{m}$ on the success of recovering $x$ by solving (\ref{eq:min_tv}).

\thm{\label{thm:main_tv}
Let $\epsilon>0$ and $\cL = (\log(s\epsilon^{-1}) +1)\cdot \log(q^{-1} N^{3/2} \sqrt{s}) $.
 Suppose that $\mathbf{M}$ and $\mathbf{m}$  satisfy the following.
\begin{itemize}
\item[(i)] 
For $k=1,\ldots, r$,
 \eas{
\frac{m_k}{M_k-M_{k-1}} \gtrsim   \cL
\cdot \left(\frac{S\left(\Lambda, \max\br{M_{k-1},1}\right)}{\max\br{M_{k-1},1}} 
+ \frac{M_{k}}{\max\br{M_{k-1},1}}\cdot \frac{(s-1)}{N} \right).  
}
\item[(ii)] For $k=1,\ldots, r$,
  $ m_k \gtrsim (\log(s\epsilon^{-1}) +1) \cdot \log(q^{-1} N^{3/2} \sqrt{s}) \cdot \hat m_k$
  such that $\br{\hat m_k}_{k=1}^r$ satisfies
  $$
1\gtrsim \sum_{k=1}^r\left(\frac{M_k - M_{k-1}}{\hat m_k} -1\right)\cdot \left(\frac{N}{\max\br{M_{k-1}^2,1}} \cdot \hat s_k +    \frac{s-1}{N}\cdot\left(\frac{M_k}{\max\br{M_{k-1},1}}\right)^2\right)
$$
for any $\br{\hat s_k}_{k=1}^r$ such that
$$
\sum_{k=1}^r\hat s_k \leq F(\Lambda).
$$

\end{itemize}
Then, with probability exceeding $(1-\epsilon)$, any solution to (\ref{eq:min_tv}) with  $\Omega = \Omega_{\mathbf{M}, \mathbf{m}}$  satisfies
\be{\label{eq:tv_err}
N^{-1/2}\cdot \norm{\xi - x}_{2}  \lesssim  \left(\left( 1 +\sqrt{s}\cdot L\right) \cdot \frac{\delta}{\sqrt{q}} + \norm{\rP_\Lambda  \rD x}_{1}\right).
}
with
$$
L= \sqrt{\frac{\log(\epsilon^{-1}) + \log_2(8N\sqrt{s}q^{-1})}{\log_2(4N\sqrt{s}q^{-1})}} , \quad q=\min_{k=1}^r \frac{m_k}{M_k - M_{k-1}}.
$$

}

\begin{remark}
 Note that $x$ is often considered to be the discretization of some function $f$ on a compact interval, say $[0,1]$, with $x_j = f(j/N)$. Then it is natural to define the discrete $\ell^2$ norm (see for example, \cite{leveque2007finite}) of $f$ as $$ \left(\sum_{j=1}^N N^{-1} \abs{f(j/N)}^2\right)^{1/2} = N^{-1/2}\cdot \nm{x}_2.$$
Furthermore, the discrete gradient norm of $f$ is often defined to be 
$$
\frac{1}{N}\cdot \sum_{j=1}^N \frac{\abs{f((j+1)/N) - f(j/N)}}{1/N} = \nm{\rD x}_1.
$$
So, $N^{-1/2}$ term on the right hand side of the error estimate above is a natural occurrence.
\end{remark}

\subsubsection{Interpretation of the result}
\begin{itemize}
\item[(a)]  Both conditions (i) and (ii) have  the factor of $\left(\max\br{M_{k-1},1}\right)^{-1}$ on their right hand sides, since $M_{k-1}$ increases with $k$, a direct consequence is that  the percentage of samples drawn at levels corresponding to higher Fourier frequencies should be small relative to levels corresponding to low Fourier frequencies. To offer an intuitive and informal explanation of this phenomenon, the underlying signal is such that $\rP_\Lambda \rD x \approx 0$ and we can consider it as being well approximated by an element of $\cN(\rP_\Lambda \rD)$, which consists of constant vectors. We will later show that at higher Fourier frequencies, the sampling vectors become increasingly incoherent with these constant vectors and columns of $(\rP_\Lambda \rD)^\dagger$ (which are elements of $\cN(\rP_\Lambda\rD)^\perp$. So, if the crucial information about our signal is encoded in $\cN(\rP_\Lambda \rD)$ and $\cN(\rP_\Lambda \rD)^\perp$ and this information becomes increasing spread out at higher Fourier frequencies, then this advocates for more subsampling at higher Fourier frequencies.
 \item[(b)] 
Condition (i) suggests that in order to guarantee stable recovery, the fraction of samples which we should draw from $\Gamma_k\setminus \Gamma_{k-1}$ is, up to the usual log factor and ratios between $M_k$ and $M_{k-1}$, $s/N + S(\Lambda, M_{k-1})/\max\br{M_{k-1},1}$, where we recall from Definition \ref{def:active_s_1d} that $S(\Lambda,\max\br{M_{k-1},1})$ is the active sparsity and decreases as $k$ increases. So, (i) implies that if suffices to take $s-1$ samples uniformly at random from all available Fourier samples, then after dividing the Fourier samples into $r$ levels, draw  additional samples in accordance to the active sparsity at each level.

Similarly to (i), (ii) also presents a factor of $(s-1)/N$. However, instead of a dependence on the active sparsity in each level, the sampling is now dependent on  $F(\Lambda)$, the fineness prescribed by $\Lambda$. When the widths of the constant vector components of $\cN(\rP_\Lambda \rD)$ are large, such as in the case of signal 1 of Figure \ref{fig:1d_recons}, $F(\Lambda)$ can significantly smaller than sparsity.  So, this result suggest that the amount of subsampling possible is determined by $F(\Lambda)$ and the behaviour of $S(\Lambda, \cdot)$ - smaller values of $F(\Lambda)$ and $S(\Lambda, \cdot)$ will allow for more subsampling.
 
\end{itemize}

\subsubsection{Comparison with the recovery result from \cite{tv1}}\label{sec:comparison}
We first remark that the total variation norm considered in \cite{tv1} is with  periodic boundary conditions and this was crucial to the proofs in \cite{tv1}. In contrast, this results of this paper consider the total variation norm with  Neumann boundary conditions. However,  the Neumann boundary conditions are not crucial to the proof of Theorem \ref{thm:main_tv} and the analysis can be adapted with the periodic boundary condition case with the same results. Boundary conditions aside, we now compare the result of Theorem \ref{thm:main_tv} with the results of \cite{tv1}.
In Theorem 2.1 of \cite{tv1}, it was shown that if $\Omega$ is chosen uniformly at random with \bes{
m \gtrsim  s\cdot \log \left(\epsilon^{-1}\right)\cdot \log\left(N\right),
}
then with probability exceeding $1-\epsilon$, any minimizer $\hat x$ of 
(\ref{eq:min_tv}) satisfies
\bes{
\frac{\nm{x-\hat x}_2}{\sqrt{N}}
\lesssim  \log^{1/2}(m) \log(s)\cdot \left(  \delta \cdot \frac{\sqrt{s}}{\sqrt{q}}+ \nm{\rP_\Delta^\perp\rD x}_1\right).
}
The result of Theorem \ref{thm:main_tv} improves upon this error bound by a factor of $\log^{1/2}(m) \log(s)$, furthermore, if we let $M_k = k\cdot N\cdot r^{-1}$, then since $S(\Lambda, \cdot ) \leq s$ and $F(\Lambda) \leq s$, if we consider the worst case scenario, and let $F(\Lambda) =s$ and $S(\Lambda, \max\br{k-1,1}) = s$ for each $k$, then  condition (i) is satisfied if
$$
m_k \gtrsim (\log(s\epsilon^{-1})+1) \cdot \log(q^{-1}N^{3/2} \cdot \sqrt{s}) \cdot \left(\frac{s}{\max\br{k-1,1}} + \left(\frac{k}{\max\br{k-1,1}}\right)\frac{s-1}{r}\right),
$$
and as a rather crude estimate, condition (ii) is satisfied if each of its summand is no greater that $r^{-1}$ and this is implied by
$$
m_k \gtrsim (\log(s\epsilon^{-1})+1) \cdot \log(q^{-1}N^{3/2} \cdot \sqrt{s}) \cdot \left( \frac{r^2 \cdot s}{\max\br{(k-1)^2,1}} + \left(\frac{k}{\max\br{k-1,1}}\right)^2 \cdot (s-1)\right)
$$
and the total number of samples prescribed is
$$
m\geq C \cdot s \cdot (\log(s\epsilon^{-1})+1) \cdot \log(q^{-1}N^{3/2} \cdot \sqrt{s})
$$
where the constant $C$ depends only on the number of levels $r$. Although the total number of samples does not improve upon the estimate from \cite{tv1}, the analysis reveals how the sampling pattern should depend on the sparsity structure of the underlying signal. Note also that this is a worst case estimate on the number of samples, since the terms $S(\Lambda, \max\br{k-1,1})$ and $F(\Lambda)$ can be much smaller than $s$ as demonstrated in Section \ref{sec:structure_tv}. Finally, it is likely that the bounds of Theorem \ref{thm:main_tv} are not sharp and can be improved upon - the result still considers sampling from all $N$ Fourier frequencies, however, Theorem 2.3 from \cite{tv1} reveals that one only needs to sampling high frequencies to recover `fine' details. So, it is likely that the factors of $(s-1)/N$ can be removed from estimates (i) and (ii) above.

\subsection{The two dimensional case}

\subsubsection*{The multilevel sampling model}
Let $N\in\bbN$. For $r\in\bbN$,  let $\mathbf{\Gamma}= \br{\Gamma_k}_{k=1}^r$ be $r$ disjoint sets such that
$$\bigcup_{k=1}^r \Gamma_k=\br{-\lfloor N/2\rfloor, \ldots, \lceil N/2\rceil -1}^2$$ 
and let  $\mathbf{m}= \br{m_k}_{k=1}^r \in\bbN^r$  be such that  for each $k=1,\ldots, r$,  $0\leq m_k \leq \abs{\Gamma_k }$. Let $\Omega = \Omega_1\cup \cdots \cup \Omega_r$ be such that for each $k =1,\ldots, r$,
$$
\Omega_k \subset  \Gamma_k, \qquad \abs{\Omega_k} = m_k, 
$$
is drawn uniformly at random. 
We do not place any restrictions on how $\br{\Gamma_k}_{k=1}^r$ can be chosen, however, it is natural to group Fourier samples of similar frequencies into the same $\Gamma_k$ set.

Fix $x\in\bbC^{N\times N}$. Suppose we have an index set $\Lambda\subset \br{1,\ldots, N}^2$, such that  there exists a partition of $\br{1,\ldots, N}^2$, $\br{I_j : j=1,\ldots, n}$  for which
$
\cN(\tilde \rP_\Lambda \rD) = \br{\mathbbm{1}_{I_j}: j=1,\ldots, n}.
$
Let $s = S(\Lambda, 1)$. The following theorem describes the impact of $\mathbf{\Gamma}$ and $\mathbf{m}$ on the success of recovering $x$ by solving (\ref{eq:min_tv}). We first introduce some notation.
For $k=1,\ldots, r$, let $$M^{\max}_k =  \max_{m\in\Gamma_k} \abs{m}, \qquad M^{\min}_k =  \max\br{\min_{m\in\Gamma_k} \abs{m},\, 1}.
$$
Let
$$
c_k = \max\br{1,\quad  \frac{\nm{\rD^* \tilde\rP_{\Lambda^c} \sigma}_1}{\nm{\rD^* \tilde\rP_{\Lambda^c} \sigma}_2^2}, \quad \mu_k\cdot M^{\min}_k }.
$$
where
\be{\label{def:2dcase_inc_W}
\mu_k = \mu(\rP_{\Gamma_k} \rA \rQ^\perp_{\cW_\Lambda} (\tilde\rP_\Lambda \rD)^\dagger), \qquad \cW_\Lambda = \mathrm{span}\left(\br{\rD^* \tilde\rP_{\Lambda^c} \sigma} \cup \cN(\tilde\rP_\Lambda \rD)\right),}
and
\be{\label{eq:sigma_2dtv}
 \sigma =  \left(\frac{(\rD x)_{j_1, j_2}}{\abs{(\rD x)_{j_1, j_2}}}\right)_{(j_1,j_2)\in\br{1,\ldots, N}^2}.
}

\begin{theorem}\label{thm:tv_2d}
Let $\epsilon>0$ and $$
\cL =( \log(s\epsilon^{-1})+1) \log(q^{-1} N^2 B \sqrt{s}), \qquad B = \nm{(\tilde \rP_\Lambda\rD)^\dagger}_{1\to 2} .$$
Suppose that $\mathbf{\Gamma}$ and $\mathbf{m}$ satisfy the following.

\item[(i)] For $k=1,\ldots, r$,

$$
\frac{m_k}{\abs{\Gamma_k}} \gtrsim \cL\cdot
c_k \cdot \left(\frac{S\left(\Lambda, M^{\min}_k \right)}{N\cdot M_k^{\min}}
+  \frac{s}{N^2} \cdot \frac{M^{\max}_k }{M_k^{\min}}\right).
$$

\item[(ii)] For $k=1,\ldots, r$,
$m_k \gtrsim \cL \cdot \hat m_k
$ with
$$
1 \gtrsim \sum_{k=1}^r \left(\frac{\abs{\Gamma_{k}}}{\hat m_k} -1\right) \cdot  c_k^2    \cdot \left( \frac{\hat s_k}{(M^{\min}_k)^2} +  \frac{s}{N^2} \cdot \left(\frac{M^{\max}_k }{M^{\min}_k}\right)^2\right)
$$
for any $\br{\hat s_k}_{k=1}^r$ such that
$$
\sum_{k=1}^r\hat s_k \leq F(\Lambda).
$$
Then, with probability exceeding $(1-\epsilon)$, any solution $\xi$ to (\ref{eq:min_tv_2d}) will satisfy
\bes{
\norm{\xi - x}_{2}  \lesssim  \nm{(\tilde\rP_{\Lambda}\rD)^\dagger}_{1\to 2} \cdot \left(\left( 1 +\sqrt{s}\cdot L\right) \cdot \frac{\delta}{\sqrt{q}} + \norm{\tilde \rP_{\Lambda} \rD x}_{1}\right),
}
with
$$
L= \sqrt{\frac{\log(\epsilon^{-1}) + \log_2(8N^2 \sqrt{s}q^{-1})}{\log_2(4N^2\sqrt{s }q^{-1})}} , \quad q=\min_{k=1}^r \frac{m_k}{\abs{\Gamma_k }}.
$$

\end{theorem}

\begin{remark}
There are a couple of loose ends to this theorem, as we are currently missing bounds for $c_k$ and $\nm{(\tilde\rP_\Lambda \rD)^\dagger}_{1\to 2}$. However, given an index set $\Lambda$ and test image $T$, it is possible to compute these quantities.
 As an example, in the case of the Logan-Shepp Phantom $x$ of dimension $N\times N$ with $N\leq 260$, by letting $\Lambda^c= \br{j: (\rD x)_j \neq 0}$, one can check computationally that $\nm{(\rP_\Lambda \rD)^\dagger}_{1\to 2} \leq 20$ and that $\nm{\rD^*\tilde\rP_{\Lambda^c}\sigma}_1\leq \nm{\rD^* \tilde\rP_{\Lambda^c}\sigma}_2^2$ and $\mu(\rP_{\br{j}} \rU) \leq 2/\abs{j}$ for $\rU = \rA \rQ_{\cW_\Lambda}^\perp (\tilde\rP_\Lambda \rD)^\dagger$. Thus, $c_k \lesssim M_k^{\min}/M_k^{\max}$ in the case of the Logan Shepp phantom. Figure \ref{fig:inc_phantom} demonstrates the decay in the coherence of $\rU$ for $N=160, 260$.  So, it is likely that these quantities are not detrimental to the result in practice. 

\end{remark}

\subsubsection{Interpretation of the result}
\begin{itemize}

\item[(a)] Observe the factors of $( M^{\min}_k )^{-1}$ on the right hand sides of the inequalities in (i) and (ii).  So, this suggest that the amount of subsampling in $\Gamma_k$ will be inversely proportional to  $M_k^{\min}$, suggesting the need for denser sampling in sampling domains which contain lower Fourier frequencies.

\item[(b)] To understand the dependence on the sparsity structure imposed by $\Lambda$, we first consider condition (i). This suggests that the fraction of samples we should drawn from the $k^{th}$ level $\Gamma_k$ is (up to log factors and ratios between the absolute values of consecutive levels and $c_k$) 
$$\frac{s}{N^2} + \frac{S(\Lambda, M^{\min}_k )}{N\cdot   M^{\min}_k }.
$$
By the discussion in \ref{sec:structure_tv}, $S(\Lambda, p)$ can be understood as a measure of the fine details up to  $p$ and decreases as $p$ increases. So, (i) suggests that for stable recovery, it is sufficient to take $s$ samples randomly across all samples, then at each level, sample in accordance to the  fine details in the underlying image. Recalling the active sparsity graphs from Figure \ref{fig:active_sparsity}, $S(\Lambda, p)$ decays faster in $p$ for image 1 than image 2 and so, condition (i) suggests that for accurate recovery, the number of samples required at high frequencies for image 1 will be significantly smaller than the number of samples required at high frequencies for image 2.

 Condition (ii) presents a similar message to condition (i), but instead of the dependence on active sparsities, it is dependent on the fineness prescribed by $\Lambda$, $F(\Lambda)$. If the constant components which make up the space $\cN(\rP_\Lambda \rD)$ are large in their areas relative to their perimeters then $F(\Lambda)$ can be significantly smaller than sparsity. Thus, allowing for increased subsampling.
\end{itemize}

\begin{figure}[ht] 
\centering
$\begin{array}{cc}
\includegraphics[ trim=0.8cm 0.2cm 0.8cm 0.2cm,clip=true,width=0.45\textwidth]{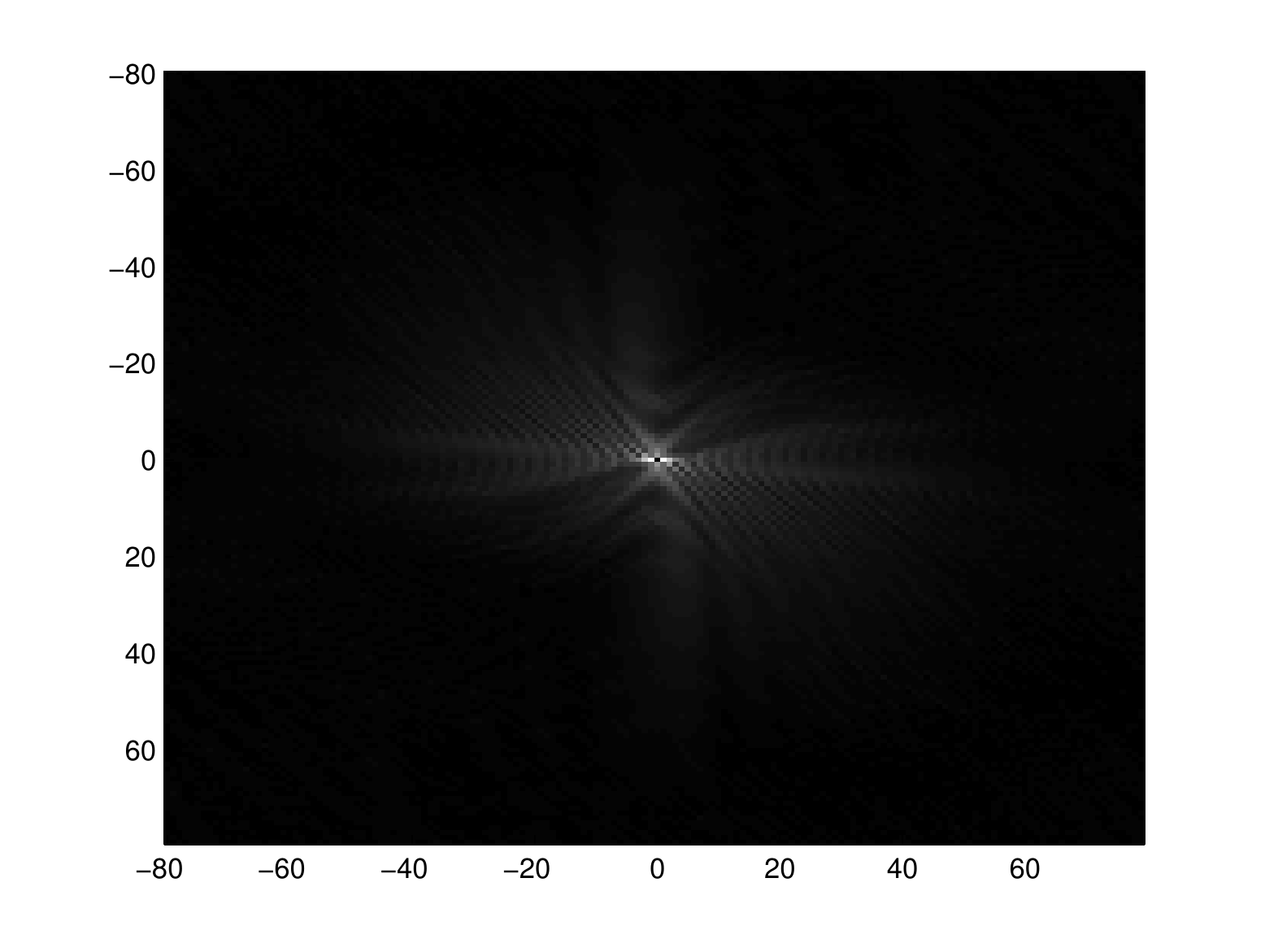} &
\includegraphics[ trim=0.8cm 0.2cm 1cm 0.2cm,clip=true,width=0.45\textwidth]{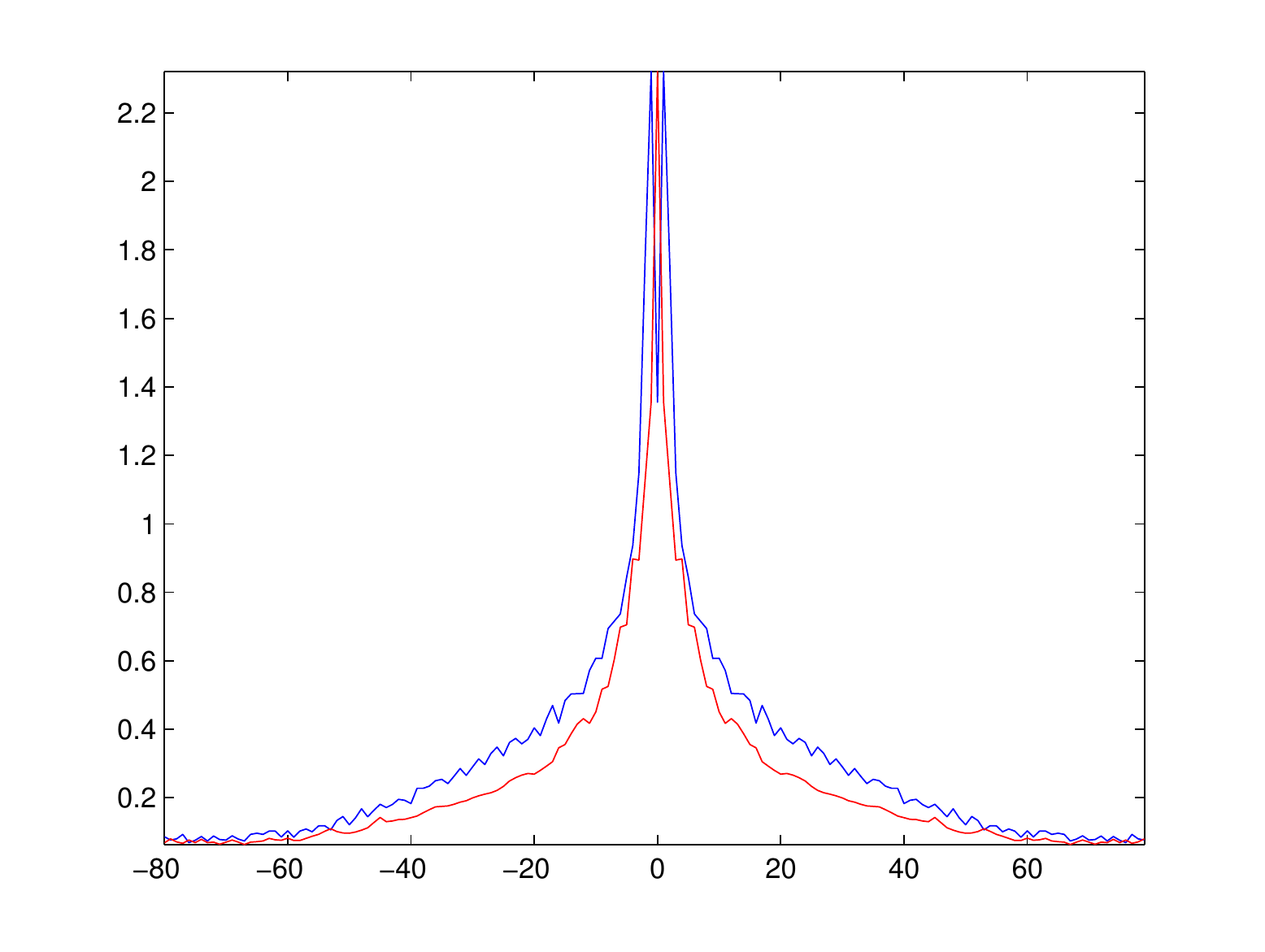} \\
\includegraphics[ trim=0.8cm 0.2cm 0.8cm 0.2cm,clip=true,width=0.45\textwidth]{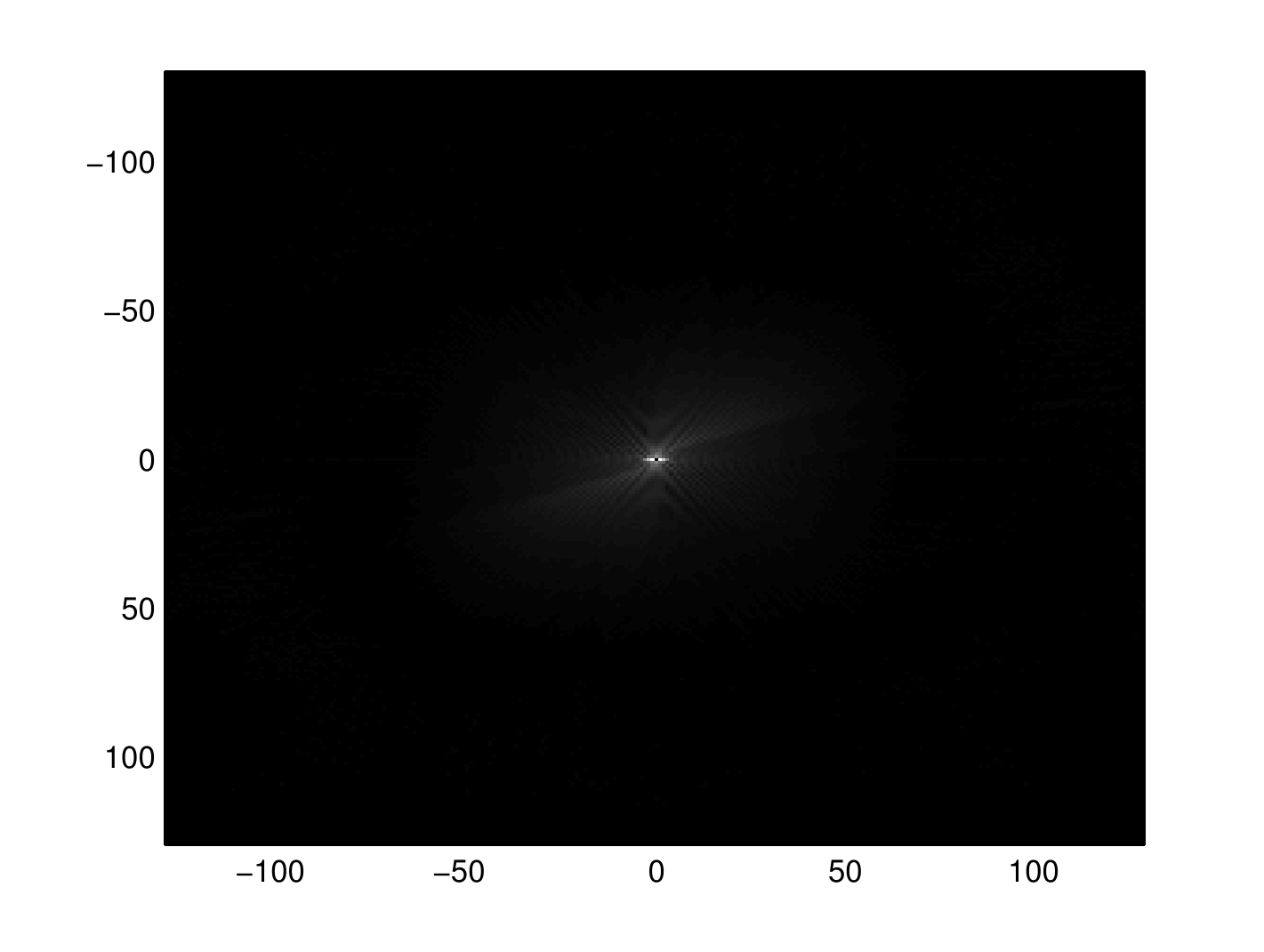} &
\includegraphics[ trim=0.8cm 0.2cm 1cm 0.2cm,clip=true,width=0.45\textwidth]{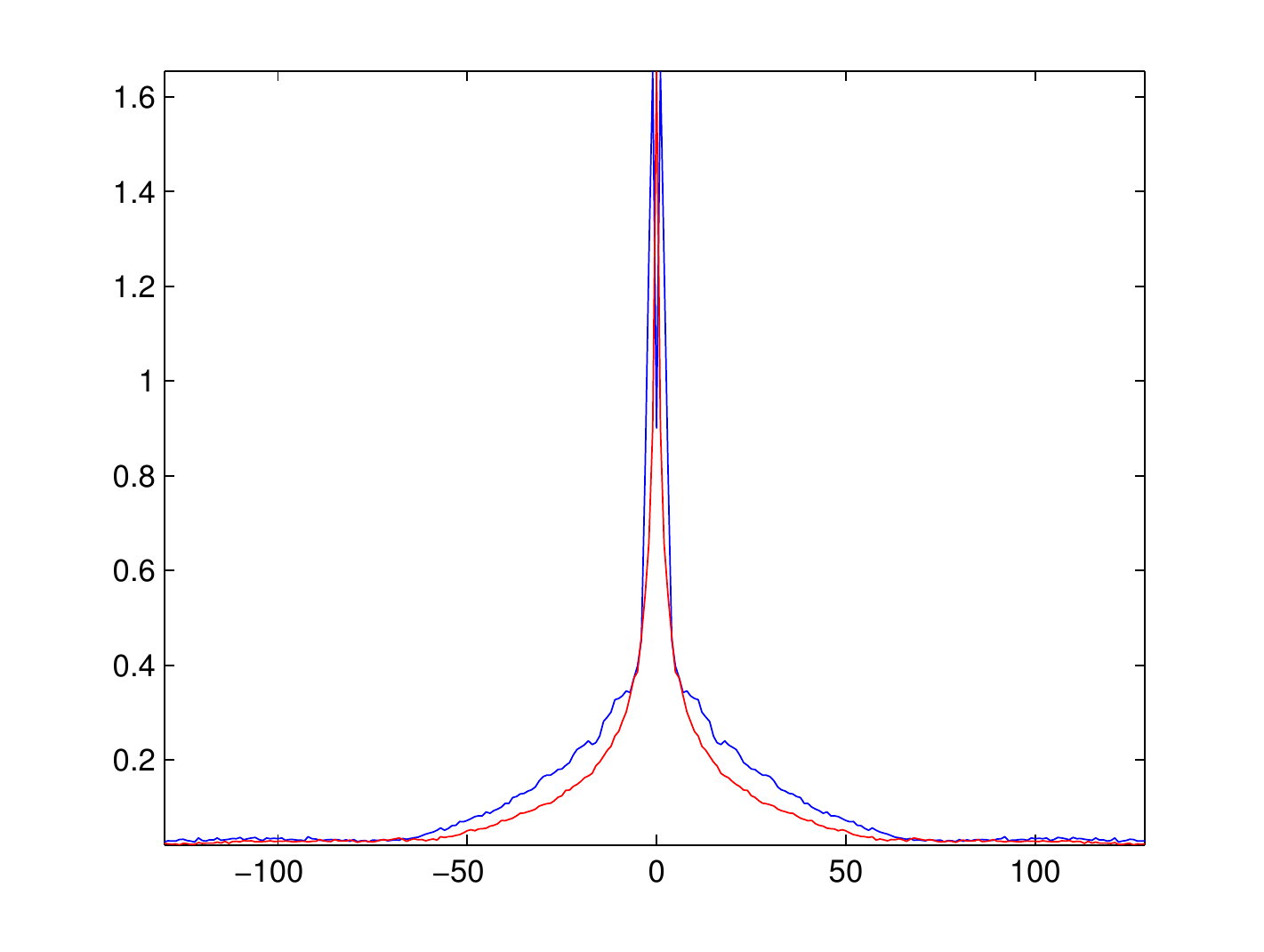} 
\end{array}$
\caption{ 
 for the Logan-Shepp phantom of size $N\times N$, the left plots show the absolute value of the matrix entries in $\rU = \rA \rQ_{\cW_\Lambda}^\perp (\tilde\rP_\Lambda \rD)^\dagger$, where $\Lambda$ is the index set on which $\rD x = 0$ and $\cW_\Lambda$ is as defined in (\ref{def:2dcase_inc_W}). The right plots show $\max_{j} \abs{\rU_{k,j}}$ for $k=-N/2, \ldots, N/2-1$ (red) and $\max_{k} \abs{\rU_{k,j}}$ for $j=-N/2, \ldots, N/2-1$ (blue).
 The top plots are for $N=160$ and the bottoms plots are for $N=260$.
\label{fig:inc_phantom}}
\end{figure}

\section{A general compressed sensing framework}
In order to understand the performance of (\ref{eq:min_tv}) and (\ref{eq:min_tv_2d}), we will analyse a more general problem. 
Let $\rA$ and $\rD$ be bounded linear operators on $\ell^2(\bbN)$, and for $x\in \ell^2(\bbN)$, $\delta\geq 0$ and $\Omega\subset \bbN$, suppose we are given $y$ such that $\norm{\rP_\Omega \rA x - y}_2\leq \delta$ where $\rP_\Omega$ is the projection onto the canonical basis indexed by $\Omega$.

We will consider how the choice of $\Omega$ affects  minimizers (if they exist) of the following  problem. 
\be{\label{eq:prob}
\inf_{z : \rD z \in \ell^1(\bbN)} \norm{  \rD z}_1 \text{ subject to } \norm{\rP_\Omega \rA z - y}_2 \leq \delta.
}

The minimization problem (\ref{eq:prob}) which we analyse concerns operators over a Hilbert space, and since the total variation minimization problem  in this paper is finite dimensional, we will only require a corollary of the main result in this section. The reason for carrying out analysis for this more complicated problem is that  there are inverse problems which are better modelled in a Hilbert space setting and where an understanding of the solutions to this infinite dimensional minimization problem will be required. We refer to \cite{BAACHGSCS,adcockbreaking} for further details.

We also remark that without imposing additional restrictions on $\rD$, it is not clear that minimizers of (\ref{eq:prob}) exist. Moreover even when minimizers exist, in practice, one strives to find approximate minimizers instead of exact minimizers. So, we will instead consider approximate minimizers of (\ref{eq:prob}).
\begin{definition}[Approximate minimizers, \cite{cai2012image}]\label{def:approx_min}
Let $b\geq 0$. Let $\cW\subset \ell^2(\bbN)$ and $F: \ell^2(\bbN) \to \bbR$.
We say that $\xi\in \cW$ is a $b$-optimal solution to 
$
\inf_{z \in \cW} F(z)
$
if  
$
F(\xi) \leq b + \inf_{z \in \cW} F(z).
$

\end{definition}

We  introduce two definitions which will guide our analysis of (\ref{eq:prob}).

 \defn{[Multi-level sampling]
\label{multi_level_dfn}
Let $r \in \bbN$, $\mathbf{M} = (M_1,\ldots,M_r) \in \bbN^r$ with $1 \leq M_1 < \ldots < M_r$, $\mathbf{m} = (m_1,\ldots,m_r) \in \bbN^r$, with $m_k \leq M_k-M_{k-1}$, $k=1,\ldots,r$, and suppose that
\bes{
\Omega_k \subseteq \{ M_{k-1}+1,\ldots, M_{k} \},\quad | \Omega_k | = m_k,\quad k=1,\ldots,r,
}
are chosen uniformly at random, where $M_0 = 0$.  We refer to the set
\bes{
\Omega = \Omega_{\mathbf{M},\mathbf{m}} := \Omega_1 \cup \ldots \cup \Omega_r.
}
as an $(\mathbf{M},\mathbf{m})$-multilevel sampling scheme.
}
To understand the role of variable density  sampling as mentioned, we will specifically consider $\Omega$ as a multi-level sampling scheme. Typical compressed sensing statements relate to the number of samples chosen uniformly at random across all available samples. Instead, we will derive statements relating to the randomness in the samples chosen is non-uniform across the available samples.

\begin{definition}[Cosparsity]
Given any $\rD \in \cB(\ell^2(\bbN))$ and $x\in \ell^2(\bbN)$, the \textit{cosparse set} is the index set $\Lambda$ for which $\rP_\Lambda \rD x =0$.
If the dimension of $\mathrm{Null}(\rP_\Lambda \rD)$ is $s$, then we say that $x$ is $s$-cosparse.
\end{definition}
This notion of cosparsity was introduced in \cite{nam2012cosparse}. 
Much of the theoretical results in compressed sensing has been about recovery for minimization problems of the following form.
\bes{
\inf_{ \eta \in \ell^1(\bbN)} \norm{   \eta}_1 \text{ subject to } \norm{\rP_\Omega \rA \rD^* \eta - y}_2 \leq \delta.
}
Such a minimization problem promotes sparsity in the synthesis coefficients of the underlying signal $x$ with respect to $\rD$, so recovery statements are generally for the recovery of $x = \rD^* z$ such that the synthesis coefficients  $z$ is sparse.

However, in (\ref{eq:prob}), we are minimizing the analysis coefficients with respect to $\rD$, and in the case where the rows of $\rD$ are highly redundant,  one cannot expect $\rD x$ to have many zero entries unless $x=0$. So, instead of sparsity, we  characterize the information content with respect to the zeros of $\rD x$. This is the idea behind cosparsity.

\subsection{Main results II}\label{general_main_results}
Let $\rA, \rD \in\cB(\ell^2(\bbN))$ and fix $x \in\ell^2(\bbN)$. Assume that $\nm{\rA}_{2\to 2} = 1$. Given any $\Lambda \subset \bbN$, let $\cW_\Lambda \subset \ell^2(\bbN)$ be such that $\cW_\Lambda \supseteq \cN(\rP_\Lambda \rD)$ and assume that this is a subspace of dimension $s$. Let $\rW\in\cB(\ell^2(\bbN))$ be a matrix whose columns form an orthonormal basis for $\cW_\Lambda$.
Let  $\blam\in\bbR^{s}_{+}$ be such that 
\be{\label{eq:blam_cond}
 \norm{\blam\circ \rW^* \rD^* \rP_{\Lambda^c}\sgn( \rD x)}_\infty \leq 1.
 }
  Let $q\in (0,1]$, let  $B \geq \norm{(\rP_\Lambda\rD)^\dagger}_{1\to 2}$.

\begin{assumption}[Identifiability]
Assume that
\be{\label{eq:restr_Lambda}
\inf_{u\in\cN(  \rD^*\rP_{\Lambda}  )}\norm{( \rD^* \rP_\Lambda )^\dagger \rQ_{\cW_\Lambda}^\perp \rD^* \rP_{\Lambda^c}   \sgn(\rD x)  -u}_\infty\leq \frac{1}{16}
} 
\end{assumption}

\begin{assumption}[Balancing properties]\label{assu:balancing}
We  assume that there exists $\rX \in\cB(\ell^2(\bbN))$ which is invertible and $M\in\bbN$ such that  the following holds.  
\be{\label{eq:restr_X}
\nm{\rQ_{\cW_\Lambda} \rA^* \rP_{[M]} \rA \rX \rQ_{\cW_\Lambda} - \rQ_{\cW_\Lambda}}_{2 \to 2 } \leq \frac{1}{8},
}
\be{\label{eq:restr_X2}
\nm{\blam\circ \rW^*\rA^* \rP_{[M]}\rA  \rX \rW \circ\blam^{-1} - \rW^* \rW}_{\infty\to \infty} \leq \left(4\log_2^{1/2}(4C_* M\sqrt{s}/q)\right)^{-1},
}
and
\be{\label{eq:restr_D}
\norm{(\rD^* \rP_\Lambda)^\dagger \rQ_{\cW_\Lambda}^\perp \rA^* \rP_{[M]} \rA \rX \rW\circ\blam^{-1}}_{\infty \to \infty} \leq \frac{1}{16}
}
where 
$C_* = \nm{\blam^{-1}}_\infty \cdot \max\br{1, B \nm{\rX}_{2\to 2}}$.
\end{assumption}
We also define $ Y = \min\br{  q \cdot( 8\nm{\blam^{-1}}_\infty \sqrt{s}\cdot \norm{\rX \rA^* \rP_{[M]}}_{2\to 2})^{-1}, \,  B\cdot \sqrt{q} },$ and
\bes{
\tilde M =\min\br{i\in\bbN : \max_{k\geq i}  \norm{\rP_{[M]} \rA \rQ_{\cW_\Lambda}^\perp (\rP_\Lambda\rD)^\dagger e_i}_2 \leq Y}.
}
Then, we have the following result.
\thm{\label{thm:general}
Fix $\epsilon>0$.
Let $\Omega =\Omega_{\mathbf{M},\mathbf{m}}$ be an $(\mathbf{M},\mathbf{m})$-multilevel sampling scheme and let $\Gamma_k = \br{M_{k-1}+1,\ldots, M_k}$. Suppose that $\mathbf{m}=(m_1, \hdots, m_r)\in\bbN^r$ is such that
$q=\min_{k=1}^r \frac{m_k}{M_{k}-M_{k-1}}$ and
satisfies the following.

\begin{itemize}
\item[(a)] 
For each $k=1,\ldots, r$,
\be{\label{eq:cond_m_k}
 q_k \gtrsim  (\log(s\epsilon^{-1} )+1) \log(q^{-1}\tilde M C_* \sqrt{s}) \cdot \mu_k \cdot  \sum_{j=1}^s   \mu(\rP_{\Gamma_k} \rA \rX \rW \rP_{\br{j}}) \cdot\blam^{-1}_j,  \qquad t=1,2
 }
 where
 $$\mu_k = \max\br{ \mu(\rP_{\Gamma_k} \rA \rQ_{\cW_\Lambda}^\perp(\rP_\Lambda \rD)^\dagger), \quad \mu( \rP_{\Gamma_k} \rA \rX \rW \circ\blam)}$$
and $q_k \gtrsim (\log(s\epsilon^{-1}) +1) \log(q^{-1}\tilde M C_* \sqrt{s}) \cdot \hat q_k $ such that $\br{\hat q_k}_{k=1}^r$ satisfies the following.
\be{\label{eq:hat_m}
1\gtrsim  \max_{\norm{\eta}_\infty =1}\sum_{k=1}^r (\hat q_k^{-1} -1)\cdot \mu(\rP_{\Gamma_k}\rA \xi   )^2 \cdot \norm{\rP_{\Gamma_k} \rA \rX \rW \circ\blam^{-1} \cdot \eta}_2^2,  
}
for all $\xi \in \br{\rQ_{\cW_\Lambda}^\perp(\rP_\Lambda \rD)^\dagger) e_i, \, \blam_j\rW e_j  : i\in \bbN, \, j=1,\ldots, s}$.

\item[(b)] For each $k=1,\ldots, r$,

$$
q_k \gtrsim \log\left(\frac{s}{\epsilon}\right) \cdot \left( \nm{\rX}_{2\to 2} + 1\right) \cdot \sum_{l=1}^r \mu_{k,l}^2
$$
where $\mu_{k,l} = \max\br{\mu(\rP_{\Gamma_k} \rA \rW \rP_{\br{l}}), \mu(\rP_{\Gamma_k} \rA \rX \rW \rP_{\br{l}})}$.

\item[(c)]  
$$
B^2 \gtrsim \log\left(\frac{\tilde M}{\gamma}\right) \cdot \max_{k=1}^r(q_k^{-1} -1) \cdot \mu(\rP_{\Gamma_k}\rA \rQ_{\cW_\Lambda}^\perp(\rP_\Lambda \rD)^\dagger    )^2.
$$
\end{itemize}

If the given samples $y$ in (\ref{eq:prob}) satisfy $\nm{\rP_\Omega \rA x - y}_2 \leq \delta$, then  with probability exceeding $(1-\epsilon)$, any $b$-optimal solution $\xi$ to (\ref{eq:prob}) satisfies
\spls{
\norm{\xi - x}_2  \lesssim B \cdot (\nm{\rX}_{2\to 2} +1) \cdot \left(  L  \cdot \frac{\delta}{\sqrt{q}} + \norm{\rP_\Lambda \rD x}_1 + b\right)
}
with
$$
L=\norm{\rX}_{2\to 2} +\sqrt{s}\norm{\blam^{-1}}_\infty \sqrt{\frac{\log(\epsilon^{-1}) + \log_2(8M C_*\sqrt{s}q^{-1})}{\log_2(4M C_*\sqrt{s}q^{-1})}}.$$

}

\begin{remark}
\begin{itemize}
\item[(i)]  $\blam$ is  only used as a tool to mathematically understand solutions to (\ref{eq:prob}).
We are often interested in cases where either the sparsifying operator $\rD$ or its pseudoinverse $(\rP_\Lambda\rD)^\dagger$ are of arbitrarily large norm. One notable example is the case where $\rD$ is the finite differences matrix.  

$\rW$ is also only a mathematical construct of this theory, and provides a concrete way of understanding the space $\cW_\Lambda$ and is in general not unique. As we demonstrate in our analysis of the total variation case, appropriate choices of $\rW$ and $\blam$ can provide insight as to how we should subsample.

\item[(ii)] Despite the norm restriction on $\rA$ in Theorem \ref{thm:general}, we can still apply the result of Theorem \ref{thm:general} to understand solutions to (\ref{eq:prob}) where the matrix $\rA$ is of arbitrary norm. This is because $f \in\ell^2(\bbN)$ is a solution  to (\ref{eq:prob}) with $\nm{\rA}_{2\to 2} \neq 1$, if and only if $f$ is a solution to
$$
\inf_{z: \rD z \in\ell^1(\bbN)} \nm{\rD z}_1 \text{ subject to  } \nm{\rP_\Omega \tilde \rA z - \tilde y}_2 \leq \tilde \delta
$$
where $\tilde \rA = \nm{\rA}_{2\to 2}^{-1}\cdot \rA$, $\tilde y = \nm{\rA}_{2\to 2}^{-1}\cdot y$ and $\tilde \delta = \nm{\rA}_{2\to 2}^{-1}\cdot \delta$.  
\item[(iii)] Instead of solving the constrained minimization problem (\ref{eq:prob}), practitioners often consider the following  unconstrained minimization problem. This is  especially  the case in the context of total variation minimization.
 $$
 \inf_{z: \rD z \in\ell^1(\bbN)} \nm{\rD z}_1+ \alpha \nm{\rA z - y}_2^2, \qquad \alpha >0.
 $$
 We simply remark here that the techniques used to derive recoverability conditions for (\ref{eq:prob}) can readily be extended to cover this unconstrained case. Appendix \ref{sec:unconstrained} outlines how this can be done.
\end{itemize}
\end{remark}

\subsection*{Discussion on the recovery result}

Suppose that $\rD$ is the identity operator. Then, we can let $\cW_\Lambda = \rP_{\Lambda^c}$, $\blam = \mathbbm{1}_{\Lambda^c}$  and  (\ref{eq:restr_Lambda}) trivially holds since the right hand side of (\ref{eq:restr_Lambda}) is simply zero. Suppose further that $\rA^* \rA$ is invertible on $\ell^2(\bbN)$. This is the case if $\rA = \left(\rA_{k,j}\right)_{k,j\in\bbN}$ where $\rA_{k,j} =  \ip{\psi_k}{\varphi_j}$ and $\br{\varphi_j}_{j\in\bbN}$ is a Riesz basis in $\ell^2(\bbN)$ and $\br{\psi_j}_{j\in\bbN}$ is  a frame if $\ell^2(\bbN)$. So, samples of a signal $x\in\ell^2(\bbN)$ are of the form $\left(\ip{x}{\psi_j}\right)_{j\in\Omega}$ and the $\ell^1$ regularization is on the Riesz basis coefficients with respect to $\br{\varphi_j}_{j\in\bbN}$ . Then, a natural choice for $\rX$ for the balancing properties condition is $(\rA^* \rA)^{-1}$ since
$$
\nm{\rP_{\Lambda^c} \rA^* \rP_{[M]} \rA (\rA^* \rA)^{-1} \rP_{\Lambda^c} - \rP_{\Lambda^c}}_{2\to 2} \to 0
$$
and
$$
\nm{\rP_{\Lambda} \rA^* \rP_{[M]} \rA  (\rA^* \rA)^{-1} \rP_{\Lambda^c}}_{\infty \to \infty}
$$
as $M\to \infty$. So, by letting $\rX = (\rA^* \rA)^{-1}$ and letting $M\in\bbN$ be sufficiently large, the conditions in Assumption \ref{assu:balancing} as satisfied.

Then, conditions (a), (b) and (c) reduce to the following condition.
\be{\label{eq:cond_m_k_synth}
\frac{m_k}{M_k-M_{k-1}} \gtrsim \cL \cdot \mu_k \cdot \sum_{j\in\Lambda^c} \max\br{ \mu( \rP_{\Gamma_k} \rA\rP_{\br{j}}), \mu(\rP_{\Gamma_k} \rA\rX  \rP_{\br{j}})} ,  \qquad k=1,\ldots,r
}
where
$$
\mu_k =
\max\br{\mu( \rP_{\Gamma_k} \rA\rP_{\Lambda^c}),\mu( \rP_{\Gamma_k} \rA \rX \rP_{\Lambda^c}),  \mu(\rP_{\Gamma_k} \rA  \rP_{\Lambda})}
$$
and $ m_k \gtrsim \cL \cdot \hat m_k$ with
\be{\label{eq:hat_m_synth}
1\gtrsim \max_{j}\br{ \max_{\norm{\zeta}_\infty =1 } \sum_{k=1}^r \left(\frac{M_k-M_{k-1}}{\hat m_k} -1\right)\cdot   \mu(\rP_{\Gamma_k} \rA \rP_{\br{j}})^2 \cdot  \norm{\rP_{\Gamma_k}\rA \rX \rP_{\Lambda^c}  \zeta}_2^2},
}
and the constants involved are dependent on  $\nm{\rX}_{2\to 2}$.

Both (\ref{eq:cond_m_k_synth}) and (\ref{eq:hat_m_synth}) show that the number of samples drawn from the $k^{th}$ level $\Gamma_k$ depends not on the global coherence values, but only on coherence values associated with $\Gamma_k$. Note that if $\mu_k\to 0$ as $k\to \infty$, then this would support the need for denser sampling at lower sampling levels. Furthermore, the structural dependence is seen in the summation term of (\ref{eq:cond_m_k_synth}) and the norm term of (\ref{eq:hat_m_synth}), which we can understand as being linked to the number of sparsifying vectors which are incoherent with the sampling vectors in level $k$.  In the special case where $\rA$ is an isometry, then $\rX$ the identity operator and Theorem \ref{thm:general} simply reduces to the main theorem of \cite{adcockbreaking}.

If $\rD$ is not the identity operator, then condition (\ref{eq:restr_Lambda}) is no longer trivial. However, this condition can be made to hold by setting $\cW_\Lambda$ to be a sufficiently  large subspace, although doing so will correspondingly force the cardinality of $\Omega$ to increase. A similar condition has also been proposed in \cite{vaiter2011robust}. Thus, in combination with conditions in Assumption \ref{assu:balancing}, these are conditions on the subspace on which one can guarantee robust recovery and on the range of samples which we draw from. Conditions (a), (b), and (c) reveal the dependence of the sampling density in $\Gamma_k$ on the coherence values associated with $\Gamma_k$ and on signal structure (realized through $\rW$) associated with $\Gamma_k$. We demonstrate in this paper how analysis on these conditions can provide insight into the success of total variation regularization. It is likely that similar techniques  will also prove fruitful for understanding the role of analysis regularization where $\rD$ is a frame operator.

\subsection{ A recovery result for an $\nm{\cdot}_{2,1}$ norm}
The isotropic total variation norm defined in Section \ref{sec:main_tv_results}  cannot be written as $\nm{\rD \cdot}_1$ for any linear operator $\rD$ and (\ref{eq:min_tv_2d}) does not fit into the framework of (\ref{eq:prob}). However, observe that $\frac{1}{\sqrt{2}}\nm{\cdot}_{TV,1}  \leq \nm{\cdot}_{TV} \leq \nm{\cdot}_{TV,1} $
where $\rD_1, \rD_2$ are the two dimensional finite differences operators defined in Section \ref{sec:main_tv_results} and $\nm{\cdot}_{TV,1}  = \nm{\rD_1 \cdot}_1 + \nm{\rD_2 \cdot}_1$ is typically known as the anisotropic total variation norm.
Due to the equivalence between the isotropic and anisotropic total variation norms, the recovery conditions for the isotropic total variation minimization problem of (\ref{eq:min_tv_2d}) can be shown to be up to a constant equivalent to those for the anisotropic total variation minimization problem which is covered by Theorem \ref{thm:general}. We have the following result.

\begin{theorem}\label{thm:general_21}

Let $\rA, \rD \in\cB(\ell^2(\bbN))$, and assume that $\nm{\rA}_{2\to 2} = 1$. Let $\delta\geq 0$, $x, y\in\ell^2(\bbN)$ 
be such that $\nm{\rA x - y}_2 \leq \delta$. Consider
\be{\label{eq:prob21}
\inf \norm{\rD z}_{2,1} \text{ subject to } \norm{\rA z - y}_2 \leq \delta
}
where
$$
\norm{z}_{2,1} := \sum_{i\in\bbN} \sqrt{\sum_{k\in \Delta_i} \abs{z_k}^2}
$$
and
$\br{\Delta_i: i\in\bbN}$ are finite disjoint subsets of $\bbN$ such that $\cup_{i\in\bbN}\Delta_i = \bbN$ and $\abs{\Delta_i} = 2$ for all $i\in\bbN$.

 Let $\Lambda = \cup_{i \in J} \Delta_i$ for some $J\subset \bbN$ and let $\rW$ and $\cW_\Lambda$ be defined as for Theorem \ref{thm:general}.  
Let $\sigma \in\ell^\infty(\bbN)$ be such that
$$
\rP_{\Delta_i} \sigma =  \frac{\rP_{ \Delta_i} \rD x}{\norm{\rP_{\Delta_i} \rD x}_2}, \qquad i\in\bbN.
$$ 
Then under the assumptions of Theorem \ref{thm:general} with
(\ref{eq:blam_cond}) and (\ref{eq:restr_Lambda})
replaced by
\bes{
\nm{\blam\circ \rW^* \rD^* \rP_{\Lambda^c}\sigma }_\infty \leq 1
}
and
\bes{
\inf_{u\in\cN(  \rD^*\rP_{\Lambda}  )}\norm{( \rD^* \rP_\Lambda )^\dagger \rQ_{\cW_\Lambda}^\perp \rD^* \rP_{\Lambda^c}   \sigma -u}_\infty\leq \frac{1}{16}
}
 respectively, any $b$-optimal solution $\xi$ to (\ref{eq:prob21}) satisfies
\spls{
\norm{\xi - x}_2  \lesssim  (1+ \nm{\rX}_{2\to 2}  )\cdot B \cdot \left(L \cdot \frac{\delta}{\sqrt{q}} + \norm{\rP_\Lambda \rD x}_{2,1} + b\right)
}
with
$$
L=\norm{\rX}_{2\to 2} +\sqrt{s}\cdot\norm{\blam^{-1}}_\infty \sqrt{\frac{\log(\epsilon^{-1}) + \log_2(8M C_*\sqrt{s}q^{-1})}{\log_2(4M C_*\sqrt{s}q^{-1})}}.$$

\end{theorem}

\section{Proof of Theorem \ref{thm:main_tv}}\label{sec:proof_tv1}
We assume throughout that $x\in\bbC^N$ and $\Lambda^c = \br{t_1,\ldots, t_{s-1}}\in\bbN^{s-1}$ with $1\leq t_1 < t_2 <\cdots < t_{s-1} <N-1$. Let $t_0 = 0$ and $t_s = N$.

\subsubsection*{Definition of $\cW_\Lambda$}
\begin{lemma}\label{lem:defn_1dtv_W}
Let $\cW_\Lambda = \mathrm{ span}\br{\eta_j \in\bbC^N: j=1,\ldots, s+1}$ where

$$
(\eta_k)_j = \begin{cases}
\frac{1}{\sqrt{t_{k}-t_{k-1}}} &j= t_{k-1}+1,\ldots, t_{k}\\
0 &\text{otherwise.}
\end{cases}, \qquad k=1,\ldots, s,$$
   and
\be{\label{eq:xi}
\eta_{s+1} = \begin{cases} \xi/\nm{\xi} &\xi \neq 0\\
0 &\xi = 0
\end{cases}, \qquad \xi =  \rD^* \rP_{\Lambda^c}   \sgn( \rD x)
- \sum_{j=1}^s \ip{\rD^* \rP_{\Lambda^c}    \sgn( \rD x)}{ \eta_j}\eta_j.
}
If $\xi \neq 0$, let $\rW \in \bbC^{N\times (s+1)}$ and $\blam \in\bbR_{>0}^{s+1}$ be defined as follows.
$$
\rW = (\eta_1|\eta_2|\hdots| \eta_{s+1}), \qquad
\blam_l = \begin{cases} 2^{-1}\cdot(t_l -t_{l-1})^{1/2} & 1\leq l \leq s\\
 \nm{\xi}^{-1} & l = s+1.
 \end{cases} 
$$
Otherwise, if $\xi  = 0$, then let 
$\rW \in \bbC^{N\times s}$ and $\blam \in\bbR_{>0}^{s}$ be defined as follows.
$$
\rW = (\eta_1|\eta_2|\hdots| \eta_s), \qquad
\blam_l =  2^{-1}\cdot(t_l -t_{l-1})^{1/2} ,\quad 1\leq l \leq s.
$$
Then 
\begin{itemize}
\item[(i)] $\cW_\Lambda \supset \cN(\rP_\Lambda \rD)$, $\rW$ is an isometry and
$
\inf_{v\in\cN(  \rD^*\rP_{\Lambda}  )}\norm{( \rD^* \rP_\Lambda )^\dagger \rQ_{\cW_\Lambda}^\perp  \rD^* \rP_{\Lambda^c}   \sgn(\rP_{\Lambda^c} \rD x)  -v}_\infty = 0,
$
\item[(ii)] For each $j=1,\ldots, s$, either $\nm{\rP_{\Delta_j} \xi}_2 = 0$ or 
$\nm{\rP_{\Delta_j} \xi}_2 \geq \sqrt{2}/3$, where $\Delta_j = \br{t_{j-1}+1,\ldots, t_j}$.
Furthermore, if $t_j - t_{j-1}=1$, then $\nm{\rP_{\Delta_j} \xi}_2=0$.
\item[(iii)]  $\norm{ \blam\circ \rW^*\rD^* \rP_{\Lambda^c} \sgn(\rP_{\Lambda^c} \rD x)}_\infty \leq 1$.
\end{itemize}
\end{lemma}
\begin{proof}
\item[(i)]
The first two assertions are simply by definition of $\rW$.  The third assertion follows because  $$
 \rD^* \rP_{\Lambda^c}   \sgn( \rD x)
= \xi +\sum_{j=1}^s \ip{\rD^* \rP_{\Lambda^c}    \sgn( \rD x)}{ \eta_j}\eta_j \in \cW_\Lambda.
$$

\item[(ii)]

Let $\zeta =  \sgn( \rP_{\Lambda^c}\rD x)$.
Then, $\rD^* \rP_{\Lambda^c} \zeta = (\alpha_j)_{j=1}^N$ where $\alpha_j = - \zeta_{j} + \zeta_{j-1}$ for $j=2,\ldots, n-1$, $\alpha_1 = -\zeta_1$ and $\alpha_N = \zeta_{N-1}$. Note that $\zeta_j = 0$ unless $j\in\Lambda^c$. If $t_j - t_{j-1}=1$, then $\Delta_j = \br{t_j}$ and $\rP_{\Delta_j} \rD^* \rP_{\Lambda^c} \zeta = 0$.
If $t_j - t_{j-1} = 2$, then
$$
\rP_{\Delta_j}\rD^* \rP_{\Lambda^c} \zeta
=\left(\zeta_{t_{j-1}},  -\zeta_{t_j}\right)^T, \qquad \ip{\rD^* \rP_{\Lambda^c} \zeta}{\eta_j} = \frac{\zeta_{t_{j-1}} - \zeta_{t_j}}{\sqrt{2}}
$$
and
$
\nm{\rP_{\Delta_j}\xi }_2^2 = \abs{ \frac{\zeta_{t_j} - \zeta_{t_{j-1}}}{2}}^2 + \abs{\frac{\zeta_{t_j} + \zeta_{t_{j-1}}}{2}}^2 = \frac{1}{2}.
$
If $t_j - t_{j-1} \geq 3$, then
$$
\rP_{\Delta_j}\rD^* \rP_{\Lambda^c} \zeta
=\left( \zeta_{t_{j-1}}, \, 0, \cdots,0,\, -\zeta_{t_j}\right)^T, \qquad \ip{\rD^* \rP_{\Lambda^c} \zeta}{\eta_j} = \frac{\zeta_{t_{j-1}} - \zeta_{t_j}}{\sqrt{t_j - t_{j-1}}}
$$
and
\eas{
\nm{\rP_{\Delta_j}\xi}_2^2 & =
\abs{\zeta_{t_j} -\zeta_{t_{j-1}}}^2 \cdot \frac{t_j -t_{j-1}-1}{t_j -t_{j-1}} + \abs{\zeta_{t_{j-1}} + \frac{\zeta_{t_j} - \zeta_{t_{j-1}}}{t_j -t_{j-1}}}^2 + \abs{-\zeta_{t_{j}} + \frac{\zeta_{t_j} - \zeta_{t_{j-1}}}{t_j -t_{j-1}}}^2\\
&\geq 2\left(1-\frac{2}{3}\right)^2 = \frac{2}{9}
}
So, 
$\nm{\rP_{\Delta_j} \xi }_2^2 $ is either zero  or at least $2/9$.

\item[(iii)]

Let $\zeta = \sgn(\rP_{\Lambda^c} \rD x)$.
For $j=1,\ldots, s$,
$$
\abs{\ip{\rW^*\rD^* \rP_{\Lambda^c} \zeta}{e_j}} = \frac{\abs{\zeta_{t_j}-\zeta_{t_{j-1}}}}{\sqrt{t_j - t_{j-1}}}
\leq \frac{2}{\sqrt{t_j - t_{j-1}}},
$$
so, $\abs{ \left(\blam\circ \rW^*\rD^* \rP_{\Lambda^c} \zeta\right)_j} \leq 1$.
Recalling the definition of $\xi$ from (\ref{eq:xi}),  if $\xi \neq 0$, then we have that
\eas{
\abs{\ip{\rW^*\rD^* \rP_{\Lambda^c} \zeta}{e_{s+1}}}
= \nm{\xi}_2^{-1} \ip{\xi}{\rD^* \rP_{\Lambda^c} \zeta}
= \frac{\nm{\xi}_2^2}{\nm{\xi}_2}
}
and since $\blam_{s+1} = \nm{\xi}_2^{-1}$,
$$\abs{\blam_{s+1} \cdot \ip{\rW^*\rD^* \rP_{\Lambda^c} \zeta}{e_{s+1}}} = 1.
$$

\end{proof}

\subsubsection*{On the pseudo-inverse $(\rP_\Lambda \rD)^\dagger$}
For $n\in\bbN$, let 
$\rD_n \in\bbR^{(n-1) \times n }$  such that
$$
\rD_n = \begin{pmatrix}
-1& +1& & & 0 \\
 & -1 & +1 \\
& &  \ddots & \ddots\\
0& & &-1 & +1
\end{pmatrix}.
$$
Then, one can check that the pseudoinverse of $\rD_n$ is
$$
\rD_n^\dagger  =  \begin{pmatrix}
-(n-1)/n& -(n-2)/n& \hdots & -1/n\\
1/n& -(n-2)/n& & -1/n\\
\vdots& 2/n& & \vdots\\
\vdots& \vdots& & \vdots\\
\vdots& \vdots& & -1/n\\
1/n& 2/n& \hdots & (n-1)/n
\end{pmatrix} \in\bbR^{n\times (n-1)}.
$$
Moreover, for $\rD:= \rD_N \in\bbR^{(N-1)\times N}$,
$$
\rP_\Lambda \rD =  \begin{pmatrix}[c|c|c|c]
   \tilde \rD_1 &\tilde \rD_2 &\hdots &\tilde \rD_s
  \end{pmatrix} \in\bbR^{(N-s)\times N}
$$
where 
$$
\tilde \rD_1 = \begin{pmatrix} \rD_{t_1}\\
0_{N-s-t_1+1\times t_1}
  \end{pmatrix}
   \quad
  \tilde \rD_s = 
  \begin{pmatrix} 
  0_{(t_{s-1} -s+1)\times (N-t_{s-1})}\\
  \rD_{N-t_{s-1}}
  \end{pmatrix} 
$$
$$
  \tilde \rD_k = 
  \begin{pmatrix} 0_{(t_{k-1} -k+1) \times (t_k - t_{k-1})}\\
  \rD_{t_k - t_{k-1}}\\
0_{(N-s-t_k+k)\times (t_k - t_{k-1})}
  \end{pmatrix}, \quad k=2,\ldots, s-1,
$$
 and $0_{m\times n}\in\bbR^{m\times n}$ whose entries are all zero.
Thus, by computing the pseudo-inverse of block matrices,
 $$
( \rP_\Lambda \rD )^\dagger =
\begin{pmatrix}
   \tilde \rD_1^\dagger \\\tilde \rD_2^\dagger \\\vdots \\\tilde \rD_s^\dagger
  \end{pmatrix}
 $$
 where 
$$
\tilde \rD_1^\dagger = \begin{pmatrix}[c|c] \rD_{t_1}^\dagger &
0_{t_1 \times (N-s-t_1+1)}
  \end{pmatrix}
   \quad
  \tilde \rD_s^\dagger = 
  \begin{pmatrix} [c|c]
  0_{ (N-t_{s-1})\times (t_{s-1} -s+1)}&
  \rD_{N-t_{s-1}}^\dagger
  \end{pmatrix} 
$$ 
 
 $$
\tilde \rD_k^\dagger = 
  \begin{pmatrix}[c|c|c] 
  0_{(t_k - t_{k-1})\times (t_{k-1} -k +1) } &  \rD_{t_k - t_{k-1}}^\dagger &
0_{(t_k - t_{k-1})\times (N-s-t_k+k)}
  \end{pmatrix}, \quad k=2,\ldots, s-1
$$
and it is straighforward to check that
\be{\label{lem:d_dag_norm}
\norm{(\rP_\Lambda \rD)^\dagger}_{1\to 2} \leq \sqrt{N}.
}

\subsubsection*{Incoherence estimates}
\begin{lemma} \label{lem:inc_tv}
Let $\cW_\Lambda$, $\rW$ and $\xi$ be as defined in Lemma \ref{lem:defn_1dtv_W}.

\begin{itemize}
\item[(i)] For $k=-N/2+1,\ldots, N/2$ and  $m=1,\ldots, s$,  
$$\abs{\ip{\rA\rW e_m}{e_k}} \lesssim \min\br{\frac{ \sqrt{t_m - t_{m-1}}}{\sqrt{N}},  \frac{\sqrt{N}}{ \abs{k}\cdot  \sqrt{t_m - t_{m-1}}}},$$

If $\xi \neq 0$, then
$$
\frac{\abs{\ip{\rA\rW e_{s+1}}{e_{k}}}}{\nm{\xi}_2}  \lesssim \frac{1}{\sqrt{N}}
$$
and
$$
\nm{\xi}_2\cdot \abs{\ip{\rA\rW e_{s+1}}{e_{k}}}
\lesssim (s-1) \cdot\frac{1}{\sqrt{N}}\cdot \frac{  \abs{k}}{N} + \sum_{j=1}^s \frac{1}{\sqrt{t_j - t_{j-1}}} \abs{\ip{\rA \eta_j}{e_k}}.
$$

\item[(ii)]
For  $j=1,\ldots, N$,
 \eas{
\abs{\ip{\rA (\rP_\Lambda \rD)^\dagger e_j}{e_k}} &\lesssim
\frac{\sqrt{N} }{ \abs{k} }
}
\item[(iii)] $$\abs{\left(\rA \rQ_{\cW_\Lambda}(\rP_\Lambda \rD)^\dagger \right)_{k,j}}
\lesssim \frac{1}{\sqrt{N}}
$$
and so, for $j\in\br{t_{m-1}-m+2, \ldots t_m-m+1}$,
\eas{
&\abs{\ip{\rA \rQ_{\cW_\Lambda}^\perp (\rP_\Lambda \rD)^\dagger e_j}{e_k}}  
\leq \abs{\ip{\rA (\rP_\Lambda \rD)^\dagger e_j}{e_k}}  + \abs{\ip{\rA \rQ_{\cW_\Lambda}  (\rP_\Lambda \rD)^\dagger e_j}{e_k}} \\
&\lesssim \min\br{ \frac{\sqrt{N} }{\abs{ k} }, \frac{(t_m - t_{m-1})}{\sqrt{N}}}
}
\end{itemize}
\end{lemma}

\prf{

\item[(i)]

For $m=1,\ldots, s$ and $k \in\br{ -N/2+1,\ldots, N/2}\setminus \br{0}$ ,  
\eas{
\ip{\rA\rW e_m}{e_k}& = 
 \ip{\eta_m}{\rA^* e_k} = \frac{1}{\sqrt{N}\sqrt{t_m - t_{m-1}}}  \sum_{l=t_{m-1}+1}^{t_m} e^{2\pi i lk/N}\\
 & =  \frac{1}{\sqrt{N}\sqrt{t_m - t_{m-1}}} e^{2\pi i (t_{m-1}+1)k/N} \frac{1-e^{2\pi i (t_m - t_{m-1})k/N}}{1-e^{2\pi i k/N}} \\
 &=  \frac{1}{\sqrt{N}\sqrt{t_m - t_{m-1}}} e^{\pi ik/N} e^{\pi i (t_m + t_{m-1})k/N} \frac{\sin\left(\frac{\pi (t_m - t_{m-1})k}{N}\right)}{\sin\left(\frac{\pi k}{N}\right)}.
 }
 By Lemma \ref{lem:sin_property}, if $\abs{k}\leq N/4$, then
\eas{
\abs{\ip{\rA\rW e_m}{e_k}} &\leq  \frac{1}{\sqrt{N}\sqrt{t_m - t_{m-1}}}\frac{\min\br{1, \pi (t_m-t_{m-1}) \abs{k}/N}}{\abs{\sin\left(\frac{\pi k}{N}\right)}}\\
&\leq \sqrt{2} \min\br{\frac{\sqrt{t_m - t_{m-1}}}{\sqrt{N}}, \frac{\sqrt{N}}{\pi \abs{k}\sqrt{t_m - t_{m-1}}}}
}
and if $N/2 \geq \abs{k} > N/4$, then
\eas{
\abs{\ip{\rA\rW e_m}{e_k}} &\leq  \frac{\sqrt{2}}{\sqrt{N}\sqrt{t_m - t_{m-1}}}\min\br{1, \pi (t_m-t_{m-1}) \abs{k}/N}\\
&\leq \sqrt{2} \min\br{\frac{\pi\sqrt{t_m - t_{m-1}}}{2\sqrt{N}}, \frac{\sqrt{N}}{2 \abs{k}\sqrt{t_m - t_{m-1}}}}
}
where we have used the assumption that $\abs{k}/N \leq \frac{1}{2}$.
If $k=0$, then  $\abs{\ip{\rA\rW e_m}{e_k}} \leq \frac{\sqrt{t_m - t_{m-1}}}{\sqrt{N}}.$
Thus,
\be{\label{eq:inc_i}
\abs{\ip{\rA\rW e_m}{e_k}} \lesssim \min\br{\frac{ \sqrt{t_m - t_{m-1}}}{\sqrt{ N}},  \frac{\sqrt{N}}{\abs{k}\cdot \sqrt{t_m - t_{m-1}}}}
}

For the second part of (i), recalling the definition of $\eta_{s+1}$ and $\xi$ from (\ref{eq:xi}), 
\eas{
\abs{\ip{\rA\rW e_{s+1}}{ e_k}}
=  \frac{\abs{\ip{\rA \xi}{e_k}}}{\nm{\xi}_2}.
}
For $j=1,\ldots, s$, let $\Delta_j \br{t_{j-1}+1,\ldots, t_j}$. If $t_j = t_{j-1}+1$, then $\rP_{\Delta_j}\xi = 0$ by (ii) of Lemma \ref{lem:defn_1dtv_W}. Assuming that $t_{j} \neq t_{j-1}+1$, it follow that
$$
\abs{\ip{\rA \rP_{\Delta_j} \xi}{e_k}}\leq 
\abs{\ip{\rA e_{1+ t_{j-1}}}{e_k}} + \abs{\ip{\rA e_{ t_{j}}}{e_k}}
+ \frac{2}{\sqrt{t_j -t_{j-1}}} \abs{\ip{\rA \eta_j}{e_k}}
\leq \frac{4}{\sqrt{N}}.
$$
Furthermore, by Lemma \ref{lem:defn_1dtv_W}, either $\nm{\rP_{\Delta_j}\xi}_2=0$ or  $\nm{\rP_{\Delta_j}\xi}_2^2\geq \frac{2}{9}$. So, if $\abs{\br{j: \nm{\rP_{\Delta_j}\xi}_2 \neq 0}} = p$, then $\frac{1}{\nm{\xi}_2^2} \leq \frac{9}{2p}$.
Thus,
\be{\label{eq:inc_xi}
\frac{\abs{\ip{\rA\rW e_{s+1}}{ e_k}}}{\nm{\xi}_2} \leq \frac{1}{\nm{\xi}_2^2}\sum_{j=1}^s\abs{\ip{\rA \rP_{\Delta_j} \xi}{e_k}} \leq \frac{9}{2p}\cdot \frac{4p}{\sqrt{N}} \leq \frac{18}{\sqrt{N}}.
}
Finally,
\eas{
\nm{\xi}\abs{\ip{\rA\rW e_{s+1}}{ e_k}}
&=  \abs{\ip{\rA \xi}{e_k}}
\leq \sum_{j=1}^{s-1}\abs{\ip{\rA \rD^* e_{t_j}}{e_k}} + \sum_{j=1}^s \frac{2}{\sqrt{t_j - t_{j-1}}} \abs{\ip{\rA \eta_j}{e_k}}.
}
Since
$\ip{\rA \rD^* e_{t_j}}{e_k} = \frac{\sqrt{2}i}{\sqrt{N}} e^{2 \pi i k (t_j+1)/N} \sin(\pi k /N)$
we have that
\eas{
\nm{\xi}\abs{\ip{\rA\rW e_{s+1}}{ e_k}}
&=  \abs{\ip{\rA \xi}{e_k}}
\leq (s-1) \cdot\sqrt{\frac{2}{N}}\min\br{\frac{\pi \abs{k}}{N}, 1} + \sum_{j=1}^s \frac{2}{\sqrt{t_j - t_{j-1}}} \abs{\ip{\rA \eta_j}{e_k}}.
}

\item[(ii)] For $\abs{k}\geq 1$,
\eas{
\abs{\ip{\rA (\rP_{\Lambda} \rD)^\dagger e_j}{e_k}} \leq  \frac{\sqrt{N}\cdot \nm{\rD (\rP_{\Lambda} \rD)^\dagger e_j }_1}{\sqrt{2}\abs{k}} = \frac{\sqrt{2} \sqrt{N}}{ \abs{k}}.
}
where the first inequality follows by Lemma \ref{lem:ft_decay_1d} and the equality is true since it is straightforward to verify that $\nm{\rD (\rP_{\Lambda} \rD)^\dagger e_j }_1 \leq 2$.
Also, for $k=0$,
\eas{
\abs{\ip{\rA (\rP_{\Lambda} \rD)^\dagger e_j}{e_k}} \leq \nm{(\rP_{\Lambda} \rD)^\dagger e_j}_1 \leq \sqrt{N}
}
by (\ref{lem:d_dag_norm}).

\item[(iii)]
\be{\label{eq:lem_iii_inc_a}
\rA \rW \rW^* (\rP_\Lambda \rD)^\dagger e_j
= \rA \left(\sum_{l=1}^{s+1}  \ip{\eta_l}{(\rP_\Lambda \rD)^\dagger e_j}\eta_l\right)
= \ip{\eta_{s+1}}{(\rP_\Lambda \rD)^\dagger e_j}\rA \eta_{s+1}
}
since for $l=1,\ldots, s$, $\eta_l \in\cN(\rP_\Lambda \rD) = \cR((\rP_\Lambda \rD)^\dagger)^\perp$. Recall that $\eta_{s+1 } = \frac{\xi}{\nm{\xi}}$.
By definition of $\xi$ and again using the fact that $l=1,\ldots, s$, $\eta_l \in\cN(\rP_\Lambda \rD) = \cR((\rP_\Lambda \rD)^\dagger)^\perp$, 
\be{\label{eq:lem_iii_inc_b}
\ip{\xi }{(\rP_\Lambda \rD)^\dagger e_j} 
= \ip{\rD^* \rP_{\Lambda^c} \sgn(\rP_{\Lambda^c}\rD x)}{(\rP_\Lambda \rD)^\dagger e_j} \leq \nm{\rD (\rP_\Lambda \rD)^\dagger e_j}_1 \leq 2.
}
Recall from (\ref{eq:inc_xi}) that for each $k = -N/2+1,\ldots, N/2$,
$$
\frac{\abs{\ip{\rA \xi}{e_k}}}{\nm{\xi}^2} \leq \frac{12}{\sqrt{N}}.
$$
Thus, by applying this with (\ref{eq:lem_iii_inc_a}) and (\ref{eq:lem_iii_inc_b}), we have that for each $k = -N/2+1,\ldots, N/2$
$$
\abs{\ip{\rA \rW \rW^* (\rP_\Lambda \rD)^\dagger e_j}{e_k}} \leq \frac{24}{\sqrt{N}}.
$$

}

\subsubsection*{Norm estimates}

\begin{lemma}\label{lem:rel_sparsity_tv}
Let $\bigcup_{k=1}^r\Gamma_k = \br{1,\ldots, N}$ be the union of $r$ disjoint subsets and let $M_k = \max_{m\in\Gamma_k}\abs{m}$.
Let $\rW$, $\blam$, $\xi$ be as defined in Lemma \ref{lem:defn_1dtv_W}.
If $\xi \neq 0$, then given $\alpha \in \bbC^{s+1}$ be such that $\nm{\alpha}_\infty =1$, there exists $\br{\hat s_k}_{k=1}^r \in\bbR_+^r$ such that
\be{\label{eq:rel_sparsity1d_a}
\norm{\rP_{\Gamma_k} \rA \rW \circ\blam^{-1}\cdot \alpha}_2^2 \lesssim
 \hat s_k +
 \left(\frac{M_k}{N}\right)^2 \cdot (s-1), \qquad \sum_{k=1}^r\hat s_k \leq \sum_{j=1}^s \frac{1}{t_j-t_{j-1}}.
 }
If $\xi = 0$, then given $\alpha \in \bbC^{s}$ be such that $\nm{\alpha}_\infty =1$, there exists $\br{\hat s_k}_{k=1}^r \in\bbR_+^r$ such that
\be{\label{eq:rel_spasity1d_b}
\norm{\rP_{\Gamma_k} \rA \rW \circ\blam^{-1}\cdot \alpha}_2^2 \lesssim
 \hat s_k , \qquad \sum_{k=1}^r\hat s_k \leq \sum_{j=1}^s \frac{1}{t_j-t_{j-1}}.
 }

\end{lemma}

\begin{proof}
First assume that $\xi\neq 0$.
By letting $\zeta = \sgn(\rP_{\Lambda^c} \rD x)$ and recalling the definition of $\rW$ and $\blam$ from Lemma \ref{lem:defn_1dtv_W}, we have that
\spl{\label{eq:lem:rel_sparsity_tv_bd1}
\nm{(\rP_{\Gamma_k} \rA \rW \circ\blam^{-1} ) \alpha}_2^2
&= \nm{\rP_{\Gamma_k} \rA \left(\sum_{j=1}^s \frac{2\alpha_j}{\sqrt{t_j - t_{j-1}}} \eta_j + \alpha_{s+1} \rD^* \rP_{\Lambda^c} \zeta - \sum_{j=1}^s \frac{\alpha_{s+1}(\zeta_{t_{j-1}} - \zeta_{t_j})}{\sqrt{t_j - t_{j-1}}} \right)}^2_2\\
&\leq 2\nm{\sum_{j=1}^s \frac{2\tilde \alpha_j}{\sqrt{t_j - t_{j-1}}} \rP_{\Gamma_k} \rA  \eta_j}_2^2 +2 \nm{\rP_{\Gamma_k} \rA  \rD^* \rP_{\Lambda^c}\zeta}_2^2 \\
&= 32 \hat s_k + +2 \nm{\rP_{\Gamma_k} \rA  \rD^* \rP_{\Lambda^c}\zeta}_2^2,
}
where $\nm{\tilde{ \alpha}}_\infty \leq 2$ and
$$
\hat s_k = \frac{1}{16} \nm{\sum_{j=1}^s \frac{2\tilde \alpha_j}{\sqrt{t_j - t_{j-1}}} \rP_{\Gamma_k} \rA  \eta_j}_2^2.
$$
Then, since $\bigcup_{k=1}^r \Gamma_k = \br{1,\ldots, N}$ and $\rA$ is  unitary,
$$
\sum_{k=1}^r \hat s_k = \frac{1}{16}\sum_{k=1}^r \nm{ \rP_{\Gamma_k} \rA\left(\sum_{j=1}^s \frac{2\tilde \alpha_j}{\sqrt{t_j - t_{j-1}}}   \eta_j\right)}_2^2 = \norm{\sum_{j=1}^s \frac{2\tilde \alpha_j}{4\sqrt{t_j - t_{j-1}}}   \eta_j}_2^2
\leq \sum_{j=1}^s \frac{1}{t_j -t_{j-1}}.
$$
To bound the second term on the right hand side of (\ref{eq:lem:rel_sparsity_tv_bd1}),
\eas{
\nm{\rP_{\Gamma_k} \rA  \rD^* \rP_{\Lambda^c}\zeta}_2^2
&\leq \sum_{\abs{w}\leq M_k} \abs{\sum_{j=1}^{s-1}\frac{\zeta_{t_j}}{\sqrt{N}} (e^{2\pi i t_j w/N} - e^{2\pi i (t_j+1)w/N}) }^2\\
&=  \sum_{\abs{w}\leq M_k} \frac{1}{N}\abs{\sin\left(\frac{\pi w}{N}\right)}^2 \abs{\sum_{j=1}^{s-1}  \zeta_{t_j} e^{2\pi i w t_j/N}}^2\\
&\leq \min\br{\left(\frac{\pi M_{k}}{N}\right)^2, 1}\cdot ( s-1),
}
where the last inequality follows because
$$
\abs{\sin\left(\frac{\pi w}{N}\right)}^2  \leq  \min\br{\left(\frac{\pi M_{k}}{N}\right)^2, 1}
$$
for all $\abs{w} \leq M_k$ and
$$\sum_{\abs{w}\leq M_k} \frac{1}{N} \abs{\sum_{j=1}^{s-1}  \zeta_{t_j} e^{2\pi i w t_j/N}}^2 \leq \sum_{j=1}^{s-1} \abs{\zeta_{t_j}}^2 \leq s-1
$$
by the Parseval property of the discrete Fourier transform.

If $\xi = 0$, then since $\br{\Gamma_k}_{k=1}^r$ is a disjoint partition of $\br{1,\ldots, N}$ and $\rA$, $\rW$ are unitary,
\spls{
\sum_{k=1}^r\nm{(\rP_{\Gamma_k} \rA \rW \circ\blam^{-1} ) \alpha}^2_2
&=\nm{\blam\circ \alpha }_2^2 \leq \sum_{j=1}^s \frac{2}{t_j -t_{j-1}}.
}
So  (\ref{eq:rel_spasity1d_b}) is holds with $\hat s_k = \frac{1}{2} \nm{(\rP_{\Gamma_k} \rA \rW \circ\blam^{-1} ) \alpha}^2_2$.
\end{proof}

\subsection{Proof of  Theorem \ref{thm:main_tv}}

\begin{proof}[Proof of  Theorem \ref{thm:main_tv}]
We will show that conditions (i) and (ii) of this theorem imply the conditions of Theorem \ref{thm:general}. Then, since $\nm{(\rP_\Lambda \rD)^\dagger}_{1\to 2} \leq \sqrt{N}$ by  (\ref{lem:d_dag_norm}) and a minimizer necessarily exists (see \cite{tv1}), we can let $b=0$ in the conclusion of Theorem \ref{thm:general} and the conclusion of this theorem will follow from the conclusion of Theorem \ref{thm:general}.

Since $\rA$ and $\rD$ are defined as finite dimensional matrices, when applying Theorem \ref{thm:general}, we will consider the infinite dimensional extension of them, so $\rD e_i = 0$ and $\rA e_i = 0$ for $i>N$. So, we can let the $M$ and $\tilde M$ values of Theorem \ref{thm:general} be $ N$.
Recall also Definition \ref{def:active_s_1d} and observe that for $\Lambda^c = \br{t_1,\ldots, t_{s-1}}$, $S(\Lambda, 0 ) = s$.
 Let $\cL =( \log(s\epsilon^{-1}) +1)\cdot \log(q^{-1} N^{3/2} \sqrt{s}) $.
Let $0 = M_0 < M_1<\cdots < M_r = N$ and let for each $k=1,\ldots, r$,
let 
$$
\Gamma_k = \br{j\in\bbZ: -\lfloor M_{k}/2\rfloor \leq  j \leq \lfloor M_{k-1}/2\rfloor -1, \, \lceil M_{k-1}/2\rceil \leq j\leq \lceil M_k/2 \rceil -1},
$$
and note that $\abs{\Gamma_k } = M_k - M_{k-1}$.
Let $\cW_\Lambda$, $\rW$, $\blam$, $\xi$ be defined as in Lemma \ref{lem:defn_1dtv_W}. By (i) by Lemma \ref{lem:defn_1dtv_W}, (\ref{eq:restr_Lambda}) is satisfied.
Now, $\rX = \rQ_{\cW_\Lambda}$ because $\rA$ is unitary. So, the left hand side of (\ref{eq:restr_D}) is simply zero.

 Assume for now that $\xi \neq 0$.
We first consider (a) of Theorem \ref{thm:general}. We are required to show that for each $k=1,\ldots, r$,  $m_k$ satisfies
\be{\label{eq:cond_m_tv1d}
\frac{m_k}{M_k-M_{k-1}} \gtrsim \cL \cdot \max\br{\mu( \rP_{\Gamma_k} \rA\rW \circ\blam), \mu(\rP_{\Gamma_k} \rA \rQ_{\cW_\Lambda}^\perp(\rP_\Lambda \rD)^\dagger)}\sum_{j=1}^{s+1} \blam_j^{-1}\mu(\rP_{\Gamma_k} \rA \rW \rP_{\br{j}}),  
}
and $ m_k \gtrsim \cL \cdot \hat m_k$ with
\be{\label{eq:hat_m_tv1d}
1\gtrsim \max_{j}\br{ \max_{\norm{\zeta}_\infty =1 } \sum_{k=1}^r \left(\frac{M_k-M_{k-1}}{\hat m_k} -1\right)\cdot  \mu_{k,j}^2 \cdot  \norm{\rP_{\Gamma_k}\rA \rW \circ\blam^{-1} \cdot \zeta}_2^2},
}
where
$$
\mu_{k,j}=   \max\br{\blam_j\cdot \mu(\rP_{\Gamma_k} \rA\rW \rP_{\br{j}}), \mu( \rP_{\Gamma_k} \rA \rQ_{\cW_\Lambda}^\perp(\rP_\Lambda \rD)^\dagger \rP_{\br{j}})}.
$$
Recall that $\blam_j =\frac{ \sqrt{t_j - t_{j-1}}}{2}$ for $j=1,\ldots, s$ and by (i) of Lemma \ref{lem:inc_tv}, 
$$
\blam_{s+1}^{-1} \cdot \mu(\rP_{\Gamma_k} \rA \rW \rP_{\br{s+1}})
\lesssim (s-1) \cdot \sqrt{\frac{1}{N}}\cdot  \frac{\pi M_k}{N} + \sum_{j=1}^s \frac{1}{\sqrt{t_j - t_{j-1}}} \cdot \mu(\rP_{\Gamma_k} \rA \rW \rP_{\br{j}}).
$$
So,
$$
\sum_{j=1}^{s+1} \blam_j^{-1} \cdot \mu(\rP_{\Gamma_k} \rA \rW \rP_{\br{j}})
\lesssim (s-1) \cdot \frac{M_k}{N^{3/2}}+ \sum_{j=1}^s \frac{1}{\sqrt{t_j - t_{j-1}}} \cdot \mu(\rP_{\Gamma_k} \rA \rW \rP_{\br{j}}).
$$
Furthermore, by (i) of Lemma \ref{lem:inc_tv}, for $j=1,\ldots, s$,
$$
\blam_j^{-1} \mu(\rP_{\Gamma_k} \rA \rW \rP_{\br{j}}) \lesssim  \min\br{ \frac{1}{\sqrt{N}}, \frac{\sqrt{N}}{(t_j - t_{j-1}) \cdot \max\br{M_{k-1},1}} }
$$
and 
$$
\max_{j=1}^s \br{\blam_j \cdot \mu(\rP_{\Gamma_k} \rA \rW \rP_{\br{j}})} \lesssim \frac{\sqrt{N}}{\max\br{M_{k-1},1}}, \qquad 
\br{\blam_j \cdot \mu(\rP_{\Gamma_k} \rA \rW \rP_{\br{s+1}})} \lesssim \frac{1}{\sqrt{N}}\leq \frac{\sqrt{N}}{\max\br{M_{k-1},1}}.
$$
We also apply (iii) of Lemma \ref{lem:inc_tv} to obtain
$$
\mu(\rP_{\Gamma_k} \rA \rQ_{\cW_\Lambda}^\perp (\rP_\Lambda \rD)^\dagger) \lesssim \frac{\sqrt{N}}{\max\br{M_{k-1},1}}.
$$
Therefore, by letting $$\Delta_k = \br{j: t_j -t_{j-1} \leq \frac{N}{\max\br{M_{k-1},1}}},$$
(\ref{eq:cond_m_tv1d}) holds if
\eas{
\frac{m_k}{M_k-M_{k-1}} \gtrsim &\cL \cdot \frac{\sqrt{N}}{\max\br{M_{k-1},1}} \cdot \left( \sum_{j\not\in\Delta_k} \frac{\sqrt{N}}{(t_j - t_{j-1}) \cdot \max\br{M_{k-1},1}}
+ \sum_{j\in\Delta_k} \frac{1}{\sqrt{N}} + \frac{s-1}{N}\cdot \frac{M_k}{\sqrt{N}}\right)\\
& \gtrsim \cL \cdot \frac{1}{\max\br{M_{k-1},1}} \cdot \left( \sum_{j\not\in\Delta_k} \frac{N}{(t_j - t_{j-1}) \cdot \max\br{M_{k-1},1}}
+ \abs{\Delta_k} +  \frac{s-1}{N}\cdot M_k\right)\\
&= \cL \cdot \frac{1}{\max\br{M_{k-1},1}}\left(S(\Lambda,\max\br{M_{k-1},1})+  \frac{s-1}{N}\cdot M_k\right).
}

To understand when (\ref{eq:hat_m_tv1d}), first observe that by Lemma \ref{lem:rel_sparsity_tv}, given any $\alpha\in\bbC^N$ such that $\nm{\alpha}_\infty \leq 1$, there exists $\br{\hat s_k}_{k=1}^r \in\bbR_+^r$ such that
\eas{
\norm{\rP_{\Gamma_k} \rA \rW \circ\blam^{-1}\cdot \alpha}^2 \lesssim
 \hat s_k +
 \left(\frac{M_k}{N}\right)^2 \cdot (s-1), \qquad \sum_{k=1}^r\hat s_k \leq \sum_{j=1}^s \frac{1}{t_j-t_{j-1}}.
 }
Thus, combining with the previous observation that
$$
\max_{j=1}^s \br{\blam_j \cdot \mu(\rP_{\Gamma_k} \rA \rW \rP_{\br{j}})} \lesssim \frac{\sqrt{N}}{\max\br{M_{k-1},1}}, \qquad
\mu(\rP_{\Gamma_k} \rA \rQ_{\cW_\Lambda}^\perp (\rP_\Lambda \rD)^\dagger) \lesssim \frac{\sqrt{N}}{\max\br{M_{k-1},1}},
$$
we have that (\ref{eq:hat_m_tv1d}) holds if $m_k \gtrsim \cL \hat m_k$ with
$$
1\geq \sum_{k=1}^r\left(\frac{M_k - M_{k-1}}{\hat m_k} -1\right)\cdot \left(\frac{N}{(1+M_{k-1})^2} \cdot \hat s_k +    \frac{s-1}{N}\cdot\left(\frac{M_k}{\max\br{M_{k-1},1}}\right)^2\right)
$$
for all $\br{\hat s_k}_{k=1}^r \in\bbR_+^r$ such that  
$$
\sum_{k=1}^r\hat s_k \leq \sum_{j=1}^s \frac{1}{t_j - t_{j-1}} = F(\Lambda),
$$
where   $F(\Lambda)$ is defined in  Definition \ref{def:fineness_1d}.
To show that (b) of Theorem \ref{thm:general} holds, note that by (i) of Lemma \ref{lem:inc_tv},
$$
\sum_{j=1}^{s+1} \mu(\rP_{\Gamma_k} \rA \rW \rP_{\br{j}})^2
\lesssim \sum_{j=1}^s (\rP_{\Gamma_k} \rA \rW \rP_{\br{j}})^2 + (s-1) \cdot \frac{M_k^2}{N^{3}}
$$
and for each $k=1,\ldots, s$
$$
\mu(\rP_{\Gamma_k} \rA \rW \rP_{\br{j}}) \lesssim \min\br{ \frac{N}{(t_j - t_{j-1})\cdot \max\br{M_{k-1},1}^2}, \frac{1}{\max\br{M_{k-1},1}}
}.
$$
Therefore, 
$$
\sum_{j=1}^{s+1} \mu(\rP_{\Gamma_k} \rA \rW \rP_{\br{j}})^2
\lesssim \frac{1}{\max\br{M_{k-1},1}} \left(\sum_{j\not\in\Delta_k} \frac{N}{\max\br{M_{k-1},1}\cdot (t_j - t_{j-1})} + \abs{\Delta_l} + (s-1) \cdot \frac{M_k^2}{N^{2}}\right)
$$
and condition (b) holds provided that
$$
\frac{m_k}{M_k - M_{k-1}}\gtrsim \cL \cdot
\frac{1}{\max\br{M_{k-1},1}} \left(\sum_{j\not\in\Delta_k} \frac{N}{\max\br{M_{k-1},1}\cdot (t_j - t_{j-1})} + \abs{\Delta_l} + (s-1) \cdot \frac{M_k^2}{N^{2}}\right),
$$
which is implied by condition (i) of our assumptions.
Finally, for condition (c), since $\nm{(\rP_\Lambda \rD)^\dagger}_{1\to 2}\leq \sqrt{N}$, it suffices to let $B = \sqrt{N}$ and combining with the incoherence estimate
$$
\mu(\rP_{\Gamma_k} \rA \rQ_{\cW_\Lambda}^\perp (\rP_\Lambda \rD)^\dagger) \lesssim \frac{\sqrt{N}}{\max\br{M_{k-1},1}},
$$ (c) is implied by (i) of our theorem.

Finally, if $\xi = 0$, then by retracing the steps of this proof, we have that the assumptions of Theorem \ref{thm:general} are implied by (i') and (ii') below.
\begin{itemize}

\item[(i')] 
For $k=1,\ldots, r$,
 \eas{
\frac{m_k}{M_k-M_{k-1}} \gtrsim   \frac{\cL}{\max\br{M_{k-1},1}} 
\cdot S\left(\Lambda, \max\br{M_{k-1},1}\right),
}
\item[(ii')] For $k=1,\ldots, r$,
  $ m_k \gtrsim (\log(s\epsilon^{-1}) +1) \cdot \log(q^{-1} N^{3/2} \sqrt{s}) \cdot \hat m_k$
  such that $\br{\hat m_k}_{k=1}^r$ satisfies
  $$
1\gtrsim \sum_{k=1}^r\left(\frac{M_k - M_{k-1}}{\hat m_k} -1\right)\cdot \left(\frac{N}{\max\br{M_{k-1},1}^2} \cdot \hat s_k\right)
$$
for any $\br{\hat s_k}_{k=1}^r$ such that
$$
\sum_{k=1}^r\hat s_k \leq F(\Lambda)
$$
\end{itemize}
which are less restrictive conditions than the assumptions of Theorem \ref{thm:main_tv}.

\end{proof}

\section{Proof of Theorem \ref{thm:tv_2d}} \label{sec:proof_tv2}
Let $\br{I_1,\ldots, I_n}$ be defined such that 
 $\mathrm{span}\br{\mathbbm{1}_{I_j}: j=1,\ldots, n} = \cN(\tilde \rP_\Lambda \rD)$. Assume that $\Lambda^c \neq \emptyset$.
\begin{lemma} \label{lem:defn_2dtv_W}
Let $\cW_\Lambda = \mathrm{span}\br{f_{j}: j=1,\ldots, n+1}$ where
$$f_j = \frac{\mathbbm{1}_{I_j}}{\abs{I_j}^{1/2}}, \quad j=1,\ldots, n, \qquad f_{n+1} = \begin{cases} \xi/\nm{\xi} &\xi \neq 0\\
0 &\xi = 0 \end{cases}
$$
 where
\be{\label{eq:xi_2d}
\xi = \rD^* \rP_{\Lambda^c} \sigma - \sum_{j=1}^n \ip{\rD^* \rP_{\Lambda^c}\sigma}{f_j}f_{j}
}
and $\sigma$ is as defined in (\ref{eq:sigma_2dtv}). 
If $\xi \neq 0$, then let 
\spl{\label{eq:defn_null_2dtv}
&\rW: \bbC^{n+1} \to \bbC^{N\times N}, \qquad \rW \alpha = \sum_{j=1}^{n+1} \alpha_s f_{j} , \quad \forall \alpha \in\bbC^{n+1}\\
&\blam := \left(\nm{\rD f_{1}}_{1}^{-1},\ldots, \nm{ \rD f_{n}}_{1}^{-1}, \nmu{ \xi}_2^{-1} \right)^T
= \left(\frac{\sqrt{\abs{I_1}}}{\mathrm{Per}(I_1)},\ldots,\frac{\sqrt{\abs{I_n}}}{\mathrm{Per}(I_n)}, \nmu{ \xi}_2^{-1} \right)^T,
}
otherwise, if $\xi = 0$, let
\spls{
&\rW: \bbC^{n} \to \bbC^{N\times N}, \qquad \rW \alpha = \sum_{j=1}^{n} \alpha_s f_{j} , \quad \forall \alpha \in\bbC^{n}\\
&\blam := \left( \nm{ f_{1}}_{TV}^{-1},\ldots, \nm{ f_{n}}_{TV}^{-1} \right)^T
= \left(\frac{\sqrt{\abs{I_1}}}{\mathrm{Per}(I_1)},\ldots,\frac{\sqrt{\abs{I_n}}}{\mathrm{Per}(I_n)} \right)^T.
}
Then
\item[(i)] $\rW$ is an isometry and
$
\inf_{v\in\cN(  \rD^* \tilde \rP_{\Lambda}  )}\norm{( \rD^* \rP_\Lambda )^\dagger \rQ_{\cW_\Lambda}^\perp  \rD^* \tilde\rP_{\Lambda^c}   \sigma
  -v}_\infty = 0,
$

\item[(ii)]  $\norm{ \blam\circ \rW^*\rD^* \rP_{\Lambda^c} \sigma}_\infty \leq 1$.
\end{lemma}
\begin{proof}

We first remark that $\blam$ is well defined because by the assumption that $\Lambda^c\neq \emptyset$, it follows that $I_j \subsetneq \br{1,\ldots, N}^2$ for all $j=1,\ldots, n$, which implies that $\mathrm{Per}(I_j) \neq 0$. 
Claim (i) is trivial.
For (ii), observe that for each $j=1,\ldots, n$,
$$
\abs{\ip{\rW^* \rD^* \tilde\rP_{\Lambda^c} \sigma}{ e_j}} = \abs{\ip{f_j}{\rD^* \tilde\rP_{\Lambda^c}\sigma}} \leq \nm{\rD f_j}_1 =\blam_j^{-1}
$$
and if $\xi \neq 0$,
$$
\abs{\ip{\rW^* \rD^* \tilde\rP_{\Lambda^c} \sigma}{ e_{n+1}}}
= \frac{\nm{\rD^* \tilde\rP_{\Lambda} \sigma}_2^2 - \sum_{j=1}^s \abs{\ip{\rD^* \tilde\rP_{\Lambda^c} \sigma}{f_j}}^2}{\nm{\xi}_2} = \nm{\xi}_2.
$$
\end{proof}

\subsubsection*{Incoherence estimates}

\begin{lemma}\label{lem:2d_tv_inc}
Let $\rW$, $\xi$ be defined as in Lemma \ref{lem:defn_2dtv_W}.
\begin{itemize}
\item[(i)] Let $k=(k_1,k_2)\in \br{-N/2+1,\ldots, N/2}^2$. Then
$$
\abs{(\rA f_{j})_{k_1,k_2}}   \lesssim  \min\br{ \frac{\mathrm{Per}(I_j)}{\abs{k} \cdot \sqrt{\abs{I_j}}},
\frac{\sqrt{\abs{I_j}}}{N}}
$$

\item[(ii)] If $\xi \neq 0$, then
$$
\abs{(\rA f_{n+1})_{k_1,k_2}} \lesssim \nmu{ \xi}^{-1} \left(
 \frac{\abs{\Lambda^c}\cdot  \abs{k}}{N^2}  +
 \sum_{j=1}^n \frac{\mathrm{Per}(I_j)}{\sqrt{\abs{I_j}}}  \abs{(\rA f_{j})_{k_1,k_2}}
\right)
$$
and
$$
\abs{(\rA f_{n+1})_{k_1,k_2}} \leq \frac{\nm{\xi}_1}{N\cdot \nm{\xi}_2}.
$$
\end{itemize}
\end{lemma}
\begin{proof}
\item[(i)]
For each $k\in \br{-N/2+1,\ldots, N/2}^2$,
we have that
\be{\label{pfinc_i3}
(\rA f_{I_j})_{k_1,k_2} = \abs{\frac{1}{N \cdot \sqrt{ \abs{I_j}}} \sum_{(l_1,l_2)\in I_j } e^{2\pi i (k_1 l_1+k_2 l_2)/N}} \leq \frac{\sqrt{\abs{I_j}}}{N}.
}

Also,  by  Lemma \ref{lem:ft_decay}, for $k=(k_1,k_2)\in \br{-N/2+1,\ldots, N/2}^2\setminus \br{(0,0)} $,
\be{\label{pfinc_i}
\abs{(\rA f_{I_j})_{k_1,k_2}} \leq \nm{ \rD f_j}_1 \cdot \frac{1}{\abs{k}}.
}

Now, $\nm{ \rD f_j}_1 \leq \abs{I_j}^{-1/2} \nm{\rD \mathbbm{1}_{I_j}}_1 =\frac{ \mathrm{Per}(I_j)}{\sqrt{\abs{I_j}}}$. So,
\be{\label{pfinc_i2}
\abs{(\rA f_{I_j})_{k_1,k_2}} \leq \frac{\mathrm{Per}(I_j)}{\sqrt{\abs{I_j}}\cdot \abs{k}}.
}

The conclusion of (i) follows from  combining (\ref{pfinc_i3})  with (\ref{pfinc_i2}).

\item[(ii)]
Note that
\be{\label{eq:inc_2d_ii}
(\rA f_{n+1})_{k_1,k_2} = \nmu{ \xi}^{-1} \left((\rA \rD^* \rP_{\Lambda^c}\sigma)_{k_1,k_2}
- \sum_{j=1}^n \ip{ \rD^* \rP_{\Lambda^c}\sigma}{f_{I_j}} (\rA f_{I_j})_{k_1,k_2}\right).
}
We first consider the first term on the right hand side of (\ref{eq:inc_2d_ii}),
$$
\abs{(\rA \rD^* \rP_{\Lambda^c}\sigma)_{k_1,k_2}} \leq \sum_{j\in\Lambda^c}\abs{(\rA \rD_1^* e_j)_{k_1,k_2}} +  \sum_{j\in\Lambda^c}\abs{(\rA \rD_2^* e_j)_{k_1,k_2}}.
$$
To bound each summand in the above inequality, note that for each $j\in\Lambda^c$,
\eas{
\abs{(\rA \rD_1^* e_j)_{k_1,k_2}} &= \frac{1}{N}\abs{e^{2\pi i k_2 j_2/N} (e^{2\pi i k_1(j_1+1)/N} - e^{2\pi i k_1 j_1/N})}\\
&\leq \frac{2}{N} \abs{\sin\left(\frac{\pi k_1}{N}\right)} \leq \frac{2}{N}\min\br{\frac{\pi \abs{ k_1}}{N}, 1}.
}
and similarly, for each $j\in \Lambda^c$,
\eas{
\abs{(\rA \rD_2^* e_j)_{k_1,k_2}}  \leq \frac{2}{N}\min\br{\frac{\pi \abs{k_2}}{N}, 1}.
}
To bound the second term on the right hand side of (\ref{eq:inc_2d_ii}), observe that \eas{
\abs{\ip{ \rD^* \rP_{\Lambda^c}\sigma}{f_{I_j}}}
= \abs{ \ip{\rP_{\Lambda^c}\sigma}{\rD f_{I_j}}}
 \leq \norm{\rD f_{I_j}}_1  = \frac{\mathrm{Per}(I_j)}{\sqrt{\abs{I_j}}}.
 }
Thus, we have that
$$
\abs{(\rA f_{n+1})_{k_1,k_2}} \lesssim \nmu{ \xi}^{-1} \left(
 \frac{\abs{\Lambda^c}}{N}\cdot \frac{ \abs{k}}{N}  +
 \sum_{j=1}^n \frac{\mathrm{Per}(I_j)}{\sqrt{\abs{I_j}}} \abs{(\rA f_{I_j})_{k_1,k_2}}
\right).
$$
To prove the last part  of (ii),  observe that
$$
\frac{\abs{(\rA f_{n+1})_{k_1,k_2}} }{\nm{\xi}} \leq \frac{\abs{(\rA \xi)_{k_1,k_2}}}{\nm{\xi}^2} \leq \frac{\nm{\xi}_1}{N\cdot \nm{\xi}^2}.
$$

\end{proof}

\begin{lemma}\label{lem_2d_rel_sparsity}
Let $\rW, \xi, \blam$ be defined as in Lemma \ref{lem:defn_2dtv_W}. Let $\bigcup_{k=1}^r \Gamma_k = \br{1,\ldots, N}^2$ be $r$ disjoint sets. Given $\zeta$ such that $\nm{\zeta}_\infty \leq 1$, there exists $\br{\hat s_k}_{k=1}^r\in\bbR^r_+$ such that
\eas{
\nm{\rP_{\Gamma_k  } \rA \rW (\blam^{-1} \cdot \zeta)}_2^2
\lesssim 
\hat s_k
+  S_k,  \quad k=1,\ldots, r,
\qquad
\sum_{k=1}^r \hat s_k \leq
 \sum_{j=1}^n \frac{\mathrm{Per}(I_j)^2}{\abs{I_j}}
}
where
$$
S = \begin{cases}
\max_{j=(j_1,j_2)\in\Gamma_k} \frac{\abs{j}^2}{N^2} \cdot  \abs{\Lambda^c} &\text{if }\xi \neq 0\\
0 &\text{if } \xi =0.
\end{cases}
$$
\end{lemma}

\begin{proof}
If $\xi = 0$, then since $\rA$ and $\rW$ are unitary matrices,
\eas{
\sum_{k=1}^r\nm{\rP_{\Gamma_k} \rA \rW (\blam^{-1} \cdot \zeta)}^2_2
= \nm{\blam^{-1} \cdot \zeta}^2_2 \leq  \sum_{j=1}^n \frac{\mathrm{Per}(I_j)^2}{\abs{I_j}}.
}
It remains to consider the case when $\xi \neq 0$. By definition of $\rW$,
\eas{
\nm{\rP_{\Gamma_k} \rA \rW (\blam^{-1} \cdot \zeta)}_2
=
\nm{\rP_{\Gamma_k} \rA \left(\sum_{j=1}^{n+1} \blam_j^{-1} \zeta_j f_{j}\right)}_2.
}
Recall that $
f_{n+1} =   \xi/\nm{\xi}_2$ where $$\xi 
= \rD^* \rP_{\Lambda^c}\sigma - \sum_{j=1}^n \ip{\rD^* \rP_{\Lambda^c}\sigma}{f_j}f_j
$$
with  $\sigma$  as defined in (\ref{eq:sigma_2dtv}),
and for $j=1,\ldots, n$, $\blam_j^{-1} = \frac{\mathrm{Per(I_j)}}{ \sqrt{\abs{I_j}}}$, $\blam_{n+1}^{-1} = \nmu{ \xi}_2$.
  So,
\eas{
\nm{\rP_{\Gamma_k } \rA \rW \blam^{-1} \zeta}_2^2
\leq
2\nm{\rP_{\Gamma_k } \rA \left(\sum_{j=1}^n \tilde \zeta_j f_{j}\right) }_2^2+ 2\nm{\rP_{\Gamma_{k} } \rA  \xi}_2^2,
}
where for $j=1,\ldots,n$, $\abs{\tilde \zeta_j} \leq \frac{2\mathrm{Per(I_j)}}{\sqrt{\abs{I_j}}}$.
  Let $\hat s_k = \frac{1}{4} \nm{\rP_{\Gamma_k } \rA \left(\sum_{j=1}^n \tilde \zeta_j f_{j}\right) }_2^2$. Then,  since $\rA$ is unitary,
\eas{
\sum_{k=1}^r \hat s_k =
&\frac{1}{4}\sum_{k=1}^r\nm{\rP_{\Gamma_k  } \rA \sum_{j=1}^n \tilde  \zeta_j f_{j} }_2^2
= \frac{1}{4}\nm{\sum_{j=1}^n \tilde \zeta_j f_{j} }^2_2
\leq
\sum_{j=1}^n \frac{\mathrm{Per}(I_j)^2}{\abs{I_j}}.
}
Finally,
\eas{
&\nm{\rP_{\Gamma_k} \rA \xi}_2^2 \leq \sum_{j=(j_1,j_2)\in\Gamma_k} \abs{\sum_{l\in\Lambda^c} (\rA \rD_1^* e_l)_{j} + \sum_{l\in\Lambda^c}(\rA \rD_2^* e_l)_j}^2\\
&=\sum_{j=(j_1,j_2)\in\Gamma_k} \frac{\abs{(e^{2\pi i j_1/N} -1)(e^{2\pi i j_2/N} -1)}}{N^2} \abs{\sum_{l\in\Lambda^c}e^{2\pi i N^{-1} \ip{j}{l}}+ \sum_{l\in\Lambda^c}e^{2\pi i N^{-1} \ip{j}{l}}}^2\\
&\leq \max_{(j_1,j_2)\in\Gamma_k}\frac{2\abs{(e^{2\pi i j_1/N} -1)(e^{2\pi i j_2/N} -1)}}{N^2} \sum_{j\in\Gamma_k}\abs{\sum_{l\in\Lambda^c}e^{2\pi i N^{-1} \ip{j}{l}}}+\sum_{j\in\Gamma_k}\abs{ \sum_{l\in\Lambda^c}e^{2\pi i N^{-1} \ip{j}{l}}}^2\\
&\lesssim \max_{j=(j_1,j_2)\in\Gamma_k} \frac{\abs{j}^2}{N^2} \cdot \abs{\Lambda^c}.
}
\end{proof}

\subsubsection*{Proof of Theorem \ref{thm:tv_2d}}
We will prove this theorem by showing that our assumptions imply the assumptions of Theorem \ref{thm:general_21}. Since $\rA$ and $\rD$ are defined as finite dimensional matrices, when applying Theorem \ref{thm:general}, we will consider the infinite dimensional extension of them, so $\rD e_i = (0,0)$ and $\rA e_i =(0, 0)$ for $i\not\in\br{1,\ldots,N}^2$ and we can let $M=\tilde M = N^2$ in Theorem \ref{thm:general_21}.

Recall the definitions of $\cW_\Lambda$, $\rW$, $\blam$ and $\xi$ from Lemma \ref{lem:defn_2dtv_W}. We first need to show that the identifiability and balancing properties hold. Since $\rA$ is a unitary operator, $\rA^*\rA = \rI$, we can let $\rX$ be the identity and in this case, conditions  (\ref{eq:restr_X}), (\ref{eq:restr_X2}) and (\ref{eq:restr_D}) are trivially true. By our choice of $\cW_\Lambda$, we have from Lemma \ref{lem:defn_2dtv_W}
that the following  condition is also trivially true.
$$
\inf_{u\in\cN(\rD^* \rP_\Lambda)} \nm{(\rD^* \rP_\Lambda)^\dagger \rQ_{\cW_\Lambda}^\perp \rD^* \tilde \rP_{\Lambda^c}\sigma - u}_\infty \leq \frac{1}{16}
$$
Recalling Definition \ref{def:active_s_2d}, we have that $s=S(\Lambda, 0) = \ord{\abs{\Lambda^c}}$.

First assume that $\xi\neq 0$.
We will now show that conditions (a), (b) and (c) for Theorem \ref{thm:general_21} are satisfied, that is, for $$
\cL = (\log(s\epsilon^{-1})+1) \log(q^{-1} N^2 B \sqrt{s}), \qquad B = \nm{(\tilde \rP_\Lambda\rD)^\dagger}_{1\to 2} ,$$ conditions (a'), (b') and (c') listed below are satisfied.
\begin{itemize}
\item[(a')] For each $k=1,\ldots,r$
\spl{\label{eq:cond_m_k_tv2d}
\frac{m_k}{\abs{\Gamma_k }} \gtrsim \cL \cdot \max\br{\mu( \rP_{\Gamma_k } \rA\rW \circ\blam), \mu(\rP_{\Gamma_k } \rA \rQ_{\cW_\Lambda}^\perp(\tilde \rP_\Lambda \rD)^\dagger)}
\cdot\sum_{j=1}^{n+1} \blam_j^{-1}\mu(\rP_{\Gamma_k } \rA \rW \rP_{\br{j}}) ,  
}
and $ m_k \gtrsim \cL \cdot \hat m_k$ with
\be{\label{eq:hat_m_tv2d}
1\gtrsim \max_{\norm{\zeta}_\infty =1 } \sum_{k=1}^r \left(\frac{\abs{\Gamma_k}}{\hat m_k} -1\right)\cdot  \mu_{k,j}^2 \cdot  \norm{\rP_{\Gamma_k }\rA  \rW \circ\blam^{-1} \cdot \zeta}_2^2, \qquad j=1,\ldots, N
}
where
$$
\mu_{k,j}=   \max\br{\blam_j\cdot \mu(\rP_{\Gamma_k } \rA\rW \rP_{\br{j}}), \mu( \rP_{\Gamma_k } \rA \rQ_{\cW_\Lambda}^\perp(\tilde \rP_\Lambda \rD)^\dagger \rP_{\br{j}})}.
$$

\item[(b')] For each $k=1,\ldots, r$,
$$\frac{m_k}{\abs{\Gamma_k }}\gtrsim \log\left(\frac{\abs{\Lambda^c}}{\epsilon}\right) \cdot \sum_{j=1}^{n+1} \left(\mu(\rP_{\Gamma_k }\rA \rW \rP_{\br{j}})\right)^2 .
$$
\item[(c')] For each $k=1,\ldots, r$ and $j\in\bbN$
$$\max\br{1, \nm{(\tilde \rP_\Lambda \rD)^\dagger}_{1\to 2}}\gtrsim  \left(\frac{\abs{\Gamma_k }}{m_k} -1\right)\cdot\log\left(\frac{N}{\epsilon}\right) \cdot \left(\mu(\rP_{\Gamma_k } \rA \rQ_{\cW_\Lambda}^\perp(\tilde \rP_\Lambda \rD)^\dagger \rP_{\br{j}})\right)^2.
$$
\end{itemize} 
For each $k=1,\ldots, r$, let $$M_k^{\min} = \max\br{\min_{m\in\Gamma_k} \abs{m}, \, 1}, \quad
M_k^{\max} = \max_{m\in\Gamma_k} \abs{m}.$$
We first consider (a'). By Lemma \ref{lem:2d_tv_inc},
\be{\label{eq:f_*_inc_upper}
\blam_{n+1}\cdot\mu(\rP_{\Gamma_k } \rA \rW \rP_{\br{n+1}})
\lesssim \frac{\abs{\Lambda^c}}{N} \cdot \frac{M_k^{\max}}{N} +
\sum_{j=1}^n \frac{\mathrm{Per}(I_j)}{\sqrt{\abs{I_j}}} \cdot \mu(\rP_{\Gamma_k  } \rA \rW \rP_{\br{j}}).
}
So,
\spl{\label{tv_2dthmprf_cond_a_i}
&\sum_{j=1}^{n+1} \blam_j^{-1}\cdot \mu(\rP_{\Gamma_k } \rA \rW \rP_{\br{j}})\\
&\lesssim  \frac{\abs{\Lambda^c}}{N} \cdot \frac{M_k^{\max}}{N} +
\sum_{j=1}^n \frac{\mathrm{Per}(I_j)}{\sqrt{\abs{I_j}}} \cdot \mu(\rP_{\Gamma_k } \rA \rW \rP_{\br{j}}).
}
By (i) of Lemma \ref{lem:2d_tv_inc}, for $j=1,\ldots, n$,
\be{\label{tv_2dthmprf_cond_a_ii}
\frac{\mathrm{Per}(I_j)}{\sqrt{\abs{I_j}}} \cdot \mu(\rP_{\Gamma_k  } \rA \rW \rP_{\br{j}})
\lesssim \mathrm{Per}(I_j)\cdot  \min\br{ \frac{\mathrm{Per}(I_j)}{\abs{I_j}}\cdot   \frac{1}{M_k^{\min} } \, , \, \frac{1}{N} }.
}

Let 
$$\Delta_k := \br{j: \frac{\mathrm{Per}(I_j)}{\abs{I_j}} \geq \min_{m\in\Gamma_k }\frac{  \abs{m}}{N}}.
$$
Then (\ref{tv_2dthmprf_cond_a_i}) and (\ref{tv_2dthmprf_cond_a_ii}) gives
\eas{
&\sum_{j=1}^{n+1} \blam_j^{-1}\cdot \mu(\rP_{\Gamma_k } \rA \rW \rP_{\br{j}})\\
&\leq 
\frac{1}{N}\cdot \left(
\sum_{j\not\in\Delta_k} \frac{\mathrm{Per}(I_j)^2}{\abs{I_j}} \cdot   \frac{N}{M_k^{\min} }
+
\cdot\sum_{j\in\Delta_k} \mathrm{Per}(I_j)
+  \abs{\Lambda^c} \cdot \frac{M_k^{\max}}{N}\right).
}
Again, by Lemma \ref{lem:2d_tv_inc}, for $j=1,\ldots, s$
\be{ \label{inc_upper_AWlambda}
\blam_j\cdot \mu(\rP_{\Gamma_k } \rA \rW \rP_{\br{j}})\lesssim \frac{1}{M_k^{\min}}
}
and
\be{ \label{eq:f_*_inc_upper2}
\blam_{n+1}\cdot \mu(\rP_{\Gamma_k } \rA \rW \rP_{\br{s+1}})\lesssim \frac{\nm{\xi}_1}{\nm{\xi}_2^2}\cdot \frac{1}{N}.
}
Therefore, the first part of condition (a') becomes
\eas{
\frac{m_k}{\abs{\Gamma_k}} &\gtrsim \cL\cdot
\frac{\nu_k}{N} \cdot \left(
\sum_{j\not\in\Delta_k} \frac{\mathrm{Per}(I_j)^2}{\abs{I_j}} \cdot   \frac{N}{M_k^{\min} }
+
\cdot\sum_{j\in\Delta_k} \mathrm{Per}(I_j)
+  \abs{\Lambda^c} \cdot \frac{M_k^{\max}}{N}\right)\\
&= \cL \cdot \frac{\nu_k}{N} \cdot \left( S(\Lambda, M_{k}^{\min})
+  \abs{\Lambda^c} \cdot \frac{M^{\max}_k }{N}\right).
}
where
$$
\nu_k = \max\br{1,\,  \frac{\nm{\xi}_1}{\nm{\xi}_2^2}, \, \mu_k\cdot  M^{\min}_k }\cdot \frac{1}{  M^{\min}_k },
$$
which is exactly condition (i).

For the second part of (a'), first observe that by Lemma \ref{lem_2d_rel_sparsity}, for any $\zeta \in\bbC^{s+1}$ such that $\nm{\zeta}_\infty = 1$, there exists $\br{\hat s_k}_{k=1}^r\in\bbR^r_+$ such that
 for $k=1,\ldots, r$
\eas{
\nm{\rP_{\Gamma_k } \rA \rW (\blam^{-1} \cdot \zeta)}_2^2
\lesssim 
\hat s_k
+   \frac{(M_k^{\max})^2}{N^2} \cdot  \abs{\Lambda^c}, 
\qquad
\sum_{k=1}^r \hat s_k \leq  \sum_{j=1}^n \frac{2\mathrm{Per}(I_j)^2}{\abs{I_j}}.
}
So, in conjunction with (\ref{inc_upper_AWlambda}), the second part of condition (a') becomes
$m_k \gtrsim \cL \cdot \hat m_k
$ with
$$
1 \geq \sum_{k=1}^r \left(\frac{\abs{\Gamma_{k} }}{\hat m_k} -1\right) \cdot  \nu_k^2 \cdot \left(\hat s_k
+   \frac{(M_k^{\max})^2}{N^2}\right)
$$
for all $\br{\hat s_k}_{k=1}^r\in\bbR^r_+$  satisfying
$$
\sum_{k=1}^r \hat s_k \leq 
\sum_{j=1}^n \frac{\mathrm{Per}(I_j)^2}{\abs{I_j}} = F(\Lambda)
$$
which is exactly condition (ii).

For condition (b'), by (\ref{eq:f_*_inc_upper}) and (\ref{eq:f_*_inc_upper2}), we have that
\eas{
&\mu(\rP_{\Gamma_k } \rA \rW \rP_{\br{n+1}})^2\\
&\lesssim \frac{\mu(\rP_{\Gamma_k } \rA \rW \rP_{\br{n+1}}) }{\nm{\xi}_2}\cdot\left( \frac{\abs{\Lambda^c}}{N} \cdot \frac{M_k^{\max}}{N} +
\sum_{j=1}^n \frac{\mathrm{Per}(I_j)}{\sqrt{\abs{I_j}}} \cdot \mu(\rP_{\Gamma_k  } \rA \rW \rP_{\br{j}})\right)\\
&\lesssim \frac{\nm{\xi}_2 }{N\cdot \nm{\xi}^2_2}\cdot \left(\frac{\abs{\Lambda^c}}{N} \cdot \frac{M_k^{\max}}{N} +
\sum_{j=1}^n \frac{\mathrm{Per}(I_j)}{\sqrt{\abs{I_j}}} \cdot \mu(\rP_{\Gamma_k } \rA \rW \rP_{\br{j}})\right) .
}
By combining this with the following upper bound $$\mu(\rP_{\Gamma_k } \rA \rW \rP_{\br{j}})\leq \frac{\mathrm{Per}(I_j)}{\sqrt{\abs{I_j}}} \frac{1}{M_k^{\min}}, \qquad j=1,\ldots, n,$$
we have that
\eas{
&\sum_{j=1}^{n+1}\mu(\rP_{\Gamma_k } \rA \rQ \rP_{\br{j}})^2\\
&\lesssim \sum_{j=1}^n \left(  \frac{1}{M_k^{\min}} + \frac{\nm{\xi}_1}{N\cdot \nm{\xi}_2^2}\right)\cdot \frac{\mathrm{Per}(I_j)}{\sqrt{\abs{I_j}}} \cdot \mu(\rP_{\Gamma_k } \rA \rQ \rP_{\br{j}}) +  \frac{\nm{\xi}_2 \cdot \abs{\Lambda^c}}{N^2\cdot \nm{\xi}^2_2}\cdot  \frac{M_k^{\max}}{N}\\
&\lesssim
\nu_k \cdot \frac{1}{N}\cdot \left(\sum_{j\in\Delta_k} \frac{\mathrm{Per}(I_j)^2}{\abs{I_j}}\cdot  \frac{N}{ M_k^{\min}} + \sum_{j\in\Delta_k} \mathrm{Per}(I_j) + \abs{\Lambda^c}\cdot \frac{M_k^{\max}}{N}\right).
}
So, condition (b') is implied by
$$
\frac{m_k}{\abs{\Gamma_k }} \gtrsim \cL \cdot \nu_k\cdot \frac{1}{N}\cdot \left(\sum_{j\in\Delta_k} \frac{\mathrm{Per}(I_j)^2}{\abs{I_j}}\cdot  \frac{N}{M_k^{\min}} + \sum_{j\in\Delta_k} \mathrm{Per}(I_j) + \abs{\Lambda^c}\cdot  \frac{M_k^{\max}}{N}\right)
$$
which is condition (i).
Finally, condition (c') is implied by condition (i).

In the case where $\xi = 0$, by retracing the steps in the case of $\xi\neq 0$, we can show that the conditions of Theorem \ref{thm:general_21} are true if (i') and (ii') below hold.
\begin{itemize}
\item[(i')] For $k=1,\ldots, r$,

$$
\frac{m_k}{\abs{\Gamma_k}} \gtrsim \cL\cdot
\frac{c_k}{N}\cdot  \frac{1}{M^{\min}_k } \cdot S\left(\Lambda, M^{\min}_k \right).
$$

\item[(ii')] For $k=1,\ldots, r$,
$m_k \gtrsim \cL \cdot \hat m_k
$ with
$$
1 \gtrsim \sum_{k=1}^r \left(\frac{\abs{\Gamma_{k}}}{\hat m_k} -1\right) \cdot  c_k^2 \cdot  \frac{1}{ (M^{\min}_k)^2} \cdot  \hat s_k 
$$
for any $\br{\hat s_k}_{k=1}^r$ such that
$$
\sum_{k=1}^r\hat s_k \leq F(\Lambda).
$$

\end{itemize}
 Note that (i') and (ii') are less restrictive than (i) and (ii).

\section{Proof of Theorem \ref{thm:general}}\label{sec:proofs}
There has been some recent analysis on the minimizers for problems of the form (\ref{eq:prob}) both in the context of compressed sensing \cite{candes2011compressed} and for general linear inverse problems \cite{vaiter2011robust,haltmeier2013stable,grasmair2011necessary}. 
The approach of the latter three cited works is to show that  robust recovery  is implied by the existence of a vector which satisfies certain properties. This vector is often referred to as the dual certificate, and much of the work in proving robust recovery is in deriving the conditions under which this dual certificate exists. We will follow this approach to prove Theorem \ref{thm:general}.

\subsection{Existence of dual certificate implies stable recovery}

\begin{lemma}\label{lem:Q_perp}
Given index set $\Lambda$, let $\cW_\Lambda$ be such that $\cW_\Lambda \supset \cN(  \rP_\Lambda \rD)$. Then $$\norm{\rQ_{\cW_\Lambda}^\perp z}_{2} \leq \norm{(  \rP_\Lambda \rD)^\dagger}_{1\to 2}\norm{  \rP_\Lambda \rD z}_1.$$

\end{lemma}
\begin{proof}
First note that $\cW_\Lambda^\perp \subset \cN(  \rP_\Lambda \rD)^\perp$ and  
$$\rQ_{\cW_\Lambda}^\perp \rQ_{\cN(  \rP_\Lambda \rD)}^\perp = \rQ_{\cN(   \rP_\Lambda \rD)}^\perp \rQ_{\cW_\Lambda}^\perp = \rQ_{\cW_\Lambda}^\perp.$$ 
Furthermore, since $ \rQ_{\cN(  \rP_\Lambda \rD)}^\perp = (  \rP_\Lambda \rD)^\dagger   \rP_\Lambda \rD$, we have that $\norm{\rQ_{\cW_\Lambda}^\perp z}_2 \leq \norm{(  \rP_\Lambda \rD)^\dagger}_{1\to 2}\norm{  \rP_\Lambda \rD z}_1$.
\end{proof}

\begin{proposition}[Dual vector for the constrained problem]\label{prop:dual}

Let  $\rA, \rD\in\cB(\ell^2(\bbN))$, $\Lambda\subset \bbN$ and let $\cW_\Lambda\subset \ell^2(\bbN)$ be such that $\cW_\Lambda \supset \cN(  \rP_\Lambda \rD)$.
Let $\Omega := \Omega_1 \cup \cdots \cup \Omega_r \subset \bbN$ be the union of $r$ disjoint subsets and $\br{q_k}_{k=1}^r \in (0,1]^r$.  Define $q= \min_{j=1}^r q_j$ and
$$\rP_{\Omega,\mathbf{ q}} := q_1^{-1}\rP_{\Omega_1}\oplus \ldots \oplus q_r^{-1} \rP_{\Omega_r}, \quad \rP_{\Omega, \sqrt{\mathbf{q}}} := q_1^{-1/2}\rP_{\Omega_1}\oplus \ldots \oplus q_r^{-1/2} \rP_{\Omega_r}.$$
Let  $y = \rP_\Omega \rA x + \xi$ in (\ref{eq:prob}) with $\norm{\xi}_2\leq \delta$ and let $\hat x = x + z$ be a $b$-optimal solution to (\ref{eq:prob}). 
Let 
$$1-\left(c_0+ 2c_1 c_2 q K  \right)\geq \gamma$$
for $c_0, c_1, c_2, K >0$,
 and suppose that $\rQ_{\cW_\Lambda}\rA^* \rP_{\Omega,\mathbf{ q}} \rA \rQ_{\cW_\Lambda}$ is invertible on $\rQ_{\cW_\Lambda}(\ell^2(\bbN))$ with
\be{\label{eq:dual_inv_cond}
\norm{(\rQ_{\cW_\Lambda}\rA^* \rP_{\Omega,\mathbf{ q}} \rA \rQ_{\cW_\Lambda} )^{-1}}_{2\to 2} \leq \frac{4}{3}K
}
\be{\label{eq:dual_sym}
\norm{\rQ_{\cW_\Lambda}  \rA^* \rP_{\Omega,\mathbf{q}} \rA \rQ_{\cW_\Lambda}}_{2\to 2} \leq \frac{5}{4} 
}
\be{ \label{eq:dual_col_nm_cond}
\max_{j=1,\ldots,N} \norm{ \rP_{\Omega,\sqrt{\mathbf{q}}} \rA \rQ_{\cW_\Lambda}^\perp (\rP_\Lambda \rD)^\dagger e_j}_2 \leq c_2
}
and that there exists some $\rho = \rA^*\rP_\Omega w$ such that the following holds:
\begin{enumerate}
\item[(i)] $\norm{ \rQ_{\cW_\Lambda} \rD^* \rP_{\Lambda^c}   \sgn(\rP_{\Lambda^c} \rD x) - \rQ_{\cW_\Lambda} \rho}_2\leq c_1\cdot q$
\item[(ii)]$\inf_{u\in\cN(  \rD^*\rP_{\Lambda}  )}\norm{( \rD^* \rP_\Lambda )^\dagger \rQ_{\cW_\Lambda}^\perp ( \rD^* \rP_{\Lambda^c}   \sgn(\rP_{\Lambda^c} \rD x) -\rho) -u}_\infty\leq  c_0$.
\end{enumerate}
Then, 
\spls{
\norm{z}_2 &\lesssim \delta\cdot \left( \frac{  K}{\sqrt{q}}  +
C\cdot \left( c_1 \sqrt{q} K + \norm{w}_2 \right)  \right)
 +
 C\cdot \left( \norm{\rP_\Lambda \rD x}_1 + b\right).
}
where $C =  \gamma^{-1}\left(c_2  K + \nm{(\rP_\Lambda \rD)^\dagger}_{1\to 2}\right) 
$.

\end{proposition}

\begin{proof}
First note that by  (\ref{eq:dual_sym}),
$$\norm{\rQ_{\cW_\Lambda}  \rA^* \rP_{\Omega,\mathbf{q}}}_2 \leq \sqrt{\frac{5}{4q}} , \quad \norm{\rQ_{\cW_\Lambda}  \rA^* \rP_{\Omega,\sqrt{\mathbf{q}}}}_2 \leq \sqrt{\frac{5}{4}} $$
 Furthermore, by Lemma \ref{lem:Q_perp},
\spl{\label{eq:bound_noise}
\norm{z}_2 \leq \norm{\rQ_{\cW_\Lambda} z}_2 + \norm{\rQ_{\cW_\Lambda}^\perp z}_2  \leq \norm{\rQ_{\cW_\Lambda} z}_2 + \norm{(  \rP_\Lambda \rD)^\dagger}_{1\to 2}\norm{  \rP_\Lambda \rD z}_1.
}
We seek to bound $\norm{\rQ_{\cW_\Lambda} z}_2$ and $\norm{\rP_\Lambda \rD z}_1$. 
To bound $\norm{\rQ_{\cW_\Lambda} z}_2$, observe that by (\ref{eq:dual_inv_cond}), (\ref{eq:dual_sym}) and (\ref{eq:dual_col_nm_cond}),
\spl{ \label{eq:bound_cosparse_perp}
\norm{\rQ_{\cW_\Lambda} z}_2 &= \norm{(\rQ_{\cW_\Lambda} \rA^* \rP_{\Omega,\mathbf{q}} \rA \rQ_{\cW_\Lambda})^{-1} \rQ_{\cW_\Lambda}  \rA^* \rP_{\Omega,\mathbf{q}} \rA \rQ_{\cW_\Lambda} z}_2\\
&\leq  \norm{(\rQ_{\cW_\Lambda} \rA^* \rP_{\Omega,\mathbf{q}} \rA \rQ_{\cW_\Lambda})^{-1}}_2 \norm{\rQ_{\cW_\Lambda}  \rA^* \rP_{\Omega,\mathbf{q}} \rA (z-\rQ_{\cW_\Lambda}^\perp z)}_2\\
&\leq \frac{2\sqrt{5}  K}{3\sqrt{q}}\nm{\rP_\Omega \rA z}_2 + \frac{4K}{3} \max_{j\in\bbN}\norm{\rQ_{\cW_\Lambda}  \rA^* \rP_{\Omega,\mathbf{q}} \rA \rQ_{\cW_\Lambda}^\perp (\rP_\Lambda \rD)^\dagger e_j}_2 \norm{\rP_\Lambda \rD z}_1 \\
&\leq \delta\cdot  \frac{4\sqrt{5}  K}{3\sqrt{q}}  + \frac{ 2\sqrt{5} K c_2}{3}\norm{\rP_\Lambda \rD z}_1\\
&\leq \delta\cdot  \frac{4  K}{\sqrt{q}}  +  2 K c_2\norm{\rP_\Lambda \rD z}_1.
}
To bound $\norm{ \rP_\Lambda \rD z}_1$, note that
\spl{\label{eq:dual_consq_min}
\norm{  \rD(x+z)}_1 &= \norm{  \rP_\Lambda \rD(x+z)}_1 +\norm{  \rP_{\Lambda^c} \rD (x+z)}_1 \\
&\geq \norm{  \rP_\Lambda \rD z}_1 - \norm{  \rP_\Lambda \rD x}_1 +\norm{  \rP_{\Lambda^c} \rD x}_1 +\Re \ip{  \rP_{\Lambda^c} \rD  z}{\sgn(  \rP_{\Lambda^c} \rD x)}\\
&\geq \norm{  \rP_\Lambda \rD z}_1 - 2\norm{  \rP_\Lambda \rD x}_1 +\norm{\rD x}_1 +\Re \ip{\rP_{\Lambda^c} \rD  z}{\sgn(\rP_{\Lambda^c} \rD x)}
}
and since $x+z$ is a $b$-optimal solution  by assumption, it follows that $b+ \norm{\rD x}_1 \geq \norm{\rD(x+z)}_1 $, and we have that
\spl{ \label{eq:bound_cosparse}
\norm{\rP_\Lambda \rD z}_1 \leq b+ 2\norm{\rP_\Lambda \rD x}_1  +\abs{ \ip{\rP_{\Lambda^c} \rD  z}{\sgn(\rP_{\Lambda^c} \rD x)}}.
}
Using the existence of a dual vector $\rho$ with $\rho = \rA^* \rP_\Omega w$, we have that
\spl{\label{eq:dual_sgn}
&\abs{ \ip{\rP_{\Lambda^c} \rD  z}{\sgn(\rP_{\Lambda^c} \rD x)}} = \abs{ \ip{z}{ \rD^* \rP_{\Lambda^c} \sgn(\rP_{\Lambda^c} \rD x)}}\\
&= \abs{ \ip{ z}{\rQ_{\cW_\Lambda} \rD^* \rP_{\Lambda^c} \sgn(\rP_{\Lambda^c} \rD x) - \rQ_{\cW_\Lambda} \rho} + \ip{ z}{\rho} + \ip{z}{\rQ_{\cW_\Lambda}^\perp ( \rD^* \rP_{\Lambda^c} \sgn(\rP_{\Lambda^c} \rD x) -\rho ) } }\\
&= \abs{ \ip{ z}{\rQ_{\cW_\Lambda} \rD^* \rP_{\Lambda^c} \sgn(\rP_{\Lambda^c} \rD x) - \rQ_{\cW_\Lambda} \rho} + \ip{ z}{\rho} + \ip{z}{\rQ_{\cN(\rP_\Lambda \rD)}^\perp \rQ_{\cW_\Lambda}^\perp ( \rD^* \rP_{\Lambda^c} \sgn(\rP_{\Lambda^c} \rD x) -\rho ) } }\\
&\leq  \norm{\rQ_{\cW_\Lambda} z} \norm{ \rQ_{\cW_\Lambda} \rD^* \rP_{\Lambda^c} \sgn(\rP_{\Lambda^c} \rD x) - \rQ_{\cW_\Lambda} \rho}+ \norm{\rP_\Omega \rA x} \norm{w} \\
&\qquad\qquad+\abs{\ip{\rP_\Lambda \rD z}{((\rP_\Lambda \rD)^\dagger)^*\rQ_{\cW_\Lambda}^\perp ( \rD^* \rP_{\Lambda^c} \sgn(\rP_{\Lambda^c} \rD x) -\rho)  } }\\
&\leq c_1 q \norm{\rQ_{\cW_\Lambda} z} + 2\delta \norm{w} + \norm{\rP_\Lambda \rD z}_1 \inf_{u\in\cN(\rD^* \rP_{\Lambda})}\norm{((\rP_\Lambda \rD)^\dagger)^*\rQ_{\cW_\Lambda}^\perp ( \rD^* \rP_{\Lambda^c} \sgn(\rP_{\Lambda^c} \rD x) -\rho) -u}_\infty.
}
Thus, by (\ref{eq:bound_cosparse_perp}) and assumption (ii),
\spls{
\abs{ \ip{\rP_{\Lambda^c} \rD  z}{\sgn(\rP_{\Lambda^c} \rD x)}} 
&\leq \left( 4\sqrt{q} c_1 K+ 2 \norm{w} \right) \cdot \delta + \left(c_0+ 2c_1 c_2 q K  \right) \cdot \norm{\rP_\Lambda \rD z}_1.
}
By plugging this back into (\ref{eq:bound_cosparse}), and recalling that $1- \left(c_0+ 2c_1 c_2 q K  \right)\geq \gamma$, we have that
$$
\norm{\rP_\Lambda \rD z}_1 \leq \gamma^{-1} \left(b+ 2\norm{\rP_\Lambda \rD x}_1 + \left( 4\sqrt{q} c_1  K+ 2 \norm{w} \right)\cdot \delta \right).
$$
So, having obtained bounds for $\nm{\rP_\Lambda \rD z}_2$ and $\nm{\rQ_{\cW_\Lambda} z}_2$,  (\ref{eq:bound_noise}) yields
\spls{
\norm{z}_2 &\lesssim \delta\cdot \left( \frac{  K}{\sqrt{q}}  +
C\cdot \left( c_1 \sqrt{q} K + \norm{w}_2 \right)  \right)
 +
 C\cdot  \norm{\rP_\Lambda \rD x}_1 + C\cdot b.
}
where $C =  \gamma^{-1}\left(c_2   K + \nm{(\rP_\Lambda \rD)^\dagger}_{1\to 2}\right) 
$.

\end{proof}

\subsection{Verification of the conditions in Proposition \ref{prop:dual}} \label{sec:verification_of_dualprop}

In this section, we will derive conditions under which the conditions of Proposition \ref{prop:dual} (with appropriate values for $c_0$, $c_1$ and $c_2$)  are satisfied.

We first remark that the probability that the conditions to Proposition \ref{prop:dual} hold for $\Omega$ chosen in accordance to a uniform sampling model as described in Definition \ref{multi_level_dfn}
is up to a constant bounded by the probability that they are satisfied when $\Omega$ is chosen in accordance to a Bernoulli model (described subsequently). Such equivalence has become standard in the compressed sensing literature and we refer to \cite{candes2006robust} for further details.
 Therefore, throughout this section, we will assume that the sampling set $\Omega = \Omega_1 \cup \cdots \cup \Omega_r$ adheres to a Bernoulli model, that is, 
 for $0=M_0 < M_1 <\cdots < M_r = M$, and $q_k = \frac{m_k}{M_k - M_{k-1}} \in [0,1]$, 
 $$
 \Omega_k =\left( \br{\delta_j \cdot j: j=M_{k-1}+1,\ldots, M_k}\cap\br{M_{k-1}+1,\ldots, M_k}\right) ,\quad k=1,\ldots, r
 $$
where $\delta_j$ is a random variable such that $\bbP( \delta_j = 1) = q_k$ and  $\bbP(\delta_k = 0) = 1-q_k$ for $j\in\br{M_{k-1}+1,\ldots, M_k}$. We write $\Omega_k \sim \mathrm{Ber}(q_k)$.

We also let $\Gamma_k = \br{M_{k-1}+1,\ldots, M_k}$, $q= \min_{j=1}^r q_j$ and
$$\rP_{\Omega,\mathbf{ q}} := q_1^{-1}\rP_{\Omega_1}\oplus \ldots \oplus q_r^{-1} \rP_{\Omega_r}, \quad \rP_{\Omega, \sqrt{\mathbf{q}}} := q_1^{-1/2}\rP_{\Omega_1}\oplus \ldots \oplus q_r^{-1/2} \rP_{\Omega_r}.$$ 
We will show that it suffices to let $\br{q_k}_{k=1}^r$ satisfy the following conditions.

\begin{assumption}\label{assumption_sampling}
Suppose that $\rA, \rD\in \cB(\ell^2(\bbN))$ with $\nm{\rA}_{2\to 2}=1$. Let  $\Lambda\subset \bbN$, $\cW_\Lambda \supset \cN(\rP_\Lambda \rD)$ be a subspace of dimension $s$ and let $\rW\in\cB(\ell^2(\bbN))$ be such that its columns form an orthonormal basis of $\cW_\Lambda$.   Let $q\in (0,1]$, let  $B \geq \norm{(\rP_\Lambda\rD)^\dagger}_{1\to 2}$.
Assume that there exists $\rX \in\cB(\ell^2(\bbN))$ which is invertible and $M\in\bbN$ such that  the following holds.  
\bes{
\nm{\rQ_{\cW_\Lambda} \rA^* \rP_{[M]} \rA \rX \rQ_{\cW_\Lambda} - \rQ_{\cW_\Lambda}}_{2 \to 2 } \leq \frac{1}{8},
}
\bes{
\nm{\blam\circ \rW^*\rA^* \rP_{[M]}\rA  \rX \rW \circ\blam^{-1} - \rW^* \rW}_{\infty\to \infty} \leq \left(4\log_2^{1/2}(4C_* M\sqrt{s}/q)\right)^{-1},
}
and
\bes{
\norm{(\rD^* \rP_\Lambda)^\dagger \rQ_{\cW_\Lambda}^\perp \rA^* \rP_{[M]} \rA \rX \rW\circ\blam^{-1}}_{\infty \to \infty} \leq \frac{1}{16}
}
where 
$C_* = \nm{\blam^{-1}}_\infty \cdot \max\br{1, B \nm{\rX}_{2\to 2}}$.

Defined $ Y = \min\br{  q \cdot( 8\nm{\blam^{-1}}_\infty \sqrt{s}\cdot \norm{\rX \rA^* \rP_{[M]}}_{2\to 2})^{-1}, \,  B\cdot \sqrt{q} },$ and
\bes{
\tilde M =\min\br{i\in\bbN : \max_{k\geq i}  \norm{\rP_{[M]} \rA \rQ_{\cW_\Lambda}^\perp (\rP_\Lambda\rD)^\dagger e_i}_2 \leq Y}.
}
Assume that the following coherence conditions hold.

\begin{itemize}
\item[(a)] 
For each $k=1,\ldots, r$,
\bes{
 q_k \gtrsim  (\log(s\epsilon^{-1}) +1) \log(q^{-1}\tilde M C_* \sqrt{s}) \cdot \mu_k \cdot  \sum_{j=1}^s  \cdot \mu(\rP_{\Gamma_k} \rA \rX \rW \rP_{\br{j}}) \cdot\blam^{-1}_j,  \qquad t=1,2
 }
 where
 $$\mu_k = \max\br{ \mu(\rP_{\Gamma_k} \rA \rQ_{\cW_\Lambda}^\perp(\rP_\Lambda \rD)^\dagger), \quad \mu( \rP_{\Gamma_k} \rA \rX \rW \circ\blam)}$$
and $q_k \gtrsim (\log(s\epsilon^{-1}) +1) \log(q^{-1}\tilde M C_* \sqrt{s}) \cdot \hat q_k $ such that $\br{\hat q_k}_{k=1}^r$ satisfies the following.
\bes{
1\gtrsim  \max_{\norm{\eta}_\infty =1}\sum_{k=1}^r (\hat q_k^{-1} -1)\cdot \mu(\rP_{\Gamma_k}\rA \xi   )^2 \cdot \norm{\rP_{\Gamma_k} \rA \rX \rW \circ\blam^{-1} \cdot \eta}_2^2,  
}
for all $\xi \in \br{\rQ_{\cW_\Lambda}^\perp(\rP_\Lambda \rD)^\dagger) e_i, \, \blam_j\rW e_j  : i\in \bbN, \, j=1,\ldots, s}$.

\item[(b)] For each $k=1,\ldots, r$,

$$
q_k \gtrsim \log\left(\frac{s}{\epsilon}\right) \cdot \left( \nm{\rX}_{2\to 2} +1\right) \cdot \sum_{l=1}^r \mu_{k,l}^2
$$
where $\mu_{k,l} = \max\br{\mu(\rP_{\Gamma_k} \rA \rW \rP_{\br{l}}), \mu(\rP_{\Gamma_k} \rA \rX \rW \rP_{\br{l}})}$.

\item[(c)]  
$$
B^2 \gtrsim \log\left(\frac{\tilde M}{\gamma}\right) \cdot \max_{k=1}^r(q_k^{-1} -1) \cdot \mu(\rP_{\Gamma_k}\rA \rQ_{\cW_\Lambda}^\perp(\rP_\Lambda \rD)^\dagger    )^2.
$$
\end{itemize}

\end{assumption}

The proofs in this section will make use of two Bernstein inequalities which we state here.

\begin{theorem}[Bernstein inequality for random variables \cite{foucart2013mathematical}]\label{thm:bernstein}
Let $Z_1,\ldots, Z_M \in\bbC$ be independent random variables with zero mean such that
$
\abs{Z_j} \leq K
$ almost surely for  all $l=1,\ldots, M$ and some constant $K>0$. Assume also that $\sum_{j=1}^M\bbE\abs{Z_j}^2 \leq \sigma^2$ for some constant $\sigma^2 >0$. Then for $t>0$,
$$
\bbP\left(\abs{\sum_{j=1}^M Z_j} \geq t\right) \leq 4\exp\left(-\frac{t^2/4}{\sigma^2 + Kt/(3\sqrt{2})}\right).
$$
If $Z_1,\ldots, Z_M \in\bbR$ are real instead of complex random variables, then 
$$
\bbP\left(\abs{\sum_{j=1}^M Z_j} \geq t\right) \leq 2\exp\left(-\frac{t^2/2}{\sigma^2 + Kt/3}\right).
$$
\end{theorem}

\begin{theorem}[Bernstein inequality for rectangular matrices \cite{tropp2012user}]\label{thm:matrixBernstein}
Let $Z_1,\ldots, Z_M \in\bbC^{d_1\times d_2}$ be independent random matrices  such that $\bbE Z_j = 0$ for each  $j=1,\ldots, M$ and $\nm{Z_j}_{2\to 2} \leq K$ almost surely for each  $j=1,\ldots, M$ and some constant $K>0$. Let $$\sigma^2 := \max \br{\nm{\sum_{j=1}^M \bbE(Z_j Z_j^*)}_{2\to 2}, \nm{\sum_{j=1}^M \bbE(Z_j^* Z_j)}_{2\to 2}}.$$
Then, for $t>0$,
$$
\bbP\left(\nm{\sum_{j=1}^M Z_j}_{2\to 2}\geq t\right) \leq 2(d_1+d_2)\exp\left(\frac{-t^2/2}{\sigma^2 + Kt/3}\right)
$$
\end{theorem}

\subsection*{Analysis of conditions (\ref{eq:dual_inv_cond}), (\ref{eq:dual_sym}) and (\ref{eq:dual_col_nm_cond})}
\begin{lemma} \label{lem:verif_propcondns}

Let $\epsilon\in (0,1]$ and let $ B\geq \nm{(\rP_\Lambda \rD)^\dagger}_{1\to 2}$. Let $E$ be the event that $\rQ_{\cW_\Lambda}\rA^* \rP_{\Omega,\mathbf{ q}} \rA \rQ_{\cW_\Lambda}$ is invertible on $\rQ_{\cW_\Lambda}(\ell^2(\bbN))$,
\be{\label{eq:dual_inv_cond_e1}
  \norm{(\rQ_{\cW_\Lambda}\rA^* \rP_{\Omega,\mathbf{ q}} \rA \rQ_{\cW_\Lambda} )^{-1}}_{2\to 2} \leq \frac{4}{3}\nm{\rX}_{2\to 2},
}
\be{\label{eq:dual_sym_e1}
\norm{\rQ_{\cW_\Lambda}  \rA^* \rP_{\Omega,\mathbf{q}} \rA \rQ_{\cW_\Lambda}}_{2\to 2} \leq \frac{5}{4},
}
and
\be{ \label{eq:dual_col_nm_cond_e1}
\max_{j \in\bbN} \norm{ \rP_{\Omega,\sqrt{\mathbf{q}}} \rA \rQ_{\cW_\Lambda}^\perp (\rP_\Lambda \rD)^\dagger e_j}_2 \leq B.
}
Suppose that conditions in Assumption \ref{assumption_sampling} are satisfied. Then, $\bbP(E^c) \leq \epsilon/6$.

\end{lemma}
\begin{proof}
Let $p=\epsilon/6$.
We first remark that (\ref{eq:dual_inv_cond_e1}) is satisfied if
\bes{
\nm{\rQ_{\cW_\Lambda}\rX \rA^*  \rP_{\Omega,\mathbf{q}} \rA \rQ_{\cW_\Lambda}  - \rQ_{\cW_\Lambda} }_{2\to 2} \leq \frac{1}{4}
}
since this would imply that $(\rQ_{\cW_\Lambda}\rX \rA^*  \rP_{\Omega,\mathbf{q}} \rA \rQ_{\cW_\Lambda})^{-1}$ exists on $\rQ_{\cW_\Lambda}(\ell^2(\bbN))$. Now, since $\rX$ is invertible, given any $y\in\ell^2(\bbN)$,
$$
 \nm{\rQ_{\cW_\Lambda}\rA^*  \rP_{\Omega,\mathbf{q}} \rA \rQ_{\cW_\Lambda} y}_2 \geq \nm{\rX}_{2\to 2}^{-1} \cdot
\nm{\rQ_{\cW_\Lambda}\rX \rA^*  \rP_{\Omega,\mathbf{q}} \rA \rQ_{\cW_\Lambda} y}_2
 \geq \nm{\rX}_{2\to 2}^{-1}\cdot \left(1-\frac{1}{4}\right)\cdot \nm{y}_2,
$$
which implies that  $(\rQ_{\cW_\Lambda}\rA^*  \rP_{\Omega,\mathbf{q}} \rA \rQ_{\cW_\Lambda})^{-1}$ exists and satisfies the norm bound in (\ref{eq:dual_inv_cond_e1}).
Second, observe that (\ref{eq:dual_col_nm_cond_e1}) holds provided that
\bes{
\max_{j \in \bbN} \norm{ \rP_{\br{j}}(\rD^* \rP_\Lambda )^\dagger \rQ_{\cW_\Lambda}^\perp \rA^* \rP_{\Omega,\mathbf{q}} \rA \rQ_{\cW_\Lambda}^\perp (\rP_\Lambda \rD)^\dagger \rP_{\br{j}}}_{2\to 2} \leq B^2.
}
So, to prove this lemma, it suffices to show that
\be{\label{eq:E_toshow1}
\bbP\left(\nm{\rQ_{\cW_\Lambda}\rX \rA^*  \rP_{\Omega,\mathbf{q}} \rA \rQ_{\cW_\Lambda}  - \rQ_{\cW_\Lambda} }_{2\to 2} > \frac{1}{4}\right) \leq p/3,
}
\be{\label{eq:E_toshow2}
\bbP\left(\norm{\rQ_{\cW_\Lambda}  \rA^* \rP_{\Omega,\mathbf{q}} \rA \rQ_{\cW_\Lambda}}_{2\to 2} > \frac{5}{4} \right) \leq p/3,
}
and
\be{\label{eq:E_toshow3}
\bbP\left(\max_{j \in \bbN} \norm{ \rP_{\br{j}}(\rD^* \rP_\Lambda )^\dagger \rQ_{\cW_\Lambda}^\perp \rA^* \rP_{\Omega,\mathbf{q}} \rA \rQ_{\cW_\Lambda}^\perp (\rP_\Lambda \rD)^\dagger \rP_{\br{j}}}_{2\to 2} > B^2\right) \leq p/3.
}
This is achieved by setting $\gamma = p/3$ in Propositions \ref{prop:B_3}, \ref{prop:bound_sym} and \ref{prop:bound_max_col} and observing that the assumptions of these propositions are implied by Assumption \ref{assumption_sampling}.

\end{proof}

\begin{proposition}\label{prop:B_3}

Let $\cW_\Lambda = \cR(\rW)$ be such that the columns of $\rW$ form an orthonormal set and the dimensions of $\cW_\Lambda =s$. 
Let $\rA \in\cB(\ell^2(\bbN))$ be such that $\nm{\rA}_{2\to 2} \leq 1$. Suppose that $\rX\in\cB(\ell^2(\bbN))$  and $M\in\bbN$ is such that
$$
\nm{\rQ_{\cW_\Lambda}\rX \rA^* \rP_{[M]}\rA \rQ_{\cW_\Lambda} - \rQ_{\cW_\Lambda} }_{2\to 2} \leq \frac{1}{8}.
$$
 Then, given any $\gamma>0$,
\spls{
\bbP\left( \norm{\rQ_{\cW_\Lambda} \rX \rA^* \left(q_1^{-1}\rP_{\Omega_1}\oplus \ldots\oplus q_r^{-1} \rP_{\Omega_r}\right) \rA \rQ_{\cW_\Lambda} - \rQ_{\cW_\Lambda}}_{2\to 2}  \geq \frac{1}{4}\right) \leq \gamma
}
provided that for each $k=1,\ldots, r$
$$
q_k \geq \log\left(\frac{4s}{\gamma}\right) \cdot 128 \cdot \left(  \nm{\rX}_{2\to 2} +1\right) \cdot \sum_{l=1}^r \mu_{k,l}^2
$$
where $\mu_{k,l} = \max\br{\mu(\rP_{\Gamma_k} \rA \rW \rP_{\br{l}}), \mu(\rP_{\Gamma_k} \rA \rX \rW \rP_{\br{l}})}$.

\end{proposition}

\begin{proof}
Let $\tilde q_j := q_k$ for $j=\br{M_{k-1}+1,\ldots, M_k}$ and let $\br{\delta_j}_{j=1}^M$ be Bernoulli random variables such that $\bbP(\delta_j =1) = \tilde q_j$ and $\bbP(\delta_j = 0) = 1-\tilde q_j$.
First observe that since $\rQ_{\cW_\Lambda}\rX \rA^* \rP_{[M]}\rA\rQ_{\cW_\Lambda} =\rQ_{\cW_\Lambda} $ by definition of $\rX$ and since $\rQ_{\cW_\Lambda} = \rW \rW^*$,
\eas{
&\norm{\rQ_{\cW_\Lambda} \rX \rA^* \left(q_1^{-1}\rP_{\Omega_1}\oplus \ldots\oplus q_r^{-1} \rP_{\Omega_r}\right) \rA \rQ_{\cW_\Lambda} - \rQ_{\cW_\Lambda}}_{2\to 2}\\
&= \norm{\sum_{j=1}^M (\tilde q_j^{-1} \delta_j -1) \rQ_{\cW_\Lambda} \rX \rA^* \left(e_j\otimes\overline e_j\right) \rA \rQ_{\cW_\Lambda}}_{2\to 2}\\
&\leq \norm{\sum_{j=1}^M (\tilde e_j^{-1} \delta_j -1) \rW^* \rX \rA^* \left(e_j\otimes\overline e_j\right) \rA \rW}_{2\to 2}+ \norm{\rQ_{\cW_\Lambda}\rX \rA^*\rP_{[M]}\rA\rQ_{\cW_\Lambda} - \rQ_{\cW_\Lambda} }_{2\to 2}\\
&\leq \norm{\sum_{j=1}^M (\tilde q_j^{-1} \delta_j -1) \rW^* \rX \rA^* \left(e_j\otimes\overline e_j\right) \rA \rW}_{2\to 2} + \frac{1}{8}.
}
Therefore, it suffices to show that under the assumptions of this proposition,
\spls{
\bbP\left(  \norm{\sum_{j=1}^M (\tilde q_j^{-1} \delta_j -1) \rW^* \rX \rA^* \left(e_j\otimes\overline e_j\right) \rA \rW}_{2\to 2}   \geq \frac{1}{8}\right) \leq \gamma.
}
Let $Z_j = (\tilde q_j^{-1} \delta_j -1) \rW^* \rX \rA^* \left(e_j\otimes\overline e_j\right) \rA \rW$, then $Z_1,\ldots, Z_M\in\bbC^{s\times s}$ are independent mean zero matrices. We aim to apply Theorem \ref{thm:matrixBernstein}.
Let $\xi_j = \rW^* \rX \rA^* e_j$ and $\eta_j = \rW^*  \rA^* e_j$ 
\spls{
&\norm{Z_j}_{2\to 2} \leq \max\br{1, \tilde q_j^{-1} -1} \norm{\xi_j\otimes \overline{\eta_j}}_2 \leq \max\br{1, \tilde q_j^{-1} -1} \norm{\xi_j}_2\norm{\eta_j}_2\\
&\leq \max_{k=1}^r\br{ \max\br{1, q_k^{-1} -1} \sqrt{\sum_{l=1}^s \mu( \rP_{\Gamma_k}\rA \rX \rW \rP_{\br{l}})^2}\sqrt{\sum_{l=1}^s \mu(\rP_{\Gamma_k} \rA \rW \rP_{\br{l}})^2}} =: K.
}
Also,
\spls{
&\norm{\sum_{j=1}^M\bbE (Z_j^* Z_j)}_{2\to 2} = \sup_{\norm{x}_2=1} \abs{\sum_{j=1}^M (\tilde q_j^{-1}-1)\ip{\xi_j\otimes \overline{\eta_j} x}{\xi_j\otimes \overline{\eta_j} x}}\\
&= \sup_{\norm{x}_2=1} \abs{\sum_{j=1}^M (\tilde q_j^{-1}-1)\ip{\xi_j}{ \eta_j}\ip{\xi_j}{x}\ip{\eta_j}{x}}\\
&\leq \max_{j=1}^M \abs{(\tilde q_j^{-1}-1)\ip{\xi_j}{ \eta_j}} \sup_{\norm{x}_2=1} \abs{\sum_{j=1}^M \ip{e_j}{\rA\rX\rW x}\ip{e_j}{\rA\rW x}}\\
&\leq \max_{j=1}^M \abs{(\tilde q_j^{-1}-1)\ip{\xi_j}{ \eta_j}}\cdot \norm{\rA\rX\rW}_{2\to 2}\norm{\rA\rW}_{2\to 2}\\
&\leq  \nm{\rX}_{2\to 2} \cdot \max_{j=1}^M (\tilde q_j^{-1} -1) \sum_{l=1}^s\abs{\ip{e_j}{\rA\rX\rW e_l}}\abs{\ip{e_j}{\rA\rW e_l}}\\
&\leq \nm{\rX}_{2\to 2} \cdot \max_{k=1}^r ( q_k^{-1}-1) \sum_{l=1}^s \mu(\rP_{\Gamma_k} \rA \rX \rW \rP_{\br{l}})\cdot\mu( \rP_{\Gamma_k}\rA \rW \rP_{\br{l}}) =: \sigma^2
}
where we have used the assumption that $\nm{\rA}_{2\to 2} = 1$.
Similarly,
$$
\norm{\sum_{k=1}^M\bbE (Z_k Z_k^*)}_{2\to 2} \leq \sigma^2.
$$
Thus, by Theorem \ref{thm:bernstein}
$$
\bbP( \norm{\rQ_{\cW_\Lambda} \rX \rA^* \left(q_1^{-1}\rP_{\Omega_1}\oplus \ldots\oplus q_r^{-1} \rP_{\Omega_r}\right) \rA \rQ_{\cW_\Lambda} - \rQ_{\cW_\Lambda}}_{2\to 2}  >\frac{1}{4})
\leq 4s \exp\left(\frac{-1}{128(\sigma^2 + K  /24)}\right).
$$

\end{proof}

\begin{proposition} \label{prop:bound_sym}

Let $\cW_\Lambda = \cR(\rW)$ be such that the columns of $\rW$ form an orthonormal set and the dimension of $\cW_\Lambda$ is $s$. 
Let $\rA \in\cB(\ell^2(\bbN))$ be such that $\nm{\rA}_{2\to 2} \leq 1$. Let $\gamma \in (0,1]$, then
\spls{
\bbP\left(\norm{\rQ_{\cW_\Lambda}  \rA^* \rP_{\Omega,\mathbf{q}} \rA \rQ_{\cW_\Lambda}}_{2\to 2} > \frac{5}{4} \right)\leq \gamma
}
if for each $k=1,\ldots, r$, 
$$
q_k  \geq 35 \cdot  \log \left(\frac{4s}{\gamma}\right) \cdot \sum_{j=1}^r \mu(\rP_{\Gamma_k} \rA  \rW \rP_{\br{j}})^2
$$

\end{proposition} 
\begin{proof}
 For $j=1,\ldots, M$, let $\delta_j$ be random Bernoulli variables such that $\bbP(\delta_j =1) = \tilde q_j$, where $\tilde q_j = q_k$ for $j=M_{k-1}+1,\ldots, M_k$. Observe that
 \spls{
\norm{\rQ_{\cW_\Lambda}  \rA^* \rP_{\Omega,\mathbf{q}} \rA \rQ_{\cW_\Lambda}}_{2\to 2}
&= \norm{\sum_{k=1}^M (\tilde{q}_k^{-1} \delta_k - 1)\rQ_{\cW_\Lambda}  \rA^* (e_k\otimes\overline e_k) \rA\rQ_{\cW_\Lambda} + \rQ_{\cW_\Lambda} \rA^* \rP_{[M]} \rA \rQ_{\cW_\Lambda}}_{2\to 2}\\
&\leq \norm{\sum_{k=1}^M (\tilde{q}_k^{-1} \delta_k - 1)\rW^*  \rA^* (e_k\otimes\overline e_k) \rA \rW}_{2\to 2} + 1.
} 
Let $Z_k = (\tilde{q}_k^{-1} \delta_k - 1)\rW^* \rA^* (e_k\otimes\overline e_k) \rA \rW$, then 
\spls{
\bbP\left(\norm{\rQ_{\cW_\Lambda}  \rA^* \rP_{\Omega,\mathbf{q}} \rA \rQ_{\cW_\Lambda}}_{2\to 2} > \frac{5}{4} \right)
\leq  \bbP\left(\norm{\sum_{k=1}^M Z_k}_{2\to 2} > \frac{1}{4} \right).
}
By applying Theorem \ref{thm:matrixBernstein} as in Proposition \ref{prop:B_3}, we obtain that 
$$
\bbP( \norm{\rQ_{\cW_\Lambda}\rA^* \left(q_1^{-1}\rP_{\Omega_1}\oplus \ldots\oplus q_r^{-1} \rP_{\Omega_r}\right) \rA \rQ_{\cW_\Lambda} - \rQ_{\cW_\Lambda}}_{2\to 2}  \geq \beta)
\leq 4s \exp\left(\frac{-\beta^2/2}{\sigma^2 + K \beta /3}\right).
$$
where $\beta = \frac{1}{4} $,
$$
 \sigma^2 = \max_{k=1}^r ( q_k^{-1}-1) \sum_{l=1}^s  \mu( \rP_{\Gamma_k}\rA \rW \rP_{\br{l}})^2,
$$
and
$$
K=\max_{k=1}^r\br{ \max\br{1, q_k^{-1} -1} \sum_{l=1}^s \mu(\rP_{\Gamma_k} \rA \rW \rP_{\br{l}})^2}.
$$
\end{proof}

\begin{proposition} \label{prop:bound_max_col}

Let $\cW_\Lambda = \cR(\rW)$ be such that the columns of $\rW$ form an orthonormal set and the dimension of $\cW_\Lambda$ is $s$. 
Let $\rD \in\cB(\ell^2(\bbN))$ and $\rA \in\cB(\ell^2(\bbN))$ be such that $\nm{\rA}_{2\to 2} \leq 1$. 
Let $\gamma\in (0,1]$ and let  $B \geq \norm{(\rP_\Lambda\rD)^\dagger}_{1\to 2} $. Then
\spls{
\bbP\left(\sup_{j\in\bbN}\norm{ \rP_{\br{j}} (\rD^* \rP_\Lambda)^\dagger \rQ_{\cW_\Lambda}^\perp \rA^* \rP_{\Omega,\mathbf{q}} \rA \rQ_{\cW_\Lambda}^\perp (\rP_\Lambda\rD)^\dagger \rP_{\br{j}}}_{2\to 2} > \frac{5}{4} B^2\right)\leq \gamma
}
if for each $k=1,\ldots, r$
$$
\log\left(\frac{2\tilde M}{\gamma}\right) \cdot \max_{k=1}^r(q_k^{-1} -1) \mu(\rP_{\Gamma_k}\rA \rQ_{\cW_\Lambda}^\perp(\rP_\Lambda \rD)^\dagger   )^2 \leq \frac{3}{14} B^2, 
$$
with 
$$
\tilde M = \min \br{j\in\bbN: \max_{k\geq j} \nm{\rP_{[M]} \rA \rQ_{\cW_\Lambda}^\perp (\rP_\Lambda \rD)^\dagger e_j}_{2} \leq \sqrt{\frac{5}{4}} \cdot B\cdot \sqrt{q} } <\infty.
$$

\end{proposition} 
\begin{proof}
For $j=1,\ldots, M$, let $\delta_j$ be random Bernoulli variables such that $\bbP(\delta_j =1) = \tilde q_j$, where $\tilde q_j = q_k$ for $j=M_{k-1}+1,\ldots, M_k$. Observe that for each $j \in \bbN$
\spls{
& \norm{\rP_{\br{j}} (\rD^* \rP_\Lambda)^\dagger \rQ_{\cW_\Lambda}^\perp \rA^* \rP_{\Omega,\mathbf{q}} \rA \rQ_{\cW_\Lambda}^\perp (\rP_\Lambda\rD)^\dagger \rP_{\br{j}}}_{2\to 2}\\
 &=  
 \Bigg \|
 \sum_{k=1}^M (\tilde q_k^{-1}\delta_k -1) \rP_{\br{j}} (\rD^* \rP_\Lambda)^\dagger \rQ_{\cW_\Lambda}^\perp \rA^* (e_k\otimes\overline e_k) \rA \rQ_{\cW_\Lambda}^\perp (\rP_\Lambda\rD)^\dagger \rP_{\br{j}} \\
 &\qquad\qquad\qquad\qquad+ \rP_{\br{j}} (\rD^* \rP_\Lambda)^\dagger \rQ_{\cW_\Lambda}^\perp \rA^*\rP_{[M]} \rA \rQ_{\cW_\Lambda}^\perp (\rP_\Lambda\rD)^\dagger \rP_{\br{j}}  
 \Bigg \|_{2\to 2}
 \\
 &\leq \norm{\sum_{k=1}^M (\tilde q_k^{-1}\delta_k -1) \rP_{\br{j}} (\rD^* \rP_\Lambda)^\dagger \rQ_{\cW_\Lambda}^\perp \rA^* (e_k\otimes\overline e_k) \rA \rQ_{\cW_\Lambda}^\perp (\rP_\Lambda\rD)^\dagger \rP_{\br{j}}}_{2\to 2} + \norm{(\rP_\Lambda\rD)^\dagger}_{1\to 2}^2\\
 &\leq  \abs{\sum_{k=1}^M (\tilde q_k^{-1}\delta_k -1)  \abs{\ip{\rA \rQ_{\cW_\Lambda}^\perp (\rP_\Lambda\rD)^\dagger e_j}{e_k}}^2} + B^2.
 }
 Thus, by letting $Z_k^j = (\tilde q_k^{-1}\delta_k -1)  \abs{\ip{\rA \rQ_{\cW_\Lambda}^\perp (\rP_\Lambda\rD)^\dagger e_j}{e_k}}^2$, we have that
 \spls{
 &\bbP\left(\norm{ \rP_{\br{j}} (\rD^* \rP_\Lambda)^\dagger \rQ_{\cW_\Lambda}^\perp \rA^* \rP_{\Omega,\mathbf{q}} \rA \rQ_{\cW_\Lambda}^\perp (\rP_\Lambda\rD)^\dagger \rP_{\br{j}}}_{2\to 2} > \frac{5}{4}B^2\right)\\
 &\leq \bbP\left( \abs{\sum_{k=1}^r Z_k^j} >  \frac{1}{4}B^2\right).
 }
Furthermore, we have that
$$
\abs{Z^j_k} \leq \max_{k=1}^r (q_k^{-1} -1)\abs{\mu(\rP_{\Gamma_k}\rA \rQ_{\cW_\Lambda}^\perp(\rP_\Lambda \rD)^\dagger \rP_{\br{j}}   )}^2 =: K_j
$$ 
and
\spls{
&\sum_{k=1}^M \bbE\left(\abs{Z_k^j}^2\right) = \sum_{k=1}^M (\tilde q_k^{-1} -1)\abs{\ip{\rA \rQ_{\cW_\Lambda}^\perp (\rP_\Lambda\rD)^\dagger e_j}{e_k}}^4\\
& \leq  B^2\max_{k=1}^r(q_k^{-1} -1) \mu(\rP_{\Gamma_k}\rA \rQ_{\cW_\Lambda}^\perp(\rP_\Lambda \rD)^\dagger \rP_{\br{j}}   )^2 = : \sigma_j^2.
}
Thus, letting $t = \frac{B^2}{4}$, Theorem \ref{thm:bernstein} yields
\spls{
\bbP\left( \abs{\sum_{k=1}^r Z_k^j} >  \frac{B^2}{4}\right)
\leq 2\exp\left(\frac{-t^2/2}{\sigma_j^2+ K_j t/3}\right).
 }
Now let $\Gamma \subset \bbN$ be such that
 $$
 \bbP\left(\sup_{j\in\Gamma}\norm{ \rP_{\br{j}} (\rD^* \rP_\Lambda)^\dagger \rQ_{\cW_\Lambda}^\perp \rA^* \rP_{\Omega,\mathbf{q}} \rA \rQ_{\cW_\Lambda}^\perp (\rP_\Lambda\rD)^\dagger \rP_{\br{j}}}_{2\to 2} > \frac{5 B^2}{4}\right) = 0.
 $$
 If $\abs{\Gamma^c}<\infty$, then by the union bound,
\spls{ \bbP\left(\max_{j\in\Gamma^c} \norm{ \rP_{\br{j}} (\rD^* \rP_\Lambda)^\dagger \rQ_{\cW_\Lambda}^\perp \rA^* \rP_{\Omega,\mathbf{q}} \rA \rQ_{\cW_\Lambda}^\perp (\rP_\Lambda\rD)^\dagger \rP_{\br{j}}}_{2\to 2} > \frac{5 B^2}{4}\right)\leq 2\abs{\Gamma^c} \max_{j=1}^r \exp\left(\frac{-t^2/2}{\sigma_j^2+ K_j t/3}\right).
}
To conclude this proof, we simply need to show that $\abs{\Gamma^c} < \tilde M$. Observe that
$$
\norm{ \rP_{\br{j}} (\rD^* \rP_\Lambda)^\dagger \rQ_{\cW_\Lambda}^\perp \rA^* \rP_{\Omega,\mathbf{q}} \rA \rQ_{\cW_\Lambda}^\perp (\rP_\Lambda\rD)^\dagger \rP_{\br{j}}}_{2\to 2}
\leq \frac{1}{q}\norm{\rP_{[M]}\rA \rQ_{\cW_\Lambda}^\perp (\rP_\Lambda\rD)^\dagger e_j}_{2}^2 \to 0 
$$
as $j\to \infty$. 
Therefore, $\tilde M$ is finite and
$$
\Upsilon = \br{j\in\bbN :  \frac{1}{\sqrt{q}}\norm{\rP_{[M]}\rA \rQ_{\cW_\Lambda}^\perp (\rP_\Lambda\rD)^\dagger e_j}_{2} > \sqrt{\frac{5}{4}} \cdot B}
$$
is a finite subset. Finally, $\Gamma^c$ is finite since $\abs{\Gamma^c} \leq \abs{\Upsilon} \leq \tilde M$.
\end{proof}

\subsection*{Construction of the dual certificate}
As explained in \cite{BAACHGSCS,adcockbreaking}, we may replace the Bernoulli sampling model stated at the start of this section \ref{sec:verification_of_dualprop} with  the following equivalent sampling model: $\Omega = \Omega_1\cup \cdots \cup \Omega_r$ with
$$
\Omega_k = \Omega_k^1 \cup \Omega_k^2 \cup \cdots \cup \Omega^\mu_k, \quad \br{M_{k-1}+1, \ldots, M_k} \supset \Omega_k^j\sim \mathrm{Ber}(q_k^j), \quad k=1,\ldots,r, \quad j=1,\ldots,\mu
$$
for some $\mu\in\bbN$ and $\br{q_k^j}_{j=1}^\mu$ such that 
$$
(1-q_k^1)(1-q_k^2) \ldots (1-q_k^\mu) = 1-q_k.
$$
We will assume this alternative model throughout the following theorem.

\begin{theorem}\label{thm:verif_dualcertif}
Let $\epsilon\in (0,1]$ and $\sigma \in \ell^\infty(\bbN)$ be such that $\nm{\sigma}_\infty \leq 1$. Let $\blam \in\bbR_+^s$ be such that 
$$
\norm{\blam\circ \rW^* \rD^* \rP_{\Lambda^c} \sigma}_\infty \leq 1.
$$
Suppose that the conditions of Assumption \ref{assumption_sampling} are satisfied. Then with probability exceeding $1-5\epsilon/6$, there exists $\rho = \rA^* \rP_\Omega w$ such that
\begin{enumerate}
\item[(i)] $\nm{\rQ_{\cW_\Lambda} \rD^* \rP_{\Lambda^c} \sigma - \rQ_{\cW_\Lambda} \rho}_2 \leq \frac{q}{8} \cdot \min\br{1, \left(c_2 \nm{\rX}_{2\to 2}\right)^{-1}}$
\item[(ii)] $\norm{ (\rD^* \rP_\Lambda)^\dagger \rQ_{\cW_\Lambda}^\perp  \rho}_\infty \leq \frac{1}{8}$
\item[(iii)] $\nm{w}_2 \leq \sqrt{\frac{s}{q}} \cdot \nm{\blam^{-1}}_\infty \cdot \sqrt{\frac{\log(p^{-1}) +\log(8M C_* \sqrt{s} q^{-1})}{\log_2(5MC_* \sqrt{s}q^{-1})}}$
\end{enumerate}
where  $C_* = \nm{\blam^{-1}}_\infty \cdot \max\br{1,c_2 \nm{\rX}_{2\to 2}}$.
\end{theorem}

\begin{proof}
We will construct $\rho$ using  a recursive golfing technique introduced in \cite{gross2011recovering,kueng2014ripless}. We first describe to the construction of $\rho$, then show that with probability exceeding $1-5\epsilon/6$, this construction satisfies conditions (i), (ii) and (iii) of this theorem.

 Let $\gamma = \epsilon/6$. Define
$\nu \in\bbN$, $\nu\leq \mu$, $\br{\alpha_j}_{j=1}^\mu$ and $\br{\beta_j}_{j=1}^\mu$ as follow:
\spls{
\mu = 8\lceil 3\nu + \log(\gamma^{-1})\rceil,\qquad \nu = \lceil \log_2(8 C_* M\sqrt{s}/q) \rceil\\
q_k^1 = q_k^2 = \frac{1}{4}q_k,\qquad \tilde q_k = q_k^3=\cdots = q_k^\mu, \qquad q_k = \frac{N_k - N_{k-1}}{m_k}\\
\alpha_1 = \alpha_2 = \left(2\log_2^{1/2}(4C_* M\sqrt{s}/q)\right)^{-1}, \qquad \alpha_i = \frac{1}{2}\\
\beta_1=\beta_2 = \frac{1}{8}, \qquad \beta_i = \frac{1}{8}\log_2(4 C_* M\sqrt{s}/q), \quad 3\leq i\leq \mu
}
Let $Z_0= \rQ_{\cW_\Lambda} \rD^* \rP_{\Lambda^c} \sigma$ and for $i=1,2$ define
\bes{
Z_i = \rQ_{\cW_\Lambda} \rD^* \rP_{\Lambda^c} \sigma - \rQ_{\cW_\Lambda}Y_i, \quad
Y_i = \sum_{j=1}^i \rA^* \rP_{\Omega^j,q^j} \rA \rX \rQ_{\cW_\Lambda} Z_{j-1}.
} 
Let $\Theta_1 = \br{1}, \ \Theta_2 = \br{1,2}$ and for $i\geq 3$, define
\spls{
\Theta_i &= \begin{cases}
\Theta_{i-1}\cup\br{i} & \norm{\blam\circ(\rW^* (\rQ_{\cW_\Lambda} - \rQ_{\cW_\Lambda} \rA^*\rP_{\Omega^i, \mathbf{q}^i} \rA \rX \rQ_{\cW_\Lambda}) Z_{i-1}}_\infty \leq \alpha_i \norm{\blam\circ \rW^* Z_{i-1}}_\infty \\
&\norm{ (\rD^* \rP_\Lambda )^\dagger \rQ_{\cW_\Lambda}^\perp \rA^* \rP_{\Omega^i, \mathbf{q}^i} \rA \rX \rQ_{\cW_\Lambda} Z_{i-1}}_\infty \leq \beta_i \norm{\blam\circ \rW^* Z_{i-1}}_\infty\\
\Theta_{i-1}&\text{otherwise.}
\end{cases}\\
Y_i &= \begin{cases}
\sum_{j\in\Theta_i} \rA^* \rP_{\Omega^j,q^i} \rA \rX \rQ_{\cW_\Lambda} Z_{j-1} & i\in\Theta_i\\
Y_{i-1}&\text{otherwise.}
\end{cases}\\
Z_i &= \begin{cases}
\rQ_{\cW_\Lambda} \rD^* \rP_{\Lambda^c}\sigma- \rQ_{\cW_\Lambda}Y_i & i\in\Theta_i\\
Z_{i-1}&\text{otherwise.}
\end{cases}
}
Define the following events
\spls{
A_i &: \  \norm{\blam\circ (\rW^* (\rQ_{\cW_\Lambda} - \rQ_{\cW_\Lambda} \rA^*\rP_{\Omega^i, \mathbf{q}^i} \rA \rX \rQ_{\cW_\Lambda}) Z_{i}}_\infty \leq \alpha_i \norm{\blam\circ \rW^* Z_{i}}_\infty, \quad i=1,2\\
B_i &: \ \norm{ (\rD^* \rP_\Lambda )^\dagger \rQ_{\cW_\Lambda}^\perp \rA^* \rP_{\Omega^i,\mathbf{q}^i} \rA \rX \rQ_{\cW_\Lambda} Z_{i}}_\infty \leq \beta_i \norm{\blam\circ \rW^* Z_{i}}_\infty, \quad i=1,2\\
B_3 &: \ \abs{\Theta_\mu} \geq \nu, \\
B_4 &: \ \mathop{\cap}_{i=1}^2 A_i \cap \mathop{\cap}_{i=1}^3 B_i.
}
Let $\tau(j)$ denote the $j^{th}$ element in $\Theta_\mu$ and if $B_4$ occurs, then we let 
$
\rho =Y_{\tau(\nu)}$, otherwise $\rho$ is simply the zero vector.

\subsubsection*{In the event of $B_4$.}
Assume that  event $B_4$ occurs. We  now demonstrate that $\rho$ satisfies properties (i), (ii) and (iii).
Observe that
 $$
 Z_{\tau(i)} = (\rQ_{\cW_\Lambda} - \rQ_{\cW_\Lambda}\rA^* \rP_{\Omega^i, \mathbf{q}^i} \rA \rX \rQ_{\cW_\Lambda})Z_{\tau(i-1)}.
 $$
 Then,
 \spls{
 &\norm{\rQ_{\cW_\Lambda} \rD^* \rP_{\Lambda^c} \sigma - \rQ_{\cW_\Lambda}\rho }_{2} = \norm{Z_{\tau(\nu)}}_2\\
 &= \norm{ \rW^*(\rQ_{\cW_\Lambda} - \rQ_{\cW_\Lambda}\rA^* \rP_{\Omega^i, \mathbf{q}^i} \rA \rX \rQ_{\cW_\Lambda}) Z_{\tau(\nu-1)} }_2\\
 &\leq \sqrt{s}\cdot \nm{\blam^{-1}}_\infty \cdot \norm{\blam\circ\rW^* (\rQ_{\cW_\Lambda} - \rQ_{\cW_\Lambda} A^*\rP_{\Omega^i,\mathbf{q}^i} \rA \rX \rQ_{\cW_\Lambda}) Z_{i-1}}_\infty \\
 &\leq \sqrt{s} \cdot \nm{\blam^{-1}}_\infty\cdot \prod_{i=2}^\nu \alpha_{\tau(i)} \norm{\blam\circ \rW^* Z_1}_\infty\\
 &\leq \sqrt{s} \cdot \nm{\blam^{-1}}_\infty \cdot  \prod_{i=1}^\nu \alpha_{\tau(i)} \norm{\blam\circ \rW^* \rD^* \rP_{\Lambda^c} \sigma}_\infty
 \leq  \nm{\blam^{-1}}_\infty \cdot \frac{\sqrt{s}}{2^\nu} \leq \frac{q}{8} \cdot \min\br{1, \left(c_2 \nm{\rX}_{2\to 2}\right)^{-1}}
 }
 where we have recalled the definition of $\nu$ and also,  by definition of $\blam $, $\norm{\blam\circ \rW^* \rD^* \rP_{\Lambda^c} \sigma}_\infty \leq 1$.
So, condition (i) is satisfied by $\rho$.

For condition (ii),
\spls{
&\norm{( \rD^* \rP_\Lambda )^\dagger \rQ_{\cW_\Lambda}^\perp \rho }_\infty 
\leq \sum_{i=1}^\nu \norm{( \rD^* \rP_\Lambda )^\dagger \rQ_{\cW_\Lambda}^\perp \rA^* \rP_{\Omega^i, \mathbf{q}^i} \rA \rX \rQ_{\cW_\Lambda} Z_{\tau(i-1)} }_\infty\\
&\leq \sum_{i=1}^\nu \beta_{\tau(i)}\norm{\blam\circ \rW^* Z_{\tau(i-1)}}_\infty 
\leq  \sum_{i=1}^\nu \beta_{\tau(i)} \prod_{j=1}^{i-1}\alpha_{\tau(i)} \norm{\blam\circ \rW^*Z_{\tau(1)}}_\infty \\
&\leq  \sum_{i=1}^\nu \beta_{\tau(i)} \prod_{j=1}^{i-1}\alpha_{\tau(i)} \cdot \norm{\blam\circ \rW^* \rD^* \rP_{\Lambda^c}\sigma}_\infty \\
&\leq \frac{1}{8}\left( 1+ \frac{1}{2\log_2^{1/2}(8M\sqrt{s}q^{-1/2})} +\sum_{j=2}^{\nu -1} \frac{1}{2^j}\right) \leq \frac{1}{8}.
} 

To show $\rho$ satisfies condition (iii) in the event of $B_4$, observe that by definition, $\rho = \rA^* \rP_\Omega w$ where $w = \sum_{j=1}^\nu w_j$ with $w_j = \rP_{\Omega^{\tau(j)}, \mathbf{q}^{\tau(j)}} \rA \rX Z_{\tau(j-1)}$.
For each $j=1,\ldots, \nu$,
\eas{
\norm{w_j}_2^2 = \ip{\rP_{\Omega^{\tau(j)}, \mathbf{q}^{\tau(j)}} \rA \rX Z_{\tau(j-1)}}{\rP_{\Omega^{\tau(j)}, \mathbf{q}^{\tau(j)}} \rA \rX Z_{\tau(j-1)}}
\leq \sum_{k=1}^r \left(\frac{1}{q_k^{\tau(j)}}\right)^{2} \norm{\rP_{\Omega^{\tau(j)}_k} \rA \rX Z_{\tau(j-1)}}_2^2
}
and 
\eas{
&\sum_{k=1}^r \left(\frac{1}{q_k^{\tau(j)}}\right)^{2} \norm{\rP_{\Omega^{\tau(j)}_k} \rA \rX Z_{\tau(j-1)}}_2^2
= \sum_{k=1}^r\left(\frac{1}{q_k^{\tau(j)}}\right)^{2} \ip{\rA^*\rP_{\Omega^{\tau(j)}_k} \rA \rX Z_{\tau(j-1)}}{ \rX Z_{\tau(j-1)}}\\
&=\sum_{k=1}^r\left(\frac{1}{q_k^{\tau(j)}}\right)^{2} \ip{\rQ_{\cW_\Lambda}\rA^*\rP_{\Omega^{\tau(j)}_k} \rA \rX Z_{\tau(j-1)}}{ \rX Z_{\tau(j-1)}} \\
&\leq \max_{k=1}^r\br{ \frac{1}{q_k^{\tau(j)}} }  \ip{\rQ_{\cW_\Lambda}\rA^*\rP_{\Omega^{\tau(j)},\mathbf{q}^{\tau(j)}} \rA \rX Z_{\tau(j-1)}}{ \rX Z_{\tau(j-1)}} \\
&\leq \max_{k=1}^r\br{ \frac{1}{q_k^{\tau(j)}} } \left( \norm{Z_{\tau(j-1)}}_2\norm{Z_{\tau(j)}}_2 + \nm{\rX}_2\norm{Z_{\tau(j-1)}}_2^2\right) \\
&\leq \max_{k=1}^r\br{ \frac{1}{q_k^{\tau(j)}} } \left(\norm{\rW^*Z_{\tau(j-1)}}_2\norm{\rW^*Z_{\tau(j)}}_2 + \nm{\rX}_2 \norm{\rW^*Z_{\tau(j-1)}}_2^2 \right)\\
&\leq  \max_{k=1}^r\br{ \frac{1}{q_k^{\tau(j)}} } \cdot s\cdot \left(\norm{\blam^{-1}\circ \blam\circ \rW^* Z_{\tau(j-1)}}_\infty\norm{\blam^{-1}\circ \blam\circ\rW^* Z_{\tau(j)}}_\infty +  \nm{\rX}_2\norm{\blam^{-1}\circ\blam\circ\rW^*Z_{\tau(j-1)}}_\infty^2 \right)\\
&\leq \nm{ \blam^{-1}}_\infty^2\cdot  \nm{\rX}_2\cdot \max_{k=1}^r\br{ \frac{1}{q_k^{\tau(j)}} } \cdot s\cdot (\alpha_{\tau(j)} + 1) \cdot \left(\prod_{i=1}^{j-1}\alpha_{\tau(j)}\right)^2.
}
Thus,
\eas{
\norm{w}_2 \leq \nm{\blam^{-1}}_\infty \cdot\nm{\rX}_2^{1/2}\cdot\sqrt{s}\cdot \sum_{j=1}^\nu
\max_{k=1}^r\br{ \frac{1}{\sqrt{q_k^{\tau(j)}}} }  \cdot\sqrt{\alpha_{\tau(j)} + 1} \cdot \prod_{i=1}^{j-1}\alpha_{\tau(j)} 
}
and by plugging in our choice of parameters and carrying out some  algebraic manipulations (see \cite{adcockbreaking}), we have
$$
\norm{w}_2 \lesssim \nm{\blam^{-1}}_\infty \cdot\nm{\rX}_2^{1/2}\cdot\sqrt{s}\cdot \sqrt{\frac{\log(\gamma^{-1}) + \log_2(8M\sqrt{s}q^{-1})}{\log_2(4M\sqrt{s}q^{-1})} }\cdot \max_{k=1}^r \sqrt{\frac{M_k - M_{k-1}}{m_k}}.
$$
So, to prove this theorem, we need to show that $\bbP(B_4^c)\leq 5\gamma$ and this is true if
$$
\bbP(A_i^c)\leq \gamma, \quad i=1,2, \qquad \bbP(B_j^c) \leq \gamma, \quad j=1,2,3.
$$

\subsubsection*{Bounding the probability that event $B_3$ does not occur.}
We aim to show that $\mathbb{P}(B_3^c) < \gamma$.
We first define
the random variables $X_1, \hdots X_{\mu-2}$ by 
\begin{equation}\label{theX}
X_j = 
\begin{cases}
0 & \Theta_{j+2} \neq \Theta_{j+1},\\
1 & \text{otherwise}.
\end{cases}
\end{equation}
and observe that
\begin{equation}\label{bound_B4}
\mathbb{P}(B_3^c) = \mathbb{P}(|\Theta_{\mu}| < \nu) = \mathbb{P}(X_1+\hdots +X_{\mu-2} > \mu-\nu).
\end{equation}
Suppose that $P$ is such that 
\be{\label{prob_bound}
P\geq \bbP(X_j =1 | X_{l_1}= \ldots= X_{l_g}=1),
} for any $j=1,\ldots, \mu-2$, $l_1,\ldots, l_g\in\br{1,\ldots, \mu-2}$ such that $j\not\in\br{l_1,\ldots, l_g}$. Then,
$$
\bbP(\sum_{i=1}^{\mu-2}X_i \geq \mu-\nu) \leq\binom{\mu-2}{\mu-\nu} P^{\mu-\nu}.
$$
Now let $\{\tilde X_k\}_{k=1}^{\mu-2}$ be independent binary variables taking values $0$ and $1$, such that $\bbP(\tilde X_k = 1) = P$.
Then, since it can be shown that \cite[Lemma 7.14]{adcockbreaking}
$$
\binom{\mu-2}{\mu-\nu} P^{\mu-\nu} \leq \left(\frac{(\mu-2)e}{\mu-\nu}\right)^{\mu-\nu} \bbP(\sum_{i=1}^{\mu-2}\tilde X_i \geq \mu-\nu),
$$
we have that
\be{\label{bound_B4_II}
\bbP(\sum_{i=1}^{\mu-2}X_i \geq \mu-\nu) \leq \left(\frac{(\mu-2)e}{\mu-\nu}\right)^{\mu-\nu} \bbP(\sum_{i=1}^{\mu-2}\tilde X_i \geq \mu-\nu).
}
By the standard Chernoff bound ([Theorem 2.1, equation 2, McDiarmid]) that for $t > 0$,
\begin{equation}\label{tony}
\mathbb{P}\left(\tilde X_1+\hdots + \tilde X_{\mu-2} \geq  (\mu -2)(t + P)\right) \leq e^{-2(\mu-2) t^2}.
\end{equation}
Hence, if we let $t = (\mu-\nu)/(\mu-2) - P$, it follows from (\ref{bound_B4_II}) and (\ref{tony}) that 
$$
\mathbb{P}(B_4^c) \leq e^{-2(\mu-2) t^2 + (\mu-\nu)(\log(\frac{\mu-2}{\mu-\nu})+1)} \leq e^{-2(\mu-2) t^2 + \mu-2}.
$$
Thus, by choosing $P = 1/4$ we get that 
$$
\mathbb{P}(B_3^c) \leq \gamma
$$
whenever $\mu \geq x$ and $x$ is the largest root satisfying 
$$
(x-\mu)\left(\frac{x-\nu}{\mu-2} - \frac{1}{4} \right) - \log(\gamma^{-1/2}) - \frac{x-2}{2} = 0,
$$
so $\mu \geq 8\lceil 3\nu + \log(\gamma^{-1/2})\rceil$ which is satisfied by our choice of $\mu$.

It remains to verify that (\ref{prob_bound}) holds with $P = 1/4$:
Observe that $X_j=1$ whenever
$$\norm{\blam\circ (\rW^* (\rQ_{\cW_\Lambda} - \rQ_{\cW_\Lambda} \rA^*\rP_{\Omega^i, \mathbf{q}^i} \rA \rX \rQ_{\cW_\Lambda}) Z_{i-1}}_\infty \leq \frac{1}{2} \norm{\blam\circ\rW^* Z_{i-1}}_\infty
$$ and
$$\norm{ (\rD^* \rP_\Lambda )^\dagger \rQ_{\cW_\Lambda}^\perp \rA^* \rP_{\Omega^i,\mathbf{\tilde q}} \rA \rX \rQ_{\cW_\Lambda} Z_{i-1}}_\infty \leq \frac{1}{8}\log_2(4\tilde M C_* \sqrt{s}/q) \norm{\blam\circ \rW^* Z_{i-1}}_\infty$$
for $i=j+2$. Thus, by Proposition \ref{prop:A_i} and Proposition \ref{prop:B_i}, $\bbP(X_j=1)\leq \frac{1}{4}$ if for each $i\in\bbN$,
\be{\label{eq:con1}
1\gtrsim \frac{\log(32 \tilde M)}{\log_2\left(4\tilde M C_* \sqrt{s}q^{-1}\right)} \cdot \max_{\norm{\eta}_\infty =1}\sum_{k=1}^r (\tilde q_k^{-1} -1)\cdot \mu(\rP_{\Gamma_k}\rA \rQ_{\cW_\Lambda}^\perp(\rP_\Lambda \rD)^\dagger \rP_{\br{i}}   )^2 \cdot \norm{\rP_{\Gamma_k} \rA \rX \rW \circ\blam^{-1} \cdot \eta}_2^2, 
}
and 
\be{\label{eq:con2}
\tilde q_k \gtrsim  \frac{\log(32\tilde M)}{\log_2\left(4\tilde M\sqrt{s} C_* q^{-1}\right)^2} \cdot \mu(\rP_{\Gamma_k}\rA \rQ_{\cW_\Lambda}^\perp(\rP_\Lambda \rD)^\dagger    ) \cdot \sum_{j=1}^s  \blam_j^{-1} \cdot \mu(\rP_{\Gamma_k} \rA \rX \rW \rP_{\br{j}}), \qquad k=1,\ldots,r 
}
as well as for each $i=1,\ldots,s$,
\be{\label{eq:con3}
1\gtrsim \log(32s) \cdot \max_{\norm{\eta}_\infty =1}\sum_{k=1}^r (\tilde q_k^{-1} -1)\cdot  (\mu(\rP_{\Gamma_k} \rA \rW \rP_{\br{i}})\cdot\blam_i)^2 \cdot \norm{\rP_{\Gamma_k} \rA \rX \rW \circ\blam^{-1} \cdot \eta}_2^2
}
and
\be{\label{eq:con4}
\tilde q_k \gtrsim  \log(32s) \cdot \mu(\rP_{\Gamma_k} \rA \rW \circ\blam) \cdot \sum_{j=1}^s  \mu(\rP_{\Gamma_k} \rA \rX \rW \rP_{\br{j}})\cdot\blam_j^{-1} , \qquad k=1,\ldots,r
}
It now remains to show that the assumptions of this theorem imply (\ref{eq:con1}), (\ref{eq:con2}), (\ref{eq:con3}) and (\ref{eq:con4}).

\subsubsection*{Stage 1:}
We  show that (\ref{eq:con1}) and (\ref{eq:con3}) are satisfied if 
$q_k \gtrsim (\log(s\epsilon^{-1}) +1) \log(q^{-1}\tilde M C_* \sqrt{s}) \cdot \hat q_k $ such that $\br{\hat q_k}_{k=1}^r$ satisfies the following.
 For each $i\in\bbN$,
\be{\label{eq:hat_m_laststep1}
1\gtrsim  \max_{\norm{\eta}_\infty =1}\sum_{k=1}^r (\hat q_k^{-1} -1)\cdot (\mu(\rP_{\Gamma_k}\rA \rQ_{\cW_\Lambda}^\perp(\rP_\Lambda \rD)^\dagger \rP_{\br{i}}   ))^2 \cdot \norm{\rP_{\Gamma_k} \rA \rX \rW \circ\blam^{-1} \cdot \eta}_2^2,  
}
 and for $i=1,\ldots, s$
\be{\label{eq:hat_m_laststep2}
1\gtrsim \max_{\norm{\eta}_\infty =1} \sum_{k=1}^r ( \hat q_k^{-1} -1)\cdot \max_{\norm{\eta}_\infty =1}  \left(\blam_i\cdot \mu(\rP_{\Gamma_k} \rA \rW \rP_{\br{i}})\right)^2 \cdot \norm{\rP_{\Gamma_k} \rA \rX \rW \circ\blam^{-1} \cdot \eta}_2^2.
}
First observe that $(1-q_k^1)\cdots (1-q_k^\mu) = (1-q_k)$ implies that
$q_k^1+q_k^2+\cdots + q_k^\mu \geq q_k$. So, by our choice of $\br{q_k^j}_{j=1}^\mu$, if follows that $2(\mu-2) \tilde q_k \geq q_k$. If (\ref{eq:hat_m_laststep1}) and (\ref{eq:hat_m_laststep2}) are satisfied by $q_k$, then
\eas{
&2\left(8(\lceil 3\log(9C_* M \sqrt{s}/q)+ \log(\gamma^{-1})\rceil) -2\right)\tilde q_k  \geq q_k \\
&\gtrsim \hat q_k (\log(s\epsilon^{-1}) +1) \log(q^{-1}\tilde M C_* \sqrt{s})
\geq \hat{q}_k (\log(s)+1)(\log(q^{-1}\tilde M C_* \sqrt{s})+\log(\epsilon^{-1})).
}
Since $\gamma = \epsilon/6$, it follows that
$$
\tilde q_k \gtrsim \hat q_k(\log(s)+1)
$$
and (\ref{eq:hat_m_laststep1}) implies that for each $i\in\bbN$,
\eas{
1 &\gtrsim (\log(s)+1)\left( \max_{\norm{\eta}_\infty =1}\sum_{k=1}^r (\hat q_k^{-1}(\log(s)+1)^{-1} -(\log(s)+1)^{-1})\cdot \mu_{k,i}^2 \cdot \norm{\rP_{\Gamma_k} \rA \rX \rW \circ\blam^{-1} \cdot \eta}_2^2\right)\\
&\gtrsim
(\log(s)+1)\left( \max_{\norm{\eta}_\infty =1}\sum_{k=1}^r (\tilde q_k^{-1} -1)\cdot \mu_{k,i}^2 \cdot \norm{\rP_{\Gamma_k} \rA \rX \rW \circ\blam^{-1} \cdot \eta}_2^2\right)
}
where $\mu_{k,i}=\mu(\rP_{\Gamma_k}\rA \rQ_{\cW_\Lambda}^\perp(\rP_\Lambda \rD)^\dagger \rP_{\br{i}}   )$, so (\ref{eq:hat_m_laststep1}) implies (\ref{eq:con1}). A similar argument which replaces $\br{\mu_{k,i}}_{i\in\bbN}$ by
$ \br{\blam_i\cdot \mu(\rP_{\Gamma_k} \rA \rW \rP_{\br{i}})}_{i=1}^s$ will show that (\ref{eq:hat_m_laststep2}) implies (\ref{eq:con3}).

\subsubsection*{Stage 2:}

We show that (\ref{eq:con2}) and (\ref{eq:con4}) are satisfied if for each $k=1,\ldots, r$,
\be{\label{eq:m_k_laststep1}
 q_k \gtrsim  (\log(s\epsilon^{-1}) +1) \log(q^{-1}\tilde M C_* \sqrt{s}) \cdot \sum_{j=1}^s \max_i \left(\mu(\rP_{\Gamma_k}\rA \rQ_{\cW_\Lambda}^\perp(\rP_\Lambda \rD)^\dagger \rP_{\br{i}}   )\right) \cdot \mu(\rP_{\Gamma_k} \rA \rX \rW \rP_{\br{j}}) \cdot\blam^{-1}_j, 
 }
 and
\be{\label{eq:m_k_laststep2}
 q_k \gtrsim  (\log(s\epsilon^{-1}) +1) \log(q^{-1}\tilde M C_* \sqrt{s})\cdot \sum_{j=1}^s \left(\max_i \blam_i\cdot \mu(\rP_{\Gamma_k} \rA \rX \rW \rP_{\br{i}}) \right)\cdot \mu(\rP_{\Gamma_k} \rA \rX \rW \rP_{\br{j}})\cdot\blam^{-1}_j. 
}
As in Stage 1, we have that
$$
2\left(8(\lceil 3\log(8C_* M \sqrt{s}/q)+ \log(\gamma^{-1})\rceil) -2\right)\tilde q_k  \geq q_k,
$$
and since $$(\log(s\epsilon^{-1}) +1) \log(q^{-1}\tilde M C_* \sqrt{s}) \geq  (\log(s)+1)(\log(q^{-1}\tilde M C_* \sqrt{s})+\log(\epsilon^{-1}),$$
it follows that (\ref{eq:m_k_laststep2}) implies that
$$
\tilde q_k \gtrsim 
\log(s +1) \cdot \sum_{j=1}^s \left(\max_i \blam_i\cdot \mu(\rP_{\Gamma_k} \rA \rX \rW \rP_{\br{i}}) \right)\cdot \mu(\rP_{\Gamma_k} \rA \rX \rW \rP_{\br{j}})\cdot\blam^{-1}_j. 
$$
which is up to a constant equivalent to (\ref{eq:con4})
In the same way, (\ref{eq:m_k_laststep1}) implies (\ref{eq:con2}).

\subsubsection*{Bounding the probability that one of the events $A_1, A_2, B_1, B_2$ does not occur}

By Proposition \ref{prop:A_i}, for $i=1,2$, $\bbP(A_i^c)\leq \gamma$ if for each $j$,
$$
1\gtrsim \log_2\left(\frac{4\tilde M\sqrt{s}}{q}\right) \cdot \log\left(\frac{4s}{\gamma}\right)\cdot \max_{\norm{\eta}_\infty =1}\sum_{k=1}^r (q_k^{-1}-1) \left(\blam_j\cdot\mu(\rP_{\Gamma_k} \rA \rW \rP_{\br{j}})\right)^2 \norm{\rP_{\Gamma_k}\rA\rX\rW \circ\blam^{-1} \cdot \eta}_2^2, 
$$
and for each $k=1,\ldots,r$
$$
1\gtrsim \log_2^{1/2}\left(\frac{4\tilde M\sqrt{s}}{q}\right) \cdot \log\left(\frac{4s}{\gamma}\right)\cdot \br{q_k^{-1} \cdot  \mu(\rP_{\Gamma_k} \rA \rW \circ\blam)\cdot  \sum_{l=1}^s  \blam^{-1}_l \cdot \mu(\rP_{\Gamma_k} \rA \rX \rW \rP_{\br{l}})}.
$$
By Proposition \ref{prop:B_i} and assumption (\ref{eq:restr_D}), for $i=1,2$,  $\bbP(B_i^c)\leq \gamma$ whenever
$$
1\gtrsim \max_{j=1}^r\br{\log\left(\frac{4\tilde M}{\gamma}\right)\cdot  \max_{\norm{\eta}_\infty =1}\sum_{k=1}^r (q_k^{-1} -1) \left(\mu(\rP_{\Gamma_k}\rA \rQ_{\cW_\Lambda}^\perp(\rP_\Lambda \rD)^\dagger \rP_{\br{j}}   )\right)^2 \norm{\rP_{\Gamma_k}\rA\rX\rW \circ\blam^{-1} \cdot \eta}_2^2},
$$
and for each $k=1,\ldots,r$
$$
1\gtrsim  \log\left(\frac{4\tilde M}{\gamma}\right)\cdot q_k^{-1} \cdot \mu(\rP_{\Gamma_k}\rA \rQ_{\cW_\Lambda}^\perp(\rP_\Lambda \rD)^\dagger    )\cdot  \sum_{l=1}^s \blam^{-1}_l \cdot \mu(\rP_{\Gamma_k} \rA \rX \rW \rP_{\br{l}}).
$$

\end{proof}

\begin{proposition}\label{prop:B_i}
Let $\blam \in\bbR_+^s$.
Let $\cW_\Lambda = \cR(\rW)$ be such that the columns of $\rW$ form an orthonormal set and the dimensions of $\cW_\Lambda =s$. 
Let $\rA \in\cB(\ell^2(\bbN))$ be such that $\nm{\rA}_{2\to 2} \leq 1$. 
 Let $\alpha>0$ and suppose that $\rX\in\cB(\ell^2(\bbN))$ and $M\in\bbN$ are such that
$$\norm{(\rD^*\rP_\Lambda)^\dagger \rQ_{\cW_\Lambda}^\perp \rA^* \rP_{[M]} \rA \rX \rW \circ\blam^{-1}}_\infty \leq \alpha/2.$$
Given any $\xi\in \rQ_{\cW_\Lambda}(\ell^2(\bbN))$, we have that
$$ 
\bbP\left( \norm{  (\rD^*\rP_\Lambda)^\dagger \rQ_{\cW_\Lambda}^\perp  \rA^* \left(q_1^{-1}\rP_{\Omega_1}\oplus \ldots\oplus q_r^{-1} \rP_{\Omega_r}\right) \rA \rX \rQ_{\cW_\Lambda} \xi}_\infty > \alpha \norm{\blam\circ \rW^* \xi}_\infty \right) \leq \gamma$$
if for each $k=1,\ldots, r$
$$
q_k \geq \frac{18}{\alpha^2} \cdot \log\left(\frac{4\tilde{M}}{\gamma}\right)  \cdot 
\mu(\rP_{\Gamma_k}\rA \rQ_{\cW_\Lambda}^\perp(\rP_\Lambda \rD)^\dagger   )
\cdot  \sum_{l=1}^s \blam^{-1}_l \cdot \mu(\rP_{\Gamma_k} \rA \rX \rW \rP_{\br{l}}),
$$
and for each $j\in\bbN$,
$$
1\geq \frac{18}{\alpha} \cdot \log\left(\frac{4\tilde{M}}{\gamma}\right)  \cdot  \max_{\norm{\eta}_\infty = 1}\sum_{k=1}^r (q_k^{-1} -1) \left(\mu(\rP_{\Gamma_k}\rA \rQ_{\cW_\Lambda}^\perp(\rP_\Lambda \rD)^\dagger \rP_{\br{j}}   )\right)^2 \norm{\rP_{\Gamma_k}\rA\rX\rW ( \blam^{-1} \cdot \eta)}_2^2,
$$
where
\bes{
\tilde M =\min\br{i : \max_{k\geq i}  \norm{\rP_{[M]} \rA \rQ_{\cW_\Lambda}^\perp (\rP_\Lambda\rD)^\dagger e_i}_2 \leq \frac{\alpha \cdot q}{  \nm{\blam^{-1}}_\infty \sqrt{s}\cdot \norm{\rX \rA^*\rP_{[M]}}_{2\to 2}}}.
}

\end{proposition}
\begin{proof}
Since we are required to derive conditions under which
$$ 
\bbP\left( \norm{  (\rD^*\rP_\Lambda)^\dagger \rQ_{\cW_\Lambda}^\perp  \rA^* \left(q_1^{-1}\rP_{\Omega_1}\oplus \ldots\oplus q_r^{-1} \rP_{\Omega_r}\right) \rA \rX \rQ_{\cW_\Lambda} (\rW\circ\blam^{-1} \circ\blam\circ \rW^* \xi)}_\infty > \alpha \norm{\blam\circ \rW^* \xi}_\infty \right) \leq \gamma$$
we may assume, without loss of generality that $\norm{\blam\circ \rW^* \xi}_\infty = 1$
and consider for $\tilde{\xi}:= \blam\circ \rW^* \xi$, conditions under which the following hold.
$$ 
\bbP\left( \norm{  (\rD^*\rP_\Lambda)^\dagger \rQ_{\cW_\Lambda}^\perp  \rA^* \left(q_1^{-1}\rP_{\Omega_1}\oplus \ldots\oplus q_r^{-1} \rP_{\Omega_r}\right) \rA \rX \rQ_{\cW_\Lambda} \rW (  \blam^{-1} \cdot \tilde \xi)}_\infty > \alpha  \right) \leq \gamma$$
 For $j=1,\ldots, M$, let $\delta_j$ be random Bernoulli variables such that $\bbP(\delta_j =1) = \tilde q_j$, where $\tilde q_j = q_k$ for $j=M_{k-1}+1,\ldots, M_k$. Observe that
\spls{  
&(\rD^*\rP_\Lambda)^\dagger \rQ_{\cW_\Lambda}^\perp  \rA^* \left(q_1^{-1}\rP_{\Omega_1}\oplus \ldots\oplus q_r^{-1} \rP_{\Omega_r}\right) \rA \rX \rQ_{\cW_\Lambda} \rW (  \blam^{-1}\cdot \tilde \xi)\\
&= \sum_{j=1}^M (\tilde q_j^{-1} \delta_j -1) (\rD^*\rP_\Lambda)^\dagger \rQ_{\cW_\Lambda}^\perp \rA^* (e_j\otimes\overline e_j) \rA \rX \rW ( \blam^{-1}\cdot \tilde \xi) + (\rD^*\rP_\Lambda)^\dagger \rQ_{\cW_\Lambda}^\perp \rA^* \rP_{[M]}  \rA \rX \rW  (\blam^{-1} \cdot \xi)
}
where we have used the facts that $\rQ_{\cW_\Lambda} \rW = \rW$.
For $j=1,\ldots, M$ and $i\in\bbN$, define the random variables
$$
Z_j^i = (\tilde q_j^{-1} \delta_j -1)\ip{(\rD^*\rP_\Lambda)^\dagger \rQ_{\cW_\Lambda}^\perp \rA^* (e_j\otimes\overline e_j) \rA \rX  \rW ( \blam^{-1}\cdot \tilde \xi)}{e_i} 
$$
For $t>0$ and $i\in\bbN$, we will use Theorem \ref{thm:bernstein} to obtain an upper bound for
\be{
\bbP\left(\abs{\sum_{j=1}^M Z_j^i}>t\right). 
}
To bound $\sum_{j=1}^M \bbE\left(\abs{Z_j^i}^2\right)$, first observe that
\spls{
\bbE\left(\abs{Z_j^i}^2\right) &= (\tilde q_j^{-1} -1) \abs{\ip{(\rD^*\rP_\Lambda)^\dagger \rQ_{\cW_\Lambda}^\perp \rA^* (e_j\otimes\overline e_j) \rA \rX \rW (  \blam^{-1}\cdot \tilde \xi)}{e_i}}^2\\
&= (\tilde q_j^{-1} -1) \abs{\ip{\rA^* e_j}{ \rQ_{\cW_\Lambda}^\perp(\rP_\Lambda\rD)^\dagger e_i}}^2 \abs{\ip{ e_j}{\rA \rX  \rW ( \blam^{-1}\cdot \tilde \xi)}}^2
}
Thus,  we have that
\spls{
\sum_{j=1}^M \bbE\left(\abs{Z_j^i}^2\right) &\leq \sum_{k=1}^r (q_k^{-1} -1) \left(\mu(\rP_{\Gamma_k}\rA \rQ_{\cW_\Lambda}^\perp(\rP_\Lambda \rD)^\dagger \rP_{\br{i}}   )\right)^2 \norm{\rP_{\Gamma_k} \rA \rX  \rW ( \blam^{-1}\cdot \tilde \xi)}_2^2\\
&\leq \sup_{\norm{\eta}_\infty =1} \sum_{k=1}^r (q_k^{-1} -1) \left(\mu(\rP_{\Gamma_k}\rA \rQ_{\cW_\Lambda}^\perp(\rP_\Lambda \rD)^\dagger \rP_{\br{i}}   )\right)^2 \norm{\rP_{\Gamma_k} \rA \rX  \rW ( \blam^{-1}\cdot  \eta)}_2^2=:C_{1,i}.
}
To bound $\abs{Z_j^i}$, observe that
\spls{
\abs{Z_j^i} &\leq \max\br{\tilde q_j^{-1}-1,1} \abs{\ip{\rA^* e_j}{ \rQ_{\cW_\Lambda}^\perp(\rP_\Lambda\rD)^\dagger e_i}} \abs{\ip{ e_j}{\rA \rX  \rW (  \blam^{-1}\cdot \tilde \xi)}}\\
&\leq \max_{k=1}^r\br{\max\br{q_k^{-1}-1,1} \cdot \max_{i} \mu(\rP_{\Gamma_k}\rA \rQ_{\cW_\Lambda}^\perp(\rP_\Lambda \rD)^\dagger \rP_{\br{i}}   )\cdot  \sum_{l=1}^s \blam^{-1}_l \mu(\rP_{\Gamma_k} \rA\rX \rW \rP_{\br{l}})} =: C_2.
}
We now let $\Gamma \subset \bbN$ be such that
$$\bbP\left( \sup_{i\in\Gamma}\abs{ \ip{ (\rD^*\rP_\Lambda)^\dagger \rQ_{\cW_\Lambda}^\perp  \rA^* \left(q_1^{-1}\rP_{\Omega_1}\oplus \ldots\oplus q_r^{-1} \rP_{\Omega_r}\right) \rA \rX \rW ( \blam^{-1} \cdot \tilde\xi)}{e_i}} > \alpha  \right) = 0.
$$
Suppose that $\abs{\Gamma^c}<\infty$. Then by Theorem \ref{thm:bernstein} and the union bound,
\spls{
&\bbP\left( \norm{  (\rD^*\rP_\Lambda)^\dagger \rQ_{\cW_\Lambda}^\perp  \rA^* \left(q_1^{-1}\rP_{\Omega_1}\oplus \ldots\oplus q_r^{-1} \rP_{\Omega_r}\right) \rA \rX \rW \circ\blam^{-1} \cdot \tilde\xi}_\infty > \alpha  \right)\\
&\bbP\left( \sup_{i\in\Gamma^c} \abs{ \ip{ (\rD^*\rP_\Lambda)^\dagger \rQ_{\cW_\Lambda}^\perp  \rA^* \left(q_1^{-1}\rP_{\Omega_1}\oplus \ldots\oplus q_r^{-1} \rP_{\Omega_r}\right) \rA \rX \rW \circ\blam^{-1} \cdot \tilde\xi}{e_i}} > \alpha  \right)\\
&\leq 4\abs{\Gamma^c} \max_{j=1}^r \exp\left( - \frac{\alpha^2/16}{C_{1,j} + C_2 \cdot \alpha /(6\sqrt{2})}\right)
}
whenever
$\norm{(\rD^*\rP_\Lambda)^\dagger \rQ_{\cW_\Lambda}^\perp \rA^* \rP_{[M]} \rA \rX \rW \circ\blam^{-1} \cdot \tilde\xi}_\infty \leq \frac{\alpha}{2}$.

To show that $\Gamma^c$ is a finite set, note that
\spls{
 &\abs{ \ip{ (\rD^*\rP_\Lambda)^\dagger \rQ_{\cW_\Lambda}^\perp  \rA^* \left(q_1^{-1}\rP_{\Omega_1}\oplus \ldots\oplus q_r^{-1} \rP_{\Omega_r}\right) \rA \rX \rW ( \blam^{-1} \cdot \tilde\xi)}{e_i}}\\
 &\leq \norm{\blam^{-1} \cdot \tilde \xi}_2\norm{ \rW^* \rQ_{\cW_\Lambda} \rX \rA^*\left(q_1^{-1}\rP_{\Omega_1}\oplus \ldots\oplus q_r^{-1} \rP_{\Omega_r}\right)\rA \rQ_{\cW_\Lambda}^\perp (\rP_\Lambda\rD)^\dagger e_i }_2\\
 &\leq \sqrt{s} \cdot \norm{\blam^{-1}}_\infty\cdot \norm{\rX \rA^*\rP_{[M]}}_{2\to 2} \cdot \max_{k=1}^r q_k^{-1} \cdot \norm{\rP_{[M]} \rA \rQ_{\cW_\Lambda}^\perp (\rP_\Lambda\rD)^\dagger e_i}_2 \to 0
}
as $i\to \infty$. Thus,
$$\Upsilon := \br{i : \norm{\blam^{-1}}_\infty\cdot \sqrt{s}\cdot  \norm{\rX \rA^*\rP_{[M]}}_{2\to 2} \cdot q^{-1} \cdot  \norm{\rP_{[M]} \rA \rQ_{\cW_\Lambda}^\perp (\rP_\Lambda\rD)^\dagger e_i}_2 > \alpha}$$
is a finite set and $\Gamma^c \subset \Upsilon$.
Finally, the observation that $\abs{\Upsilon} \leq \tilde M$  yields the desired result.

\end{proof}

\begin{proposition}\label{prop:A_i}
Let $\xi\in \rQ_{\cW_\Lambda}(\ell^2(\bbN))$ and $\alpha>0$.
Let $\cW_\Lambda = \cR(\rW)$ be such that the columns of $\rW$ form an orthonormal set and the dimensions of $\cW_\Lambda =s$ and let $\blam\in\bbR_+^s$.
Let $\rA \in\cB(\ell^2(\bbN))$ be such that $\nm{\rA}_{2\to 2} \leq 1$. Suppose that $\rX\in\cB(\ell^2(\bbN))$  and $M\in\bbN$ is such that
$$
\nm{\blam\circ \rW^*\rA^* \rP_{[M]}\rA  \rX \rW \circ\blam^{-1} - \rW^* \rW}_{\infty\to \infty} \leq \frac{\alpha}{2}.
$$
Then
$$ 
\bbP\left( \norm{ \blam\circ \rW^* (\rQ_{\cW_\Lambda}  \rA^* \left(q_1^{-1}\rP_{\Omega_1}\oplus \ldots\oplus q_r^{-1} \rP_{\Omega_r}\right) \rA \rX \rQ_{\cW_\Lambda} - \rQ_{\cW_\Lambda} ) \xi}_{\infty} > \alpha \norm{\blam\circ \rW^* \xi}_\infty \right) \leq \gamma$$
if for each $k=1,\ldots, r$,
$$
q_k \geq \frac{18}{\alpha} \cdot \log\left(\frac{4s}{\gamma}\right)
\cdot  \mu(\rP_{\Gamma_k} \rA \rW \circ\blam )\cdot 
\sum_{l=1}^s \blam_l^{-1} \cdot \mu(\rP_{\Gamma_k} \rA\rX \rW \rP_{\br{l}})
$$
and for each $j=1,\ldots, s$
$$
1\geq \frac{18}{\alpha^2}\cdot\log\left(\frac{4s}{\gamma}\right) \cdot
\max_{\norm{\eta}_\infty =1} \sum_{k=1}^r (q_k^{-1}-1)\cdot \mu(\rP_{\Gamma_k} \rA \rW \circ\blam \rP_{\br{j}})^2 \norm{\rP_{\Gamma_k}\rA\rX\rW \circ\blam^{-1} \cdot\eta}_2^2
$$

\end{proposition} 
\begin{proof}
Without loss of generality, assume that  $\nm{\blam\circ \rW^*\xi}_\infty=1$ and let $\tilde \xi = \blam\circ \rW^* \xi$. For $j=1,\ldots, M$, let $\delta_j$ be random Bernoulli variables such that $\bbP(\delta_j =1) = \tilde q_j$, where $\tilde q_j = q_k$ for $j=M_{k-1}+1,\ldots, M_k$. Observe that
\spl{\label{eq:obsv_Ai}
 & \nm{ \blam\circ \rW^* (\rQ_{\cW_\Lambda}  \rA^* \left(q_1^{-1}\rP_{\Omega_1}\oplus \ldots\oplus q_r^{-1} \rP_{\Omega_r}\right) \rA \rX \rQ_{\cW_\Lambda}  -  \rQ_{\cW_\Lambda}) \xi }_{\infty \to \infty}\\
 &= \nm{ \sum_{j=1}^M (\tilde q_j^{-1} \delta_j-1) \blam\circ \rW^*   \rA^* (e_j\otimes\overline e_j) \rA \rX  \rW (\blam^{-1} \cdot \tilde \xi )+ \blam\circ \rW^*   \rA^*  \rP_{[M]}\rA \rX   \rQ_{\cW_\Lambda}  \xi -\blam\circ\rW^* \rQ_{\cW_\Lambda} \xi }_{\infty \to \infty} \\
 &= \nm{ \sum_{j=1}^M (\tilde q_j^{-1} \delta_j-1) \blam\circ \rW^*   \rA^* (e_j\otimes\overline e_j) \rA \rX  \rW ( \blam^{-1} \cdot \tilde \xi)
 }_{\infty \to \infty} + \frac{\alpha}{2}
}
since $\nm{\blam\circ \rW^* \xi}_\infty =$ and $\nm{\blam\circ \rW^*   \rA^*  \rP_{[M]}\rA \rX  \rW \circ\blam^{-1} - \rW^* \rW}_{\infty \to \infty }\leq \frac{\alpha}{2}$.
For $j=1,\ldots, M$ and $i=1,\ldots,s$, define the random variables
$$
Z_j^i = \ip{(\tilde q_j^{-1} \delta_j-1) \blam\circ \rW^*   \rA^* (e_j\otimes\overline e_j) \rA \rX  \rW (\blam^{-1} \cdot \tilde \xi)}{e_i}. 
$$
We will apply Theorem \ref{thm:bernstein} to obtain an upper bound on 
$$
\bbP\left(\abs{\sum_{j=1}^M Z_j^i}>t\right)
$$
for $t>0$.
First,
$$
\bbE\left(\abs{Z^i_j}^2\right) = (\tilde q_j^{-1} -1) \abs{\ip{\rA\rW \circ\blam\circ e_i}{e_j}}^2 \abs{\ip{e_j}{\rA \rX \rW ( \blam^{-1} \cdot \tilde \xi)}}^2
$$
and so,
$$
\sum_{j=1}^M \bbE\left(\abs{Z^i_j}^2\right) \leq \max_{\norm{\eta}_\infty =1}\sum_{k=1}^r (q_k^{-1}-1)\blam_i^2\cdot \left(\mu(\rP_{\Gamma_k} \rA \rW \rP_{\br{i}})\right)^2 \norm{\rP_{\Gamma_k} \rA\rX\rW ( \blam^{-1} \cdot \eta)}_2^2 =: C_{1,i}
$$
Also,
$$
\abs{Z^i_j} \leq  \max_{k=1}^r\br{\max\br{q_k^{-1}-1,1} \cdot \max_{i=1}^s\blam_i\cdot \mu(\rP_{\Gamma_k} \rA \rW \rP_{\br{i}})\cdot  \sum_{l=1}^s \blam^{-1}_l \cdot \mu(\rP_{\Gamma_k} \rA \rX \rW \rP_{\br{l}})} =: C_2.
$$
Finally, by (\ref{eq:obsv_Ai}),  Theorem \ref{thm:bernstein}  and the union bound,
\spls{
&\bbP\left( \norm{\blam\circ \rW^* (\rQ_{\cW_\Lambda}  \rA^* \left(q_1^{-1}\rP_{\Omega_1}\oplus \ldots\oplus q_r^{-1} \rP_{\Omega_r}\right) \rA \rX \rQ_{\cW_\Lambda} - \rQ_{\cW_\Lambda} )\rW (\blam^{-1} \cdot \tilde \xi)}_{\infty} > \alpha  \right) \\
&\leq 4s \max_{i=1}^r \exp\left( - \frac{\alpha^2/16}{C_{1, i} + C_2 \cdot \alpha /(6\sqrt{2})}\right).
}
\end{proof}

\section{Regularization with mixed norms}
In this section, we let $\rA, \rD \in\cB(\ell^2(\bbN))$ and consider the following minimization problem.
\be{\label{eq:prob21_sec}
\min_{z} \norm{\rD z}_{2,1} \text{ subject to } \norm{\rA z - y}_2 \leq \delta
}
where
$$
\norm{x}_{2,1} := \sum_{i\in\bbN} \sqrt{\sum_{w\in \Delta_i} \abs{x_w}^2}
$$
and
$\br{\Delta_i: i\in\bbN}$ are finite disjoint subsets of $\bbN$ such that $\cup_{i\in\bbN}\Delta_i = \bbN$.
 
\begin{lemma}\label{lem:l12norm}
Let $\rA \in \cB(\ell^2(\bbN))$ and suppose that $x,y\in\ell^1(\bbN)$ such that $\max_{i\in\bbN} \norm{\rP_{\Delta_i} y}_2<\infty$ and $\norm{x}_{2,1}<\infty$. Then,
\begin{itemize}
\item[(i)]$$\norm{\rA x}_2 \leq \norm{x}_{2,1} \max_{l\in\bbN} \norm{\rA \rP_{\Delta_l}}_{2\to 2}.
$$
\item[(ii)] $$
\ip{x}{y} \leq \norm{x}_{2,1}\max_{i\in\bbN} \norm{\rP_{\Delta_i} y}_2
$$
\end{itemize}

\end{lemma} 
\prf{
Let $a_{i,j} = \ip{ \rA e_j}{e_i}$.
For (i),
\eas{
\norm{\rA x}_2 = \sqrt{\sum_{i\in\bbN} \abs{\sum_{l\in\bbN} \sum_{j\in\Delta_l} a_{i,j} x_j}^2}
\leq \sqrt{\sum_{i\in\bbN}\abs{\sum_{l\in\bbN} \norm{\rP_{\Delta_l}x} \sqrt{\sum_{j\in\Delta_l} \abs{a_{i,j}}^2}}^2}.
}
Let $z=(z_l)_{l\in\bbN}$ where $z_l = \norm{\rP_{\Delta_l}x}_2$ and let $\rC = (c_{i,l})_{i,l\in\bbN}$ where $c_{i,l} = \sqrt{\sum_{j\in\Delta_l} \abs{a_{i,j}}^2}$.
Then,
\eas{
\norm{\rA x}_2 \leq \norm{\rC z}_2\leq \norm{z}_1 \max_{l\in\bbN} \norm{(c_{l,i})_{i\in\bbN}}_2
= \norm{x}_{2,1} \max_{l\in\bbN} \norm{\rA \rP_{\Delta_l}}_2.
}
For (ii), 
\eas{
\ip{x}{y} = \sum_{i\in\bbN} \sum_{j\in\Delta_i} x_j y_j
\leq \sum_{i\in\bbN} \sqrt{\sum_{j\in\Delta_i} \abs{x_j}^2}\sqrt{\sum_{j\in\Delta_i} \abs{y_j}^2}
\leq \norm{x}_{2,1} \max_{l\in\bbN} \norm{\rA \rP_{\Delta_l}}_2.
}
}

 We show here that robust recovery is implied by the existence of a dual certificate. We also refer the reader to \cite{halmeierl12} for a related result.

 \begin{proposition}[Dual vector for $\ell^{2,1}$ regularization]\label{prop:dual21}
Let $\sigma \in\ell^\infty(\bbN)$ be such that
$$
\rP_{\Delta_i} \sigma =  \frac{\rP_{ \Delta_i} \rD x}{\norm{\rP_{\Delta_i} \rD x}_2}, \quad i\in\bbN.
$$ 
 Let $\Lambda = \cup_{i\in J} \Delta_i$ for some $J\subset \bbN$ and let $\cW_\Lambda\subset \ell^2(\bbN)$ be such that $\cW_\Lambda \supset \cN(  \rP_\Lambda \rD)$.
Let $\Omega := \Omega_1 \cup \cdots \cup \Omega_r \subset \bbN$ be the union of $r$ disjoint subsets and $\br{q_k}_{k=1}^r \in [0,1]^r$.  Define $q= \min_{j=1}^r q_j$ and
$$\rP_{\Omega,\mathbf{ q}} := q_1^{-1}\rP_{\Omega_1}\oplus \ldots \oplus q_r^{-1} \rP_{\Omega_r}, \quad \rP_{\Omega, \sqrt{\mathbf{q}}} := q_1^{-1/2}\rP_{\Omega_1}\oplus \ldots \oplus q_r^{-1/2} \rP_{\Omega_r}.$$
Let $y = \rP_\Omega \rA x + \xi$, with $\norm{\xi}\leq \delta$ and let $\hat x = x + z$ be a $b$-optimal solution to (\ref{eq:prob21_sec}). Let
$c_0,c_1,c_2>0$ such that
$1-  \left(c_0+ 2c_1 c_2 q K  \right)\geq \gamma$,   and suppose that $\rQ_{\cW_\Lambda} \rA^* \rP_{\Omega,\mathbf{ q}} \rA \rQ_{\cW_\Lambda} $ is invertible on $\rQ_{\cW_\Lambda}(\ell^2(\bbN))$ with
\be{\label{eq:inv_cond21}
\norm{(\rQ_{\cW_\Lambda} \rA^* \rP_{\Omega,\mathbf{ q}} \rA \rQ_{\cW_\Lambda} )^{-1}}_{2\to 2} \leq \frac{4K}{3}
}
\be{\label{eq:sym21}
\norm{\rQ_{\cW_\Lambda} \rA^* \rP_{\Omega,\mathbf{q}} \rA \rQ_{\cW_\Lambda}}_{2\to 2} \leq \frac{5}{4} 
}
\be{ \label{eq:col_nm_cond21}
\max_{l\in\bbN} \norm{\rP_{\Omega,\sqrt{\mathbf{q}}} \rA \rQ_{\cW_\Lambda}^\perp (\rP_\Lambda \rD)^\dagger\rP_{\Delta_l}}_{2\to 2} \leq  c_2
}
and that there exists some $\rho = \rA^*\rP_\Omega w$ such that the following holds:
\begin{enumerate}
\item[(i)] $\norm{ \rQ_{\cW_\Lambda} \rD^* \rP_{\Lambda^c}   \sigma- \rQ_{\cW_\Lambda} \rho}_2\leq c_1\cdot q$
\item[(ii)]$\inf_{u\in\cN(\rD^* \rP_{\Lambda})} \max_{i\in\bbN}\norm{\rP_{\Lambda \cap \Delta_i}( \rD^* \rP_\Lambda)^\dagger \rQ_{\cW_\Lambda}^\perp ( \rD^* \rP_{\Lambda^c} \sigma -\rho) -u}_2\leq  c_0$.
\end{enumerate}
Then, 
\spls{
\norm{z} \lesssim \delta\cdot \left( \frac{ K}{\sqrt{q}}  +
C\cdot \left( c_1 \sqrt{q}  K + \norm{w} \right)  \right)
 +
 C\cdot  \norm{\rP_\Lambda \rD x}_{2,1} + C \cdot b.
}
where $C =  \gamma^{-1}\left(c_2  K + \max_{l\in\bbN}\nm{(\rP_\Lambda \rD)^\dagger\rP_{\Lambda_l}}_{2\to 2 }\right) 
$.

\end{proposition}
 \begin{proof}
 By  (\ref{eq:sym21}),
$$\norm{\rQ_{\cW_\Lambda}  \rA^* \rP_{\Omega,\mathbf{q}}}_{2\to 2} \leq \sqrt{\frac{5}{4q}} , \quad \norm{\rQ_{\cW_\Lambda}  \rA^* \rP_{\Omega,\sqrt{\mathbf{q}}}}_{2\to 2} \leq \sqrt{\frac{5}{4}} $$
  Further, observe that  by Lemma \ref{lem:l12norm}
\spl{\label{eq:bound_noise21}
\norm{z}_2 \leq \norm{\rQ_{\cW_\Lambda} z}_2 + \norm{\rQ_{\cW_\Lambda}^\perp z}_2  \leq \norm{\rQ_{\cW_\Lambda} z} +\max_{l\in\bbN} \norm{(  \rP_\Lambda \rD)^\dagger \rP_{\Delta_l}}_2 \norm{  \rP_\Lambda \rD z}_{2,1}.
}
We seek to bound $\norm{\rQ_{\cW_\Lambda} z}_2$ and $\norm{\rP_\Lambda \rD z}_1$. To bound $\norm{\rQ_{\cW_\Lambda} z}_2$, 
\spl{ \label{eq:bound_cosparse_perp21}
\norm{\rQ_{\cW_\Lambda} z}_2 &= \norm{(\rQ_{\cW_\Lambda}  \rA^* \rP_{\Omega,\mathbf{q}} \rA \rQ_{\cW_\Lambda})^{-1} \rQ_{\cW_\Lambda} \rA^* \rP_{\Omega,\mathbf{q}} \rA \rQ_{\cW_\Lambda} z}_2\\
&\leq  \norm{(\rQ_{\cW_\Lambda}  \rA^* \rP_{\Omega,\mathbf{q}} \rA \rQ_{\cW_\Lambda})^{-1}}_2 \norm{\rQ_{\cW_\Lambda}  \rA^* \rP_{\Omega,\mathbf{q}} \rA (z-\rQ_{\cW_\Lambda}^\perp z)}_2\\
&\leq  \frac{4\delta\sqrt{5} K}{3\sqrt{q}}  + \frac{2\sqrt{5} K}{3}\cdot  \max_{l\in\bbN} \norm{\rU \rP_{\Delta_l}}_2\cdot \norm{\rP_\Lambda \rD z}_{2,1}\\
&\leq \delta \cdot \frac{4 K}{\sqrt{q}}  + K\cdot  \max_{l\in\bbN} \norm{\rU \rP_{\Delta_l}}_2\cdot \norm{\rP_\Lambda \rD z}_{2,1}.
}
where $\rU = \rP_{\Omega,\sqrt{\mathbf{q}}} \rA \rQ_{\cW_\Lambda}^\perp (\rP_\Lambda \rD)^\dagger$.
To bound $\norm{ \rP_\Lambda \rD z}_1$, note that
\spls{
\norm{  \rD(x+z)}_{2,1} &= \norm{  \rP_\Lambda \rD(x+z)}_{2,1} +\norm{  \rP_{\Lambda^c} \rD (x+z)}_{2,1} \\
&\geq \norm{  \rP_\Lambda \rD z}_{2,1} - \norm{  \rP_\Lambda \rD x}_{2,1} +\norm{  \rP_{\Lambda^c} \rD (x + z)}_{2,1}. 
}
Observe that
\eas{
&\norm{\rP_{\Lambda^c} \rD (x+z)}_{2,1} = \sum_{i\in\bbN}\norm{ \rP_{\Lambda^c \cap \Delta_i} \rD(x+z)}_2\\
&\geq \Re \sum_{i\in\bbN} \ip{ \rP_{\Lambda^c \cap \Delta_i} \rD(x+z)}{
 \frac{\rP_{\Lambda^c \cap \Delta_i} \rD x}{\norm{\rP_{\Lambda^c \cap \Delta_i} \rD x}_2}}
= \norm{\rP_{\Lambda^c} \rD x}_{2,1} + \Re \ip{\rP_{\Lambda^c} \rD z}{\sigma}
}
Thus,
\spl{\label{eq:dual_consq_min_21}
\norm{  \rD(x+z)}_{2,1}
&\geq \norm{  \rP_\Lambda \rD z}_{2,1} - \norm{  \rP_\Lambda \rD x}_{2,1} + \norm{\rP_{\Lambda^c} \rD x}_{2,1} + \Re \ip{\rP_{\Lambda^c} \rD z}{\sigma}\\
&\geq \norm{  \rP_\Lambda \rD z}_{2,1} - 2\norm{  \rP_\Lambda \rD x}_{2,1}
+ \norm{ \rD x}_{2,1} + \Re \ip{\rP_{\Lambda^c} \rD z}{\sigma}. 
}
and since $b+ \norm{\rD x}_{2,1} \geq \norm{\rD(x+z)}_{2,1} $ by assumption, we have that
\spl{ \label{eq:bound_cosparse21}
\norm{\rP_\Lambda \rD z}_{2,1} \leq b+  2\norm{\rP_\Lambda \rD x}_{2,1}  +\abs{ \ip{\rP_{\Lambda^c} \rD z}{\sigma}}.
}
Using the existence of a dual vector $\rho$ with $\rho = \rA^* \rP_\Omega w$, we have that
\spl{\label{eq:dual_sgn21}
&\abs{ \ip{\rP_{\Lambda^c} \rD  z}{\sigma}} = \abs{ \ip{z}{ \rD^* \rP_{\Lambda^c} \sigma}}\\
&= \abs{ \ip{ z}{\rQ_{\cW_\Lambda} \rD^* \rP_{\Lambda^c} \sigma - \rQ_{\cW_\Lambda} \rho} + \ip{ z}{\rho} + \ip{z}{\rQ_{\cW_\Lambda}^\perp ( \rD^* \rP_{\Lambda^c} \sigma -\rho ) } }\\
&= \abs{ \ip{ z}{\rQ_{\cW_\Lambda} \rD^* \rP_{\Lambda^c}\sigma - \rQ_{\cW_\Lambda} \rho} + \ip{ z}{\rho} + \ip{z}{\rQ_{\cN(\rP_\Lambda \rD)}^\perp \rQ_{\cW_\Lambda}^\perp ( \rD^* \rP_{\Lambda^c} \sigma -\rho ) } }\\
&\leq  \norm{\rQ_{\cW_\Lambda} z}_2 \norm{ \rQ_{\cW_\Lambda} \rD^* \rP_{\Lambda^c} \sigma - \rQ_{\cW_\Lambda} \rho}_2 + \norm{\rP_\Omega \rA x}_2 \norm{w}_2 +\abs{\ip{\rP_\Lambda \rD z}{((\rP_\Lambda \rD)^\dagger)^*\rQ_{\cW_\Lambda}^\perp ( \rD^* \rP_{\Lambda^c} \sigma -\rho)  } }\\
&\leq c_1 q \norm{\rQ_{\cW_\Lambda} z}_2 + 2\delta \norm{w} + \norm{\rP_\Lambda \rD z}_{2,1} \inf_{u\in\cN(\rD^* \rP_{\Lambda})} \max_{i\in\bbN}\norm{\rP_{\Lambda \cap \Delta_i}((\rP_\Lambda \rD)^\dagger)^*\rQ_{\cW_\Lambda}^\perp ( \rD^* \rP_{\Lambda^c} \sigma -\rho) -u}_2.
}
 Thus, by (\ref{eq:bound_cosparse_perp21}) and assumptions,
\spls{
\abs{ \ip{\rP_{\Lambda^c} \rD  z}{\sigma}}
\leq (4c_1 \sqrt{q} K + 2\nm{w}_2) \cdot \delta +  
 \left(c_0+ 2c_1 c_2 q K  \right)\norm{\rP_\Lambda \rD z}_{2,1}.
}
By plugging this back into (\ref{eq:bound_cosparse21}), and recalling that $1-  \left(c_0+ 2c_1 c_2 q K  \right)\geq \gamma$, we have that
$$
\norm{\rP_\Lambda \rD z}_{2,1} \lesssim \gamma^{-1} \left( b+ \norm{\rP_\Lambda \rD x}_{2,1} + \left( c_1 \sqrt{q} K+ \norm{w}_2 \right)\cdot \delta \right).
$$
So, this bound along with  (\ref{eq:bound_noise21}) and (\ref{eq:bound_cosparse_perp}) gives that 
\spls{
\norm{z} \lesssim \delta\cdot \left( \frac{ K}{\sqrt{q}}  +
C\cdot \left( c_1 \sqrt{q} K + \norm{w} \right)  \right)
 +
 C\cdot  \norm{\rP_\Lambda \rD x}_{2,1} + C\cdot b.
}
where $C =  \gamma^{-1}\left(c_2 K + \max_{l\in\bbN}\nm{(\rP_\Lambda \rD)^\dagger\rP_{\Lambda_l}}_{2\to 2 }\right) 
$.

\end{proof}

 \section{Proofs of Main results II}
 \begin{proof}[Proof of Theorem \ref{thm:general}]
 We will proceed by showing that under the assumptions of this theorem, the conditions of Proposition \ref{prop:dual} are satisfied with probability exceeding $1-\epsilon$.
 Let $B$ be a constant such that $B \geq \nm{(\rP_\Lambda \rD)^\dagger}_{1\to 2} $ and let $E_1$ be the event that conditions (\ref{eq:dual_inv_cond}), (\ref{eq:dual_sym}) and (\ref{eq:dual_col_nm_cond}) with $c_2 = B$  are all satisfied. 
 Then, under the assumptions of this theorem, Lemma \ref{lem:verif_propcondns} shows that $\bbP(E_1^c) \leq \epsilon/6$.
 
 Let $E_2$ be the event that there exists $\rho = \rA^* \rP_\Omega w$ such that it satisfies (i) and (ii) of Proposition \ref{prop:dual} where we let $K=\nm{\rX}_{2\to 2}$, $c_1 = \frac{1}{8}\cdot \min\br{1,(c_2   \nm{\rX}_{2\to 2})^{-1}}$ and $c_0 = 1/4$, and
 $\nm{w}_2 \leq \sqrt{\frac{s}{q}} \cdot \nm{\blam^{-1}}_\infty \cdot \sqrt{\frac{\log(p^{-1}) +\log(8M C_* \sqrt{s} q^{-1})}{\log_2(5MC_* \sqrt{s}q^{-1})}}$
where  $C_* = \nm{\blam^{-1}}_\infty \cdot \max\br{1,c_2 \nm{\rX}_{2\to 2}}$.
 Then, since (\ref{eq:restr_Lambda}) holds, it follows that
\eas{
 &\inf_{u\in\cN(\rD^* \rP_\Lambda)} \nm{(\rD^* \rP_\Lambda)^\dagger \rQ_{\cW_\Lambda}^\perp (\rD^* \rP_{\Lambda^c} \sgn(\rD x) - \rho) - u}_\infty\\
 &\leq  \inf_{u\in\cN(\rD^* \rP_\Lambda)} \nm{(\rD^* \rP_\Lambda)^\dagger \rQ_{\cW_\Lambda}^\perp \rD^* \rP_{\Lambda^c} \sgn(\rD x)  - u}_\infty
 +
 \nm{(\rD^* \rP_\Lambda)^\dagger \rQ_{\cW_\Lambda}^\perp \rho}_\infty \leq \frac{1}{16} +  \nm{(\rD^* \rP_\Lambda)^\dagger \rQ_{\cW_\Lambda}^\perp \rho}_\infty,
 }
 and
 by letting $\sigma = \sgn(\rD x)$ in Theorem \ref{thm:verif_dualcertif}, it follows that $\bbP(E_2^c) \leq 5\epsilon/6$.
 Therefore, $\bbP(E_1 \cap E_2 ) \geq 1-\epsilon$ and by plugging in the conclusion of Proposition \ref{prop:dual} with $\gamma = \frac{1}{2}$, the conclusion of this theorem follows.
 \end{proof}

 \begin{proof}[Proof of Theorem \ref{thm:general_21}]
 We first consider the assumptions of Proposition \ref{prop:dual21}  where $\abs{\Delta_i}=2$ for all $i\in\bbN$. Then, (\ref{eq:col_nm_cond21}) holds provided that
 $$
 \max_{l\in\bbN} \norm{\rP_{\Omega,\sqrt{\mathbf{q}}} \rA \rQ_{\cW_\Lambda}^\perp (\rP_\Lambda \rD)^\dagger e_l}_\infty \leq  \frac{c_2}{\sqrt{2}} 
 $$
 and condition (ii) holds provided that
 $\inf_{u\in\cN(\rD^* \rP_{\Lambda})} \norm{(\rP_\Lambda \rD)^\dagger)^*\rQ_{\cW_\Lambda}^\perp ( \rD^* \rP_{\Lambda^c} \sigma -\rho) -u}_\infty\leq \frac{ c_0}{\sqrt{2}}$.
 Therefore, we may now proceed as in the proof of Theorem \ref{thm:general} to show that the conditions of Proposition \ref{prop:dual21}
  hold with probability exceeding $1-\epsilon$. Specifically, we employ Proposition \ref{lem:verif_propcondns} and Theorem \ref{thm:verif_dualcertif} with $\sigma \in\ell^\infty(\bbN)$  such that
$$
\rP_{\Delta_i} \sigma =  \frac{\rP_{ \Delta_i} \rD x}{\norm{\rP_{\Delta_i} \rD x}_2}, \quad i\in\bbN.
$$ 
So, with probability exceeding $1-\epsilon$, the conclusion of Proposition \ref{prop:dual21} hold with $\gamma = 1-\frac{\sqrt{2}}{2}$.
 \end{proof}

 \section{Concluding remarks}
 In practice, when applying total variation regularization for the purpose of subsampling in the recovery of signals from their  Fourier data,  the Fourier samples are chosen in a random and non-uniform manner. Through some numerical examples, this paper demonstrated that this choice cannot be dependent on sparsity alone, but the sparsity structure of the underlying signal. To capture the necessary structure dependence, the notions of fineness and active sparsities were introduced and we derived theoretical statements on how these notions impact the choice of the sampling set $\Omega$. There are two ways in which the work presented here can be extended
 \begin{enumerate}
 \item As discussed in \ref{sec:comparison}, the results of Theorem \ref{thm:main_tv} and Theorem \ref{thm:tv_2d} are not sharp, and it would be desirable to investigate whether the $s$ terms in bounds on the number of samples can be removed so that the structure dependence is reduced to  active sparsity and fineness only.
 \item The general theoretical framework of Theorem \ref{thm:general} can be seen as a generalization of the main theorem from \cite{adcockbreaking} and it would be of interest to analyse the conditions of this theorem in the case of sampling with frames, or reconstructing in Riesz bases.
 \end{enumerate}

\section{Acknowledgements}
 This work was supported by the UK Engineering and Physical Sciences
Research Council (EPSRC) grant EP/H023348/1 for the University of Cambridge
Centre for Doctoral Training, the Cambridge Centre for Analysis. The author would also like to thank Ben Adcock and Anders Hansen for useful discussions.

 \section*{Appendix}
 
 \appendix

\section{A dual certificate result for an unconstrained minimization problem}\label{sec:unconstrained}

We consider the following minimization problem.
\be{\label{eq:prob_unconstr}
\inf_{x : \rD x \in \ell^1(\bbN)} \norm{\rP_\Omega \rA x - y}_2^2 + \alpha\norm{  \rD x}_1, \qquad \alpha>0.
}

\begin{proposition}
Consider the setting of Proposition \ref{prop:dual}. Let $x$ be such that 
$ \norm{\rP_\Omega \rA x - y}_2 \leq \delta$ and $\hat x=x+z$ be a $b$-optimal minimizer of
(\ref{eq:prob_unconstr}). Then
\eas{
&\nm{z}_2 \lesssim C \cdot\left(
(\delta  + \sqrt{b}) \cdot C_q + \frac{\delta^2 + b}{\alpha}
+ \alpha\cdot C_q^2+  \nm{\rP_\Lambda \rD x}_1 \right)
}
where
$$
C = 1+\gamma^{-1} \cdot (K c_2 + \nm{(\rP_\Lambda \rD)^\dagger}_{1\to 2}), \quad
C_q = 
\left( K\cdot  \left(\frac{1}{\sqrt{q}} + c_1 \sqrt{q}\right) + \nm{w}_2\right)
$$
Furthermore, if $\alpha = \sqrt{q}\cdot( \delta+ \sqrt{b})$, then
\be{\label{eq:error_unconstr}
\nm{z}_2 \lesssim C \cdot \left( \max \br{1, \tilde C_q}\cdot \frac{\delta+ \sqrt{b}}{\sqrt{q}}  +  \nm{\rP_\Lambda \rD x}_1
\right)
}
where 
$$
\tilde C_q  = 
\left(  K\cdot  \left(1 + c_1 q\right) + \nm{w}_2\sqrt{q}\right)^2.
$$

\end{proposition}
\begin{remark}
Since the main theorems in this paper are proved by  proving that the conditions of Proposition \ref{prop:dual} hold, this proposition shows that if the conditions of Theorem \ref{thm:general} are satisfied, then any $\xi$ solution to (\ref{eq:prob_unconstr}) with $\alpha = \sqrt{q} \cdot(\delta + \sqrt{b})$ satisfies (\ref{eq:error_unconstr}).
If we let $b=0$, then this affirms the finding in \cite{benning2014phase}, which numerically demonstrates that in order to obtain the error bound in (\ref{eq:error_unconstr}), a linear relation between $\alpha$ and $\delta$ is required and the linear scaling increases as $q$ increases.
\end{remark}

\prf{

Let $q= \min_{j=1}^r q_j$ and
$$\rP_{\Omega,\mathbf{ q}} := q_1^{-1}\rP_{\Omega_1}\oplus \ldots \oplus q_r^{-1} \rP_{\Omega_r}, \quad \rP_{\Omega, \sqrt{\mathbf{q}}} := q_1^{-1/2}\rP_{\Omega_1}\oplus \ldots \oplus q_r^{-1/2} \rP_{\Omega_r}.$$ 
This proof is reuses many of the steps in the proof of Proposition \ref{prop:dual}.
 First note that from (\ref{eq:bound_cosparse_perp}), we have that
\spl{\label{dual_unconstr_Q}
\norm{\rQ_{\cW_\Lambda} z} &\leq \frac{2\sqrt{5 K}}{3\sqrt{q}} \nm{\rP_{\Omega, \mathbf{q}} \rA z} + \frac{2c_2\sqrt{5K}}{3} \nm{\rP_\Lambda \rD z}_1\\
&\leq \frac{2  K}{\sqrt{q}}\nm{\rP_{\Omega, \mathbf{q}} \rA z}_2  + ( 2 K c_2)\norm{\rP_\Lambda \rD z}_1\\
&\leq \frac{2 K}{\sqrt{q}} (\nm{\rP_\Omega \rA \hat x -y}_2  + \nm{\rP_\Omega\rA  x -y}_2) + ( 2 K c_2)\norm{\rP_\Lambda \rD z}_1\\
&\leq (2 K) \cdot \left(\frac{(\delta + \lambda)}{\sqrt{q}} +  c_2\nm{\rP_\Lambda \rD z}_1\right)
}
where we have let $\lambda := \nm{y-\rP_\Omega \rA \hat x}_2$.

By following (\ref{eq:dual_consq_min}), we have that
\eas{
\alpha \nm{\rD \hat x}_1 &\geq \alpha \nm{\rP_\Lambda \rD z}_1 - 2\alpha \nm{\rP_\Lambda\rD x}_1 + \alpha \Re\ip{\rP_{\Lambda^c} z}{\sgn(\rP_{\Lambda^c} \rD x)} + \alpha \nm{\rD x}_1\\
&+\lambda^2 -\lambda^2 +\nm{y-\rP_\Omega \rA \hat x}^2 -\delta^2.
}
Since $\hat x$ is a $b$-optimal solution of (\ref{eq:prob_unconstr}), $\alpha \nm{\rD \hat x}_1 + \lambda^2 \leq  \alpha \nm{\rD x}_1 +\nm{y-\rP_\Omega \rA x}_2^2 + b$. Thus,
\be{\label{eq:prop_unrestr_1}
b+ 2\alpha \nm{\rP_\Lambda\rD x}_1 + \alpha \abs{\ip{\rP_{\Lambda^c} z}{\sgn(\rP_{\Lambda^c} \rD x)}}  + \delta^2 \geq \alpha \nm{\rP_\Lambda \rD z}_1 +\lambda^2.
}
Following the argument in (\ref{eq:dual_sgn}), we have that
\spl{
\abs{\ip{\rP_{\Lambda^c} z}{\sgn(\rP_{\Lambda^c} \rD x)}}  &\leq c_1 q\nm{\rQ_{\cW_\Lambda}z}_2+ \nm{\rP_\Omega\rA z}_2 \nm{w}_2 + c_0\nm{\rP_\Lambda \rD z}_1 \\
&\leq c_1 q\nm{\rQ_{\cW_\Lambda}z}_2+ (\lambda + \delta) \nm{w}_2 + c_0\nm{\rP_\Lambda \rD z}_1 
}
So, plugging in the estimate for $\nm{\rQ_{\cW_\Lambda}z}_2$ from (\ref{dual_unconstr_Q}), we have that
\spl{ \label{eq:prop_unrestr_2}
\abs{\ip{\rP_{\Lambda^c} z}{\sgn(\rP_{\Lambda^c} \rD x_0)}}  \leq (2K c_1\sqrt{q} + \nm{w}_2) (\delta+\lambda) + (2K c_1 c_2 q +c_0)\nm{\rP_\Lambda \rD z}_1 
}
Thus, since $1- (2K c_1 c_2 q +c_0)\geq \gamma$, (\ref{eq:prop_unrestr_1}) and (\ref{eq:prop_unrestr_2}) yields
\be{\label{eq:estim_lamb_from}
b+ 2\alpha \nm{\rP_\Lambda\rD x}_1 + \alpha (2K c_1\sqrt{q} + \nm{w}_2) (\delta+\lambda) + \delta^2 \geq \gamma \cdot\alpha\cdot  \nm{\rP_\Lambda \rD z}_1 +\lambda^2.
}
We now estimate $\lambda$. By (\ref{eq:estim_lamb_from}), we have that
$$
0 \geq \lambda^2 - \alpha (2K c_1\sqrt{q} + \nm{w}_2) \lambda -  (b+ 2\alpha \nm{\rP_\Lambda\rD x}_1 + \alpha (2K c_1\sqrt{q} + \nm{w}_2) \delta + \delta^2).
$$
This implies that the quadratic formula and algebraic manipulations
\be{\label{eq:lambda_upper_bd}
\lambda \leq \alpha(2K c_1\sqrt{q} + \nm{w}_2) + \delta + \sqrt{2\alpha\nm{\rP_\Lambda \rD x}_1} + \sqrt{b}.
}
Also, from (\ref{eq:estim_lamb_from}), we have that
\eas{
\nm{\rP_\Lambda \rD z}_1&\leq \gamma^{-1} \left( \frac{b}{\alpha} + 2 \nm{\rP_\Lambda\rD x}_1 +  (2K c_1\sqrt{q} + \nm{w}_2) (\delta+\lambda) +\frac{ \delta^2}{\alpha}\right)
}
Therefore,
\eas{
\nm{z}_2 &\leq \nm{\rQ_{\cW_\Lambda}z}_2 + \nm{\rQ_{\cW_\Lambda}^\perp z}_2\\
&\leq \nm{\rQ_{\cW_\Lambda}z}_2 + \nm{(\rP_\Lambda \rD)^\dagger}_{1\to 2}\nm{\rP_\Lambda\rD z}_1\\
&\lesssim  K\left(\frac{\delta+\lambda}{\sqrt{q}} + c_2 \nm{\rP_\Lambda \rD z}_1\right) + \nm{(\rP_\Lambda \rD)^\dagger}_{1\to 2}\nm{\rP_\Lambda\rD z}_1\\
&= \frac{ K}{\sqrt{q}}\cdot (\delta + \lambda)
+ ( K c_2 + \nm{(\rP_\Lambda \rD)^\dagger}_{1\to 2})\cdot \nm{\rP_\Lambda \rD z}_1.
}
Let $C = \gamma^{-1} \cdot ( K c_2 + \nm{(\rP_\Lambda \rD)^\dagger}_{1\to 2})$. Plugging in our bound on $ \nm{\rP_\Lambda \rD z}_1$ yields
\eas{
&\nm{z}_2 \lesssim 
\frac{ K}{\sqrt{q}}\cdot (\delta + \lambda)
+ C\cdot  \left( \frac{b}{\alpha}+ \nm{\rP_\Lambda\rD x}_1 +  ( K c_1\sqrt{q} + \nm{w}_2) (\delta+\lambda) +\frac{ \delta^2}{\alpha}\right)\\
&= (\delta +\lambda) \cdot \left(\frac{ K}{\sqrt{q}} + C\cdot  (K c_1\sqrt{q} + \nm{w}_2)\right) +  C\cdot   \left( \frac{b}{\alpha}+ \nm{\rP_\Lambda\rD x}_1 + \frac{\delta^2}{\alpha}\right).
}
From (\ref{eq:lambda_upper_bd}), we have that
\eas{
&\lambda\cdot \left(\frac{ K}{\sqrt{q}} + C \cdot ( K c_1\sqrt{q} + \nm{w}_2)\right)\\
&\lesssim
\alpha \cdot ( K c_1\sqrt{q} + \nm{w}_2)
\cdot \left(\frac{ K}{\sqrt{q}} + C \cdot ( K c_1\sqrt{q} + \nm{w}_2)\right)\\
&+ (\delta + \sqrt{b}) \cdot \left(\frac{ K}{\sqrt{q}} + C \cdot ( K c_1\sqrt{q} + \nm{w}_2)\right)\\
&+ \alpha \cdot \frac{ K^2}{q} +  \nm{\rP_\Lambda \rD x}_1\\
&+\alpha \cdot C \cdot ( K c_1\sqrt{q} + \nm{w}_2)^2 + \nm{\rP_\Lambda \rD x}_1 \cdot C
}

Therefore,
\eas{
&\nm{z}_2 \lesssim 
 (\delta + \sqrt{b}) \cdot \left(\frac{K}{\sqrt{q}} + (C+1)\cdot  ( K c_1\sqrt{q} + \nm{w}_2)\right) +  (C+1)\cdot   \left(\nm{\rP_\Lambda\rD x}_1 + \frac{\delta^2}{\alpha} + \frac{b}{\alpha}\right)\\
 &+ \alpha \cdot (C+1) \cdot \left(( K c_1\sqrt{q} + \nm{w}_2) +  \frac{K}{\sqrt{q}}\right)^2.
}

}

\section{Basic estimates}

\begin{lemma}\label{lem:sin_property}
For $\abs{x} \leq \frac{\pi}{4}$, $\abs{\sin(x)} \geq \frac{\abs{x}}{\sqrt{2}}$. For $\abs{x} \in (\pi/4, \pi/2]$, $\abs{\sin(x)} > \frac{1}{\sqrt{2}}$.
\end{lemma}

\begin{lemma}\label{lem:ft_decay_1d}
Let $\rA$ be the one dimensional unitary discrete Fourier transform on $\bbC^N$ and let $\rD$ be the one dimensional finite differences operator defined in Section \ref{sec:main_tv_results}. Let $x\in\bbC^{N}$. Then, given $k\in\bbZ\setminus\br{0}$ such that $\abs{k}\leq \lceil N/2\rceil$, we have that
$$
\abs{(\rA x)_{k}} \leq   \frac{\sqrt{N} \nm{ \rD x}_1}{\sqrt{2}\abs{k}}.
$$
\end{lemma}
\begin{proof} First, by the choice of $k$, $e^{2\pi i k/N} \neq 1$.
By applying summation by parts to the definition of 
$(\rA x)_{k}$, we obtain
\eas{
(\rA x)_{k} &= \frac{1}{\sqrt{N}}\sum_{j=1}^N x_{j}e^{\frac{2\pi i jk}{N}}\\
&= \frac{1}{\sqrt{N}}  \left(x_{N}\left(\sum_{j=1}^N e^{ \frac{2\pi i j k}{N}}\right) - \sum_{n=1}^{N-1} \left(\sum_{j=1}^n e^{ \frac{2\pi i j k}{N}}\right) (\rD x)_{ n}\right)\\
&= \frac{-1}{\sqrt{N}} \sum_{n=1}^{N-1} \left(\frac{e^{\pi i (n+1) k/N} \sin\left(\frac{\pi  n k}{N}\right)}{\sin\left(\frac{\pi k}{N}\right)}\right) (\rD x)_{ n}.
}
since 
$\sum_{j=1}^N e^{ \frac{2\pi i j k}{N}} = 0$ and
$$
\sum_{j=1}^n e^{ \frac{2\pi i j k}{N}} = \frac{e^{\pi i (n+1) k/N} \sin\left(\frac{\pi  n k}{N}\right)}{\sin\left(\frac{\pi k}{N}\right)}.
$$
By Lemma \ref{lem:sin_property}, for $k$ such that $\abs{k}\leq N/4$,
$$
\frac{1}{\abs{\sin(\pi k/N)}} \leq \begin{cases}
\frac{\sqrt{2}N}{\pi \abs{k}} & \abs{k}\leq N/4\\
\frac{N}{\sqrt{2}\abs{k}} & N/4< \abs{k} \leq N/2.
\end{cases}
$$
Therefore,
$$
\abs{(\rA x)_{k}} \leq \frac{\sqrt{N}\cdot \norm{\rD x}_1}{\sqrt{2}\cdot \abs{k}}.
$$

\end{proof}

\begin{lemma}\label{lem:ft_decay}
Let $\rA$ be the two dimensional unitary discrete Fourier transform on $\bbC^{N\times N}$ and let $\rD_1$, $\rD_2$ be the two dimensional finite differences operators defined in Section \ref{sec:main_tv_results}. Let $x\in\bbC^{N\times N}$. Then, given $(k_1,k_2) \in\bbZ\setminus\br{ (0,0)}$ such that $\abs{k_1}, \abs{k_2} \leq \lceil N/2\rceil$, we have that
$$
\abs{(\rA x)_{k_1,k_2}} \leq   \min\br{\frac{ \nm{ \rD_1 x}_1}{\sqrt{2}\abs{k_1}},\frac{ \nm{ \rD_2 x}_1}{\sqrt{2}\abs{k_2}}}
\leq \frac{ \nm{ \rD x}_1}{\abs{(k_1,k_2)}}
$$
where $\nm{\rD x}_1 = \nm{\rD_1 x}_1 + \nm{\rD_2 x}_1$.
\end{lemma}
\begin{proof}
Without loss of generality, assume that $k_2 \neq 0$. It suffices to show that
\be{\label{eq:to_show_lemftdecay}
\abs{(\rA x)_{k_1,k_2}} \leq \frac{ \nm{ \rD_2 x}_1}{\sqrt{2}\abs{k_2}},
}
since if $k_1 \neq 0$, then by symmetry
\eas{
\abs{(\rA x)_{k_1,k_2}} \leq \frac{ \nm{ \rD_1 x}_1}{\sqrt{2}\abs{k_1}}
}
and 
\bes{
\abs{(\rA x)_{k_1,k_2}} \leq   \min\br{\frac{ \nm{ \rD_1 x}_1}{\sqrt{2}\abs{k_1}},\frac{ \nm{ \rD_2 x}_1}{\sqrt{2}\abs{k_2}}}
\leq \frac{\nm{\rD x}_1}{\abs{(k_1,k_2)}}.
}
If $k_1 = 0$, then clearly,
\eas{
\abs{(\rA x)_{k_1,k_2}} \leq \frac{ \nm{ \rD_2 x}_1}{\sqrt{2}\abs{k_2}}\leq \frac{\nm{\rD x}_1}{\abs{(k_1,k_2)}}.
}
We now proceed to prove (\ref{eq:to_show_lemftdecay}).
By applying summation by parts to the definition of 
$(\rA x)_{k_1,k_2}$, we obtain
\eas{
(\rA x)_{k_1,k_2} &= \frac{1}{N}\sum_{j_1,j_2\in\br{1,\ldots,N}} x_{j_1,j_2}e^{2\pi i \left(\frac{j_1k_1}{N}+ \frac{j_2k_2}{N}\right)}\\
&= \frac{1}{N} \sum_{j_1=1}^N e^{ \frac{2\pi i j_1 k_1}{N}}\left(x_{j_1,N}\left(\sum_{j_2=1}^N e^{ \frac{2\pi i j_2 k_2}{N}}\right) - \sum_{n=1}^{N-1} \left(\sum_{j_2=1}^n e^{ \frac{2\pi i j_2 k_2}{N}}\right) (\rD_2 x)_{j_1, n}\right).
}
Observe that
$\sum_{j_2=1}^N e^{ \frac{2\pi i j_2 k_2}{N}} = 0$ and
$$
\sum_{j_2=1}^n e^{ \frac{2\pi i j_2 k_2}{N}} = \frac{e^{\pi i (n+1) k_2/N} \sin\left(\frac{\pi  n k_2}{N}\right)}{\sin\left(\frac{\pi k_2}{N}\right)}.
$$
Thus, by applying Lemma \ref{lem:sin_property},
\eas{
\abs{(\rA x)_{k_1,k_2}} &\leq  \left(N \abs{\sin\left(\frac{\pi k_2}{N}\right)}\right)^{-1} 
\sum_{j_1=1}^{N} \sum_{n=1}^{N-1} \abs{(\rD_2 x)_{j_1, n} }
 \leq \frac{ \nm{ \rD_2 x}_1}{\sqrt{2}\abs{k_2}}.
}

\end{proof}

 \addcontentsline{toc}{section}{References}
\bibliographystyle{abbrv}
\bibliography{References}

\begin{thebibliography}{10}

\bibitem{adcockbreaking}
B.~Adcock, A.~Hansen, C.~Poon, and B.~Roman.
\newblock Breaking the coherence barrier: A new theory for compressed sensing.
\newblock {\em Preprint}, 2014.

\bibitem{BAACHGSCS}
B.~Adcock and A.~C. Hansen.
\newblock Generalized sampling and infinite-dimensional compressed sensing.
\newblock {\em Technical report NA2011/02, DAMTP, University of Cambridge},
  2011.

\bibitem{benning2014phase}
M.~Benning, L.~Gladden, D.~Holland, C.-B. Sch{\"o}nlieb, and T.~Valkonen.
\newblock Phase reconstruction from velocity-encoded mri measurements--a survey
  of sparsity-promoting variational approaches.
\newblock {\em Journal of Magnetic Resonance}, 238:26--43, 2014.

\bibitem{cai2012image}
J.-F. Cai, B.~Dong, S.~Osher, and Z.~Shen.
\newblock Image restoration: Total variation, wavelet frames, and beyond.
\newblock {\em Journal of the American Mathematical Society}, 25(4):1033--1089,
  2012.

\bibitem{candes2011compressed}
E.~J. Candes, Y.~C. Eldar, D.~Needell, and P.~Randall.
\newblock Compressed sensing with coherent and redundant dictionaries.
\newblock {\em Applied and Computational Harmonic Analysis}, 31(1):59--73,
  2011.

\bibitem{candes2006robust}
E.~J. Cand{\`e}s, J.~Romberg, and T.~Tao.
\newblock Robust uncertainty principles: Exact signal reconstruction from
  highly incomplete frequency information.
\newblock {\em Information Theory, IEEE Transactions on}, 52(2):489--509, 2006.

\bibitem{chambolle2009total}
A.~Chambolle and J.~Darbon.
\newblock On total variation minimization and surface evolution using
  parametric maximum flows.
\newblock {\em International journal of computer vision}, 84(3):288--307, 2009.

\bibitem{chambolle2010continuous}
A.~Chambolle, A.~Giacomini, and L.~Lussardi.
\newblock Continuous limits of discrete perimeters.
\newblock {\em ESAIM: Mathematical Modelling and Numerical Analysis},
  44(02):207--230, 2010.

\bibitem{foucart2013mathematical}
S.~Foucart and H.~Rauhut.
\newblock {\em A mathematical introduction to compressive sensing}.
\newblock Springer, 2013.

\bibitem{goldstein2009split}
T.~Goldstein and S.~Osher.
\newblock The split bregman method for l1-regularized problems.
\newblock {\em SIAM Journal on Imaging Sciences}, 2(2):323--343, 2009.

\bibitem{grasmair2011necessary}
M.~Grasmair, O.~Scherzer, and M.~Haltmeier.
\newblock Necessary and sufficient conditions for linear convergence of
  l1-regularization.
\newblock {\em Communications on Pure and Applied Mathematics}, 64(2):161--182,
  2011.

\bibitem{gross2011recovering}
D.~Gross.
\newblock Recovering low-rank matrices from few coefficients in any basis.
\newblock {\em Information Theory, IEEE Transactions on}, 57(3):1548--1566,
  2011.

\bibitem{halmeierl12}
M.~Haltmeier.
\newblock Block-sparse analysis regularization of ill-posed problems via
  l2,1-minimization.
\newblock In {\em Methods and Models in Automation and Robotics (MMAR), 2013
  18th International Conference on}, pages 520--523, Aug 2013.

\bibitem{haltmeier2013stable}
M.~Haltmeier.
\newblock Stable signal reconstruction via $ \ell^ 1$-minimization in
  redundant, non-tight frames.
\newblock 2013.

\bibitem{kueng2014ripless}
R.~Kueng and D.~Gross.
\newblock Ripless compressed sensing from anisotropic measurements.
\newblock {\em Linear Algebra and its Applications}, 441:110--123, 2014.

\bibitem{Lustig2}
P.~E.~Z. Larson, S.~Hu, M.~Lustig, A.~B. Kerr, S.~J. Nelson, J.~Kurhanewicz,
  J.~M. Pauly, and D.~B. Vigneron.
\newblock Fast dynamic {3D} {MR} spectroscopic imaging with compressed sensing
  and multiband excitation pulses for hyperpolarized 13c studies.
\newblock {\em Magn. Reson. Med.}, 2010.

\bibitem{leveque2007finite}
R.~J. LeVeque.
\newblock {\em Finite difference methods for ordinary and partial differential
  equations: steady-state and time-dependent problems}, volume~98.
\newblock Siam, 2007.

\bibitem{Lustig}
M.~Lustig, D.~L. Donoho, J.~M. Santos, and J.~M. Pauly.
\newblock {Compressed Sensing MRI}.
\newblock {\em IEEE Signal Process. Mag.}, 25(2):72--82, March 2008.

\bibitem{nam2012cosparse}
S.~Nam, M.~E. Davies, M.~Elad, and R.~Gribonval.
\newblock The cosparse analysis model and algorithms.
\newblock {\em Applied and Computational Harmonic Analysis}, 2012.

\bibitem{tv1}
C.~Poon.
\newblock On the role of total variation in compressed sensing - uniform random
  sampling.
\newblock {\em Preprint}, 2014.

\bibitem{puy2011variable}
G.~Puy, P.~Vandergheynst, and Y.~Wiaux.
\newblock On variable density compressive sampling.
\newblock {\em Signal Processing Letters, IEEE}, 18(10):595--598, 2011.

\bibitem{tropp2012user}
J.~A. Tropp.
\newblock User-friendly tail bounds for sums of random matrices.
\newblock {\em Foundations of Computational Mathematics}, 12(4):389--434, 2012.

\bibitem{vaiter2011robust}
S.~Vaiter, G.~Peyr{\'e}, C.~Dossal, and J.~Fadili.
\newblock Robust sparse analysis regularization.
\newblock 2011.

\bibitem{ward2013stable}
R.~Ward and F.~Krahmer.
\newblock Stable and robust sampling strategies for compressive imaging.
\newblock 2013.

\end{thebibliography}
 
\end{document}